\documentclass[oneside,english]{amsart}

\usepackage[T1]{fontenc}
\usepackage[latin9]{inputenc}
\usepackage{refstyle}
\usepackage{textcomp}
\usepackage{amstext}
\usepackage{amsthm}
\usepackage{amssymb}

\makeatletter

\AtBeginDocument{\providecommand\eqref[1]{\ref{eq:#1}}}
\AtBeginDocument{\providecommand\thmref[1]{\ref{thm:#1}}}
\AtBeginDocument{\providecommand\defref[1]{\ref{def:#1}}}
\AtBeginDocument{\providecommand\secref[1]{\ref{sec:#1}}}
\AtBeginDocument{\providecommand\remref[1]{\ref{rem:#1}}}
\AtBeginDocument{\providecommand\propref[1]{\ref{prop:#1}}}
\AtBeginDocument{\providecommand\corref[1]{\ref{cor:#1}}}
\AtBeginDocument{\providecommand\exaref[1]{\ref{exa:#1}}}
\AtBeginDocument{\providecommand\subsecref[1]{\ref{subsec:#1}}}
\AtBeginDocument{\providecommand\lemref[1]{\ref{lem:#1}}}
\RS@ifundefined{subsecref}
  {\newref{subsec}{name = \RSsectxt}}
  {}
\RS@ifundefined{thmref}
  {\def\RSthmtxt{theorem~}\newref{thm}{name = \RSthmtxt}}
  {}
\RS@ifundefined{lemref}
  {\def\RSlemtxt{lemma~}\newref{lem}{name = \RSlemtxt}}
  {}

\numberwithin{equation}{section}
\numberwithin{figure}{section}
\theoremstyle{plain}
\newtheorem{thm}{\protect\theoremname}[section]
  \theoremstyle{definition}
  \newtheorem{defn}[thm]{\protect\definitionname}
  \theoremstyle{remark}
  \newtheorem{rem}[thm]{\protect\remarkname}
  \theoremstyle{definition}
  \newtheorem{example}[thm]{\protect\examplename}
  \theoremstyle{plain}
  \newtheorem{lem}[thm]{\protect\lemmaname}
  \theoremstyle{plain}
  \newtheorem{prop}[thm]{\protect\propositionname}
  \theoremstyle{plain}
  \newtheorem{cor}[thm]{\protect\corollaryname}

\usepackage{amsthm, amssymb, amsmath, color, marvosym,
rotating, graphicx, mathrsfs}

\usepackage{epsfig}
\usepackage{amscd}

\makeatletter
\newcommand*\leftdash{\rotatebox[origin=c]{-45}{$\dabar@\dabar@\dabar@$}}
\newcommand*\rightdash{\rotatebox[origin=c]{45}{$\dabar@\dabar@\dabar@$}}
\makeatother

\makeatother

\usepackage{babel}
  \providecommand{\corollaryname}{Corollary}
  \providecommand{\definitionname}{Definition}
  \providecommand{\examplename}{Example}
  \providecommand{\lemmaname}{Lemma}
  \providecommand{\propositionname}{Proposition}
  \providecommand{\remarkname}{Remark}
\providecommand{\theoremname}{Theorem}

\begin{document}

\title{Differential Operators, Gauges, and Mixed Hodge Modules}

\author{Christopher Dodd}
\begin{abstract}
The purpose of this paper is to develop a new theory of gauges in
mixed characteristic. Namely, let $k$ be a perfect field of characteristic
$p>0$ and $W(k)$ the $p$-typical Witt vectors. Making use of Berthelot's
arithmetic differential operators, we define for a smooth formal scheme
$\mathfrak{X}$ over $W(k)$, a new sheaf of algebras $\widehat{\mathcal{D}}_{\mathfrak{X}}^{(0,1)}$
which can be considered a higher dimensional analogue of the (commutative)
Dieudonne ring. Modules over this sheaf of algebras can be considered
the analogue (over $\mathfrak{X}$) of the gauges of Ekedahl and Fontain-Jannsen.
We show that modules over $\widehat{\mathcal{D}}_{\mathfrak{X}}^{(0,1)}$
admit all of the usual $\mathcal{D}$-module operations, and we prove
a robust generalization of Mazur's theorem in this context. Finally,
we show that an integral form of a mixed Hodge module of geometric
origin admits, after a suitable $p$-adic completion, the structure
of a module over $\widehat{\mathcal{D}}_{\mathfrak{X}}^{(0,1)}$.
This allows us to prove a version of Mazur's theorem for the intersection
cohomology and the ordinary cohomology of an arbitrary quasiprojective
variety defined over a number field. 
\end{abstract}

\maketitle
\tableofcontents{}

\section{Introduction}

In this work, we will develop the technology needed to state and prove
\emph{Mazur's theorem for a mixed Hodge module}. In order to say what
this means, we begin by recalling the original Mazur's theorem. Fix
a perfect field $k$ of positive characteristic; let $W(k)$ denote
the $p$-typical Witt vectors. Let $X$ be a smooth proper scheme
over $k$. To $X$ is attached its crystalline cohomology groups $\mathbb{H}_{crys}^{i}(X)$,
which are finite type $W(k)$-modules; the complex $\mathbb{H}_{crys}^{\cdot}(X)$
has the property that $\mathbb{H}_{crys}^{\cdot}(X)\otimes_{W(k)}^{L}k\tilde{\to}\mathbb{H}_{dR}^{\cdot}(X)$
(the de Rham cohomology of $X$ over $k$). Furthermore, if $\mathfrak{X}$
is a smooth, proper formal scheme over $W(k)$, whose special fibre
is $X$, then there is a canonical isomorphism
\[
\mathbb{H}_{crys}^{i}(X)\tilde{\to}\mathbb{H}_{dR}^{i}(\mathfrak{X})
\]
for any $i$. In particular, the action of the Frobenius endomorphism
on $X$ endows $\mathbb{H}_{dR}^{i}(\mathfrak{X})$ with an endomorphism
$\Phi$ which is semilinear over the Witt-vector Frobenius $F$. It
is known that $\Phi$ becomes an automorphism after inverting $p$;
the ``shape'' of the map $\Phi$ is an interesting invariant of
the pair $(\mathbb{H}_{crys}^{i}(X),\Phi)$. To make this precise,
one attaches, to any $r\in\mathbb{Z}$, the submodule $(\mathbb{H}_{crys}^{i}(X))^{r}=\{m\in\mathbb{H}_{crys}^{i}(X)|\Phi(m)\in p^{r}\mathbb{H}_{crys}^{i}(X)\}$
(the equality takes place in $\mathbb{H}_{crys}^{i}(X)[p^{-1}]$).
Thus we have a decreasing, exhaustive filtration, whose terms measure
how far $\Phi$ is from being an isomorphism. 

On the other hand, the de Rham cohomology of $X$ comes with another
filtration, the Hodge filtration, which comes from the Hodge to de
Rham spectral sequence $E_{1}^{r,s}=\mathbb{H}^{s}(X,\Omega_{X}^{r})\Rightarrow\mathbb{H}_{dR}^{r+s}(X)$.
Then we have the following remarkable
\begin{thm}
\label{thm:(Mazur)}(Mazur) Suppose that, for each $i$, the group
$\mathbb{H}_{crys}^{i}(X)$ is $p$-torsion-free, and that the Hodge
to de Rham spectral sequence of $X$ degenerates at $E_{1}$. Then
the image of the filtration $(\mathbb{H}_{crys}^{i}(X))^{r}$ in $\mathbb{H}_{dR}^{i}(X)$
is the Hodge filtration. 
\end{thm}

This is (the first half of) \cite{key-13}, theorem 3 (in fact, under
slightly weaker hypotheses; compare \cite{key-14} corollary 3.3,
and \cite{key-10}, theorem 8.26). The theorem also includes a similar
description of the conjugate filtration (the filtration coming from
second spectral sequence of hypercohomology) on $\mathbb{H}_{dR}^{i}(X)$;
we will address this as part of the more general theorem 1.2 below.
This result allowed Mazur to prove Katz's conjecture relating the
slopes of $\Phi$ to the Hodge numbers of $X$. 

In the years following \cite{key-13}, it was realized that the theorem
can be profitably rephrased in terms of certain additional structures
on $\mathbb{H}_{crys}^{i}(X)$. Let $A$ be a commutative ring. Denote
by $D(A)$ the commutative ring $A[f,v]/(fv-p)$; put a grading on
this ring by placing $A$ in degree $0$, $f$ in degree $1$, and
$v$ in degree $0$. Then a\emph{ gauge }(over \emph{$A$}) is a graded
module over $D(A)$, ${\displaystyle M=\bigoplus_{i\in\mathbb{Z}}M^{i}}$.
Set ${\displaystyle M^{\infty}:=M/(f-1)\tilde{=}\lim_{\to}M^{i}}$,
and ${\displaystyle M^{-\infty}:=M/(v-1)\tilde{=}\lim_{\to}M^{-i}}$.
One says that $M$ is an $F$-gauge if there is an isomorphism $F^{*}M^{\infty}\tilde{\to}M^{-\infty}$
(c.f. \cite{key-20}, definition 2.1, \cite{key-5}, chapter 1, or
section 2.1 below). 

Then, in the above situation, one associates the $W(k)$- gauge 
\begin{equation}
\mathbb{H}_{\mathcal{G}}^{i}(X):=\bigoplus_{r\in\mathbb{Z}}(\mathbb{H}_{crys}^{i}(X))^{r}\label{eq:Basic-Gauge-defn}
\end{equation}
where $f:(\mathbb{H}_{crys}^{i}(X))^{r}\to(\mathbb{H}_{crys}^{i}(X))^{r+1}$
acts by multiplication by $p$, and $v:(\mathbb{H}_{crys}^{i}(X))^{r}\to(\mathbb{H}_{crys}^{i}(X))^{r-1}$
acts as the inclusion. One has $\mathbb{H}_{\mathcal{G}}^{i}(X)^{\infty}\tilde{=}\mathbb{H}_{crys}^{i}(X)$,
and the isomorphism\linebreak{}
 $F^{*}(\mathbb{H}_{crys}^{i}(X))^{\infty}\to(\mathbb{H}_{crys}^{i}(X))^{-\infty}$
comes from the action of $\Phi$. 

Remarkably, it turns out that there is a reasonable definition of
$\mathbb{H}_{\mathcal{G}}^{i}(X)$ for any $X$, even without the
assumption that each group $\mathbb{H}_{crys}^{i}(X)$ is $p$-torsion-free,
or that the Hodge to de Rham spectral sequence degenerates at $E_{1}$.
To state the result, note that for any gauge $M$ (over any $A$),
$M^{-\infty}$ carries a decreasing filtration defined by $F^{i}(M^{-\infty})=\text{image}(M^{i}\to M^{\infty})$.
Passing to derived categories, we obtain a functor $D(\mathcal{G}(D(A)))\to D((A,F)-\text{mod})$
(here $\mathcal{G}(D(A))$ is the category of gauges, and $D((A,F)-\text{mod})$
is the filtered derived category of $A$); we will denote this functor
$M^{\cdot}\to M^{\cdot,-\infty}$. The analogous construction can
be carried out for $+\infty$ as well using the increasing filtration
$C^{i}(M^{\infty})=\text{image}(M^{i}\to M^{\infty})$. In particular,
if $M^{\cdot}\in D(\mathcal{G}(D(A)))$, then each $H^{i}(M^{\cdot,-\infty})$
and each $H^{i}(M^{\cdot,\infty})$ is a filtered $A$-module.
\begin{thm}
\label{thm:=00005BFJ=00005D} For any smooth $X$ over $k$, there
is a functorially attached complex of $W(k)$-gauges, $\mathbb{H}_{\mathcal{G}}^{\cdot}(X)$,
such that $\mathbb{H}_{\mathcal{G}}^{i}(X)^{\infty}\tilde{=}\mathbb{H}_{crys}^{i}(X)$
for all $i$. Further, there is an $F$-semilinear isomorphism $H^{i}((\mathbb{H}_{\mathcal{G}}^{\cdot}(X)\otimes_{W(k)}^{L}k))^{-\infty}\tilde{\to}(\mathbb{H}_{dR}^{i}(X),F)$
and a linear isomorphism $H^{i}((\mathbb{H}_{\mathcal{G}}^{\cdot}(X)\otimes_{W(k)}^{L}k))^{\infty}\tilde{\to}(\mathbb{H}_{dR}^{i}(X),C)$,
where $F$ and $C$ denote the Hodge and conjugate filtrations, respectively.

When $\mathbb{H}_{crys}^{i}(X)$ is torsion-free for all $i$ and
the Hodge to de Rham spectral sequence degenerates at $E_{1}$, then
this functor agrees with the gauge constructed above in \eqref{Basic-Gauge-defn}. 
\end{thm}

As far as I am aware, the first proof of this theorem appears in Ekedahl's
book \cite{key-20}. This is also the first place that the above notion
of gauge is defined; Ekedahl points out that Fontaine discovered the
notion independantly. Ekedahl's proof relies on deep properties of
the de Rham-Witt complex and on the results of the paper \cite{key-37};
in that paper, it is shown that there is attached to $X$ a complex
inside another category $D^{b}(\mathcal{R}-\text{mod})$ where $\mathcal{R}$
is the so-called Raynaud ring; then, in definition 2.3.1 of \cite{key-20}
Ekehahl constructs a functor from $D^{b}(\mathcal{R}-\text{mod})$
to the derived category $D^{b}(\mathcal{G}(D(A)))$; the composition
of these two functors yeilds the construction of the theorem. Another,
rather different proof of the theorem is given in \cite{key-5}, section
7.

Now let us turn to $\mathcal{D}$-modules and Hodge modules. From
at least the time of Laumon's work (\cite{key-19}), it has been understood
that the filtered complex $\mathbb{H}_{dR}^{\cdot}(X)$ (with its
Hodge filtration) can be understood as an object of filtered $\mathcal{D}$-module
theory. To explain this, let $\mathcal{D}_{X}^{(0)}$ denote the level
zero PD-differential operators on $X$. Then $\mathcal{D}_{X}^{(0)}$
acts on $\mathcal{O}_{X}$, and we have a canonical isomorphism 
\[
\int_{\varphi}\mathcal{O}_{X}[d_{X}]\tilde{\to}\mathbb{H}_{dR}^{\cdot}(X)
\]
where $\varphi$ denotes the map $X\to\text{Spec}(k)$, $d_{X}=\text{dim}(X)$,
and ${\displaystyle \int_{\varphi}}$ is the push-forward for $\mathcal{D}_{X}^{(0)}$-modules.
In addition, $\mathcal{D}_{X}^{(0)}$ comes equipped with a natural
\emph{increasing }filtration, the symbol filtration. Laumon's work\footnote{Strictly speaking, Laumon works in characteristic zero. But the same
formalism works for $\mathcal{D}_{X}^{(0)}$ in positive characteristic;
I'll address this below in the paper} upgrades the push-forward functor to a functor from filtered $\mathcal{D}_{X}^{(0)}$-modules
to filtered $k$-vector spaces; and we have that\emph{ 
\[
\int_{\varphi}\mathcal{O}_{X}[d_{X}]\tilde{\to}(\mathbb{H}_{dR}^{\cdot}(X),F')
\]
}where $F'$ is the Hodge filtration, suitably re-indexed to make
it an increasing filtration. Furthermore, Laumon works in the relative
setting; i.e., he constructs a filtered push-forward for any morphism
$\varphi:X\to Y$ of smooth varieties. 

This leads to the question, of weather the construction of \thmref{=00005BFJ=00005D}
can be understood in terms of some sort of upgrade of filtered $\mathcal{D}$-modules
to a category of graded modules. The main body of this work shows
that, at least when the schemes in question lift to smooth formal
schemes over $W(k)$, the answer is yes\footnote{In fact, the answer is always yes. But we will adress the non-liftable
case in future work}. To state the first result, recall that, in addition to the symbol
filtration, the algebra $\mathcal{D}_{X}^{(0)}$ carries a decreasing
filtration by two sided ideals, the conugate filtration, denoted $\{C^{i}(\mathcal{D}_{X}^{(0)})\}_{i\in\mathbb{Z}}$
(it was first defined in \cite{key-11}, section 3.4, c.f. also \defref{Hodge-and-Con}
below). 
\begin{thm}
\label{thm:D01}Let $\mathfrak{X}$ be a smooth formal scheme over
$W(k)$. Then there is a locally noetherian sheaf of algebras $\widehat{\mathcal{D}}_{\mathfrak{X}}^{(0,1)}$
with the following properties: 

1) ${\displaystyle \widehat{\mathcal{D}}_{\mathfrak{X}}^{(0,1)}=\bigoplus_{i}\widehat{\mathcal{D}}_{\mathfrak{X}}^{(0,1),i}}$
is a graded $D(W(k))$-algebra, and $\widehat{\mathcal{D}}_{\mathfrak{X}}^{(0,1)}/(v-1)\tilde{=}\widehat{\mathcal{D}}_{\mathfrak{X}}^{(0)}$,
while the sheaf $\widehat{\mathcal{D}}_{\mathfrak{X}}^{(0,1)}/(f-1)$
has $p$-adic completion equal to $\widehat{\mathcal{D}}_{\mathfrak{X}}^{(1)}$. 

2) Let $\widehat{\mathcal{D}}_{\mathfrak{X}}^{(0,1)}/p:=\mathcal{D}_{X}^{(0,1)}$,
a graded sheaf of $k$-algebras on $X$. The filtration $\text{im}(\mathcal{D}_{X}^{(0,1),i}\to\mathcal{D}_{X}^{(0)}\tilde{=}\mathcal{D}_{X}^{(0,1)}/(v-1))$
agrees with the conugate filtration on $\mathcal{D}_{X}^{(0)}$. 

3) We have $\mathcal{D}_{X}^{(1)}=\mathcal{D}_{X}^{(0,1)}/(f-1)$.
Consider the filtration $F^{i}(\mathcal{D}_{X}^{(1)})=\text{im}(\mathcal{D}_{X}^{(0,1),i}\to\mathcal{D}_{X}^{(0)}\tilde{=}\mathcal{D}_{X}^{(0,1)}/(f-1))$.
Then filtered modules over $(\mathcal{D}_{X}^{(1)},F^{\cdot})$ are
equivalent to filtered modules over $(\mathcal{D}_{X}^{(0)},F^{\cdot})$
(the symbol filtration on $\mathcal{D}_{X}^{(0)}$). 
\end{thm}

This sheaf of algebras is constructed in \secref{The-Algebra} below;
part $2)$ of the theorem is proved in \remref{Description-of-conjugate},
and part $3)$ is \thmref{Filtered-Frobenius}. This theorem shows
that a graded module over $\mathcal{D}_{X}^{(0,1)}$ is a simultanious
generalization of a conugate-filtered and a Hodge-filtered $\mathcal{D}_{X}^{(0)}$-module. 

The algebra $\widehat{\mathcal{D}}_{\mathfrak{X}}^{(0,1)}$ admits
analogues of all of the usual $\mathcal{D}$-module operations; namely,
tensor product, duality, left-right interchange, and well as push-forward
and pull-back over arbitrary morphisms (between smooth formal schemes).
By construction the sheaf $D(\mathcal{O}_{\mathfrak{X}})=\mathcal{O}_{\mathfrak{X}}[f,v]/(fv-p)$
carrries an action of $\widehat{\mathcal{D}}_{\mathfrak{X}}^{(0,1)}$.
Let $D(\mathcal{G}(\widehat{\mathcal{D}}_{\mathfrak{X}}^{(0,1)}))$
denotes the derived category of graded $\widehat{\mathcal{D}}_{\mathfrak{X}}^{(0,1)}$-modules;
then we have
\begin{thm}
For any morphism $\varphi:\mathfrak{X}\to\mathfrak{Y}$ of smooth
formal schemes we denote the pushforward by ${\displaystyle \int_{\varphi}:D(\mathcal{G}(\widehat{\mathcal{D}}_{\mathfrak{X}}^{(0,1)}))\to D(\mathcal{G}(\widehat{\mathcal{D}}_{\mathfrak{Y}}^{(0,1)}))}$.
If $\varphi$ is proper, then the pushforward takes $D_{coh}^{b}(\mathcal{G}(\widehat{\mathcal{D}}_{\mathfrak{X}}^{(0,1)}))$
to $D_{coh}^{b}(\mathcal{G}(\widehat{\mathcal{D}}_{\mathfrak{Y}}^{(0,1)}))$.
We have ${\displaystyle (\int_{\varphi}\mathcal{M})^{-\infty}\tilde{=}(\int_{\varphi}\mathcal{M}^{-\infty})}$,
where the pushforward on the right is in the category of $\widehat{\mathcal{D}}_{\mathfrak{X}}^{(0)}$-modules.
In particular, if $\mathfrak{Y}$ is $\text{Specf}(W(k))$, then ${\displaystyle {\displaystyle \int_{\varphi}}D(\mathcal{O}_{\mathfrak{X}})}$
is a bounded complex of finite type gauges, and we have isomorphisms
\[
({\displaystyle \int_{\varphi}}D(\mathcal{O}_{\mathfrak{X}}))^{-\infty}[d_{X}]\tilde{=}\mathbb{H}_{dR}^{\cdot}(\mathfrak{X})
\]
and 
\[
({\displaystyle \int_{\varphi}}D(\mathcal{O}_{\mathfrak{X}}))^{\infty}[d_{X}]\tilde{=}F^{*}\mathbb{H}_{dR}^{\cdot}(\mathfrak{X})
\]
where $F$ is the Witt-vector Frobenius. After passing to $k$ we
obtain isomorphisms in the filtered derived category
\[
({\displaystyle \int_{\varphi}}D(\mathcal{O}_{\mathfrak{X}})\otimes_{W(k)}^{L}k)^{-\infty}[d_{X}]\tilde{=}(\mathbb{H}_{dR}^{\cdot}(X),C')
\]
(where $C'$ in the conjugate filtration, appropriately re-indexed
to make it a decreasing filtration), and 
\[
({\displaystyle \int_{\varphi}}D(\mathcal{O}_{\mathfrak{X}})\otimes_{W(k)}^{L}k)^{\infty}[d_{X}]\tilde{=}F^{*}(\mathbb{H}_{dR}^{\cdot}(X),F')
\]
where where $F'$ is the Hodge filtration, suitably re-indexed to
make it an increasing filtration
\end{thm}

This theorem is proved in \secref{Push-Forward} below.

In fact $\widehat{\mathcal{D}}_{\mathfrak{X}}^{(0,1)}$ has many more
favorable properties which are developed extensively in this paper;
including a well-behaved pull-back for arbitrary maps, an internal
tensor product which satisfies the projection formula, and a relative
duality theory; these are sections five through eight below. Simultaneously,
we develop the analogous theory $\mathcal{D}_{X}^{(0,1)}$-modules;
here $\widehat{\mathcal{D}}_{\mathfrak{X}}^{(0,1)}/p$; technically,
we do a little more than that, and develop the theory of $\mathcal{D}_{X}^{(0,1)}$-modules
over smooth varieties which do not have to lift to $W(k)$. The two
theories play off each other nicely- we often use reduction mod $p$
and various versions of Nakayama's lemma to reduce statements about
$\widehat{\mathcal{D}}_{\mathfrak{X}}^{(0,1)}$ to statements about
$\mathcal{D}_{X}^{(0,1)}$; on the other hand, there are always local
lifts of a smooth variety over $k$, so local questions about $\mathcal{D}_{X}^{(0,1)}$
often reduce to questions about $\widehat{\mathcal{D}}_{\mathfrak{X}}^{(0,1)}$.
There is also an interesting and rich theory over the truncated Witt
vectors $W_{n}(k)$, but, given the length of this paper, we will
undertake a detailed study of it in another work. 

We also have a comparison with the gauge constructed in \thmref{=00005BFJ=00005D};
however, we will defer the proof of this result to a later paper.
That is because it seems best to prove it as a consequence of a more
general comparison theorem between the category of gauges constructed
here and the one constructed in \cite{key-5}; and this general statement
is still a work in progress. It also seems that there is a close connection
with the recent works of Drinfeld \cite{key-39} and Bhatt-Lurie \cite{key-40}
via a kind of Koszul duality formalism; again, the details are a work
in progress\footnote{The author has been discussing these topics with Bhargav Bhatt }. 

Now we discuss Mazur's theorem in the relative context. We begin with
the 
\begin{defn}
A module $\mathcal{M}\in\mathcal{G}(\widehat{\mathcal{D}}_{\mathfrak{X}}^{(0,1)})$
is standard if if $\mathcal{M}^{-\infty}$ and $\mathcal{M}^{\infty}$
are $p$-torsion-free, each map $f_{\infty}:\mathcal{M}^{i}\to\mathcal{M}^{\infty}$
is injective; and, finally, there is a $j_{0}\in\mathbb{Z}$ so that
\[
f_{\infty}(\mathcal{M}^{i+j_{0}})=\{m\in\mathcal{M}^{\infty}|p^{i}m\in f_{\infty}(\mathcal{M}^{j_{0}})\}
\]
for all $i\in\mathbb{Z}$. 
\end{defn}

Note that, over $W(k)$, this is a generalization of the construction
of the gauge in \eqref{Basic-Gauge-defn}; with the roles of $f$
and $v$ reversed (this is related to the re-indexing of the Hodge
and conjugate filtrations; c.f. also \remref{basic-equiv} below).
Thus a general version of Mazur's theorem will give conditions on
a complex of gauges which ensure that each cohomology group is standard.
In order to state such a theorem, we need to note that there is a
notion of $F$-gauge in this context, or, to be more precise, a notion
of $F^{-1}$-gauge: 
\begin{defn}
(c.f. \defref{Gauge-Defn!}) Let $F^{*}:\widehat{\mathcal{D}}_{\mathfrak{X}}^{(0)}-\text{mod}\to\widehat{\mathcal{D}}_{\mathfrak{X}}^{(1)}-\text{mod}$
denote Berthelot's Frobenius pullback (c.f. \thmref{Berthelot-Frob}
below for details). Then an $F^{-1}$-gauge over $\mathfrak{X}$ is
an object of $\mathcal{G}(\mathcal{\widehat{D}}_{\mathfrak{X}}^{(0,1)})$
equipped with an isomorphism $F^{*}\mathcal{M}^{-\infty}\tilde{\to}\widehat{\mathcal{M}^{\infty}}$
(here $\widehat{?}$ denotes $p$-adic completion). There is also
a version for complexes in $D(\mathcal{G}(\mathcal{\widehat{D}}_{\mathfrak{X}}^{(0,1)}))$,
namely, an $F^{-1}$-gauge in $D(\mathcal{G}(\mathcal{\widehat{D}}_{\mathfrak{X}}^{(0,1)}))$
is a complex $\mathcal{M}^{\cdot}$ equipped with an isomorphism $F^{*}\mathcal{M}^{\cdot,-\infty}\tilde{\to}\widehat{\mathcal{M}^{\cdot,\infty}}$
(here $\widehat{?}$ denotes the cohomolocial or derived completion,
c.f. \defref{CC} and \propref{Basic-CC-facts} below). 
\end{defn}

We denote by $D_{F^{-1}}(\mathcal{G}(\widehat{\mathcal{D}}_{\mathfrak{X}}^{(0,1)}))$
the category of complexes for which there exists an isomorphism $F^{*}\mathcal{M}^{\cdot,-\infty}\tilde{\to}\widehat{\mathcal{M}^{\cdot,\infty}}$
as above. Then we have the following rather general version of Mazur's
theorem:
\begin{thm}
(c.f. \thmref{F-Mazur}) Let $\mathcal{M}^{\cdot}\in D_{\text{coh},F^{-1}}^{b}(\mathcal{G}(\widehat{\mathcal{D}}_{\mathfrak{X}}^{(0,1)}))$.
Suppose that $\mathcal{H}^{n}(\mathcal{M}^{\cdot})^{-\infty}$ is
$p$-torsion-free for all $n$, and suppose that $\mathcal{H}^{n}((\mathcal{M}^{\cdot}\otimes_{W(k)}^{L}k)\otimes_{D(k)}^{L}k[f])$
is $f$-torsion-free for all $n$. Then $\mathcal{H}^{n}(\mathcal{M}^{\cdot})$
is standard for all $n$. 
\end{thm}

Using the formalism of filtered $\mathcal{D}$-modules one verifies
that the condition that $\mathcal{H}^{n}((\mathcal{M}^{\cdot}\otimes_{W(k)}^{L}k)\otimes_{D(k)}^{L}k[f])$
is $f$-torsion-free for all $n$ is a generalization of the degeneration
of the Hodge-to-de Rham spectral sequence. Therefore this theorem,
along with the previous one, provide a robust generalization of Mazur's
theorem, which allows much more general kinds of coefficients. 

The conditions of the theorem are satisfied in several important cases.
Suppose $R$ is a finitely generated $\mathbb{Z}$-algebra, and suppose
that $X_{R}$ is a smooth $R$ scheme, and let $\varphi:X_{R}\to Y_{R}$
be a proper map. Suppose that $(\mathcal{M}_{R},F)$ is a filtered
coherent $\mathcal{D}_{X_{R}}^{(0)}$-module on $X_{R}$. If the associated
complex filtered $\mathcal{D}$-module, $(\mathcal{M}_{\mathbb{C}},F)$
undergirds a mixed Hodge module, then by Saito's theory the Hodge-to-de
Rham spectral sequence for ${\displaystyle \int_{\varphi}(\mathcal{M}_{\mathbb{C}},F)}$
degenerates at $E_{1}$. Thus the same is true over $R$, after possibly
localizing. Further localization ensures that each ${\displaystyle \mathcal{H}^{i}(\int_{\varphi}(\mathcal{M}_{R},F))}$
is flat over $R$. 

Now suppose we have a map $R\to W(k)$. Let $\varphi:\mathfrak{X}\to\mathfrak{Y}$
denote the formal completion of the base change to $W(k)$. Then the
theorem applies if there exist a $p$-torsion-free gauge $\mathcal{N}$
over $\mathfrak{X}$ such that $\mathcal{N}^{-\infty}\tilde{=}\widehat{\mathcal{M}\otimes_{R}W(k)}$
and $F^{*}(\mathcal{M}_{k},F)\tilde{\to}\mathcal{N}^{\infty}/p$.
By a direct construction, this happens for $\mathcal{M}=\mathcal{O}_{X}$
as well as $\mathcal{M}=j_{\star}\mathcal{O}_{U}$ and $\mathcal{M}=j_{!}\mathcal{O}_{U}$
(where $U\subset X$ is an open inclusion whose compliment is a normal
crossings divisor, and $j_{\star}$ and $j_{!}$ denote the pushforwards
in mixed Hodge module theory). Therefore, by the theorem itself, it
happens when $(\mathcal{M}_{\mathbb{C}},F)$ is itself a Hodge module
``of geometric origin'' (c.f. \corref{Mazur-for-Hodge-1}). In this
paper we give some brief applications of this to the case where $\mathcal{M}_{\mathbb{C}}$
is the local cohomology along some subcheme; but we expect that there
are many more. 

Finally, let's mention that Hodge modules of geometric origin control
both the intersection cohomology and singular cohomology of singular
varieties over $\mathbb{C}$. So we can obtain 
\begin{thm}
\label{thm:Mazur-for-IC-Intro}Let $X_{R}$ be a (possibly singular)
quasiprojective variety over $R$. Then, after possibly localizing
$R,$ there is a filtered complex of $R$-modules $I\mathbb{H}^{\cdot}(X_{R})$,
whose base change to $\mathbb{C}$ yields $I\mathbb{H}^{\cdot}(X_{\mathbb{C}})$,
with its Hodge filtration. Now suppose $R\to W(k)$ for some perfect
field $k$. Then for each $i$, there is a standard gauge $\tilde{I\mathbb{H}}^{i}(X)_{W(k)}$
so that 
\[
\tilde{I\mathbb{H}}^{i}(X)_{W(k)}^{-\infty}\tilde{=}I\mathbb{H}^{\cdot}(X_{R})\otimes_{R}W(k)
\]
 and so that 
\[
\tilde{I\mathbb{H}}^{i}(X)_{W(k)}^{\infty}\tilde{=}F^{*}(I\mathbb{H}^{\cdot}(X_{R})\otimes_{R}W(k))
\]
Under this isomorphism, the Hodge filtration on $\tilde{I\mathbb{H}}^{i}(X)_{W(k)}^{\infty}/p$
agrees with the Frobenius pullback of the image of the Hodge filtration
in $I\mathbb{H}^{\cdot}(X_{R})\otimes_{R}k$. 

The analogous statement holds for the ordinary cohomology of a quasiprojective
variety $X_{R}$, with its Hodge filtration; as well as the compactly
supported cohomology. 
\end{thm}

This is proved in \corref{Mazur-for-IC} and \corref{Mazur-for-Ordinary}below.
As in \cite{key-13} and \cite{key-38}, \cite{key-55} this result
implies that the ``Newton polygon'' lies on or above the ``Hodge
polygon'' for both the ordinary and the intersection cohomology of
quasiprojective varieties, in the circumstances of the above theorem.
We note here that the theorem gives an $F$-semilinear action on the
groups $I\mathbb{H}^{\cdot}(X_{R})\otimes_{R}W(k)[p^{-1}]$, as well
as the ordinary cohomology groups $\mathbb{H}^{\cdot}(X_{R})\otimes_{R}W(k)[p^{-1}]$,
and the compactly supported cohomology as well. This action has already
been constructed as a consequence of the formalism of rigid cohomology
(c.f. \cite{key-80},\cite{key-81}). However, to my knowledge this
``integral'' version of the action has not been considered before. 

\subsection{Plan of the Paper}

The first chapter has two sections. In the first, we quickly review
the theory of gauges over $W(k)$, and in particular give the equivalence
between $F$-guages and $F^{-1}$-guages in this context. In the second,
we give a quick recollection of some generalities on graded modules,
before reviewing and extending (to the case of graded modules) the
very important technical notion of cohomological completeness (also
known as derived completeness). The Nakayama lemma is key here, as
the reduction mod $p$ will be one of our main technical tools for
proving theorems. 

The next chapter introduces $\widehat{\mathcal{D}}_{\mathfrak{X}}^{(0,1)}$,
as well as its analogue $\mathcal{D}_{X}^{(0,1)}$ over a smooth $k$-variety
$X$ (which does not have to lift to a smooth formal scheme), and
performs some basic local calculations. In particular, we prove \corref{Local-coords-over-A=00005Bf,v=00005D},
which provides a local description of $\mathcal{D}_{X}^{(0,1)}$ which
is analogous to the basic descriptions of differential operators ``in
local coordinates'' that one finds in other contexts.

In chapter $4$, we study the categories of graded modules over $\widehat{\mathcal{D}}_{\mathfrak{X}}^{(0,1)}$
and $\mathcal{D}_{X}^{(0,1)}$, importing and generalizing some key
results of \cite{key-5}. We prove the ``abstract'' version of Mazur's
theorem (\thmref{Mazur!}) for a complex of gauges. Then we go on
to introduce the notion of an $F^{-1}$-gauge over $X$ (and $\mathfrak{X}$),
which makes fundamental use of Berthelot's Frobenius descent. We explain
in \thmref{Filtered-Frobenius} how this Frobenius descent interacts
with the natural filtrations coming from the grading on $\mathcal{D}_{X}^{(0,1)}$.
Along the way, we look at the relationship between modules over $\mathcal{D}_{X}^{(0,1)}$
and modules over two important Rees algebras:$\mathcal{R}(\mathcal{D}_{X}^{(0)})$
and $\overline{\mathcal{R}}(\mathcal{D}_{X}^{(0)})$, the Rees algebras
of $\mathcal{D}_{X}^{(0)}$ with respect to the symbol and conjugate
filtrations, respectively. 

Chapters $5$ through $8$ introduce and study the basic $\mathcal{D}$-module
operations in this context: pullback, tensor product, left-right interchange,
pushforward, and duality. Much of this is similar to the story for
algebraic $\mathcal{D}$-modules (as covered in \cite{key-49}, for
instance). For instance, even though $\mathcal{D}_{X}^{(0,1)}$ and
$\widehat{\mathcal{D}}_{\mathfrak{X}}^{(0,1)}$ do not have finite
homological dimension, we show that the pushforward, pullback, and
duality functors do have finite homological dimension. As usual, the
study of the pushforward (chapter $7$ below) is the most involved,
and we spend some time exploring the relationship with the pushforwards
for $\mathcal{R}(\mathcal{D}_{X}^{(0)})$ and $\overline{\mathcal{R}}(\mathcal{D}_{X}^{(0)})$,
respectively; these admit descriptions in terms of the more standard
filtered pushforwards of $\mathcal{D}$-modules. 

Finally, in the last chapter we put everything together and prove
Mazur's theorem for a Hodge module of geometric origin; this uses,
essentially, all of the theory built in the previous sections. In
addition to the applications explained in the introduction, we give
some applications to the theory of the Hodge filtration on the local
cohomology of a subcheme of a smooth complex variety. 

There is one appendix to the paper- in which we prove a technical
result useful for constructing the gauge $j_{\star}(D(\mathcal{O}))$,
the pushforward of the trivial gauge over a normal crossings divisor. 

\subsection{Notations and Conventions}

Let us introduce some basic notations which are used throughout the
paper. For any ring (or sheaf of rings) $\mathcal{R}$, we will denote
by $D(\mathcal{R})$ the graded ring in which $\mathcal{R}$ has degree
$0$, $f$ has degree $1$, $v$ has degree $-1$, and $fv=p$. The
symbol $k$ will always denote a perfect field of positive characteristic,
and $W(k)$ the $p$-typical Witt vectors. Letters $X$, $Y$,$Z$
will denote smooth varieties over $k$, while $\mathfrak{X}$,$\mathfrak{Y}$,$\mathfrak{Z}$
will denote smooth formal schemes over $W(k)$. When working with
formal schemes, we let $\Omega_{\mathfrak{X}}^{1}$ denote the sheaf
of continuous differentials (over $W(k)$), and $\mathcal{T}_{\mathcal{X}}$
denote the continuous $W(k)$-linear derivations; we set $\Omega_{\mathfrak{X}}^{i}=\bigwedge^{i}\Omega_{\mathfrak{X}}^{1}$
and $\mathcal{T}_{\mathfrak{X}}^{i}=\bigwedge^{i}\mathcal{T}_{\mathfrak{X}}$. 

We denote by $X^{(i)}$ the $i$th Frobenius twist of $X$; i.e.,
the scheme $X\times_{\text{Spec}(k)}\text{Spec}(k)$, where $k$ map
to $k$ via $F^{i}$. Since $k$ is perfect, the natural map $\sigma:X^{(i)}\to X$
is an isomorphism. On the other hand, the relative Frobenius $X\to X^{(i)}$
is a bijection on topological spaces, which allows us to identify
$\mathcal{O}_{X^{(i)}}\tilde{=}\mathcal{O}_{X}^{p^{i}}$; we shall
tacitly use this below. 

Now we introduce some conventions on differential operators. If $\mathfrak{X}$
is a smooth formal scheme over $W(k)$, then for each $i\geq0$ we
have Berthelot's ring of differential operators of level $i$, $\widehat{\mathcal{D}}_{\mathfrak{X}}^{(i)}$,
introduced in \cite{key-1} This is a $p$-adically complete, locally
noetherian sheaf of rings on $\mathfrak{X}$. In general, this sheaf
is somewhat complicated to define, but when $\mathfrak{X}=\text{Specf}(\mathcal{A})$
is affine and admits local coordinates\footnote{i.e., $\Gamma(\Omega_{\mathfrak{X}}^{1})$ is a free module over $\mathcal{A}$}
one has the following description of its global sections: let $D_{\mathcal{A}}^{(\infty)}$
denote the subring of $\text{End}_{W(k)}(\mathcal{A})$ consisting
of the the finite order, continuous differential operators on $\mathcal{A}$.
Define $D_{\mathcal{A}}^{(i)}\subset D_{\mathcal{A}}^{(\infty)}$
to be the subring generated by differential operators of level $\leq p^{i}$.
Then we have 
\[
\Gamma(\widehat{\mathcal{D}}_{\mathfrak{X}}^{(i)})=\widehat{D_{\mathcal{A}}^{(i)}}
\]
where $\widehat{?}$ stands for $p$-adic completion. For each $i\geq0$
there is a natural, injective map $\widehat{\mathcal{D}}_{\mathfrak{X}}^{(i)}\to\widehat{\mathcal{D}}_{\mathfrak{X}}^{(i+1)}$;
when $\mathfrak{X}=\text{Specf}(\mathcal{A})$ is as above it is given
by the $p$-adic completion of the tautological inclusion $D_{\mathcal{A}}^{(i)}\subset D_{\mathcal{A}}^{(i+1)}$. 

Similarly, we have the sheaves of algebras $\mathcal{D}_{X}^{(i)}$
when $X$ is smooth over $k$. In the case $i=0$, this is simply
the usual sheaf of pd-differential operators on $X$ (c.f. \cite{key-10}).
This sheaf can be rather rapidly defined (as in \cite{key-3} chapter
1, though there they are called crystalline differential operators)
as the enveloping algebroid of the tangent sheaf $\mathcal{T}_{X}$. 

Finally let us mention that we will be often working with derived
categories of graded modules in this work. In that context, the symbol
$[i]$ denotes a shift in homological degree, while $(i)$ denotes
a shift in the grading degree. 

\subsection{Acknowledgements}

I would like to thank Mircea Musta\c{t}\u{a} and Bhargav Bhatt for
looking at an early version of the manuscript, and for many helpful
discussions thereafter. 

\section{Preliminaries}

\subsection{Gauges over $W(k)$}

In this section we set some basic notation and terminology; all of
which is essentially taken from the paper \cite{key-5}. Let $k$
be a perfect field of characteristic $p>0$; and let $W(k)$ be the
$p$-typical Witt vectors. Let $S$ be a noetherian $W(k)$-algebra.
We recall from \cite{key-5} (also \cite{key-20}) that a gauge over
$S$ is a graded module ${\displaystyle M=\bigoplus_{i=\infty}^{\infty}M^{i}}$
over the graded ring $D(S)$ where, (as always) we suppose $\text{deg}(f)=1$,
$\text{deg}(v)=-1$, and $fv=p$. A morphism of gauges is a morphism
in the category of graded modules. 

If $M$ is a gauge, we denote the resulting multiplication maps by
$f:M^{i}\to M^{i+1}$ and $v:M^{i}\to M^{i-1}$ for all $i$. 

As explained in \cite{key-5}, lemma 1.1.1, such a module is finitely
generated over $R$ iff each $M^{i}$ is finite over $S$ and the
maps $f:M^{r}\to M^{r+1}$ and $v:M^{-r}\to M^{-r-1}$ are isomorphisms
for $r>>0$. It follows that in this case the map $v:M^{r}\to M^{r-1}$
is $p\cdot$ for $r>>0$, and $f:M^{-r}\to M^{-r+1}$ is $p\cdot$
for $r>>0$. In the terminology of \cite{key-5}, such a gauge is
\emph{concentrated in a finite interval}. 
\begin{defn}
\label{def:endpoints} Let $M$ be a gauge. 

1) Set ${\displaystyle M^{\infty}:=M/(f-1)M\tilde{\to}\lim_{r\to\infty}M^{r}}$
and ${\displaystyle M^{-\infty}:=M/(v-1)M\tilde{\to}\lim_{r\to-\infty}M^{r}}$. 

2) For each $i$, denote by $f_{\infty}:M^{i}\to M^{\infty}$ and
$v_{-\infty}:M^{i}\to M^{-\infty}$ the induced maps. 

3) Define $F^{i}(M^{\infty}):=\text{image}(M^{i}\to M^{\infty})$
and $C^{i}(M^{-\infty}):=\text{image}(M^{i}\to M^{-\infty})$. In
particular, $F^{i}$ is an increasing filtration on $M^{\infty}$
and $C^{i}$ is a decreasing filtration on $M^{-\infty}$. Clearly
any morphism of gauges $M\to N$ induces morphisms of filtered modules
$(M^{\infty},F^{\cdot})\to(N^{\infty},F^{\cdot})$ and $(M^{-\infty},C^{\cdot})\to(N^{-\infty},C^{\cdot})$. 
\end{defn}

If $M$ is finitely generated we see that $M^{r}\tilde{=}M^{\infty}$
and $M^{-r}\tilde{=}M^{-\infty}$ for all $r>>0$. 

Many gauges arising in examples posses an additional piece of structure-
a Frobenius semi-linear isomorphism from $M^{\infty}$ to $M^{-\infty}$.
So let us now suppose that $S$ is equipped with an endomorphism $F$
which extends the Frobenius on $W(k)$. 
\begin{defn}
\label{def:F-gauge} (\cite{key-5}, section 1.4) An $F$-gauge is
a gauge $M$ equipped with an isomorphism $\varphi:F^{*}M^{\infty}\tilde{\to}M^{-\infty}$.
A morphism of $F$-gauges is required to respect the isomorphism $\varphi$.
More precisely, given a morphism $G:M\to N$, it induces $G^{\infty}:M^{\infty}\to N^{\infty}$
and $G^{-\infty}:M^{-\infty}\to N^{-\infty}$, and we demand $\varphi\circ F^{*}G^{\infty}=G^{\infty}\circ\varphi$.
This makes the category of $F$-gauges into an additive category,
which is abelian if $F^{*}$ is an exact functor. 
\end{defn}

Now suppose in addition that $F:S\to S$ is an isomorphism. Then: 
\begin{rem}
\label{rem:basic-equiv}There is an equivalence of categories from
$F$-gauges to $F^{-1}$-gauges; namely, send $M$ to the gauge $N$
where $N^{i}=M^{-i}$, $f:N^{i}\to N^{i+1}$ is defined to be $v:M^{-i}\to M^{-i-1}$
, $v:N^{i}\to N^{i-1}$ is defined to be $f:M^{-i}\to M^{-i+1}$.
Then $M^{\infty}=N^{-\infty}$ , $M^{-\infty}=N^{-\infty}$, and the
isomorphism $\varphi:F^{*}M^{\infty}\tilde{\to}M^{-\infty}$ yields
an isomorphism $\psi^{-1}:F^{*}N^{-\infty}\tilde{\to}N^{\infty}$;
which is equivalent to giving an isomorphism $\psi:(F^{-1})^{*}N^{\infty}\tilde{\to}N^{-\infty}$. 
\end{rem}

Finally, we want to quickly review an important construction of gauges.
We suppose here that $S=W(k)$; equipped with its Frobenius automorphism
$F$. We use the same letter $F$ to denote the induced automorphism
of the field $B=W(k)[p^{-1}]$. We will explain how gauges arrive
from lattices of $B$-vector spaces: 
\begin{example}
\label{exa:BasicGaugeConstruction}Let $D$ be a finite dimensional
$B$-vector space, and let $M$ and $N$ be two lattices (i.e., finite
free $W(k)$-modules which span $D$) in $D$. To this situation we
may attach a gauge over $W(k)$ as follows: for all $i\in\mathbb{Z}$
define
\[
M^{i}=\{m\in M|p^{i}m\in N\}
\]
We let $f:M^{i}\to M^{i+1}$ be the inclusion, and $v:M^{i}\to M^{i-1}$
be the multiplication by $p$. For $i>>0$ we have $p^{i}M\subset N$
and so $M^{i}=M$ for all such $i$. For $i<<0$ we have $p^{-i}N\subset M$
and so $M^{i}=p^{-i}N\tilde{=}N$ for such $i$. In particular we
obtain $M^{-\infty}\tilde{=}N$ and $M^{\infty}\tilde{=}M$. This
is evidently a finite-type gauge over $W(k)$. Now suppose that there
is an $F$-semi-linear automorphism $\Phi:D\to D$ so that $M=\Phi(N)$.
Then the previous construction gives an $F^{-1}$ gauge via the isomorphism
$\Phi:N=M^{-\infty}\to M^{\infty}=M$. 
\end{example}

\begin{rem}
\label{rem:=00005BFJ=00005D-standard}In \cite{key-5}, section 2.2,
there is associated an $F$-gauge to a finite dimensional $B$ vector
space $D$, equipped with lattice $M\subset D$ and a semi-linear
automorphism $\Phi:D\to D$. We recall that their construction is 
\end{rem}

\[
M^{i}=\{m\in M|\Phi(m)\in p^{i}M\}=\{m\in M|m\in p^{i}\Phi^{-1}(M)\}
\]
for all $i\in\mathbb{Z}$. In this instance $f:M^{i}\to M^{i+1}$
is the multiplication by $p$, and $v:M^{i}\to M^{i-1}$ is the inclusion.
If we set $N=\Phi^{-1}(M)$ then this is exactly the $F$-gauge which
corresponds to the $F^{-1}$ gauge constructed in \exaref{BasicGaugeConstruction}
above, via the equivalence of categories of \remref{basic-equiv}. 

In \cite{key-5} this construction is referred to as the standard
construction of gauges. We will generalize this below in \subsecref{Standard}. 

\subsection{Cohomological Completion of Graded Modules}

In this section we give some generalities on sheaves of graded modules.
Throughout this section, we let $X$ be a noetherian topological space
and $\tilde{\mathcal{R}}=\bigoplus_{i\in\mathbb{Z}}\tilde{\mathcal{R}}^{i}$
a $\mathbb{Z}$-graded sheaf of rings on $X$. The noetherian hypothesis
ensures that, for each open subset $U\subset X$, the functor $\mathcal{F}\to\mathcal{F}(U)$
respects direct sums; although perhaps not strictly necessary, it
simplifies the discussion of graded sheaves (and it always applies
in this paper). Denote $\tilde{\mathcal{R}}^{0}=\mathcal{R}$, a sheaf
of rings on $X$. 

Let $\mathcal{G}(\tilde{\mathcal{R}})$ denote the category of graded
sheaves of modules over $\tilde{\mathcal{R}}$. This is a Grothendieck
abelian category; the direct sum is given by the usual direct sum
of sheaves. To construct the product of sheaves $\{\mathcal{M}_{i}\}_{i\in I}$,
one takes the sheafification of the pre-sheaf of local sections of
the form $(m_{i})_{i\in I}$ for which there is a bound on the degree;
i.e. $-N\leq\text{deg}(m_{i})\le N$ for a fixed $N\in\mathbb{N}$
and all $i\in I$. Since $X$ is a noetherian space, this pre-sheaf
is actually already a sheaf. 

It follows formally that $\mathcal{G}(\tilde{\mathcal{R}})$ has enough
injectives; this can also be proved in the traditional way by constructing
enough injective in the category of modules over a graded ring and
then noting that the sheaf ${\displaystyle \prod_{x\in X}\mathcal{I}_{x}}$
is injective if $\mathcal{I}_{x}$ is an injective object in the category
of graded $\tilde{\mathcal{R}}_{x}$-modules. We note that an injective
in $\mathcal{G}(\tilde{\mathcal{R}})$ might not be an injective $\mathcal{\tilde{R}}$-module.
However, from the previous remark it follows that any injective in
$\mathcal{G}(\tilde{\mathcal{R}})$ is a summand of a sheaf of the
form $\prod_{x\in X}\mathcal{I}_{x}$; as such sheaves are clearly
flasque it follows that any injective in $\mathcal{G}(\tilde{\mathcal{R}})$
is flasque.

For each $i\in\mathbb{Z}$ we have the exact functor $\mathcal{M}\to\mathcal{M}^{i}$
which takes $\mathcal{G}(\tilde{\mathcal{R}})\to\mathcal{R}-\text{mod}$;
the direct sum of all of these functors is isomorphic to the identity
(on the underlying sheaves of $\mathcal{R}$-modules). Note that the
functor $\mathcal{M}\to\mathcal{M}^{0}$ admits the left adjoint $\mathcal{N}\to\tilde{\mathcal{R}}\otimes_{\mathcal{R}}\mathcal{N}$. 

Let $D(\mathcal{G}(\tilde{\mathcal{R}}))$ denote the (unbounded)
derived category of $\mathcal{G}(\tilde{\mathcal{R}})$. Then the
exact functor $\mathcal{M}\to\mathcal{M}^{i}$ derives to a functor
$\mathcal{M}^{\cdot}\to\mathcal{M}^{\cdot,i}$, and we have $\mathcal{M}^{\cdot}={\displaystyle \bigoplus_{i}\mathcal{M}^{\cdot,i}}$
for any complex in $D(\mathcal{G}(\tilde{\mathcal{R}}))$. 
\begin{lem}
Let $\varphi:X\to Y$ be a continuous map, and let $\tilde{\mathcal{R}}_{X}$
and $\tilde{\mathcal{R}}_{Y}$ be graded sheaves of algebras on $X$
and $Y$, respectively. Suppose there is a morphism of graded rings
$\varphi^{-1}(\tilde{\mathcal{R}}_{Y})\to\tilde{\mathcal{R}}_{X}$.
Then we can form the derived functor $R\varphi_{*}:D(\mathcal{G}(\tilde{\mathcal{R}}_{X}))\to D(\mathcal{G}(\tilde{\mathcal{R}}_{Y}))$,
as well as $R\varphi_{*}:D(\tilde{\mathcal{R}}_{X}-\text{mod})\to D(\tilde{\mathcal{R}}_{Y}-\text{mod})$. 

1) Let $\mathcal{F}_{X}$ denote the forgetful functor from $\mathcal{G}(\tilde{\mathcal{R}}_{X})$
to $\tilde{\mathcal{R}}_{X}-\text{mod}$ (and similarly for $\mathcal{F}_{Y}$).
Then for any $\mathcal{M}^{\cdot}\in D^{+}(\mathcal{G}(\tilde{\mathcal{R}}_{X}))$,
we have $\mathcal{F}_{Y}R\varphi_{*}(\mathcal{M}^{\cdot})\tilde{\to}R\varphi_{*}(\mathcal{F}_{X}\mathcal{M}^{\cdot})$;
where on the right hand side $R\varphi_{*}$ denotes the pushforward
$D^{+}(\tilde{\mathcal{R}}_{X}-\text{mod})\to D^{+}(\tilde{\mathcal{R}}_{Y}-\text{mod})$.
If $X$ and $Y$ have finite dimension, then this isomorphism holds
for all $\mathcal{M}^{\cdot}\in D(\mathcal{G}(\tilde{\mathcal{R}}_{X}))$. 

2) Again assuming $X$ and $Y$ have finite dimension; for each $i\in\mathbb{Z}$
we have $R\varphi_{*}(\mathcal{M}^{\cdot,i})\tilde{=}R\varphi_{*}(\mathcal{M}^{\cdot})^{i}$
in $D(\mathcal{R}_{Y}-\text{mod})$. 

3) For every $\mathcal{M}^{\cdot}\in D(\mathcal{G}(\tilde{\mathcal{R}}_{X}))$
and $\mathcal{N}^{\cdot}\in D(\mathcal{G}(\tilde{\mathcal{R}}_{Y}))$
we have 
\[
R\varphi_{*}R\underline{\mathcal{H}om}_{\varphi^{-1}(\tilde{\mathcal{R}}_{Y})}(\varphi^{-1}\mathcal{N}^{\cdot},\mathcal{M}^{\cdot})\tilde{\to}R\underline{\mathcal{H}om}_{\tilde{\mathcal{R}}_{Y}}(\mathcal{N}^{\cdot},R\varphi_{*}\mathcal{M}^{\cdot})
\]
\end{lem}

\begin{proof}
1) The statement about $D^{+}(\mathcal{G}(\tilde{\mathcal{R}}_{X}))$
follows immediately from the fact that injectives are flasque. For
the unbounded derived category, the assumption implies $\varphi_{*}$
has finite homological dimension; and by what we have just proved
the forgetful functor takes acyclic objects to acyclic objects. Therefore
we can apply the composition of derived functors (as in \cite{key-9},
corollary 14.3.5), which implies that, since $\varphi_{*}$, $\mathcal{F}_{X}$,
and $\mathcal{F}_{Y}$ have finite homological dimension in this case,
there is an isomorphism $R\varphi_{*}\circ\mathcal{F}_{X}\tilde{=}R(\varphi_{*}\circ\mathcal{F}_{X})\tilde{\to}R(\mathcal{F}_{Y}\circ\varphi_{*})\tilde{=}\mathcal{F}_{Y}\circ R\varphi_{*}$. 

2) As above this follows from \cite{key-9}, corollary 14.3.5, using
$\varphi_{*}\circ\mathcal{M}^{i}\tilde{=}(\varphi_{*}\mathcal{M})^{i}$. 

3) This is essentially identical to the analogous fact in the ungraded
case.
\end{proof}
Now we briefly discuss the internal Hom and tensor on these categories.
If $\mathcal{M}$ and $\mathcal{N}$ are objects of $\mathcal{G}(\tilde{\mathcal{R}})$,
we have the sheaf of $\mathbb{Z}$-modules $\mathcal{H}om_{\mathcal{G}(\tilde{\mathcal{R}})}(\mathcal{M},\mathcal{N})$
as well as the sheaf of graded $\mathbb{Z}$-modules ${\displaystyle \underline{\mathcal{H}om}(\mathcal{M},\mathcal{N})=\bigoplus_{i\in\mathbb{Z}}\mathcal{H}om_{\mathcal{G}(\tilde{\mathcal{R}})}(\mathcal{M},\mathcal{N}(i))}$;
if $\mathcal{M}$ is locally finitely presented this agrees with $\mathcal{H}om$
on the underlying $\tilde{\mathcal{R}}$-modules. Also, if $\mathcal{M}\in\mathcal{G}(\tilde{\mathcal{R}})$
and $\mathcal{N}\in\mathcal{G}(\tilde{\mathcal{R}}^{opp})$, we have
the tensor product $\mathcal{N}\otimes_{\tilde{\mathcal{R}}}\mathcal{M}$
which is graded in the natural way. Suppose now that $\tilde{\mathcal{S}}$
is another sheaf of graded algebras on $X$, 
\begin{lem}
\label{lem:basic-hom-tensor}1) Let $\mathcal{N}$ be a graded $(\mathcal{\tilde{\mathcal{R}}},\mathcal{\tilde{\mathcal{S}}})$
bimodule, $\mathcal{M}\in\mathcal{G}(\tilde{\mathcal{S}})$, and $\mathcal{P}\in\mathcal{G}(\tilde{\mathcal{R}})$.
Then there is an isomorphism
\[
\underline{\mathcal{H}om}_{\mathcal{\tilde{R}}}(\mathcal{N}\otimes_{\mathcal{\tilde{S}}}\mathcal{M},\mathcal{P})\tilde{\to}\underline{\mathcal{H}om}_{\tilde{\mathcal{S}}}(\mathcal{M},\underline{\mathcal{H}om}_{\tilde{\mathcal{R}}}(\mathcal{N},\mathcal{P}))
\]
Now, if we consider $\mathcal{M}^{\cdot}\in D(\mathcal{G}(\tilde{\mathcal{S}}))$
and $\mathcal{P}^{\cdot}\in D(\mathcal{G}(\tilde{\mathcal{R}}))$,
we have a map
\[
R\underline{\mathcal{H}om}_{\mathcal{\tilde{R}}}(\mathcal{N}\otimes_{\mathcal{\tilde{S}}}^{L}\mathcal{M}^{\cdot},\mathcal{P}^{\cdot})\to R\underline{\mathcal{H}om}_{\tilde{\mathcal{S}}}(\mathcal{M}^{\cdot},R\underline{\mathcal{H}om}_{\tilde{\mathcal{R}}}(\mathcal{N},\mathcal{P}^{\cdot}))
\]
and if, further, $\mathcal{N}$ is flat over $\tilde{\mathcal{S}}^{opp}$,
then this map is an isomorphism. 

2) Now suppose $\tilde{\mathcal{S}}\subset\tilde{\mathcal{R}}$ is
a central inclusion of graded rings (in particular $\tilde{\mathcal{S}}$
is commutative). Then for any $\mathcal{M}\in\mathcal{G}(\tilde{\mathcal{R}})$,
$\mathcal{N}\in\mathcal{G}(\tilde{\mathcal{R}}^{opp})$, and $\mathcal{P}\in\mathcal{G}(\tilde{\mathcal{R}})$
there are isomorphisms 
\[
\underline{\mathcal{H}om}_{\mathcal{\tilde{S}}}(\mathcal{N}\otimes_{\mathcal{\tilde{R}}}\mathcal{M},\mathcal{P})\tilde{\to}\underline{\mathcal{H}om}_{\tilde{\mathcal{R}}}(\mathcal{M},\underline{\mathcal{H}om}_{\tilde{\mathcal{S}}}(\mathcal{N},\mathcal{P}))
\]
the analogous result holds at the level of complexes: if $\mathcal{M}^{\cdot}\in D(\mathcal{G}(\tilde{\mathcal{R}}))$,
$\mathcal{N}^{\cdot}\in D(\mathcal{G}(\tilde{\mathcal{R}}^{opp}))$,
and $\mathcal{P}^{\cdot}\in D(\mathcal{G}(\tilde{\mathcal{R}}))$
there are isomorphisms 
\[
R\underline{\mathcal{H}om}_{\mathcal{\tilde{S}}}(\mathcal{N}^{\cdot}\otimes_{\mathcal{\tilde{R}}}^{L}\mathcal{M}^{\cdot},\mathcal{P}^{\cdot})\tilde{\to}R\underline{\mathcal{H}om}_{\tilde{\mathcal{R}}}(\mathcal{M}^{\cdot},R\underline{\mathcal{H}om}_{\tilde{\mathcal{S}}}(\mathcal{N}^{\cdot},\mathcal{P}^{\cdot}))
\]
\end{lem}

This is proved in a nearly identical way to the ungraded case (c.f.
\cite{key-9}, theorem 14.4.8).

Throughout this work, we will make extensive use of various sheaves
of rings over $W(k)$ and derived categories of sheaves of modules
over them. One of our main techniques will be to work with complexes
of sheaves which are\emph{ }complete in a suitable sense, and then
to apply Nakayama's lemma to deduce properties of those complexes
from their $\text{mod}$ $p$ analogues. The technical set-up for
this is the theory of cohomologically complete complexes (also called
\emph{derived complete complexes }in many places) which has been treated
in the literature in many places, e.g., \cite{key-41}, \cite{key-42},
\cite{key-43}, Tag 091N, and \cite{key-82}, section 3.4. We will
use the reference \cite{key-8}, chapter 1.5, which deals with non-commutative
sheaves of algebras in a very general setting (namely, they work with
sheaves of rings over $\mathbb{Z}[h]$, which are $h$-torsion-free). 

However, we actually have to extend the theory slightly to get exactly
what we need, because our interest is in complexes of \emph{graded}
modules, and the useful notion of completeness in this setting is
to demand, essentially, that each graded piece of a module (or complex)
is complete. We will set this up in a way that we can derive the results
in a similar way to \cite{key-8} (or even derive them from \cite{key-8}
sometimes).

From now on, we impose the assumption that $\tilde{\mathcal{R}}$
is a $W(k)$-algebra (where $W(k)$ sits in degree $0$) which is
$p$-torsion-free. Note that we have the sheaf of algebras $\tilde{\mathcal{R}}[p^{-1}]$,
which we regard as an object of $\mathcal{G}(\tilde{\mathcal{R}})$
via $\tilde{\mathcal{R}}[p^{-1}]=\bigoplus_{i\in\mathbb{Z}}\tilde{\mathcal{R}}^{i}[p^{-1}]$.
There is the category $\mathcal{G}(\tilde{\mathcal{R}}[p^{-1}])$
of graded sheaves of modules over $\tilde{\mathcal{R}}[p^{-1}]$,
and there is the functor $D(\mathcal{G}(\tilde{\mathcal{R}}[p^{-1}]))\to D(\mathcal{G}(\tilde{\mathcal{R}}))$;
which is easily seen to be fully faithful, with essential image consisting
of those complexes in $D(\mathcal{G}(\tilde{\mathcal{R}}))$ for which
$p$ acts invertibly on each cohomology sheaf (compare \cite{key-8},
lemma 1.5.2); we shall therefore simply regard $D(\mathcal{G}(\tilde{\mathcal{R}}[p^{-1}]))$
as being a full subcategory of $D(\mathcal{G}(\tilde{\mathcal{R}}))$.
Then, following \cite{key-8}, definition 1.5.5, we make the
\begin{defn}
\label{def:CC}1) An object $\mathcal{M}^{\cdot}\in D(\mathcal{R}-\text{mod})$
is said to be cohomologically complete if $R\mathcal{H}om_{\mathcal{R}}(\mathcal{R}[p^{-1}],\mathcal{M}^{\cdot})=R\mathcal{H}om_{W(k)}(W(k)[p^{-1}],\mathcal{M}^{\cdot})=0$. 

2) An object $\mathcal{M}^{\cdot}\in D(\mathcal{G}(\mathcal{\tilde{R}}))$
is said to be cohomologically complete if \linebreak{}
$R\underline{\mathcal{H}om}(\tilde{\mathcal{R}}[p^{-1}],\mathcal{M}^{\cdot})=0$. 
\end{defn}

We shall see below that two notions are not quite consistent with
one another, however, we shall only use definition $2)$ when working
with graded objects, so this will hopefully cause no confusion. 

Following \cite{key-8}, proposition 1.5.6), we have:
\begin{prop}
\label{prop:Basic-CC-facts}1) The cohomologically complete objects
in $D(\mathcal{G}(\tilde{\mathcal{R}}))$ form a thick triangulated
subcategory, denoted $D_{cc}(\mathcal{G}(\tilde{\mathcal{R}}))$.
An object $\mathcal{M}^{\cdot}\in D(\mathcal{G}(\mathcal{\tilde{R}}))$
is in $D_{cc}(\mathcal{G}(\tilde{\mathcal{R}}))$ iff $R\underline{\mathcal{H}om}(\mathcal{N}^{\cdot},\mathcal{M}^{\cdot})=0$
for all $\mathcal{N}^{\cdot}\in D(\mathcal{G}(\tilde{\mathcal{R}}[p^{-1}]))$. 

2) If $\tilde{\mathcal{S}}$ is any graded sheaf of $p$-torsion-free
$W(k)$-algebras equipped with a graded algebra map $\tilde{\mathcal{S}}\to\tilde{\mathcal{R}}$,
and $\mathcal{M}^{\cdot}\in D(\mathcal{G}(\tilde{\mathcal{R}}))$,
then $\mathcal{M}^{\cdot}\in D_{cc}(\mathcal{G}(\tilde{\mathcal{R}}))$
iff $\mathcal{M}^{\cdot}\in D_{cc}(\mathcal{G}(\tilde{\mathcal{S}}))$

3) For every $\mathcal{M}^{\cdot}\in D(\mathcal{G}(\tilde{\mathcal{R}}))$
there is a distinguished triangle
\[
R\underline{\mathcal{H}om}(\tilde{\mathcal{R}}[p^{-1}],\mathcal{M}^{\cdot})\to\mathcal{M}^{\cdot}\to R\underline{\mathcal{H}om}(\tilde{\mathcal{R}}[p^{-1}]/\tilde{\mathcal{R}}[-1],\mathcal{M}^{\cdot})
\]
and we have $R\underline{\mathcal{H}om}(\tilde{\mathcal{R}}[p^{-1}]/\tilde{\mathcal{R}}[-1],\mathcal{M}^{\cdot})\in D_{cc}(\mathcal{G}(\tilde{\mathcal{R}}))$
while $R\underline{\mathcal{H}om}(\tilde{\mathcal{R}}[p^{-1}],\mathcal{M}^{\cdot})\in D(\mathcal{G}(\tilde{\mathcal{R}}[p^{-1}]))$.
In particular, the category $D_{cc}(\mathcal{G}(\tilde{\mathcal{R}}))$
is naturally equivalent to the quotient of $D(\mathcal{G}(\tilde{\mathcal{R}}))$
by $D(\mathcal{G}(\tilde{\mathcal{R}}[p^{-1}]))$. 

4) Recall that for each object $\mathcal{M}^{\cdot}\in D(\mathcal{G}(\tilde{\mathcal{R}}))$
we have, for $i\in\mathbb{Z}$, the $i$th graded piece $\mathcal{M}^{\cdot,i}\in D(\mathcal{R}-\text{mod})$.
Then $\mathcal{M}^{\cdot}\in D(\mathcal{G}(\tilde{\mathcal{R}}))$
is in $D_{cc}(\mathcal{G}(\mathcal{R}))$ iff each $\mathcal{M}^{\cdot,i}\in D_{cc}(\mathcal{R}-\text{mod})$. 
\end{prop}

\begin{proof}
1) For any $\mathcal{N}^{\cdot}\in D(\mathcal{G}(\tilde{\mathcal{R}}[p^{-1}]))$
we have 
\[
R\underline{\mathcal{H}om}_{\tilde{\mathcal{R}}}(\mathcal{N}^{\cdot},\mathcal{M}^{\cdot})\tilde{=}R\underline{\mathcal{H}om}_{\tilde{\mathcal{R}}}(\tilde{\mathcal{R}}[p^{-1}]\otimes_{\tilde{\mathcal{R}}}^{L}\mathcal{N}^{\cdot},\mathcal{M}^{\cdot})\tilde{\to}R\underline{\mathcal{H}om}_{\tilde{\mathcal{R}}}(\mathcal{N}^{\cdot},R\underline{\mathcal{H}om}_{\tilde{\mathcal{R}}}(\tilde{\mathcal{R}}[p^{-1}],\mathcal{M}^{\cdot}))
\]
here, we have used the fact $\tilde{\mathcal{R}}[p^{-1}]$ is an $(\tilde{\mathcal{R}},\tilde{\mathcal{R}})$-bimodule,
along with \lemref{basic-hom-tensor}, 1). 

Thus if $R\underline{\mathcal{H}om}(\tilde{\mathcal{R}}[p^{-1}]\mathcal{M}^{\cdot})=0$
then $R\underline{\mathcal{H}om}(\mathcal{N}^{\cdot},\mathcal{M}^{\cdot})=0$
as claimed. Therefore $D_{cc}(\mathcal{G}(\tilde{\mathcal{R}}))$
is the (right) orthogonal subcategory to the thick subcategory $D(\mathcal{G}(\tilde{\mathcal{R}}[p^{-1}]))$;
it follows that is a thick triangulated subcategory. 

2) We have 
\[
R\underline{\mathcal{H}om}_{\tilde{\mathcal{S}}}(\tilde{\mathcal{S}}[p^{-1}],\mathcal{M}^{\cdot})\tilde{\to}R\underline{\mathcal{H}om}_{\tilde{\mathcal{R}}}(\tilde{\mathcal{S}}[p^{-1}]\otimes_{\tilde{\mathcal{S}}}^{L}\tilde{\mathcal{R}},\mathcal{M}^{\cdot})\tilde{\to}R\underline{\mathcal{H}om}_{\tilde{\mathcal{R}}}(\tilde{\mathcal{R}}[p^{-1}],\mathcal{M}^{\cdot})
\]
from which the result follows.

3) This triangle follows by applying $R\underline{\mathcal{H}om}$
to the short exact sequence 
\[
\tilde{\mathcal{R}}\to\tilde{\mathcal{R}}[p^{-1}]\to\tilde{\mathcal{R}}[p^{-1}]/\tilde{\mathcal{R}}
\]
and noting that $R\underline{\mathcal{H}om}(\tilde{\mathcal{R}},)$
is the identity functor. The complex $R\underline{\mathcal{H}om}(\tilde{\mathcal{R}}[p^{-1}],\mathcal{M}^{\cdot})$
is contained in $D(\mathcal{G}(\tilde{\mathcal{R}}[p^{-1}]))$ via
the action of $\tilde{\mathcal{R}}[p^{-1}]$ on itself. On the other
hand, as above there is a canonical isomorphism
\[
R\underline{\mathcal{H}om}_{\tilde{\mathcal{R}}}(\tilde{\mathcal{R}}[p^{-1}],R\underline{\mathcal{H}om}_{\tilde{\mathcal{R}}}(\tilde{\mathcal{R}}[p^{-1}]/\tilde{\mathcal{R}}[-1],\mathcal{M}^{\cdot}))\tilde{\leftarrow}R\underline{\mathcal{H}om}_{\tilde{\mathcal{R}}}(\tilde{\mathcal{R}}[p^{-1}]\otimes_{\tilde{\mathcal{R}}}^{L}(\tilde{\mathcal{R}}[p^{-1}]/\tilde{\mathcal{R}}),\mathcal{M}^{\cdot})[1]
\]
and the term on the right is zero since $\tilde{\mathcal{R}}[p^{-1}]\otimes_{\tilde{\mathcal{R}}}^{L}(\tilde{\mathcal{R}}[p^{-1}]/\tilde{\mathcal{R}})=0$;
therefore 
\[
R\underline{\mathcal{H}om}_{\tilde{\mathcal{R}}}(\tilde{\mathcal{R}}[p^{-1}]/\tilde{\mathcal{R}}[-1],\mathcal{M}^{\cdot})\in D_{cc}(\mathcal{G}(\mathcal{R}))
\]
 This shows that the inclusion $D_{cc}(\mathcal{G}(\mathcal{R}))\to D(\mathcal{G}(\mathcal{R}))$
admits a right adjoint, and the statement about the quotient category
follows immediately.

4) For each $\mathcal{M}\in\mathcal{G}(\tilde{\mathcal{R}})$ there
is an isomorphism of functors 
\[
\mathcal{H}om_{\mathcal{R}}(\mathcal{R}[p^{-1}],\mathcal{M}^{0})\tilde{=}\mathcal{H}om_{\mathcal{G}(\tilde{\mathcal{R}})}(\tilde{\mathcal{R}}[p^{-1}],\mathcal{M})
\]
given by restricting a morphism on the right hand side to degree $0$;
this follows from the fact that an local section of $\mathcal{H}om_{\mathcal{G}(\tilde{\mathcal{R}})}(\tilde{\mathcal{R}}[p^{-1}],\mathcal{M})$
is simply a system $(m_{i})$ of local sections of $\mathcal{M}^{0}$
satisfying $pm_{i}=m_{i-1}$; which is exactly a local section of
$\mathcal{H}om_{\mathcal{R}}(\mathcal{R}[p^{-1}],\mathcal{M}^{0})$. 

Now, $\mathcal{M}\to\mathcal{M}^{0}$ admits a left adjoint (namely
$\mathcal{N}\to\tilde{\mathcal{R}}\otimes_{\mathcal{R}}\mathcal{N}$),
and $\mathcal{N}\to\mathcal{H}om_{\mathcal{R}}(\mathcal{R}[p^{-1}],\mathcal{N})$
admits a left adjoint (namely $\mathcal{M}\to\mathcal{R}[p^{-1}]\otimes_{\mathcal{R}}\mathcal{M}$).
So by \cite{key-9}, proposition 14.4.7, the derived functor of $\mathcal{H}om_{\mathcal{R}}(\mathcal{R}[p^{-1}],\mathcal{M}^{0})$
is given by the functor $R\mathcal{H}om_{\mathcal{R}}(\mathcal{R}[p^{-1}],\mathcal{M}^{\cdot,0})$
for any $\mathcal{M}^{\cdot}\in D(\mathcal{G}(\tilde{R}))$. Therefore
there is an isomorphism of functors 
\[
R\mathcal{H}om_{\mathcal{R}}(\mathcal{R}[p^{-1}],\mathcal{M}^{\cdot,0})\tilde{\to}R\mathcal{H}om_{\mathcal{G}(\tilde{\mathcal{R}})}(\tilde{\mathcal{R}}[p^{-1}],\mathcal{M})
\]
Therefore 
\[
R\underline{\mathcal{H}om}_{\tilde{\mathcal{R}}}(\tilde{\mathcal{R}}[p^{-1}],\mathcal{M}^{\cdot})=\bigoplus_{i}R\mathcal{H}om_{\mathcal{G}(\tilde{\mathcal{R}})}(\tilde{\mathcal{R}}[p^{-1}],\mathcal{M}^{\cdot}(i))\tilde{=}\bigoplus_{i}R\mathcal{H}om_{\mathcal{R}}(\mathcal{R}[p^{-1}],\mathcal{M}^{\cdot,-i})
\]
and the result follows. 
\end{proof}
We will refer to the functor $\mathcal{M}^{\cdot}\to R\underline{\mathcal{H}om}(\tilde{\mathcal{R}}[p^{-1}]/\tilde{\mathcal{R}}[-1],\mathcal{M}^{\cdot})$
as the \emph{graded derived completion} of $\mathcal{M}^{\cdot}$,
or, usually, simply the completion of $\mathcal{M}^{\cdot}$ if no
confusion seems likely; we will denote it by $\widehat{\mathcal{M}}^{\cdot}$. 

A typical example of a cohomologically complete complex in $D(\mathcal{R}-\text{mod})$
is the following: suppose $\mathcal{M}^{\cdot}=\mathcal{M}$ is concentrated
in degree $0$. Then if $\mathcal{M}$ is $p$-torsion free, then
$\mathcal{M}$ is $p$-adically complete iff $\mathcal{M}^{\cdot}$
is cohomologically complete (c.f. \cite{key-8}, lemma 1.5.4). By
part $4)$ of the proposition, if $\mathcal{M}=\mathcal{M}^{\cdot}\in D(\mathcal{G}(\tilde{\mathcal{R}}))$
is concentrated in a single degree, then if each $\mathcal{M}^{i}$
is $p$-torsion free and $p$-adically complete, then $\mathcal{M}^{\cdot}$
is cohomologically complete (in the graded sense). Therefore the two
notions are not in general compatible; an infinite direct sum of $p$-adically
complete modules is generally not complete. 

Now we develop this notion a bit more:
\begin{lem}
\label{lem:reduction-of-completion}Let $\mathcal{M}^{\cdot}\in D(\mathcal{G}(\tilde{\mathcal{R}}))$.
Then the natural map $\mathcal{M}^{\cdot}\otimes_{W(k)}^{L}k\to\mathcal{\widehat{M}}^{\cdot}\otimes_{W(k)}^{L}k$
is an isomorphism.
\end{lem}

\begin{proof}
The cone of the map $\mathcal{M}^{\cdot}\to\mathcal{\widehat{M}}^{\cdot}$
is contained in $D(\mathcal{G}(\tilde{\mathcal{R}}[p^{-1}]))$, and
therefore vanishes upon applying $\otimes_{W(k)}^{L}k$. 
\end{proof}
Now we can transfer the Nakayama lemma into the graded setting:
\begin{cor}
\label{cor:Nakayama}Let $\mathcal{M}^{\cdot}\in D_{cc}(\mathcal{G}(\mathcal{\tilde{R}}))$,
and let $a\in\mathbb{Z}$. If $\mathcal{H}^{i}(\mathcal{M}^{\cdot}\otimes_{W(k)}^{L}k)=0$
for all $i<a$, then $\mathcal{H}^{i}(\mathcal{M}^{\cdot})=0$ for
all $i<a$. In particular $\mathcal{M}^{\cdot}=0$ iff $\mathcal{M}^{\cdot}\otimes_{W(k)}^{L}k=0$. 

Therefore, if $\mathcal{M}^{\cdot},\mathcal{N}^{\cdot}\in D_{cc}(\mathcal{G}(\mathcal{\tilde{R}}))$
and $\eta:\mathcal{M}^{\cdot}\to\mathcal{N}^{\cdot}$ is a morphism
such that $\eta\otimes_{W(k)}^{L}k:\mathcal{M}^{\cdot}\otimes_{W(k)}^{L}k\to\mathcal{N}^{\cdot}\otimes_{W(k)}^{L}k$
is an isomorphism, then $\eta$ is an isomorphism.
\end{cor}

\begin{proof}
By part 4) of the previous proposition this follows immediately from
the analogous fact for cohomologically complete sheaves over $\mathcal{R}$;
which is \cite{key-8}, proposition 1.5.8. 
\end{proof}
For later use, we record a few more useful properties of cohomologically
complete sheaves, following \cite{key-8}, propositions 1.5.10 and
1.5.12.
\begin{prop}
\label{prop:Push-and-complete}1) Suppose $\mathcal{M}^{\cdot},\mathcal{N}^{\cdot}\in D_{cc}(\mathcal{G}(\tilde{\mathcal{R}}))$,
and let $\tilde{\mathcal{S}}$ be any central graded sub-algebra of
$\tilde{\mathcal{R}}$ which contains $W(k)$. Then $R\underline{\mathcal{H}om}(\mathcal{M}^{\cdot},\mathcal{N}^{\cdot})\in D_{cc}(\mathcal{G}(\tilde{\mathcal{S}}))$. 

2) Suppose $\mathcal{M}^{\cdot}\in D(\mathcal{G}(\tilde{\mathcal{R}}))$
and $\mathcal{N}^{\cdot}\in D_{cc}(\mathcal{G}(\tilde{\mathcal{R}}))$.
Then the map $\mathcal{M}^{\cdot}\to\widehat{\mathcal{M}}^{\cdot}$
induces an isomorphism 
\[
R\underline{\mathcal{H}om}(\widehat{\mathcal{M}}^{\cdot},\mathcal{N}^{\cdot})\tilde{\to}R\underline{\mathcal{H}om}(\mathcal{M}^{\cdot},\mathcal{N}^{\cdot})
\]

3) Suppose $\varphi:X\to Y$ is a continuous map, and suppose $\tilde{\mathcal{R}}$
is a graded sheaf of algebras on $Y$ (satisfying the running assumptions
of the section). Let $\mathcal{M}^{\cdot}\in D_{cc}(\mathcal{G}(\varphi^{-1}(\tilde{\mathcal{R}})))$.
Then $R\varphi_{*}(\mathcal{M}^{\cdot})\in D_{cc}(\mathcal{G}(\tilde{\mathcal{R}}))$.
Therefore, if $\mathcal{M}^{\cdot}\in D(\mathcal{G}(\varphi^{-1}(\tilde{\mathcal{R}})))$
is any complex, then we have 
\[
\widehat{R\varphi_{*}(\mathcal{M}^{\cdot})}\tilde{\to}R\varphi_{*}(\widehat{\mathcal{M}^{\cdot}})
\]
\end{prop}

\begin{proof}
1) As $\tilde{\mathcal{S}}$ is central we have 
\[
R\underline{\mathcal{H}om}_{\tilde{\mathcal{S}}}(\tilde{\mathcal{S}}[p^{-1}],R\underline{\mathcal{H}om}_{\tilde{\mathcal{R}}}(\mathcal{M}^{\cdot},\mathcal{N}^{\cdot}))\tilde{\leftarrow}R\underline{\mathcal{H}om}_{\tilde{\mathcal{S}}}(\mathcal{M}^{\cdot}\otimes_{\tilde{\mathcal{S}}}^{L}\tilde{\mathcal{S}}[p^{-1}],\mathcal{N}^{\cdot})
\]
\[
\tilde{\to}R\underline{\mathcal{H}om}_{\tilde{\mathcal{R}}}(\mathcal{M}^{\cdot},R\underline{\mathcal{H}om}_{\tilde{\mathcal{S}}}(\tilde{\mathcal{S}}[p^{-1}],\mathcal{N}^{\cdot}))
\]
where the second isomorphism follows from the flatness of $\tilde{\mathcal{S}}[p^{-1}]$
over $\tilde{\mathcal{S}}$, and first isomorphism follows directly
from 
\[
\tilde{\mathcal{S}}[p^{-1}]\tilde{=}\text{lim}(\tilde{\mathcal{S}}\xrightarrow{p}\tilde{\mathcal{S}}\xrightarrow{p}\tilde{\mathcal{S}}\cdots)\tilde{=}{\displaystyle \text{hocolim}(\tilde{\mathcal{S}}\xrightarrow{p}\tilde{\mathcal{S}}\xrightarrow{p}\tilde{\mathcal{S}}\cdots)}
\]
so the result follows from part $2)$ of \propref{Basic-CC-facts}. 

2) This follows since $\text{cone}(\mathcal{M}^{\cdot}\to\widehat{\mathcal{M}}^{\cdot})$
is contained in the orthogonal to $D_{cc}(\mathcal{G}(\tilde{\mathcal{R}}))$,
by definition.

3) For the first claim, we use the adjunction 
\[
R\varphi_{*}R\underline{\mathcal{H}om}_{\varphi^{-1}(\tilde{\mathcal{R}})}(\varphi^{-1}(\tilde{\mathcal{R}})[p^{-1}],\mathcal{M}^{\cdot})\tilde{=}R\underline{\mathcal{H}om}_{\tilde{\mathcal{R}}}(\tilde{\mathcal{R}}[p^{-1}],R\varphi_{*}(\mathcal{M}^{\cdot}))
\]
along with part $2)$ of \propref{Basic-CC-facts}. For the second,
we use the distinguished triangle 
\[
R\underline{\mathcal{H}om}(\tilde{\mathcal{R}}[p^{-1}],\mathcal{M}^{\cdot})\to\mathcal{M}^{\cdot}\to\widehat{\mathcal{M}^{\cdot}}
\]
Since $p$ acts invertibly on $R\underline{\mathcal{H}om}(\tilde{\mathcal{R}}[p^{-1}],\mathcal{M}^{\cdot})$,
it will also act invertibly on\linebreak{}
$R\varphi_{*}(R\underline{\mathcal{H}om}(\tilde{\mathcal{R}}[p^{-1}],\mathcal{M}^{\cdot}))$,
and the result follows from the fact that $R\varphi_{*}(\widehat{\mathcal{M}^{\cdot}})$
is already cohomologically complete.
\end{proof}
In using this theory, it is also useful to note the following straightforward
\begin{lem}
\label{lem:Hom-tensor-and-reduce}Let $\tilde{\mathcal{R}}$ be as
above. Then, for any $\mathcal{M}^{\cdot},\mathcal{N}^{\cdot}\in D(\mathcal{G}(\tilde{\mathcal{R}}))$
we have 
\[
R\underline{\mathcal{H}om}_{\tilde{\mathcal{R}}}(\mathcal{M}^{\cdot},\mathcal{N}^{\cdot})\otimes_{W(k)}^{L}k\tilde{\to}R\underline{\mathcal{H}om}_{\tilde{\mathcal{R}}/p}(\mathcal{M}^{\cdot}\otimes_{W(k)}^{L}k,\mathcal{N}^{\cdot}\otimes_{W(k)}^{L}k)
\]
If we have $\mathcal{N}^{\cdot}\in D(\mathcal{G}(\tilde{\mathcal{R}})^{\text{opp}})$,
then we have
\[
(\mathcal{N}^{\cdot}\otimes_{\tilde{\mathcal{R}}}^{L}\mathcal{M}^{\cdot})\otimes_{W(k)}^{L}k\tilde{\to}(\mathcal{N}^{\cdot}\otimes_{W(k)}^{L}k)\otimes_{\tilde{\mathcal{R}}/p}^{L}(\mathcal{M}^{\cdot}\otimes_{W(k)}^{L}k)
\]
\end{lem}

To close out this section, we will give an explicit description of
the (ungraded) cohomological completion functor in a special case.
Let $\mathcal{R}$ be a $p$-torsion-free sheaf of rings on $X$ as
above; suppose $\mathcal{R}$ is left noetherian. Let us suppose that,
in addition, the $p$-adic completion $\widehat{\mathcal{R}}$ is
$p$-torsion-free, left noetherian, and that there exists a base of
open subsets $\mathcal{B}$ on $X$ such that, for any $U\in\mathcal{B}$
and any coherent sheaf $\mathcal{M}$ of $\mathcal{R}_{0}=\mathcal{R}/p=\widehat{\mathcal{R}}/p$
modules on $U$, we have $H^{i}(U,\mathcal{M})=0$ for all $i>0$
(these are assumptions $1.2.2$ and $1.2.3$ of \cite{key-8}; they
are always satisfied in this paper). Then we have
\begin{prop}
\label{prop:Completion-for-noeth}Let $\mathcal{M}^{\cdot}\in D_{coh}^{b}(\mathcal{R}-\text{mod})$.
Then there is an isomorphism 
\[
\widehat{\mathcal{M}^{\cdot}}\tilde{=}\widehat{\mathcal{R}}\otimes_{\mathcal{R}}^{L}\mathcal{M}^{\cdot}
\]
where $\widehat{\mathcal{M}^{\cdot}}$ denotes the derived completion
as usual. 
\end{prop}

\begin{proof}
Let $\mathcal{M}$ be a coherent $\mathcal{R}$-module, and let $\widehat{\mathcal{M}}$
denote its $p$-adic completion. By \cite{key-8}, lemma 1.1.6 and
the assumption on $\mathcal{B}$, we have ${\displaystyle \widehat{\mathcal{M}}(U)=\lim_{\leftarrow}\mathcal{M}(U)/p^{n}}$
for any $U\in\mathcal{B}$. So, by the noetherian hypothesis, we see
that the natural map $\widehat{\mathcal{R}}(U)\otimes_{\mathcal{R}(U)}\mathcal{M}(U)\to\widehat{\mathcal{M}}(U)$
is an isomorphism for all $U\in\mathcal{B}$ . It follows that the
map $\widehat{\mathcal{R}}\otimes_{\mathcal{R}}\mathcal{M}\to\widehat{\mathcal{M}}$
is an isomorphism of sheaves; and therefore (as $p$-adic completion
is exact on $\mathcal{R}(U)-\text{mod}$) that $\widehat{\mathcal{R}}$
is flat over $\mathcal{R}$. 

Now consider an arbitrary $\mathcal{M}^{\cdot}\in D_{coh}^{b}(\mathcal{R}-\text{mod})$.
The above implies 
\[
\mathcal{H}^{i}(\widehat{\mathcal{R}}\otimes_{\mathcal{R}}^{L}\mathcal{M}^{\cdot})\tilde{=}\widehat{\mathcal{R}}\otimes_{\mathcal{R}}\mathcal{H}^{i}(\mathcal{M}^{\cdot})\tilde{\to}\widehat{\mathcal{H}^{i}(\mathcal{M}^{\cdot})}
\]
Therefore $\widehat{\mathcal{R}}\otimes_{\mathcal{R}}^{L}\mathcal{M}^{\cdot}\in D_{coh}^{b}(\widehat{\mathcal{R}}-\text{mod})$,
which is contained in $D_{cc}(\mathcal{R}-\text{mod})$ by \cite{key-8},
theorem 1.6.1. 

Let $\mathcal{C}^{\cdot}$ be the cone of the map $\mathcal{M}^{\cdot}\to\widehat{\mathcal{R}}\otimes_{\mathcal{R}}^{L}\mathcal{M}^{\cdot}$
. Then we have a long exact sequence 
\[
\mathcal{H}^{i-1}(\mathcal{C}^{\cdot})\to\mathcal{H}^{i}(\mathcal{M}^{\cdot})\to\widehat{\mathcal{H}^{i}(\mathcal{M}^{\cdot})}\to\mathcal{H}^{i}(\mathcal{C}^{\cdot})
\]
and since both the kernel and cokernel of $\mathcal{H}^{i}(\mathcal{M}^{\cdot})\to\widehat{\mathcal{H}^{i}(\mathcal{M}^{\cdot})}$
are in $\mathcal{R}[p^{-1}]-\text{mod}$, we conclude that $\mathcal{C}^{\cdot}\in D(\mathcal{R}[p^{-1}]-\text{mod})$.
Since $\widehat{\mathcal{R}}\otimes_{\mathcal{R}}^{L}\mathcal{M}^{\cdot}\in D_{cc}(\mathcal{R}-\text{mod})$,
the result follows from the fact that $D_{cc}(\mathcal{R}-\text{mod})$
is the quotient of $D(\mathcal{R}-\text{mod})$ by $D(\mathcal{R}[p^{-1}]-\text{mod})$. 
\end{proof}

\section{\label{sec:The-Algebra}The Algebra $\widehat{\mathcal{D}}_{\mathfrak{X}}^{(0,1)}$}

To define the algebra $\widehat{\mathcal{D}}_{\mathfrak{X}}^{(0,1)}$
and prove \thmref{D01}, we are going to apply the basic gauge construction
(\exaref{BasicGaugeConstruction}) to Berthelot's differential operators.
Let $\mathfrak{X}$ be a smooth formal scheme over $W(k)$, and $X$
its special fibre. If $\mathfrak{X}$ is affine, then we denote $\mathfrak{X}=\text{Specf}(\mathcal{A})$,
and $X=\text{Spec}(A)$. 
\begin{defn}
\label{def:D^(0,1)-in-the-lifted-case} We set
\[
\mathcal{\widehat{D}}_{\mathfrak{X}}^{(0,1),i}:=\{\Phi\in\mathcal{\widehat{D}}_{\mathfrak{X}}^{(1)}|p^{i}\Phi\in\mathcal{\widehat{D}}_{\mathfrak{X}}^{(0)}\}
\]
We let $f:\mathcal{\widehat{D}}_{\mathfrak{X}}^{(0,1),i}\to\mathcal{\widehat{D}}_{\mathfrak{X}}^{(0,1),i+1}$
denote the inclusion, and $v:\mathcal{\widehat{D}}_{\mathfrak{X}}^{(0,1),i}\to\mathcal{\widehat{D}}_{\mathfrak{X}}^{(0,1),i-1}$
denote the multiplication by $p$. If $\Phi_{1}\in\mathcal{\widehat{D}}_{\mathfrak{X}}^{(0,1),i}$
and $\Phi_{2}\in\mathcal{\widehat{D}}_{\mathfrak{X}}^{(0,1),j}$,
then $\Phi_{1}\cdot\Phi_{2}\in\mathcal{\widehat{D}}_{\mathfrak{X}}^{(0,1),i+j}$,
and in this way we give 
\[
\mathcal{\widehat{D}}_{\mathfrak{X}}^{(0,1)}=\bigoplus_{i\in\mathbb{Z}}\mathcal{\widehat{D}}_{\mathfrak{X}}^{(0,1),i}
\]
the structure of a sheaf of graded algebras over $D(W(k))$. 

Now suppose $\mathfrak{X}=\text{Specf}(\mathcal{A})$. Then we have
ring theoretic analogue of the above: define $\widehat{D}_{\mathcal{A}}^{(0,1),i}:=\{\Phi\in\widehat{D}_{\mathcal{A}}^{(1)}|p^{i}\Phi\in\widehat{D}_{\mathcal{A}}^{(0)}\}$,
and as above we obtain the graded ring 
\[
\widehat{D}_{\mathcal{A}}^{(0,1)}=\bigoplus_{i}\widehat{D}_{\mathcal{A}}^{(0,1),i}
\]
 over $D(W(k))$. 

In this case, we also have the finite-order analogue: define $D_{\mathcal{A}}^{(0,1),i}:=\{\Phi\in D_{\mathcal{A}}^{(1)}|p^{i}\Phi\in D_{\mathcal{A}}^{(0)}\}$,
and as above we obtain the graded ring 
\[
D_{\mathcal{A}}^{(0,1)}=\bigoplus_{i}D_{\mathcal{A}}^{(0,1),i}
\]
 over $D(W(k))$. 
\end{defn}

It is easy to see that $\widehat{D}_{\mathcal{A}}^{(0,1),i}:=\Gamma(\mathcal{\widehat{D}}_{\mathfrak{X}}^{(0,1),i})$
when $\mathfrak{X}=\text{Specf}(\mathcal{A})$. 

With the help of local coordinates, this algebra is not too difficult
to study. We now suppose $\mathfrak{X}=\text{Specf}(\mathcal{A})$
where $\mathcal{A}$ possesses local coordinates; i.e., there is a
collection $\{x_{i}\}_{i=1}^{n}\in\mathcal{A}$ and derivations $\{\partial_{i}\}_{i=1}^{n}$
such that $\partial_{i}(x_{j})=\delta_{ij}$ and such that $\{\partial_{i}\}_{i=1}^{n}$
form a free basis for the $\mathcal{A}$-module of $W(k)$-linear
derivations. We let $\partial_{i}^{[p]}:=\partial_{i}^{p}/p!$, this
is a differential operator of order $p$ on $\mathcal{A}$. 
\begin{lem}
\label{lem:Basic-structure-of-D_A^(i)} For $i\geq0$ we have that
$\widehat{D}_{\mathcal{A}}^{(0,1),i}$ is the left $\widehat{D}_{\mathcal{A}}^{(0)}$-module\footnote{In fact, it is also the right $\widehat{D}_{\mathcal{A}}^{(0)}$-module
generated by the same elements, as an identical proof shows } generated by $\{(\partial_{1}^{[p]})^{j_{1}}\cdots(\partial_{n}^{[p]})^{j_{n}}\}$
where ${\displaystyle \sum_{t=1}^{n}j_{t}\leq i}$. For $i\leq0$
we have that $\widehat{D}_{\mathcal{A}}^{(0,1),i}=p^{-i}\cdot\widehat{D}_{\mathcal{A}}^{(0)}$. 
\end{lem}

\begin{proof}
Let $i>0$. Clearly the module described is contained in $\widehat{D}_{\mathcal{A}}^{(0,1),i}$.
For the converse, we begin with the analogous finite-order version
of the statement. Namely, let $\Phi\in D_{\mathcal{A}}^{(1)}$ be
such that $p^{i}\Phi\in D_{\mathcal{A}}^{(0)}$. We can write 
\[
\Phi=\sum_{I,J}a_{I,J}\partial_{1}^{i_{1}}(\partial_{1}^{[p]})^{j_{1}}\cdots\partial_{n}^{i_{n}}(\partial_{n}^{[p]})^{j_{n}}=\sum_{I,J}a_{I,J}\frac{\partial_{1}^{i_{1}+pj_{1}}\cdots\partial_{n}^{i_{n}+pj_{n}}}{(p!)^{|J|}}
\]
where $|J|=j_{1}+\dots+j_{n}$, and the sum is finite. After collecting
like terms together, we may suppose that $0\leq i_{j}<p$. In that
case, the $a_{I,J}\in\mathcal{A}$ are uniquely determined by $\Phi$,
and $\Phi\in D_{\mathcal{A}}^{(0)}$ iff $p^{|J|}|a_{I,J}$ for all
$I,J$. So, if $p^{i}\Phi\in D_{\mathcal{A}}^{(0)}$, we have ${\displaystyle a_{I,J}\frac{p^{i}}{p^{|J|}}\in\mathcal{A}}$.
Thus whenever $|J|>i$ we have $a_{I,J}\in p^{|J|-i}\mathcal{A}$.
On the other hand
\[
p^{|J|-i}(\partial_{1}^{[p]})^{j_{1}}\cdots(\partial_{n}^{[p]})^{j_{n}}=u\cdot\partial_{1}^{pj'_{1}}\cdots\partial_{n}^{pj'_{n}}\cdot(\partial_{1}^{[p]})^{j''_{1}}\cdots\partial_{n}^{i_{n}}(\partial_{n}^{[p]})^{j''_{n}}
\]
where $u$ is a unit in $\mathbb{Z}_{p}$, ${\displaystyle \sum j'_{i}=|J|-i}$,
and ${\displaystyle \sum j''_{i}=i}$ (this follows from the relation
$p!\partial_{j}^{[p]}=\partial_{j}^{p}$). Therefore if $|J|>i$ we
have
\[
a_{I,J}\frac{\partial_{1}^{i_{1}+pj_{1}}\cdots\partial_{n}^{i_{n}+pj_{n}}}{(p!)^{|J|}}\in D_{\mathcal{A}}^{(0)}\cdot(\partial_{1}^{[p]})^{j''_{1}}\cdots\partial_{n}^{i_{n}}(\partial_{n}^{[p]})^{j''_{n}}
\]
It follows that $\Phi$ is contained in the $D_{\mathcal{A}}^{(0)}$-submodule
generated by $\{(\partial_{1}^{[p]})^{j_{1}}\cdots(\partial_{n}^{[p]})^{j_{n}}\}$
where $j_{1}+\dots+j_{n}\leq i$. So this submodule is exactly $\{\Phi\in D_{\mathcal{A}}^{(1)}|p^{i}\Phi\in D_{\mathcal{A}}^{(0)}\}$. 

Now let $\Phi\in\widehat{D}_{\mathcal{A}}^{(0,1),i}$. Then we can
write 
\[
p^{i}\Phi=\sum_{j=0}^{\infty}p^{j}\Phi_{j}
\]
where $\Phi_{j}\in D_{\mathcal{A}}^{(0)}$. Therefore, if $j\le i$
we have, by the previous paragraph, that $p^{-i}(p^{j}\Phi_{j})$
is contained in the $D_{\mathcal{A}}^{(0)}$ submodule generated by
$\{(\partial_{1}^{[p]})^{j_{1}}\cdots(\partial_{n}^{[p]})^{j_{n}}\}$
where $j_{1}+\dots+j_{n}\leq i$. So the result follows from ${\displaystyle \Phi=\sum_{j=0}^{i}p^{j-i}\Phi_{j}+\sum_{j=i+1}^{\infty}p^{j-i}\Phi_{j}}$
as the second term in this sum is contained in $\widehat{\mathcal{D}}_{\mathcal{A}}^{(0)}$.
This proves the lemma for $i\geq0$; while for $i\leq0$ it follows
immediately from the definition. 
\end{proof}
From this it follows that ${\displaystyle \widehat{D}_{\mathcal{A}}^{(0,1),\infty}:=\lim_{\rightarrow}\widehat{D}_{\mathcal{A}}^{(0,1),i}}$
is the sub-algebra of $\text{End}_{W(k)}(\mathcal{A})$ generated
by $\widehat{D}_{\mathcal{A}}^{(0)}$ and $\{\partial_{1}^{[p]},\dots,\partial_{n}^{[p]}\}$.
We have
\begin{lem}
\label{lem:Basic-Structure-of-D^(1)} The algebra $\widehat{D}_{\mathcal{A}}^{(0,1),\infty}$
is a (left and right) noetherian ring, whose $p$-adic completion
is isomorphic to $\widehat{D}_{\mathcal{A}}^{(1)}$. Further, we have
$\widehat{D}_{\mathcal{A}}^{(0,1),\infty}[p^{-1}]\tilde{=}\widehat{D}_{\mathcal{A}}^{(0)}[p^{-1}]$. 
\end{lem}

\begin{proof}
First, put a filtration on $\widehat{D}_{\mathcal{A}}^{(0,1),\infty}$
by setting $F^{j}(\widehat{D}_{\mathcal{A}}^{(0,1),\infty})$ to be
the $\widehat{D}_{\mathcal{A}}^{(0)}$-submodule generated by $\{(\partial_{1}^{[p]})^{j_{1}}\cdots(\partial_{n}^{[p]})^{j_{n}}\}$
where $j_{1}+\dots+j_{n}\leq i$. Then $\text{gr}(\widehat{D}_{\mathcal{A}}^{(0,1),\infty})$
is a quotient of a polynomial ring $\widehat{D}_{\mathcal{A}}^{(0)}[T_{1},\dots T_{n}]$
where $T_{i}$ is sent to the class of $\partial_{i}^{[p]}$ in $\text{gr}_{1}(\widehat{D}_{\mathcal{A}}^{(0,1),\infty})$.
To see this, we need to show that the image of $\partial_{i}^{[p]}$
in $\text{gr}(\widehat{D}_{\mathcal{A}}^{(0,1),\infty})$ commutes
with $\widehat{D}_{\mathcal{A}}^{(0)}=\text{gr}^{0}(\widehat{D}_{\mathcal{A}}^{(0,1),\infty})$;
this follows from the relation 
\[
[\partial_{i}^{[p]},a]=\sum_{j=1}^{p}\partial_{i}^{[j]}(a)\partial_{i}^{[p-j]}\in\widehat{D}_{\mathcal{A}}^{(0)}
\]
for any $a\in\mathcal{A}$. So the fact that $\widehat{D}_{\mathcal{A}}^{(0,1),\infty}$
is a (left and right) noetherian ring follows from the Hilbert basis
theorem and the fact that $\widehat{D}_{\mathcal{A}}^{(0)}$ is left
and right noetherian.

Now we compute the $p$-adic completion of $\widehat{D}_{\mathcal{A}}^{(0,1),\infty}$.
Inside $\text{End}_{W(k)}(\mathcal{A})$, we have 
\[
D_{\mathcal{A}}^{(1)}\subset\widehat{D}_{\mathcal{A}}^{(0,1),\infty}\subset\widehat{D}_{\mathcal{A}}^{(1)}
\]
and so, for all $n>0$ we have 
\[
D_{\mathcal{A}}^{(1)}/p^{n}\to\widehat{D}_{\mathcal{A}}^{(0,1),\infty}/p^{n}\to\widehat{D}_{\mathcal{A}}^{(1)}/p^{n}
\]
and the composition is the identity. Thus $D_{\mathcal{A}}^{(1)}/p^{n}\to\widehat{D}_{\mathcal{A}}^{(0,1),\infty}/p^{n}$
is injective. On the other hand, suppose $\Phi\in\widehat{D}_{\mathcal{A}}^{(0,1),\infty}$.
By definition we can write 
\[
\Phi=\sum_{I}\varphi_{i}\cdot(\partial_{1}^{[p]})^{i_{1}}\cdots(\partial_{n}^{[p]})^{i_{n}}
\]
where $I=(i_{1},\dots,i_{n})$ is a multi-index, $\varphi_{i}\in\widehat{D}_{\mathcal{A}}^{(0)}$,
and the sum is finite. Choose elements $\psi_{i}\in D_{A}^{(0)}$
such that $\psi_{i}-\varphi_{i}\in p^{n}\cdot\widehat{D}_{\mathcal{A}}^{(0)}$
(this is possible since $\widehat{D}_{\mathcal{A}}^{(0)}$ is the
$p$-adic completion of $D_{A}^{(0)}$). Then if we set 
\[
\Phi'=\sum_{I}\psi_{i}\cdot(\partial_{1}^{[p]})^{i_{1}}\cdots(\partial_{n}^{[p]})^{i_{n}}\in D_{\mathcal{A}}^{(1)}
\]
we see that the class of $\Phi'$ in $D_{\mathcal{A}}^{(1)}/p^{n}$
maps to the class of $\Phi\in\widehat{D}_{\mathcal{A}}^{(0,1),\infty}/p^{n}$.
Thus $D_{\mathcal{A}}^{(1)}/p^{n}\to\widehat{D}_{\mathcal{A}}^{(0,1),\infty}/p^{n}$
is onto and therefore an isomorphism, and the completion result follows
by taking the inverse limit. 

Finally, since each $\partial_{i}^{[p]}=\partial_{i}^{p}/p!$ is contained
in $\widehat{D}_{\mathcal{A}}^{(0)}[p^{-1}]$, we must have $\widehat{D}_{\mathcal{A}}^{(0)}\subset D_{\mathcal{A}}^{(0,1),\infty}\subset\widehat{D}_{\mathcal{A}}^{(0)}[p^{-1}]$,
so that $\widehat{D}_{\mathcal{A}}^{(0,1),\infty}[p^{-1}]\tilde{=}\widehat{D}_{\mathcal{A}}^{(0)}[p^{-1}]$. 
\end{proof}
\begin{cor}
$\widehat{D}_{\mathcal{A}}^{(0,1)}$ is a (left and right) noetherian
ring, which is finitely generated as an algebra over $\widehat{D}_{\mathcal{A}}^{(0)}[f,v]$.
Therefore the sheaf $\widehat{\mathcal{D}}_{\mathfrak{X}}^{(0,1)}$
is a coherent, locally noetherian sheaf of rings which is stalk-wise
noetherian. 
\end{cor}

This follows immediately from the above. Set $\widehat{\mathcal{D}}_{\mathfrak{X}}^{(0,1),-\infty}:=\widehat{\mathcal{D}}_{\mathfrak{X}}^{(0,1)}/(v-1)\tilde{=}\widehat{\mathcal{D}}_{\mathfrak{X}}^{(0)}$,
while $\widehat{\mathcal{D}}_{\mathfrak{X}}^{(0,1),\infty}:=\widehat{\mathcal{D}}_{\mathfrak{X}}^{(0,1)}/(f-1)$
has $p$-adic completion equal to $\widehat{\mathcal{D}}_{\mathfrak{X}}^{(1)}$. 

\subsection{Generators and Relations, Local Coordinates}

In addition to the description above as via endomorphisms of $\mathcal{O}_{\mathfrak{X}}$,
it is also useful to have a more concrete (local) description of $\widehat{\mathcal{D}}_{\mathfrak{X}}^{(0,1)}$
and, especially, $\mathcal{D}_{\mathfrak{X}}^{(0,1)}/p$ . Suppose
$\mathfrak{X}=\text{Specf}(\mathcal{A})$ possesses local coordinates
as above. We'll start by describing ${\displaystyle \widehat{D}_{\mathcal{A}}^{(0,1),+}:=\bigoplus_{i=0}^{\infty}\widehat{D}_{\mathcal{A}}^{(0,1),i}}$. 
\begin{defn}
Let $M_{\mathcal{A}}$ be the free graded $\mathcal{A}$-module on
generators $\{\xi_{i}\}_{i=1}^{n}$ (in degree $0$), and $f$ and
$\{\xi_{i}^{[p]}\}_{i=1}^{n}$ (in degree $1$). Let $\mathcal{B}^{(0,1),+}$
be the quotient of the tensor algebra $T_{\mathcal{A}}(M_{\mathcal{A}})$
by the relations $[f,m]$ (for all $m\in M_{\mathcal{A}}$), $[\xi_{i},a]-\partial_{i}(a)$
(for all $i$, and for any $a\in A$), $[\xi_{i},\xi_{j}]$, $[\xi_{i}^{[p]},\xi_{j}^{[p]}]$,
$[\xi_{i}^{[p]},\xi_{j}]$ (for all $i,j$), ${\displaystyle [\xi_{i}^{[p]},a]-f\cdot\sum_{r=0}^{p-1}\frac{\partial_{i}^{p-r}}{(p-r)!}(a)\cdot\frac{\xi_{i}^{r}}{r!}}$
(for all $i$, and for any $a\in A$), $f\xi_{i}^{p}-p!\xi_{i}^{[p]}$
for all $i$, and $\xi_{i}^{p}\xi_{j}^{[p]}-\xi_{j}^{p}\xi_{i}^{[p]}$
for all $i$ and $j$. 

The algebra $\mathcal{B}^{(0,1),+}$ inherits a grading from $T_{\mathcal{A}}(M_{\mathcal{A}})$.
Let $\mathcal{C}^{(0,1),+}$ be the graded ring obtained by $p$-adically
completing each component of $\mathcal{B}^{(0,1),+}$. 
\end{defn}

Then we have 
\begin{lem}
\label{lem:Reduction-is-correct}There is an isomorphism of graded
algebras $\mathcal{C}^{(0,1),+}\tilde{\to}\widehat{D}_{\mathcal{A}}^{(0,1),+}$. 
\end{lem}

\begin{proof}
There is an evident map $T_{\mathcal{A}}(M_{\mathcal{A}})\to\widehat{D}_{\mathcal{A}}^{(0,1),+}$
which is the identity on $\mathcal{A}$ and which sends $\xi_{i}\to\partial_{i}$,
$f\to f$ and $\xi_{i}^{[p]}\to\partial_{i}^{[p]}$. Clearly this
induces a graded map $\mathcal{B}^{(0,1),+}\to\tilde{\to}\widehat{D}_{\mathcal{A}}^{(0,1),+}$.
Since each graded component of $\widehat{D}_{\mathcal{A}}^{(0,1),+}$
is $p$-adically complete, we obtain a map $\mathcal{C}^{(0,1),+}\to\tilde{\to}\widehat{D}_{\mathcal{A}}^{(0,1),+}$.
Let us show that is is an isomorphism.

We begin with the surjectivity. In degree $0$, we have that $\mathcal{B}^{(0,1),0}$
is generated by $\mathcal{A}$ and $\{\xi_{i}\}_{i=1}^{n}$ and satisfies
$[\xi_{i},a]=\partial_{i}(a)$ for all $i$ and all $a\in\mathcal{A}$.
Thus the obvious map $\mathcal{B}^{(0,1),0}\to\mathcal{D}_{\mathcal{A}}^{(0)}$
is an isomorphism, and therefore so is the completion $\mathcal{C}^{(0,1),0}\to\widehat{\mathcal{D}}_{\mathcal{A}}^{0}=\widehat{\mathcal{D}}_{\mathcal{A}}^{(0,1),0}$.
Further, we saw above (in \lemref{Basic-structure-of-D_A^(i)}) that
each $\widehat{\mathcal{D}}_{\mathcal{A}}^{(0,1),i}$ (for $i\geq0$)
is generated, as a module over $\widehat{\mathcal{D}}_{\mathcal{A}}^{0}$,
by terms of the form $\{f^{i_{0}}\partial_{1}^{[p]i_{1}}\cdots\partial_{n}^{[p]i_{n}}\}$
where $i_{0}+i_{1}\dots i_{n}=i$. By definition, $\mathcal{C}^{(0,1),i}$
is exactly the $\mathcal{C}^{(0,1),0}$-module generated by terms
of the form $\{f^{i_{0}}\xi_{1}^{[p]i_{1}}\cdots\xi_{n}^{[p]i_{n}}\}$.
Thus we see that the map surjects onto the piece of degree $i$ for
all $i$; hence the map is surjective. 

To show the injectivity, consider the graded ring ${\displaystyle \mathcal{A}[f]=\bigoplus_{i=0}^{\infty}\mathcal{A}}$.
The algebra $\widehat{D}_{\mathcal{A}}^{(0,1),+}$ acts on $\mathcal{A}[f]$
as follows: if $\Phi\in\widehat{D}_{\mathcal{A}}^{(0,1),j}$ then
$\Phi\cdot(af^{i})=\Phi(a)f^{i+j}$, where $\Phi(a)$ is the usual
action of $\Phi$ on $\mathcal{A}$, coming from the fact that $\Phi\in\widehat{D}_{\mathcal{A}}^{(1)}$.
In addition, $\mathcal{C}^{(0,1),+}$ acts on $\mathcal{A}[f]$ via
$\xi_{i}(af^{j})=\xi_{i}(a)f^{j}$ and $\xi_{i}^{[p]}(af^{j})=\partial_{i}^{[p]}(a)f^{j+1}$.
This action agrees with the composed map 
\[
\mathcal{C}^{(0,1),+}\to\widehat{D}_{\mathcal{A}}^{(0,1),+}\to\text{End}_{W(k)}(\mathcal{A}[f])
\]
where the latter map comes from the action of $\widehat{D}_{\mathcal{A}}^{(0,1),+}$
on $\mathcal{A}[f]$. We will therefore be done if we can show that
this composition is injective. 

For this, we proceed by induction on the degree $i$. When $i=0$
it follows immediately from the fact that $\mathcal{C}^{(0,1),0}\tilde{=}\widehat{D}_{\mathcal{A}}^{(0)}$.
Let $\Phi\in\mathcal{C}^{(0,1),i}$. If $\Phi$ acts as zero on $\mathcal{A}[f]$,
we will show that $\Phi\in f\cdot\mathcal{C}^{(0,1),i-1}$; the induction
assumption (and the fact that $f$ acts injectively on $\mathcal{A}[f]$)
then implies that $\Phi=0$. 

Write 
\[
\Phi=\sum_{J}\Phi_{J}(\xi_{1}^{[p]})^{j_{1}}\cdots(\xi_{n}^{[p]})^{j_{n}}-f^{i}\Psi_{0}-\sum_{s=1}^{i-1}f^{i-s}\sum_{J}\Psi_{sJ}(\xi_{1}^{[p]})^{j_{1}}\cdots(\xi_{n}^{[p]})^{j_{n}}
\]
where, in the first sum, each $J$ satisfies ${\displaystyle i=|J|=\sum_{i=1}^{n}j_{i}}$,
and in the second sum we have $|J|=i-s$, and $\Phi_{J},\Psi_{0},\Psi_{sJ}\in\mathcal{C}^{(0,1),0}\tilde{=}\widehat{D}_{\mathcal{A}}^{(0)}$.
We shall show that every term in the first sum is contained in $f\cdot\mathcal{C}^{(0,1),i-1}$. 

Expanding in terms of monomials in the $\{\xi_{i}\}$, denoted $\{\xi^{I}\}$,
we obtain an equation
\[
\Phi=\sum_{I,J}a_{J,I}\xi^{I}(\xi_{1}^{[p]})^{j_{1}}\cdots(\xi_{n}^{[p]})^{j_{n}}-f^{i}\sum_{I}b_{0,I}\xi^{I}-\sum_{s=1}^{i-1}f^{i-s}\sum_{I,J}b_{s,I,J}\xi^{I}(\xi_{1}^{[p]})^{j_{1}}\cdots(\xi_{n}^{[p]})^{j_{n}}
\]
where $a_{J,I}\to0$, $b_{0,I}\to0$, and $b_{s,I,J}\to0$ (in the
$p$-adic topology on $\mathcal{A}$) as $|I|\to\infty$. For any
multi-index $J=(j_{1},\dots,j_{n})$ let $pJ=(pj_{1},\dots,pj_{n})$.
The relations $\xi_{i}^{p}\xi_{j}^{[p]}=\xi_{j}^{p}\xi_{i}^{[p]}$
(for all $i,j$ ) in $\mathcal{C}^{(0,1),+}$ ensure that $\xi^{I}(\xi_{1}^{[p]})^{j_{1}}\cdots(\xi_{n}^{[p]})^{j_{n}}=\xi^{I'}(\xi_{1}^{[p]})^{j'_{1}}\cdots(\xi_{n}^{[p]})^{j'_{n}}$
whenever $I+pJ=I'+pJ'$ and $|J|=|J'|$. Since, in the sum ${\displaystyle \sum_{I,J}a_{J,I}\xi^{I}(\xi_{1}^{[p]})^{j_{1}}\cdots(\xi_{n}^{[p]})^{j_{n}}}$,
we have $|J|=i$ for all $J$, we may collect terms together and assume
that each multi-index $I+pJ$ is represented only once. 

Now, the fact that $\Phi$ acts as zero on $\mathcal{A}[f]$ implies
that the differential operators ${\displaystyle \sum_{I,J}a_{J,I}\partial^{I}(\partial_{1}^{[p]})^{j_{1}}\cdots(\partial_{n}^{[p]})^{j_{n}}}$
and ${\displaystyle \sum_{I}b_{0,I}\partial^{I}+\sum_{s=1}^{i-1}\sum_{I,J}b_{s,I,J}\partial^{I}(\partial_{1}^{[p]})^{j_{1}}\cdots(\partial_{n}^{[p]})^{j_{n}}}$
act as the same endomorphism on $\mathcal{A}$. Therefore, for each
$a_{J,I}$ which is nonzero, we have 
\[
a_{J,I}\partial^{I}(\partial_{1}^{[p]})^{j_{1}}\cdots(\partial_{n}^{[p]})^{j_{n}}=\sum_{I'=I+pJ}b_{0,I'}\partial^{I'}+\sum_{s=1}^{i-1}\sum_{I'+pJ'=I+pJ}b_{s,I',J'}\partial^{I'}(\partial_{1}^{[p]})^{j'_{1}}\cdots(\partial_{n}^{[p]})^{j'_{n}}
\]
Now, after inverting $p$, and using $\partial_{i}^{[p]}=\partial_{i}^{p}/p!$
inside $\widehat{\mathcal{D}}_{\mathcal{A}}^{(1)}[p^{-1}]$, we obtain
the equation 
\[
\frac{a_{J,I}}{(p!)^{i}}=b_{0,I'}+\sum_{s=1}^{i-1}\sum_{I',J'}\frac{b_{s,I',J'}}{(p!)^{s}}
\]
which implies $a_{J,I}\in p\cdot\mathcal{A}$. But we have the relation
$f\xi_{i}^{p}-p!\xi_{i}^{[p]}$ in $\mathcal{C}^{(0,1),+}$; i.e.,
$p\xi_{i}^{[p]}\in f\cdot\mathcal{C}^{(0,1),+}$. Therefore $a_{J,I}\in p\cdot\mathcal{A}$
implies $a_{J,I}\xi^{I}(\xi_{1}^{[p]})^{j_{1}}\cdots(\xi_{n}^{[p]})^{j_{n}}\in f\cdot\mathcal{C}^{(0,1),i-1}$.
Since this holds for all $I,J$, we see that in fact $\Phi\in f\cdot\mathcal{C}^{(0,1),i-1}$
as desired. 
\end{proof}
\begin{rem}
Given the isomorphism of the theorem, from now on, we shall denote
$\xi_{i}$ by $\partial_{i}$ and $\xi_{i}^{[p]}$ by $\partial_{i}^{[p]}$
inside $\mathcal{C}^{(0,1),+}$. 
\end{rem}

Next, we have
\begin{lem}
\label{lem:linear-independance-over-D_0-bar} Suppose that $\{\Phi_{sJ}\}$
are elements of $\widehat{D}_{\mathcal{A}}^{(0)}$, and suppose that,
for some $i\geq1$, we have
\[
\sum_{s=0}^{i-1}\sum_{|J|=s}f^{i-s}\Phi_{sJ}(\partial_{1}^{[p]})^{j_{1}}\cdots(\partial_{n}^{[p]})^{j_{n}}\in p\cdot\widehat{D}_{\mathcal{A}}^{(0,1),i}
\]
in $\widehat{D}_{\mathcal{A}}^{(0,1),+}$. Then each $\Phi_{sJ}$
is contained in the right ideal generated by $\{\partial_{1}^{p},\dots,\partial_{n}^{p}\}$
and $p$. 
\end{lem}

\begin{proof}
As in the previous proof we may expand the $\Phi_{0}$ and $\Phi_{s,J}$
in terms of the $\{\partial^{I}\}$ to obtain 
\begin{equation}
\Phi=\sum_{s=0}^{i-1}f^{i-s}\sum_{I,J,|J|=s}b_{s,I,J}\partial^{I}(\partial_{1}^{[p]})^{j_{1}}\cdots(\partial_{n}^{[p]})^{j_{n}}\label{eq:first-form-for-phi}
\end{equation}
where $b_{s,I,J}\to0$ as $|I|\to\infty$. Om the other hand, since
$\Phi\in p\cdot\widehat{D}_{\mathcal{A}}^{(0,1),i}$, and $\widehat{D}_{\mathcal{A}}^{(0,1),i}$
is generated over $\widehat{D}_{\mathcal{A}}^{(0)}$ by $\{f^{i-s}(\partial_{1}^{[p]})^{j_{1}}\cdots(\partial_{n}^{[p]})^{j_{n}}\}_{0\leq s\leq i,|J|=s}$,
we also obtain 
\begin{equation}
\Phi=\sum_{s=0}^{i}f^{i-s}\sum_{I,J,|J|=s}a_{s,I,J}\partial^{I}(\partial_{1}^{[p]})^{j_{1}}\cdots(\partial_{n}^{[p]})^{j_{n}}\label{eq:second-form-for-phi}
\end{equation}
where each $a_{0I}$ and $a_{s,I,J}$ are contained in $p\cdot\widehat{D}_{\mathcal{A}}^{(0)}$,
and $a_{0I}\to0$, $a_{s,I,J}\to0$ as $|I|\to\infty$. 

For a multi-index $K$, let $\tilde{K}$ denote the multi-index $(\tilde{k}_{1},\dots,\tilde{k}_{n})$
such that $\tilde{k}_{i}\leq k_{i}$ for all $i$ and such that $\tilde{k}{}_{i}$
is the greatest multiple of $p$ less than or equal $k_{i}$. Write
$\tilde{K}=p\tilde{J}$, and $K=\tilde{I}+p\tilde{J}$. Then if $K=I'+pJ'$
for some $I'\neq\tilde{I}$ and $J'\neq\tilde{J}$, we must have $j_{m}<\tilde{j}_{m}$
for some $m$; which implies that $\partial^{I'}$ is contained in
the right ideal generated by $\partial_{m}^{p}$. Since $f\cdot\partial_{i}^{p}=p!\partial_{i}^{[p]}$,
we obtain 
\[
f^{i-s}b_{s,I',J'}\partial^{I'}(\partial_{1}^{[p]})^{j'_{1}}\cdots(\partial_{n}^{[p]})^{j'_{n}}=f^{i-s-1}b_{s-1,I'',J''}\partial^{I''}(\partial_{1}^{[p]})^{j''_{1}}\cdots(\partial_{n}^{[p]})^{j''_{n}}
\]
where $I''+pJ''=I'+pJ'=K$, with $j''_{i}=j_{i}'+1$, and $b_{s-1,I'',J''}\in p\cdot\mathcal{A}$.
Therefore each such term is in the right ideal generated by $\{\partial_{i}^{p}\}$
and is contained in $p\cdot\widehat{D}_{\mathcal{A}}^{(0,1),i}$,
and so we may subtract each of these terms from $\Phi$ without affecting
the statement. 

Thus we may assume that each nonzero $b_{s,I,J}$ in \eqref{first-form-for-phi}
is of the form $\tilde{I}+p\tilde{J}$ as above, and so there is only
one nonzero $b_{s,I,J}$ for each multi-index $K$. 

Now, comparing the actions of each of the expressions \eqref{first-form-for-phi}
and \eqref{second-form-for-phi} on $\mathcal{A}[f]$, we obtain,
for each multi-index $K$, the equality 
\[
b_{s\tilde{,I},\tilde{J}}=\sum_{I+pJ=K}\sum_{s}a_{s,I,J}
\]
and since each $a_{s,I,J}\in p\cdot\mathcal{A}$, we see $b_{s\tilde{,I},\tilde{J}}\in p\cdot\mathcal{A}$.
Since this is true for all $b_{s\tilde{,I},\tilde{J}}$ the result
follows. 
\end{proof}
Using these results, we can give a description of $\widehat{D}_{\mathcal{A}}^{(0,1),+}/p:=D_{A}^{(0,1),+}.$
Let $I$ be the two-sided ideal of $D_{A}^{(0)}:=\widehat{D}_{\mathcal{A}}^{(0)}/p$
generated by $\mathcal{Z}(D_{A}^{(0)})^{+}$, the positive degree
elements of the center\footnote{The center of $D_{A}^{(0)}$ is a graded algebra via the isomorphism
$\mathcal{Z}(D_{A}^{(0)})\tilde{=}A^{(1)}[\partial_{1}^{p},\dots,\partial_{n}^{p}]$,
the degree of each $\partial_{i}^{p}$ is $1$}, and let $\overline{D_{A}^{(0)}}=D_{A}^{(0)}/I$. 
\begin{thm}
\label{thm:Local-Coords-for-D+} Let ${\displaystyle D_{A}^{(0,1),+}=\bigoplus_{i=0}^{\infty}D_{A}^{(0,1),i}}$
be the decomposition according to grading. Then each $D_{A}^{(0,1),i}$
is a module over $D_{A}^{(0)}=D_{A}^{(0,1),0}$, and
\[
D_{A}^{(0,1),i}=f\cdot D_{A}^{(0,1),i-1}\oplus\sum_{|J|=i}D_{A}^{(0)}\cdot(\partial_{1}^{[p]})^{j_{1}}\cdots(\partial_{n}^{[p]})^{j_{n}}
\]
as $D_{A}^{(0)}$-modules. Further, $f\cdot D_{A}^{(0,1),i-1}$ is
free over $\overline{D_{A}^{(0)}}$, and the module\linebreak{}
 ${\displaystyle \sum_{|J|=i}D_{A}^{(0)}\cdot(\partial_{1}^{[p]})^{j_{1}}\cdots(\partial_{n}^{[p]})^{j_{n}}}$
is isomorphic, as a $D_{A}^{(0)}$-module, to $I^{i}$, via the map
which sends $(\partial_{1}^{[p]})^{j_{1}}\cdots(\partial_{n}^{[p]})^{j_{n}}$
to $\partial_{1}^{p}{}^{j_{1}}\cdots\partial_{n}^{p}{}^{j_{n}}$.
In particular, on each $D_{A}^{(0,1),i}$ we have $\text{ker}(f)=\text{im}(v)$
and $\text{im}(f)=\text{ker}(v)$. 
\end{thm}

\begin{proof}
Let $i\geq1$. By definition $D_{A}^{(0,1),i}$ is generated, over
$D_{A}^{(0)}$ terms of the form \linebreak{}
$\{f^{i-s}\{(\partial_{1}^{[p]})^{j_{1}}\cdots(\partial_{n}^{[p]})^{j_{n}}\}_{|J|=s}$;
and so it is also generated by $f\cdot D_{A}^{(0,1),i-1}$ and $\{(\partial_{1}^{[p]})^{j_{1}}\cdots(\partial_{n}^{[p]})^{j_{n}}\}_{|J|=i}$
. Suppose we have an equality of the form 
\[
\sum_{J=i}\bar{\Phi}_{J}(\partial_{1}^{[p]})^{j_{1}}\cdots(\partial_{n}^{[p]})^{j_{n}}=\sum_{s=0}^{i-1}f^{i-s}\sum_{J}\bar{\Psi}_{sJ}(\partial_{1}^{[p]})^{j_{1}}\cdots(\partial_{n}^{[p]})^{j_{n}}
\]
in $D_{A}^{(0,1),i}$ (here, $\bar{\Phi}_{J},\bar{\Psi}_{sJ}$ are
in $D_{A}^{(0)}$). Choosing lifts to $\Phi_{J},\Psi_{J}\in\widehat{D}_{\mathcal{A}}^{(0)}$
yields 
\[
\sum_{J=i}\Phi_{J}(\partial_{1}^{[p]})^{j_{1}}\cdots(\partial_{n}^{[p]})^{j_{n}}-\sum_{s=0}^{i-1}f^{i-s}\sum_{J}\Psi_{sJ}(\partial_{1}^{[p]})^{j_{1}}\cdots(\partial_{n}^{[p]})^{j_{n}}\in p\cdot D_{\mathcal{A}}^{(0,1),i}\subset f\cdot D_{\mathcal{A}}^{(0,1),i-1}
\]
(the last inclusion follows from $(p!)\partial_{i}^{[p]}=f\partial_{i}^{p}$);
and so (the proof of) \lemref{Reduction-is-correct} now forces $\Phi_{J}\in p\cdot\widehat{D}_{\mathcal{A}}^{(0,1),i}$
for all $J$ so $\bar{\Phi}_{J}=0$ as desired. The isomorphism of
${\displaystyle \sum_{|J|=i}D_{A}^{(0)}\cdot(\partial_{1}^{[p]})^{j_{1}}\cdots(\partial_{n}^{[p]})^{j_{n}}}$
with $I^{i}$ is given by the reduction of the morphism $p^{i}\cdot$
on $D_{\mathcal{A}}^{(0,1),i}$, and \lemref{linear-independance-over-D_0-bar}
yields that $f\cdot D_{A}^{(0,1),i-1}$ is free over $\overline{D}_{A}^{(0)}$;
a basis is given by $\{f^{i-|J|}(\partial_{1}^{[p]})^{j_{1}}\cdots(\partial_{n}^{[p]})^{j_{n}}\}_{0\leq|J|\leq i}$.
The last statement follows directly from this description.
\end{proof}
We now use this to describe the entire graded algebra $D_{A}^{(0,1)}:=D_{\mathcal{A}}^{(0,1)}/p$. 
\begin{cor}
\label{cor:Local-coords-over-A=00005Bf,v=00005D} The algebra $D_{A}^{(0,1)}$
is a free graded module over $D(A)$, with a basis given by the set
$\{\partial^{I}(\partial^{[p]})^{J}\}$, where $I=(i_{1},\dots,i_{n})$
is a multi-index with $0\leq i_{j}\leq p-1$ for all $j$ and $J$
is any multi-index with entries $\geq0$. 
\end{cor}

\begin{proof}
By the previous corollary, any element of $D_{A}^{(0,1),+}$ can be
written as a finite sum
\[
\sum_{I,J}a_{I,J}\partial^{I}(\partial^{[p]})^{J}
\]
where $a_{I,J}\in A[f]$ and $I$ and $J$ are arbitrary multi-indices.
As any element in $D_{A}^{(0,1),-}$ is a sum of the form 
\[
\sum_{i=1}^{m}v^{i}\sum_{J}b_{i,J}(\partial)^{J}
\]
we see that in fact any element of $D_{A}^{(0,1)}$ can be written
as a finite sum
\[
\sum_{I,J}a_{I,J}\partial^{I}(\partial^{[p]})^{J}
\]
where $a_{I,J}\in A[f,v]$ and $I$ and $J$ are arbitrary multi-indices.
Iteratively using the relations $(p-1)!\partial_{i}^{p}=v\partial_{i}^{[p]}$
we see that we may suppose that each entry of $I$ is contained in
$\{0,\dots,p-1\}$; this shows that these elements span.

To see the linear independence, suppose we have 
\begin{equation}
\sum_{I,J}a_{I,J}\partial^{I}(\partial^{[p]})^{J}=0\label{eq:lin-dep}
\end{equation}
where now each entry of $I$ is contained in $\{0,\dots,p-1\}$. Write
\[
a_{I,J}=\sum_{s\geq0}f^{s}a_{I,J,s}+\sum_{t<0}v^{t}a_{I,J,t}
\]
We have 
\[
a_{I,J}\partial^{I}(\partial^{[p]})^{J}=\sum_{s\geq0}f^{s}a_{I,J,s}\partial^{I}(\partial^{[p]})^{J}+\sum_{J',t}a_{I,J,t}\partial^{I+pJ'}(\partial^{[p]})^{J''}+\sum_{t<0}v^{-t-|J|}a_{I,J,t}\partial^{I+pJ}
\]
where, in the middle sum, $t$ satisfies $0<-t\leq|J|$; for each
such $t$ we pick $J'$ such that $J'+J''=J$ and $|J'|=-t$. Now,
the previous corollary gives an isomorphism 
\[
D_{A}^{(0,1),i}\tilde{=}D_{A}^{(0)}/I^{i}\oplus I^{i}
\]
where $I=C^{1}(D_{A}^{(0)})$, for all $i\geq0$; this in fact holds
for all $i\in\mathbb{Z}$ if we interpret $I^{i}=D_{A}^{(0)}$ for
$i<0$. This implies that the elements $\{f^{s}\partial^{I}(\partial^{[p]})^{J},\partial^{I+pJ'}(\partial^{[p]})^{J''},v^{-t-|J|}\partial^{I+pJ}\}$
where $I,J$ are multi-indices with each entry of $I$ is contained
in $\{0,\dots,p-1\}$, are linearly independent over $A$ (one may
look at each degree separately and use the above description). Thus
\eqref{lin-dep} implies $a_{I,J,s}=0=a_{I,J,t}$ for all $I,J,s,t$;
hence each $a_{I,J}=0$ as desired. 
\end{proof}
Finally, let us apply this result to describe the finite order operators
$D_{\mathcal{A}}^{(0,1)}$. Namely, we have 
\begin{cor}
\label{cor:Each-D^(i)-is-free}The algebra $D_{\mathcal{A}}^{(0,1)}$
is free over $D(\mathcal{A})$, with a basis given by the set $\{\partial^{I}(\partial^{[p]})^{J}\}$,
where $I=(i_{1},\dots,i_{n})$ is a multi-index with $0\leq i_{j}\leq p-1$
for all $j$ and $J$ is any multi-index with entries $\geq0$. 
\end{cor}

\begin{proof}
By the previous result, the images of these elements in $D_{A}^{(0,1)}=D_{\mathcal{A}}^{(0,1)}/p$
form a basis over $D(\mathcal{A})$. Since $D_{\mathcal{A}}^{(0,1)}$
is $p$-torsion-free, and $D(\mathcal{A})$ is $p$-adically separated,
it follows directly that these elements are linearly independent over
$D(\mathcal{A})$. The fact that they span follows (as in the previous
proof) from \lemref{Basic-structure-of-D_A^(i)}. 
\end{proof}

\subsection{$\mathcal{D}^{(0,1)}$-modules over $k$}

Now let $X$ be an arbitrary smooth variety over $k$; in this subsection
we make no assumption that there is a lift of $X$; however, if $U\subset X$
is an affine, there is a always a lift of $U$ to as smooth formal
scheme $\mathfrak{U}$. In this section we will construct a sheaf
of algebras $\mathcal{D}_{X}^{(0,1)}$ such that, on each open affine
$U$ which possesses local coordinates, we have $\mathcal{D}_{X}^{(0,1)}(U)=\widehat{\mathcal{D}}_{\mathfrak{U}}^{(0,1)}(\mathfrak{U})/p$. 

There is a natural action of $\mathcal{D}_{X}^{(0)}$ on $\mathcal{O}_{X}$;
inducing a map $\mathcal{D}_{X}^{(0)}\to\mathcal{E}nd_{k}(\mathcal{O}_{X})$,
and we let $\overline{\mathcal{D}_{X}^{(0)}}\subset\mathcal{E}nd_{k}(\mathcal{O}_{X})$
denote the image of $\mathcal{D}_{X}^{(0)}$ under this map. It is
a quotient algebra of $\mathcal{D}_{X}^{(0)}$, and a quick local
calculation gives 
\begin{lem}
\label{lem:Basic-description-of-D-bar} Let $U\subset X$ be an open
subset, which possesses local coordinates $\{x_{1},\dots,x_{n}\}$,
and let $\{\partial_{1},\dots,\partial_{n}\}$ denote derivations
satisfying $\partial_{i}(x_{j})=\delta_{ij}$. Then the kernel of
the map $\mathcal{D}_{X}^{(0)}(U)\to\mathcal{E}nd_{k}(\mathcal{O}_{X}(U))$
is the two sided ideal $\mathcal{I}$ generated by $\{\partial_{1}^{p},\dots,\partial_{n}^{p}\}$.
The image consists of differential operators of the form ${\displaystyle \sum a_{I}\partial^{I}}$
where the sum ranges over multi-indices $I=(i_{1},\dots,i_{n})$ for
which $0\leq i_{j}<p$ (for all $j$), the $a_{I}\in\mathcal{O}_{X}(U)$,
and $\partial^{I}=\partial_{1}^{i_{1}}\cdots\partial_{n}^{i_{n}}$. 
\end{lem}

In particular, if $U=\text{Spec}(A)$ then we have $\overline{\mathcal{D}_{X}^{(0)}}(U)=\overline{D_{A}^{(0)}}$
as defined in the previous section. 

Now let $\mathcal{D}iff_{X}^{\leq n}$ denote the sheaf of differential
operators of order $\leq n$ on $X$. This is a sub-sheaf of $\mathcal{E}nd_{k}(\mathcal{O}_{X})$. 
\begin{defn}
\label{def:L}1) Let $\tilde{\mathcal{D}iff}_{X}^{\leq p}$ denote
the sub-sheaf of $\mathcal{D}iff_{X}^{\leq p}$ defined by the following
condition: a local section $\delta$ of $\mathcal{D}iff_{X}^{\leq p}$
is contained in $\tilde{\mathcal{D}iff}_{X}^{\leq p}$ if, for any
local section $\Phi\in\overline{\mathcal{D}_{X}^{(0)}}$, we have
$[\delta,\Phi]\in\overline{\mathcal{D}_{X}^{(0)}}$ (Here, the bracket
is the natural Lie bracket on $\mathcal{E}nd_{k}(\mathcal{O}_{X})$
coming from the algebra structure). 

2) We define the sub-sheaf $\mathfrak{l}_{X}\subset\mathcal{D}iff_{X}$
to be $\tilde{\mathcal{D}iff}_{X}^{\leq p}+\overline{\mathcal{D}_{X}^{(0)}}$. 
\end{defn}

The sections in $\mathfrak{l}_{X}$ can easily be identified in local
coordinates. Suppose $U=\text{Spec}(A)$ possess local coordinates
$\{x_{1},\dots,x_{n}\}$, and coordinate derivations $\{\partial_{1},\dots,\partial_{n}\}$. 
\begin{prop}
\label{lem:O^p-action} Let $U\subset X$ be an open subset as above.
Then we have 
\[
\mathfrak{l}_{X}(U)=\bigoplus_{i=1}^{n}\mathcal{O}_{U}^{p}\cdot\partial_{i}^{[p]}\oplus\overline{\mathcal{D}_{X}^{(0)}}(U)
\]
In particular, $\mathfrak{l}_{X}$ is a sheaf of $\mathcal{O}_{X}^{p}$-modules
(via the left action of $\mathcal{O}_{X}^{p}$ on $\mathcal{E}nd_{k}(\mathcal{O}_{X})$). 
\end{prop}

\begin{proof}
First, let's show that the displayed sum is contained in $\mathfrak{l}_{X}(U)$.
By definition $\overline{\mathcal{D}_{X}^{(0)}}(U)\subset\mathfrak{l}_{X}(U)$.
Let $\Phi\in\overline{\mathcal{D}_{X}^{(0)}}(U)$, and write ${\displaystyle \Phi=\sum_{I}a_{I}\partial^{I}}$
as in \lemref{Basic-description-of-D-bar}. Then, for any $g\in O_{X}(U)$,
we have 
\[
[g^{p}\partial_{i}^{[p]},\sum_{I}a_{I}\partial^{I}]=\sum_{I}[g^{p}\partial_{i}^{[p]},a_{I}\partial^{I}]=\sum_{I}[g^{p}\partial_{i}^{[p]},a_{I}]\partial^{I}+\sum_{I}a_{I}[g^{p}\partial_{i}^{[p]},\partial^{I}]
\]
Now, ${\displaystyle [g^{p}\partial_{i}^{[p]},a_{I}]\partial^{I}=g^{p}[\partial_{i}^{[p]},a_{I}]\partial^{I}=g^{p}\sum_{r=0}^{p-1}\partial_{i}^{[p-r]}(a_{I})\partial_{i}^{[r]}\cdot\partial^{I}\in\overline{\mathcal{D}_{X}^{(0)}}(U)}$.
Further, $a_{I}[g^{p}\partial_{i}^{[p]},\partial^{I}]=a_{I}g^{p}[\partial_{i}^{[p]},\partial^{I}]+a_{I}[g^{p},\partial^{I}]\partial_{i}^{[p]}=0$.
Thus we see that each $g^{p}\partial_{i}^{[p]}\in\mathfrak{l}_{X}(U)$,
and the right hand side is contained in the left. 

For the converse, let $\Phi\in\mathcal{D}iff_{X}^{\leq p}(U)$. It
may be uniquely written as 
\[
\Phi=\sum_{i=1}^{n}a_{i}\partial_{i}^{[p]}+\sum_{I}a_{I}\partial^{I}
\]
where $a_{i}$ and $a_{I}$ are in $\mathcal{O}_{X}(U)$, and the
second sum ranges over multi-indices $I=(i_{1},\dots,i_{n})$ with
each $i_{j}<p$ and so that $i_{1}+\dots+i_{n}\leq p$. For any coordinate
derivation $\partial_{j}$, we have 
\[
[\Phi,\partial_{j}]=-(\sum_{i=1}^{n}\partial_{j}(a_{i})\partial_{i}^{[p]}+\sum_{I}\partial_{j}(a_{I})\partial^{I})
\]
For this to be contained in $\overline{\mathcal{D}_{X}^{(0)}}(U)$,
we must have $\partial_{j}(a_{i})=0$ for all $i$. Therefore, if
$[\Phi,\partial_{j}]\in\overline{\mathcal{D}_{X}^{(0)}}(U)$ for all
$j$, we must have $\partial_{j}(a_{i})=0$ for all $j$ (and all
$i$), which means that each $a_{i}\in\mathcal{O}_{X}(U)^{p}$. Therefore,
if $\Phi\in\tilde{\mathcal{D}iff}_{X}^{\leq p}(U)$, then $\Phi$
must be contained in ${\displaystyle \bigoplus_{i=1}^{n}\mathcal{O}_{U}^{p}\cdot\partial_{i}^{[p]}\oplus\overline{\mathcal{D}_{X}^{(0)}}(U)}$,
and the result follows. 
\end{proof}
\begin{cor}
$\mathfrak{l}_{X}$ is a sheaf of Lie subalgebras of $\mathcal{E}nd_{k}(\mathcal{O}_{X})$. 
\end{cor}

\begin{proof}
As the question is local, it suffices to prove that $\mathfrak{l}_{X}(U)$
is closed under the bracket for a neighborhood $U$ which possesses
local coordinates. We use the description of the previous lemma. So
we must show that all brackets of the form 
\[
[g^{p}\partial_{i}^{[p]},h^{p}\partial_{j}^{[p]}]
\]
and 
\[
[g^{p}\partial_{i}^{[p]},\sum_{I}a_{I}\partial^{I}]
\]
are contained in $\mathfrak{l}_{X}(U)$. Here the notation is as above;
so $g,h\in\mathcal{O}_{X}(U)$, and $I=(i_{1},\dots,i_{n})$ is a
multi-index with each $i_{j}<p$. In fact, we already showed that
${\displaystyle [g^{p}\partial_{i}^{[p]},\sum_{I}a_{I}\partial^{I}]\in\overline{\mathcal{D}_{X}^{(0)}}(U)}$
in the course of the proof of the previous lemma. So we are left to
analyze the first bracket. Now, 
\[
[g^{p}\partial_{i}^{[p]},h^{p}\partial_{j}^{[p]}]=h^{p}[g^{p}\partial_{i}^{[p]},\partial_{j}^{[p]}]+[g^{p}\partial_{i}^{[p]},h^{p}]\partial_{j}^{[p]}
\]
\[
=h^{p}[g^{p},\partial_{j}^{[p]}]\partial_{i}^{[p]}+g^{p}[\partial_{i}^{[p]},h^{p}]\partial_{j}^{[p]}
\]
and we have 
\[
[\partial_{i}^{[p]},h^{p}]=\sum_{r=0}^{p-1}\partial_{i}^{[p-r]}(h^{p})\partial_{i}^{[r]}=\partial_{i}^{[p]}(h^{p})
\]
and similarly, $[g^{p},\partial_{j}^{[p]}]=-\partial_{j}^{[p]}(g^{p})$.
It is a well-known fact that $\partial_{i}^{[p]}(h^{p})=(\partial_{i}(h))^{p}$
(for the sake of completeness, we include a proof directly below).
It follows immediately that $[g^{p}\partial_{i}^{[p]},h^{p}\partial_{j}^{[p]}]\in\mathfrak{l}_{X}(U)$,
and the corollary follows.

To prove that $\partial_{i}^{[p]}(h^{p})=(\partial_{i}(h))^{p}$,
recall the following formula for Hasse-Schmidt derivations acting
on powers: 
\[
\partial_{i}^{[j]}(h^{m})=\sum_{i_{1}+\dots+i_{m}=j}\partial_{i}^{[i_{1}]}(h)\cdots\partial_{i}^{[i_{m}]}(h)
\]
which is easily checked by induction. Put $m=j=p$ in the formula.
The set 
\[
\{(i_{1},\dots,i_{p})\in\mathbb{Z}_{\geq0}^{p}|i_{1}+\dots+i_{p}=p\}
\]
is acted upon by the symmetric group $S_{p}$, and, after grouping
like terms together, we see that each term $\partial_{i}^{[i_{1}]}(h)\cdots\partial_{i}^{[i_{m}]}(h)$
in the sum is repeated $N$ times, where $N$ is the size of the $S_{p}$
orbit of $(i_{1},\dots,i_{p})$. There is a unique orbit of size $1$,
namely $i_{1}=i_{2}=\cdots=i_{p}=1$; and for every other orbit, the
size is a number of the form ${\displaystyle \frac{p!}{c_{1}!\cdots c_{r}!}}$
for some numbers $c_{i}<p$ such that $\sum c_{i}=p$ . Any such is
divisible by $p$, and so all these terms are zero in the sum since
we are in characteristic $p$. Thus we obtain 
\[
\partial_{i}^{[p]}(h^{p})=\partial_{i}^{[1]}(h)\cdots\partial_{i}^{[1]}(h)=(\partial_{i}(h))^{p}
\]
as claimed. 
\end{proof}
Now we will build the ring $\mathcal{D}_{X}^{(0,1)}$ out of $\mathcal{D}_{X}^{(0)}$
and $\mathfrak{l}_{X}$, in a manner quite analogous to the way in
which $\mathcal{D}_{X}^{(0)}$ is built out of $\mathcal{O}_{X}$
and $\mathcal{T}_{X}$ as an enveloping algebra of a Lie algebroid;
in the classical case, this construction is given in \cite{key-44}
(for schemes) and \cite{key-46} (for rings). Our construction is
similar in spirit to these works (c.f. also \cite{key-45}). 
\begin{defn}
\label{def:D-=00005Cplus-L}Let $f:\mathcal{D}_{X}^{(0)}\to\mathfrak{l}_{X}$
denote the map $\mathcal{D}_{X}^{(0)}\to\overline{\mathcal{D}_{X}^{(0)}}\subset\mathfrak{l}_{X}$.
Define the sheaf 
\[
\mathfrak{L}_{X}:=\mathcal{D}_{X}^{(0)}\oplus\bigoplus_{i=1}^{\infty}\mathfrak{l}_{X}=\bigoplus_{i=0}^{\infty}\mathfrak{L}_{X}^{i}
\]
and make it into a graded $k[f]$-module by letting $f:\mathfrak{l}_{X}\to\mathfrak{l}_{X}$
be the identity in degrees $\geq1$; thus any homogenous element in
degree $i\geq1$ can be uniquely written $f^{i-1}\Psi$ for some $\Psi\in\mathfrak{l}_{X}$. 

For local sections $\Phi\in\mathcal{D}_{X}^{(0)}$ and $f^{i-1}\Psi\in\mathfrak{L}_{X}^{i}$,
define $[\Phi,f^{i-1}\Psi]:=f^{i-1}[f\circ\Phi,\Psi]$ where on the
right we have the bracket in $\mathfrak{l}_{X}$. We then make $\mathfrak{L}_{X}$
into a sheaf of graded Lie algebras by setting $[f^{i-1}\Psi_{1},f^{j-1}\Psi_{2}]=f^{i+j-1}[\Psi_{1},\Psi_{2}]$
where $\{\Psi_{1},\Psi_{2}\}$ are local sections of $\mathfrak{l}_{X}$.
The Jacobi identity can be verified by a direct computation. 
\end{defn}

Next we introduce the action of $v$:
\begin{lem}
\label{lem:Construction-of-v-1}There is a unique endomorphism $v$
of $\mathfrak{L}_{X}$ satisfying $v(\mathcal{D}_{X}^{(0)})=0$ and,
upon restriction to an open affine $U$ which possesses local coordinates,
$v(\partial_{i}^{[p]})=(p-1)!\partial_{i}^{p}$ for coordinate derivations
$\{\partial_{i}\}_{i=1}^{n}$. This endomorphism vanishes on $f(\mathcal{D}_{X}^{(0)})$,
and on ${\displaystyle \bigoplus_{i=2}^{\infty}\mathfrak{L}_{X}^{i}}$. 
\end{lem}

\begin{proof}
Since $v(\mathcal{D}_{X}^{(0)})=0$ it suffices to define $v$ on
$\mathfrak{l}_{X}$. For any $\Phi\in\mathfrak{l}_{X}$, the action
of $\Phi$ preserves $\mathcal{O}_{X}^{p}$, and the restriction of
$\Phi$ to $\mathcal{O}_{X}^{p}\tilde{=}\mathcal{O}_{X^{(1)}}$ is
a derivation on $\mathcal{O}_{X}^{p}$ (this follows immediately from
\lemref{O^p-action} and the fact that $\partial_{i}^{[p]}(g^{p})=(\partial_{i}(g))^{p}$).
Further this derivation is trivial iff $\Phi\in f(\mathcal{D}_{X}^{(0)})\subset\mathfrak{l}_{X}$. 

On the other hand, since $k$ is perfect there is a natural isomorphism
between the sheaf of derivations on $\mathcal{O}_{X^{(1)}}$ and the
sheaf of derivations on $\mathcal{O}_{X}$, given as follows: if $\partial'$
is a (local) derivation on $\mathcal{O}_{X^{(1)}}$, then we can define
a derivation of $\mathcal{O}_{X}$ by $\partial(g)=(\partial'(g^{p}))^{1/p}$;
this is possible precisely by the identification $\mathcal{O}_{X}^{p}\tilde{=}\mathcal{O}_{X^{(1)}}$.
This association is easily checked to be an isomorphism using local
coordinates; let's name it $\tau:\text{Der}(\mathcal{O}_{X^{(1)}})\to\text{Der}(\mathcal{O}_{X})$. 

Further, there is a map $\sigma:\text{Der}(\mathcal{O}_{X})\to\mathcal{Z}(\mathcal{D}_{X}^{(0)})$
defined by $\partial\to\partial^{p}-\partial^{[p]}$, where $\partial^{[p]}$
is the $p$th iterate of the derivation (c.f. \cite{key-3}, chapter
1). In particular this map takes $\partial_{i}\to\partial_{i}^{p}$
if $\partial_{i}$ is a coordinate derivation as above. 

Now we define $v(\Phi)=(p-1)!\cdot\sigma\circ\tau(\Phi|_{\mathcal{O}_{X^{(1)}}})$;
by the above discussion this satisfies all the properties of the lemma. 
\end{proof}
Now we proceed to the definition of $\mathcal{D}_{X}^{(0,1),+}$.
By the functoriality of the enveloping algebra construction, we can
now form the pre-sheaf of enveloping algebras $\mathcal{U}(\mathfrak{L}_{X})$;
this is a pre-sheaf of graded algebras with the grading inherited
from $\mathfrak{L}_{X}$. Inside this pre-sheaf is the pre-sheaf $\mathcal{U}^{+}(\mathfrak{L}_{X})$,
which is the pre-sheaf of non-unital algebras generated by $\mathfrak{L}_{X}\subset\mathcal{U}(\mathfrak{L}_{X})$. 

For any local section $\Phi\in\mathcal{D}_{X}^{(0)}$, let $\Phi'$
denote its image in $\mathcal{U}^{+}(\mathfrak{L}_{X})$, by regarding
$\Phi\in\mathcal{D}_{X}^{(0)}\subset\mathfrak{L}_{X}$; similarly,
for a local sections $\Psi\in\mathfrak{L}_{X}$, let $\Psi'\in\mathfrak{L}_{X}\subset\mathcal{U}^{+}(\mathfrak{L}_{X})$
denote its image. 
\begin{defn}
Let $\mathcal{J}$ be the pre-sheaf of homogenous two-sided ideals
in $\mathcal{U}^{+}(\mathfrak{L}_{X})$ generated by the following
sections: for any local sections $\Phi_{1},\Phi_{2}\in\mathcal{D}_{X}^{(0)}$:
$(\Phi_{1}\cdot\Phi_{2})'-\Phi_{1}'\cdot\Phi_{2}'$ , $f\cdot\Phi_{1}'-f(\Phi_{1})'$,
$\Phi'_{1}\cdot f(\Phi'_{2})-f(\Phi_{1}')\cdot\Phi'_{2}$, $\Phi'_{1}\cdot f(\Phi'_{2})-f\cdot(\Phi_{1}'\cdot\Phi'_{2})$.
Further, if $\Psi_{1},\Psi_{2}\in\mathfrak{L}_{X}$ are any local
sections, we add the elements $\Psi_{1}'\Psi_{2}'-\Psi_{2}'\Psi_{1}'-[\Psi_{1},\Psi_{2}]'$,
as well as $g'\cdot\Psi_{1}=(g\cdot\Psi_{1})'$ for any local section
$g\in\mathcal{O}_{X}^{p}$ (the action of $\mathcal{O}_{X}^{p}$ on
$\mathfrak{l}_{X}$ is that of \lemref{O^p-action}). Finally, we
add $\Phi_{1}'\cdot\Psi'_{1}-\Phi_{2}'\cdot\Psi'_{2}$ where $\Phi_{i}$
are local sections of $\mathcal{Z}(D_{X}^{(0)})$ such that $\Phi_{1}\cdot v(\Psi_{1})=\Phi_{2}\cdot v(\Psi_{2})$. 

Define $\mathcal{D}_{X}^{(0,1),+}$ to be the sheafification of the
presheaf $\mathcal{U}^{+}(\mathfrak{L}_{X})/\mathcal{J}$. It is a
graded sheaf of algebras on $X$. 
\end{defn}

Of course, such a definition is only really useful if we can write
the algebra out explicitly in the presence of coordinates. Fortunately,
this is the case; in fact, if $U=\text{Spec}(A)$ we can compare it
with the presentation of $D_{A}^{(0,1),+}=D_{\mathcal{A}}^{(0,1),+}/p$
discussed in the previous section:
\begin{thm}
\label{thm:D-is-quasi-coherent} Let $\tilde{D}_{A}^{(0,1),+}$ be
the quasi-coherent sheaf on $U$ obtained by localizing $D_{A}^{(0,1),+}$.
This a sheaf of algebras on $U$. There is an isomorphism (of graded
sheaves of algebras) $\mathcal{D}_{X}^{(0,1),+}|_{U}\tilde{=}\tilde{D}_{A}^{(0,1),+}$.
In particular, $\mathcal{D}_{X}^{(0,1),+}$ is a quasi-coherent sheaf
of algebras on $X$, and we have $\mathcal{D}_{X}^{(0,1),0}\tilde{=}\mathcal{D}_{X}^{(0)}$. 
\end{thm}

\begin{proof}
We have the algebra $\mathcal{U}^{+}(\mathfrak{L}_{X})(U)/\mathcal{J}(U)$.
It admits a map to $D_{A}^{(0,1),+}$ as follows: by \lemref{O^p-action},
the lie algebra $\mathfrak{L}_{X}(U)$ is equal to 
\[
\mathcal{D}_{X}^{(0)}(U)\oplus\bigoplus_{i=1}^{\infty}(f^{i}(\overline{\mathcal{D}_{X}^{(0)}}(U))\oplus\bigoplus_{i=1}^{n}f^{i-1}\mathcal{O}_{X}^{p}(U)\cdot\partial_{i}^{[p]})
\]
\[
=D_{A}^{(0)}\oplus\bigoplus_{i=1}^{\infty}(f^{i}(\overline{D_{A}^{(0)}})\oplus\bigoplus_{i=1}^{n}f^{i-1}A^{p}\cdot\partial_{i}^{[p]})
\]
We map this to $D_{A}^{(0,1),+}$ via the identification of $D_{A}^{(0)}$
with $D_{A}^{(0,1),0}$, and by sending $f^{i}(\overline{D_{A}^{(0)}})$
to $f^{i}\cdot D_{A}^{(0)}\tilde{=}\overline{D_{A}^{(0)}}$ and $f^{i-1}g^{p}\partial_{i}^{[p]}$
to $f^{i-1}g^{p}\partial_{i}^{[p]}\in D_{A}^{(0,1),i}$. By sending
$f$ to $f$ we get a map of algebras $\mathcal{U}^{+}(\mathfrak{L}_{X})(U)/\mathcal{J}(U)\to D_{A}^{(0,1),+}$
(one checks the relations directly). 

Conversely, we get a map $D_{A}^{(0,1),+}\to\mathcal{U}^{+}(\mathfrak{L}_{X})(U)/\mathcal{J}(U)$
by sending $A\to A\subset\mathcal{D}_{X}^{(0)}(U)$, $\partial_{i}\to\partial_{i}\in\mathcal{D}_{X}^{(0)}(U)$,
$\partial_{i}^{[p]}\to\partial_{i}^{[p]}\in\mathfrak{l}_{X}(U)$ and
$f\to f$. Again checking the relations, this is a morphism of algebras,
and the compositions in both directions are the identity on generators.
Therefore the presheaf $U\to\mathcal{U}^{+}(\mathfrak{L}_{X})(U)/\mathcal{J}(U)$,
when restricted to open affines which admit local coordinates, agrees
with the assignment $U\to D_{A}^{(0,1),+}$. But the latter, by the
description of \thmref{Local-Coords-for-D+}, clearly agrees with
the quasi-coherent sheaf $\tilde{D}_{A}^{(0,1),+}$ on $\text{Spec}(A)$,
and the result follows. 
\end{proof}
Finally, we need to define the entire algebra $\mathcal{D}_{X}^{(0,1)}$.
This entails extending the operator $v$ to an endomorphism of all
of $\mathcal{D}_{X}^{(0,1),+}$. 
\begin{lem}
\label{lem:Construction-of-v} There is a unique $\mathcal{D}_{X}^{(0)}$-linear
endomorphism $v$ of $\mathcal{D}_{X}^{(0,1),+}$ satisfying $v(\mathcal{D}_{X}^{(0,1),i})\subset D_{X}^{(0.1),i-1}$
for all $i\geq1$ (and $v(\mathcal{D}_{X}^{(0,1),0})=0$), $v(\Phi_{1}\cdot\Phi_{2})=\Phi_{1}v(\Phi_{2})$
for all $\Phi_{1},\Phi_{2}\in\bigoplus_{i=1}^{\infty}\mathcal{D}_{X}^{(0,1),i}$,
$v(f\cdot\Phi)=0$ for all $\Phi$, and such that the restriction
of $v$ to $\mathfrak{L}_{X}$ agrees with the map $v$ constructed
in \lemref{Construction-of-v-1}. 
\end{lem}

\begin{proof}
Define $v$ on $\mathcal{D}_{X}^{(0)}\oplus\mathfrak{l}_{X}$ to be
the map constructed in \lemref{Construction-of-v-1}. The claim is
that there is a unique extension of this map to all of $\mathcal{D}_{X}^{(0,1),+}$
satisfying the conditions of the lemma. 

By the uniqueness, it is enough to check this locally. Let $U=\text{Spec}(A)$
posses local coordinates. By \thmref{Local-Coords-for-D+}, if we
set $v$ to be zero on $D_{A}^{(0,1),0}$ and $f\cdot(D_{A}^{(0,1),+})$,
and we define 
\[
v((\partial_{j}^{[p]})^{i_{j}}\cdots(\partial_{n}^{[p]})^{i_{n}})=\partial_{j}^{p}\cdot(\partial_{j}^{[p]})^{i_{j}-1}\cdots(\partial_{n}^{[p]})^{i_{n}}
\]
where $j$ is the first index such that $i_{j}\geq1$, then we have
a well-defined $D_{A}^{(0)}$-linear map satisfying all the properties
of the lemma, and which agrees with the $v$ defined above on $\mathfrak{L}_{X}(U)$.
On the other hand, $D_{A}^{(0,1),+}$ is generated as a $D_{A}^{(0)}$-module
by $D_{A}^{(0)}$, $f\cdot(D_{A}^{(0,1),+})$, and elements which
are products of $\mathfrak{L}_{X}(U)$ (again by \thmref{Local-Coords-for-D+}).
So any map which satisfies the above list of properties and equals
$v$ on $\mathfrak{L}_{X}(U)$ is equal to the one we have written
down; so the uniqueness follows as well. 
\end{proof}
Now we arrive at 
\begin{defn}
\label{def:D(0,1)}The sheaf of algebras $\mathcal{D}_{X}^{(0,1)}$
is defined as the $\mathbb{Z}$-graded sheaf of $k[v,f]$-algebras,
which as a graded sheaf is given by
\[
\bigoplus_{i=-\infty}^{-1}\mathcal{D}_{X}^{(0)}\oplus\mathcal{D}_{X}^{(0,1),+}
\]
and where we extend the action of $f$ (to an operator of degree $1$)
from $\mathcal{D}_{X}^{(0,1),+}$ to $\mathcal{D}_{X}^{(0,1)}$ by
setting $f=0$ on ${\displaystyle \bigoplus_{i=-\infty}^{-1}\mathcal{D}_{X}^{(0)}}$,
and we extend the action of $v$ (to an operator of degree $-1$)
on $\bigoplus_{i=1}^{\infty}\mathcal{D}_{X}^{(0,1),i}$ by letting
$v:\mathcal{D}_{X}^{(0,1),i}\to\mathcal{D}_{X}^{(0,1),i-1}$ be the
identity whenever $i\leq0$. The product on this algebra extends the
product on $\mathcal{D}_{X}^{(0,1),+}$ as follows: on the negative
half ${\displaystyle \bigoplus_{i=-\infty}^{0}\mathcal{D}_{X}^{(0)}}=\mathcal{D}_{X}^{(0)}\otimes_{k}k[v]$,
we use the obvious graded product. For $i\leq0$, if $\Phi\in\mathcal{D}_{X}^{(0,1),i}\tilde{=}D_{X}^{(0)}$
and $\Psi\in\mathcal{D}_{X}^{(0,1),+}$, we set 
\[
\Phi\cdot\Psi=\Phi_{0}v^{i}(\Psi)
\]
where $\Phi_{0}$ is the element $\Phi\in D_{X}^{(0)}$, now regarded
as an element of degree $0$. 
\end{defn}

From this definition and \thmref{D-is-quasi-coherent}, we see that
this is a quasicoherent sheaf of algebras, and we have an isomorphism
\[
\widehat{D}_{\mathcal{A}}^{(0,1)}/p\tilde{=}\mathcal{D}_{X}^{(0,1)}(U)
\]
for any $U=\text{Spec}(A)$ which possesses local coordinates. It
follows that $\mathcal{D}_{X}^{(0,1)}$ is a coherent, locally noetherian
sheaf of rings which is stalk-wise noetherian. One sees directly the
isomorphism $\mathcal{D}_{X}^{(0)}\tilde{=}\mathcal{D}^{(0,1)}/(v-1)$,
and we may now define $\mathcal{D}_{X}^{(1)}:=\mathcal{D}^{(0,1)}/(f-1)$.
We will see below that this agrees with Berthelot's definition; this
is clear if $X$ is liftable but not quite obvious in general.

\section{Gauges Over $\widehat{\mathcal{D}}_{\mathfrak{X}}^{(0,1)}$}

We now have several locally noetherian graded rings and so we can
consider categories of modules over them; in particular we have the
category of graded $\widehat{\mathcal{D}}_{\mathfrak{X}}^{(0,1)}$-modules,
(which we refer to a gauges over $\mathfrak{X}$) $\mathcal{G}(\widehat{\mathcal{D}}_{\mathfrak{X}}^{(0,1)})$
and the category of coherent graded modules $\mathcal{G}_{coh}(\widehat{\mathcal{D}}_{\mathfrak{X}}^{(0,1)})$.
We have the analogous categories in positive characteristic as well
as $\mathcal{G}_{qcoh}(\mathcal{D}_{X}^{(0,1)})$, the graded quasicoherent
$\mathcal{D}_{X}^{(0,1)}$-modules; as $\mathcal{D}_{X}^{(0,1)}$
is itself a quasi-coherent sheaf of algebras, this is simply the category
of sheaves in $\mathcal{G}(\widehat{\mathcal{D}}_{\mathfrak{X}}^{(0,1)})$
which are quasi-coherent over $\mathcal{O}_{X}[f,v]$. 

In this chapter we develop the basic properties of these categories
of gauges; we begin by collecting a few of their most basic properties. 

For any object in $\mathcal{G}(\widehat{\mathcal{D}}_{\mathfrak{X}}^{(0,1)})$
(or $\mathcal{G}(\mathcal{D}_{X}^{(0,1)})$) set $\mathcal{M}^{\infty}:=\mathcal{M}/(f-1)$
and $\mathcal{M}^{-\infty}:=\mathcal{M}/(v-1)$; these are exact functors
to the categories of $\widehat{\mathcal{D}}_{\mathfrak{X}}^{(0,1),\infty}$
and $\widehat{\mathcal{D}}_{\mathfrak{X}}^{(0)}(=\widehat{\mathcal{D}}_{\mathfrak{X}}^{(0,1),-\infty})$-modules,
respectively; there are obvious maps $f_{\infty}:\mathcal{M}^{i}\to\mathcal{M}^{\infty}$
and $v_{-\infty}:\mathcal{M}^{i}\to\mathcal{M}^{-\infty}$ for each
$i$. We use the same notation to denote the analogous constructions
for $\mathcal{G}(\mathcal{D}_{X}^{(0,1)})$. 

We have: 
\begin{lem}
\label{lem:Basic-v}Let $\mathcal{M}\in\mathcal{G}_{coh}(\mathcal{\widehat{D}}_{\mathfrak{X}}^{(0,1)})$.
Then each $\mathcal{M}^{i}$ is coherent as a $\mathcal{\widehat{D}}_{X}^{(0)}$-module.
Further, for all $i<<0$, the map $v:\mathcal{M}^{i}\to\mathcal{M}^{i-1}$
is an isomorphism. The same holds for $\mathcal{M}\in\mathcal{G}_{coh}(\mathcal{D}_{X}^{(0,1)})$. 
\end{lem}

\begin{proof}
By definition we have, at least locally, an exact sequence
\[
\bigoplus_{i=1}^{s}\mathcal{\widehat{D}}_{\mathfrak{X}}^{(0,1)}(r_{i})\to\bigoplus_{i=1}^{m}\mathcal{\widehat{D}}_{\mathfrak{X}}^{(0,1)}(l_{i})\to\mathcal{M}\to0
\]
Now the result follows as the lemma is true for $\mathcal{\widehat{D}}_{\mathfrak{X}}^{(0,1)}$
by construction. As the same holds for $\mathcal{\widehat{D}}_{\mathfrak{X}}^{(0,1)}$,
we may prove $2)$ in an identical manner.
\end{proof}
This allows us to give: 
\begin{defn}
\label{def:Index!}Let $\mathcal{M}\in\mathcal{G}_{coh}(\mathcal{\widehat{D}}_{\mathfrak{X}}^{(0,1)})$.
Then the index of $\mathcal{M}$ in $\mathbb{Z}\cup\{\infty\}$ is
the largest integer $i$ for which $v:\mathcal{M}^{j}\to\mathcal{M}^{j-1}$
is an isomorphism for all $j\leq i$. The index is $\infty$ if $v$
is an isomorphism for all $i$ (this can indeed happen; c.f. \exaref{Exponential!}
below). We can make the same definition for $\mathcal{M}\in\mathcal{G}_{coh}(\mathcal{D}_{X}^{(0,1)})$. 
\end{defn}

We will now use show how cohomological completeness gives a convenient
criterion for a complex to be in $D_{coh}^{b}(\mathcal{G}(\mathcal{\widehat{D}}_{\mathfrak{X}}^{(0,1)}))$. 
\begin{prop}
\label{prop:coh-to-coh}We have $D_{coh}^{b}(\mathcal{G}(\mathcal{\widehat{D}}_{\mathfrak{X}}^{(0,1)}))\subset D_{cc}(\mathcal{G}(\mathcal{\widehat{D}}_{\mathfrak{X}}^{(0,1)}))$.
Further, for $\mathcal{M}^{\cdot}\in D_{cc}(\mathcal{G}(\mathcal{\widehat{D}}_{\mathfrak{X}}^{(0,1)}))$,
we have $\mathcal{M}^{\cdot}\in D_{coh}^{?}(\mathcal{\widehat{D}}_{\mathfrak{X}}^{(0,1)})$
iff $\mathcal{M}^{\cdot}\otimes_{W(k)}^{L}k\in D_{coh}^{?}(\mathcal{G}(\mathcal{\widehat{D}}_{X}^{(0,1)}))$,
where $?=+$ or $?=b$. 
\end{prop}

\begin{proof}
Recall that if $\mathcal{F}$ is a sheaf of $W(k)$-modules which
is $p$-torsion free and $p$-adically complete; or if it is annihilated
by $p^{N}$ for some fixed $N\in\mathbb{N}$, then $\mathcal{F}$
(considered as a complex concentrated in one degree) is contained
in $D_{cc}(W(k))$. It follows that a coherent $\mathcal{\widehat{D}}_{X}^{(0)}$-module
is cohomologically complete; therefore so is an element of $\mathcal{G}_{coh}(\mathcal{\widehat{D}}_{\mathfrak{X}}^{(0,1)})$
by \propref{Basic-CC-facts}, part $4)$. Since $D_{cc}(\mathcal{G}(\mathcal{D}_{\mathfrak{X}}^{(0,1)}))$
is closed under extensions, the first statement follows directly (c.f.
\cite{key-8}, theorem 1.6.1). 

For the second statement, the forward direction is obvious. For the
converse, we note that by \cite{key-8} theorem 1.6.4, since (for
either $?=+$ or $?=b$) each \linebreak{}
$(\mathcal{M}^{\cdot})^{i}\otimes_{W(k)}^{L}k\in D_{coh}^{+}(\mathcal{D}_{X}^{(0)}-\text{mod})$,
we must have $(\mathcal{M}^{\cdot})^{i}\in D_{coh}^{+}(\mathcal{D}_{\mathfrak{X}}^{(0)}-\text{mod})$.
In particular $\mathcal{H}^{j}((\mathcal{M}^{\cdot})^{i})$ is $p$-adically
complete for each $i$ and $j$. Further, we have the short exact
sequences for the functor $\otimes_{W(k)}^{L}k$ 
\[
0\to\mathcal{H}^{j}(\mathcal{M}^{\cdot})/p\to\mathcal{H}^{j}(\mathcal{M}^{\cdot}\otimes_{W(k)}^{L}k)\to\mathcal{T}or_{1}^{W(k)}(\mathcal{H}^{j+1}(\mathcal{M}^{\cdot}),k))\to0
\]
which implies also that $\mathcal{H}^{j}(\mathcal{M}^{\cdot})/p$
is coherent over $\mathcal{D}_{X}^{(0,1)}$ for all $j$ (this follows
from the fact that $\mathcal{H}^{j}((\mathcal{M}^{\cdot})^{i})/p$
is coherent, and hence quasi-coherent, for all $i$; which implies
$\mathcal{H}^{j}(\mathcal{M}^{\cdot})/p$ is a quasicoherent sub-sheaf
of the coherent $\mathcal{D}_{X}^{(0,1)}$-module $\mathcal{H}^{j}(\mathcal{M}^{\cdot}\otimes_{W(k)}^{L}k)$,
and hence coherent). 

Now, for a fixed $j$, we can consider, for any $i$ 
\[
v:\mathcal{H}^{j}((\mathcal{M}^{\cdot})^{i})\to\mathcal{H}^{j}((\mathcal{M}^{\cdot})^{i-1})
\]
and 
\[
\mathcal{D}_{\mathfrak{X}}^{(0,1),1}\otimes_{\mathcal{D}_{\mathfrak{X}}^{(0)}}\mathcal{H}^{j}((\mathcal{M}^{\cdot})^{i})\to\mathcal{H}^{j}((\mathcal{M}^{\cdot})^{i+1})
\]
Since $\mathcal{H}^{j}(\mathcal{M}^{\cdot})/p$ is coherent over $\mathcal{D}_{X}^{(0,1)}$,
we have that the reduction mod $p$ of $v$ is surjective for $i<<0$
and the the reduction mod $p$ of the second map is surjective for
$i>>0$. By the usual complete Nakayama lemma, we see that $v$ is
surjective for $i<<0$ and the second map is surjective for $i>>0$;
therefore $\mathcal{H}^{j}(\mathcal{M}^{\cdot})$ is locally finitely
generated over $\mathcal{D}_{\mathfrak{X}}^{(0,1)}$ (since each $\mathcal{H}^{j}((\mathcal{M}^{\cdot})^{i})$
is coherent over $\mathcal{D}_{\mathfrak{X}}^{(0)}$). 

Now, let $U\subset X$ be an open affine and let $D(g)\subset U$
be a principle open inside $U$; let $\tilde{g}$ be a lift of the
function $g$ to $\Gamma(\mathcal{O}_{U})$. As each $\mathcal{H}^{j}((\mathcal{M}^{\cdot})^{i})$
is coherent, we have that $\mathcal{H}^{j}((\mathcal{M}^{\cdot})^{i})(D(g))$
is isomorphic to the completion of the localization of $\mathcal{H}^{j}((\mathcal{M}^{\cdot})^{i})(U)$
at $\tilde{g}$. It follows that $\mathcal{H}^{j}(\mathcal{M}^{\cdot})(D(g))$
is given by localizing $\mathcal{H}^{j}(\mathcal{M}^{\cdot})(U)$
at $\tilde{g}$ and then completing each component. If $\mathcal{F}$
is a graded free module over $\widehat{\mathcal{D}}_{\mathfrak{X}}^{(0,1)}$,
it has the same description; and so the kernel of any map $\mathcal{F}\to\mathcal{H}^{j}(\mathcal{M}^{\cdot})|_{U}$
also has this description (as the functor of localizing and completing
is exact on coherent $\widehat{\mathcal{D}}_{\mathfrak{X}}^{(0)}$-modules);
hence it is locally finitely generated and so $\mathcal{H}^{j}(\mathcal{M}^{\cdot})$
is itself coherent. 

Finally, we note that $\mathcal{H}^{j}(\mathcal{M}^{\cdot}\otimes_{W(k)}^{L}k)=0$
implies $\mathcal{H}^{j}(\mathcal{M}^{\cdot})/p=0$ by the above short
exact sequence. So, if $\mathcal{M}^{\cdot}\otimes_{W(k)}^{L}k\in D_{coh}^{+}(\mathcal{G}(\mathcal{D}_{X}^{(0,1)}))$,
we see that $\mathcal{H}^{j}(\mathcal{M}^{\cdot})/p=0$ for all $j<<0$;
which implies $\mathcal{H}^{j}(\mathcal{M}^{\cdot})=0$ for $j<<0$
since each $\mathcal{H}^{j}(\mathcal{M}^{\cdot})^{i}$ is $p$-adically
complete; i.e., we have $\mathcal{M}^{\cdot}\in D_{\text{Coh}}^{+}(\mathcal{D}_{\mathfrak{X}}^{(0,1)})$;
the same argument applies for bounded complexes. 
\end{proof}
This proposition will be our main tool for showing that elements of
$D_{cc}(\mathcal{G}(\mathcal{D}_{\mathfrak{X}}^{(0,1)}))$ are actually
in $D_{coh}^{b}(\mathcal{\widehat{D}}_{\mathfrak{X}}^{(0,1)})$.

\subsection{\label{subsec:Standard}Standard Gauges, Mazur's Theorem}

In this subsection we discuss the analogue of (the abstract version
of) Mazur's theorem in the context of $\widehat{\mathcal{D}}_{\mathfrak{X}}^{(0,1)}$-gauges.
Since the notion of gauge was invented in order to isolate the structures
used in the proof of Mazur's theorem, it comes as no surprise that
there is a very general version of the theorem available in this context.
Before proving it, we discuss some generalities, starting with 
\begin{defn}
\label{def:Standard!}Let $\mathcal{M}\in\mathcal{G}(\widehat{\mathcal{D}}_{\mathfrak{X}}^{(0,1)})$.
We say $\mathcal{M}$ is standard if $\mathcal{M}^{-\infty}$ and
$\mathcal{M}^{\infty}$ are $p$-torsion-free, each $f_{\infty}:\mathcal{M}^{i}\to\mathcal{M}^{\infty}$
is injective; and, finally, there is a $j_{0}\in\mathbb{Z}$ so that
\[
f_{\infty}(\mathcal{M}^{i+j_{0}})=\{m\in\mathcal{M}^{\infty}|p^{i}m\in f_{\infty}(\mathcal{M}^{j_{0}})\}
\]
for all $i\in\mathbb{Z}$. 
\end{defn}

The $j_{0}$ appearing in this definition is not unique; indeed, from
the definition if $i<0$ we have $f_{\infty}(\mathcal{M}^{i+j_{0}})=p^{-i}\cdot f_{\infty}(\mathcal{M}^{j_{0}})$
which implies that we can replace $j_{0}$ with any $j<j_{0}$. In
particular the\emph{ }index\emph{ }of a standard gauge (as in \defref{Index!})
is the maximal $j_{0}$ for which the description in the definition
is true (and it takes the value $\infty$ if this description is true
for all integers). Note that if $\mathcal{M}$ is standard, so is
the shift $\mathcal{M}(j)$, and the index of $\mathcal{M}(j)$ is
equal to $\text{index}(\mathcal{M})+j$. 

As in the case where $\mathfrak{X}$ is a point (which is discussed
above in \exaref{BasicGaugeConstruction}), standard gauges are (up
to a shift of index) exactly the ones that can be constructed from
lattices: 
\begin{example}
\label{exa:Basic-Construction-over-X} Let $\mathcal{N}'$ be a $\widehat{\mathcal{D}}_{\mathfrak{X}}^{(0)}[p^{-1}]$-module,
and let $\mathcal{N}$ be a lattice; i.e., a $\widehat{\mathcal{D}}_{\mathfrak{X}}^{(0)}$-submodule
such that $\mathcal{N}[p^{-1}]=\mathcal{N}'$. Recalling the isomorphism
$\widehat{\mathcal{D}}_{\mathfrak{X}}^{(0)}[p^{-1}]\tilde{=}\widehat{\mathcal{D}}_{\mathfrak{X}}^{(0,1),\infty}[p^{-1}]$
(c.f. \lemref{Basic-Structure-of-D^(1)}), we also suppose given a
$\widehat{\mathcal{D}}_{\mathfrak{X}}^{(0,1),\infty}$-lattice of
$\mathcal{N}'$ called $\mathcal{M}^{\infty}$. Then we may produce
a standard gauge $\mathcal{M}$ via 
\[
\mathcal{M}^{i}=\{m\in\mathcal{M}^{\infty}|p^{i}m\in\mathcal{N}\}
\]
If $\mathcal{M}^{\infty}$ is coherent over $\widehat{\mathcal{D}}_{\mathfrak{X}}^{(0,1),\infty}$
and $\mathcal{N}$ is coherent over $\widehat{\mathcal{D}}_{\mathfrak{X}}^{(0)}$,
then $\mathcal{M}$ is a coherent gauge. 
\end{example}

Let us give some general properties of standard gauges:
\begin{lem}
\label{lem:Standard-is-rigid}Suppose $\mathcal{M}\in\mathcal{G}(\widehat{\mathcal{D}}_{\mathfrak{X}}^{(0,1)})$
is standard; and let $\mathcal{M}_{0}=\mathcal{M}/p$ be its reduction
mod $p$. Then $\mathcal{M}_{0}$ has $\text{ker}(f)=\text{im}(v)$
and $\text{ker}(v)=\text{im}(f)$; further, if $\overline{m}_{i}\in\mathcal{M}_{0}^{i}$,
then $f\overline{m}_{i}=0=v\overline{m}_{i}$ iff $\overline{m}_{i}=0$. 
\end{lem}

\begin{proof}
Since $fv=0$ on $\mathcal{M}_{0}$, we always have $\text{im}(f)\subset\text{ker}(v)$
and $\text{im}(v)\subset\text{ker}(f)$; so we consider the other
inclusions.

Let $m_{i}\in\mathcal{M}^{i}$, and denote its image in $\mathcal{M}_{0}^{i}$
by $\overline{m}_{i}$. Suppose $v\overline{m}_{i}=0$. Then $vm_{i}=pm_{i-1}$
for some $m_{i-1}\in\mathcal{M}^{i-1}$, so that $f_{\infty}(vm_{i})=pf_{\infty}(m_{i})=pf_{\infty}(m_{i-1})$.
Since $\mathcal{M}^{\infty}$ is $p$-torsion-free this yields $f_{\infty}(m_{i})=f_{\infty}(m_{i-1})$
so that $fm_{i-1}=m_{i}$ by the injectivity of $f_{\infty}$. Thus
$\overline{m}_{i}\in\text{im}(f)$ and we see $\text{ker}(v)\subset\text{im}(f)$
as required. 

Now suppose $f\overline{m}_{i}=0$. Then $fm_{i}=pm_{i+1}$ for some
$m_{i+1}\in\mathcal{M}^{i+1}$ so that $f_{\infty}(m_{i})=pf_{\infty}(m_{i+1})=f_{\infty}(vm_{i+1})$,
and the injectivity of $f_{\infty}$ implies $m_{i}=vm_{i+1}$ so
that $\overline{m}_{i}\in\text{im}(v)$ as required.

To obtain the last property; since $\mathcal{M}$ is standard, after
shifting the grading if necessary, we may identify $f_{\infty}(\mathcal{M}^{i})$
with $\{m\in\mathcal{M}^{\infty}|p^{i}m\in f_{\infty}(\mathcal{M}^{0})\}$.
If $m_{i}\in\mathcal{M}^{i}$ and $f\overline{m}_{i}=0=v\overline{m}_{i}$
then $fm_{i}=pm_{i+1}$ and $vm_{i}=pm_{i-1}$; therefore $f_{\infty}(m_{i})=pf_{\infty}(m_{i+1})$
and $pf_{\infty}(m_{i})=pf_{\infty}(m_{i-1})$ so that $p^{2}f_{\infty}(m_{i+1})=pf_{\infty}(m_{i-1})$
which implies $pf_{\infty}(m_{i+1})=f_{\infty}(m_{i-1})$. But $p^{i-1}f_{\infty}(m_{i-1})\in f_{\infty}(\mathcal{M}^{0})$,
so that $p^{i}f_{\infty}(m_{i+1})\in f_{\infty}(\mathcal{M}^{0})$
which forces $f_{\infty}(m_{i+1})\in f_{\infty}(\mathcal{M}^{i})$
so that $m_{i+1}=fm'_{i}$ for some $m'_{i}\in\mathcal{M}^{i}$. So
$fm_{i}=pm_{i+1}=f(pm_{i}')$ which implies $m_{i}=pm'_{i}$ and so
$\overline{m}_{i}=0$. 
\end{proof}
This motivates the 
\begin{defn}
(\cite{key-5}, definition 2.2.2) A gauge $\mathcal{M}_{0}$ over
$\mathcal{D}_{X}^{(0,1)}$ is called quasi-rigid if it satisfies $\text{ker}(f)=\text{im}(v)$
and $\text{ker}(v)=\text{im}(f)$, it is called rigid if it is quasi-rigid
and, in addition, $\text{ker}(f)\cap$$\text{ker}(v)=0$. 
\end{defn}

By the above lemma, a gauge is rigid if it is of the form $\mathcal{M}/p$
for some standard gauge $\mathcal{M}$. 

As explained in \cite{key-5}, rigidity is a very nice condition;
and in particular we have the following generalization of \cite{key-5},
lemma 2.2.5:
\begin{lem}
\label{lem:Basic-Facts-on-Rigid}Let $\mathcal{M}_{0}\in\mathcal{G}(\mathcal{D}_{X}^{(0,1)})$.
Then $\mathcal{M}_{0}$ is rigid iff $\mathcal{M}_{0}/f$ is $v$-torsion
free and $\mathcal{M}_{0}/v$ is $f$-torsion-free. 

Further, $\mathcal{M}_{0}$ is quasi-rigid iff $\mathcal{M}_{0}\otimes_{k[f,v]}^{L}k[f]\tilde{=}\mathcal{M}_{0}/v$
and $\mathcal{M}_{0}\otimes_{k[f,v]}^{L}k[v]\tilde{=}\mathcal{M}_{0}/f$. 
\end{lem}

\begin{proof}
Suppose $\mathcal{M}_{0}$ is rigid. To show $\mathcal{M}_{0}/f$
is $v$-torsion free we have to show that if $m$ is a local section
of $\mathcal{M}_{0}$ with $vm=fm'$, then $m\in\text{im}(f)$. Since
$\text{im}(f)=\text{ker}(v)$ we have $v(vm)=0$, and since also $f(vm)=0$
we must (by the second condition of rigidity) have $vm=0$. Therefore
$m\in\text{ker}(v)=\text{im}(f)$ as desired. The proof that $\mathcal{M}_{0}/v$
is $f$-torsion-free is essentially identical.

Now suppose $\mathcal{M}_{0}$ satisfies $\mathcal{M}_{0}/f$ is $v$-torsion
free and $\mathcal{M}_{0}/v$ is $f$-torsion-free. Suppose $m\in\text{ker}(f)$.
Then the image of $m$ in $\mathcal{M}_{0}/v$ is $f$-torsion, hence
$0$; and so $m\in\text{im}(v)$; therefore $\text{ker}(f)=\text{im}(v)$
and similarly $\text{ker}(v)=\text{im}(f)$. If $fm=0=vm$, then $m\in\text{ker}(f)\cap\text{ker}(v)=\text{im}(v)\cap\text{ker}(v)=\text{ker}(f)\cap\text{im}(f)$.
Since $m\in\text{im}(v)$ the image of $m$ in $\mathcal{M}_{0}/v$
is zero; also, $m=fm'$, so since $\mathcal{M}_{0}/v$ is $f$-torsion
free we see $m'\in\text{im}(v)$. So $m=fm'=fv(m'')=0$ as desired. 

Now we consider the quasi-rigidity condition: we can write the following
free resolution of $k[f]$ over $D(k)$: 
\[
\cdots\rightarrow D(k)(-1)\xrightarrow{v}D(k)\xrightarrow{f}D(k)(-1)\xrightarrow{v}D(k)
\]
so that $\mathcal{M}_{0}\otimes_{D(k)}^{L}k[f]$ has no higher cohomology
groups iff $\text{ker}(v)=\text{im}(f)$ and $\text{ker}(f)=\text{im}(v)$;
the same holds for $\mathcal{M}_{0}\otimes_{D(k)}^{L}k[v]\tilde{=}\mathcal{M}_{0}/f$. 
\end{proof}
Now we turn to conditions for checking that a gauge is standard. 
\begin{prop}
\label{prop:Baby-Mazur}Let $\mathcal{M}\in\mathcal{G}_{\text{coh}}(\widehat{\mathcal{D}}_{\mathfrak{X}}^{(0,1)})$,
and suppose that $\mathcal{M}^{-\infty}$ and $\mathcal{M}^{\infty}$
are $p$-torsion-free. Set $\mathcal{M}_{0}=\mathcal{M}/p$, and suppose
$\mathcal{M}_{0}/v$ is $f$-torsion-free. Then $\mathcal{M}$ is
standard; in particular $\mathcal{M}$ is $p$-torsion free.
\end{prop}

\begin{proof}
Let $\mathcal{N}_{i}=\text{ker}(f_{\infty}:\mathcal{M}^{i}\to\mathcal{M}^{\infty})$
and let $\mathcal{N}=\bigoplus_{i}\mathcal{N}_{i}$. Clearly $\mathcal{N}$
is preserved under $W(k)[f]$; further, since $f_{\infty}(vm)=pf_{\infty}(m)$
we see that $\mathcal{N}$ is preserved under $W(k)[f,v]$. 

For $m$ a local section of $\mathcal{M}^{i}$ let $\overline{m}$
denote its image in $\mathcal{M}_{0}$. If $m\in\mathcal{N}$ then
certainly $f_{\infty}(\overline{m})=0$ in $\mathcal{M}_{0}^{\infty}$.
Since $\mathcal{M}_{0}/v$ is $f$-torsion-free, the map $f_{\infty}:\mathcal{M}_{0}^{i}/v\mathcal{M}_{0}^{i+1}\to\mathcal{M}_{0}^{\infty}$
is injective; so $\overline{m}\in\text{im}(v)$. Thus there is some
$m'$ so that $m-vm'\in p\cdot\mathcal{M}^{i}$; since $p=fv$ we
see that $m\in\text{im}(v)$ as well; i.e., we can assume $m=vm'$.
Since $0=f_{\infty}(vm')=pf_{\infty}(m')$ and $\mathcal{M}^{\infty}$
is $p$-torsion-free, we see that $m'\in\mathcal{N}$ as well. So
in fact $\mathcal{N}=v\cdot\mathcal{N}$. 

Now, as $\mathcal{M}$ is coherent, we may choose some $j_{0}$ for
which $v_{-\infty}:\mathcal{M}^{j}\to\mathcal{M}^{-\infty}$ is an
isomorphism for all $j\le j_{0}$. Then, for each such $j$, $\mathcal{M}^{j}$
is $p$-torsion-free (since $\mathcal{M}^{-\infty}$ is). Further,
since $fv=p$, we have that $f$ and $v$ are isomorphisms after inverting
$p$, which shows $f_{\infty}:\mathcal{M}^{j}[p^{-1}]\tilde{\to}\mathcal{M}^{\infty}[p^{-1}]$.
Since $\mathcal{M}^{j}$ and $\mathcal{M}^{\infty}$ are $p$-torsion-free,
we see that $f_{\infty}$ is injective on $\mathcal{M}^{j}$. Thus
$\mathcal{N}$ is concentrated in degrees above $j_{0}$, and we see
that every element of $\mathcal{N}$ is killed by a power of $v$.
Since $\mathcal{M}$ is coherent, it is locally noetherian, so that
every local section of $\mathcal{M}$ killed by a power of $v$ is
actually killed by $v^{N}$ for some fixed $N\in\mathbb{N}$. Therefore,
we have $v^{N}\cdot\mathcal{N}=0$. Since also $\mathcal{N}=v\cdot\mathcal{N}$
we obtain $\mathcal{N}=0$. Thus each $f_{\infty}:\mathcal{M}^{i}\to\mathcal{M}^{\infty}$
is injective. It follows that each $\mathcal{M}^{i}$ is $p$-torsion-free,
and since $fv=p$ we see that $\mathcal{M}$ is $f$ and $v$-torsion-free
as well. 

Choose $j_{0}$ so that $v:\mathcal{M}^{j}\to\mathcal{M}^{j-1}$ is
an isomorphism for all $j\leq j_{0}$. To finish the proof, we have
to show that, for all $i\in\mathbb{Z}$, $f_{\infty}(\mathcal{M}^{i+j_{0}})=\{m\in\mathcal{M}^{\infty}|p^{i}m\in f_{\infty}(\mathcal{M}^{j_{0}})\}$.
If $i\leq0$, then $v^{-i}:\mathcal{M}^{j_{0}}\to\mathcal{M}^{i+j_{0}}$
is an isomorphism, and $f_{\infty}(\mathcal{M}^{i+j_{0}})=p^{-i}f_{\infty}(\mathcal{M}^{j_{0}})$
as required. If $i>0$, then for $m\in\mathcal{M}^{i+j_{0}}$ we have
$f_{\infty}(v^{i}m)=p^{i}f_{\infty}(m)\in f_{\infty}(\mathcal{M}^{j_{0}})$
so that $f_{\infty}(\mathcal{M}^{i+j_{0}})\subseteq\{m\in\mathcal{M}^{\infty}|p^{i}m\in f_{\infty}(\mathcal{M}^{j_{0}})\}$. 

For the reverse inclusion, let $m\in\mathcal{M}^{\infty}$ be such
that $p^{i}m=f_{\infty}(m_{j_{0}})$ for some $m_{j_{0}}\in\mathcal{M}^{j_{0}}$.
By definition $\mathcal{M}^{\infty}$ is the union of its sub-sheaves
$f_{\infty}(\mathcal{M}^{n})$, so suppose $m=f_{\infty}(m_{l})$
for some $m_{l}\in\mathcal{M}^{l}$, with $l>i+j_{0}$. Since $f_{\infty}(v^{i}m_{l})=p^{i}f_{\infty}(m_{l})=p^{i}m=f_{\infty}(m_{j_{0}})$,
we see that 
\[
f^{l-(i+j_{0})}(m_{j_{0}})=v^{i}m_{l}
\]
Consider the image of this equation in $\mathcal{M}_{0}$. It shows
that that $f^{l-(i+j_{0})}(\overline{m}_{j_{0}})\in v\cdot\mathcal{M}_{0}$.
Since $f$ is injective on $\mathcal{M}_{0}/v$, the assumption that
$l-(i+j_{0})>0$ implies $\overline{m}_{j_{0}}\in v\cdot\mathcal{M}_{0}$.
As above, since $fv=p$ this implies $m_{j_{0}}\in v\cdot\mathcal{M}$;
writing $m_{j_{0}}=vm_{j_{0}+1}$ we now have the equation $f^{l-(i+j_{0})}(vm_{j_{0}+1})=v^{i}m_{l}$.
Since $v$ acts injectively on $\mathcal{M}$ we see that $f^{l-(i+j_{0})}(m_{j_{0}+1})=v^{i-1}m_{l}$.
Applying $f_{\infty}$, we see that $p^{i-1}m\in f_{\infty}(\mathcal{M}^{j_{0}+1})$.
If $i=1$, this immediately proves $f_{\infty}(\mathcal{M}^{1+j_{0}})=\{m\in\mathcal{M}^{\infty}|pm\in f_{\infty}(\mathcal{M}^{j_{0}})\}$. 

For $i>1$, then by induction on $i$ we can suppose $pm\in f_{\infty}(\mathcal{M}^{j_{0}+i-1})$.
But then $f_{\infty}(vm_{l})=pf_{\infty}(m_{l})=pm=f_{\infty}(m_{j_{0}+i-1})$
for some $m_{j_{0}+i-1}\in\mathcal{M}^{j_{0}+i-i}$. This implies
$f^{l-(j_{0}+i)}(m_{j_{0}+i-1})=vm_{l}$ so if $l>j_{0}+i$ then,
arguing exactly as in the previous paragraph, we have $m_{j_{0}+i-1}=vm_{j_{0}+i}$
for some $m_{j_{0}+i}\in\mathcal{M}^{j_{0}+i}$ and so $f^{l-(j_{0}+i)}(m_{j_{0}+i})=m_{l}$
which implies $m=f_{\infty}(m_{l})\in f_{\infty}(\mathcal{M}^{j_{0}+i})$
as required. 
\end{proof}
This result implies a convenient criterium for ensuring that gauges
are standard; this is the first analogue of Mazur's theorem in this
context:
\begin{thm}
\label{thm:Mazur!}Let $\mathcal{M}^{\cdot}\in D_{\text{coh}}^{b}(\mathcal{G}(\widehat{\mathcal{D}}_{\mathfrak{X}}^{(0,1)}))$.
Suppose that $\mathcal{H}^{n}(\mathcal{M}^{\cdot})^{-\infty}$ and
$\mathcal{H}^{n}(\mathcal{M}^{\cdot})^{\infty}$ are $p$-torsion-free
for all $n$, and suppose that $\mathcal{H}^{n}((\mathcal{M}^{\cdot}\otimes_{W(k)}^{L}k)\otimes_{D(k)}^{L}k[f])$
is $f$-torsion-free for all $n$. Then $\mathcal{H}^{n}(\mathcal{M}^{\cdot})$
is standard for all $n$. 

In particular, $\mathcal{H}^{n}(\mathcal{M}^{\cdot})$ is $p$-torsion-free,
and $\mathcal{H}^{n}(\mathcal{M}^{\cdot})/p$ is rigid for all $n$.
We have $\mathcal{H}^{n}(\mathcal{M}^{\cdot})/p\tilde{=}\mathcal{H}^{n}(\mathcal{M}^{\cdot}\otimes_{W(k)}^{L}k)$
, $(\mathcal{H}^{n}(\mathcal{M}^{\cdot})/p)/v\tilde{=}\mathcal{H}^{n}((\mathcal{M}^{\cdot}\otimes_{W(k)}^{L}k)\otimes_{D(k)}^{L}k[f])$,
and $(\mathcal{H}^{n}(\mathcal{M}^{\cdot})/p)/f\tilde{=}\mathcal{H}^{n}((\mathcal{M}^{\cdot}\otimes_{W(k)}^{L}k)\otimes_{D(k)}^{L}k[v])$
for all $n$. Further, $(\mathcal{H}^{n}(\mathcal{M}^{\cdot})/p)/f$
is $v$-torsion-free and $(\mathcal{H}^{n}(\mathcal{M}^{\cdot})/p)/v$
is $f$-torion-free for all $n$. 
\end{thm}

\begin{proof}
Suppose that $b\in\mathbb{Z}$ is the largest integer so that $\mathcal{H}^{b}(\mathcal{M}^{\cdot})\neq0$.
Then $b$ is the largest integer for which $\mathcal{H}^{b}((\mathcal{M}^{\cdot}\otimes_{W(k)}^{L}k))\neq0$,
and 
\[
\mathcal{H}^{b}((\mathcal{M}^{\cdot}\otimes_{W(k)}^{L}k))\tilde{=}\mathcal{H}^{b}(\mathcal{M}^{\cdot})/p
\]
Thus we have a distinguished triangle 
\[
\tau_{\leq b-1}(\mathcal{M}^{\cdot}\otimes_{W(k)}^{L}k)\to\mathcal{M}^{\cdot}\otimes_{W(k)}^{L}k\to(\mathcal{H}^{b}(\mathcal{M}^{\cdot})/p)[-b]
\]
to which we may apply the functor $\otimes_{k[f,v]}^{L}k[f]$. This
yields 
\begin{equation}
\tau_{\leq b-1}(\mathcal{M}^{\cdot}\otimes_{W(k)}^{L}k)\otimes_{D(k)}^{L}k[f]\to(\mathcal{M}^{\cdot}\otimes_{W(k)}^{L}k)\otimes_{D(k)}^{L}k[f]\to(\mathcal{H}^{b}(\mathcal{M}^{\cdot})/p)[-b]\otimes_{D(k)}^{L}k[f]\label{eq:triangle!}
\end{equation}
Since $\otimes_{D(k)}k[f]$ is right exact, the complex on the left
is still concentrated in degrees $\leq b-1$, and the middle and right
complex are concentrated in degrees $\leq b$. Further 
\[
\mathcal{H}^{b}((\mathcal{H}^{b}(\mathcal{M}^{\cdot})/p)[-b]\otimes_{D(k)}^{L}k[f])\tilde{=}\mathcal{H}^{0}((\mathcal{H}^{b}(\mathcal{M}^{\cdot})/p)\otimes_{D(k)}^{L}k[f])\tilde{=}(\mathcal{H}^{b}(\mathcal{M}^{\cdot})/p)/v
\]
Therefore $(\mathcal{H}^{b}(\mathcal{M}^{\cdot})/p)/v\tilde{=}\mathcal{H}^{b}((\mathcal{M}^{\cdot}\otimes_{W(k)}^{L}k)\otimes_{D(k)}^{L}k[f])$
is $f$-torsion-free by assumption. Thus we may apply the previous
proposition to $\mathcal{H}^{b}(\mathcal{M}^{\cdot})$ and conclude
that it is standard.

Now, to finish the proof that $\mathcal{H}^{n}(\mathcal{M}^{\cdot})$
is standard for all $n$, we proceed by induction on the cohomological
length of $\mathcal{M}^{\cdot}$. If the length is $1$ we are done.
If not, we have the distinguished triangle 
\[
\tau_{\leq b-1}(\mathcal{M}^{\cdot})\to\mathcal{M}^{\cdot}\to\mathcal{H}^{b}(\mathcal{M}^{\cdot})[-b]
\]
which yields the triangle 
\[
\tau_{\leq b-1}(\mathcal{M}^{\cdot})\otimes_{W(k)}^{L}k\to\mathcal{M}^{\cdot}\otimes_{W(k)}^{L}k\to(\mathcal{H}^{b}(\mathcal{M}^{\cdot})/p)[-b]
\]
where we have used that $\mathcal{H}^{b}(\mathcal{M}^{\cdot})$ is
$p$-torsion-free to identify $(\mathcal{H}^{b}(\mathcal{M}^{\cdot})/p)\tilde{=}\mathcal{H}^{b}(\mathcal{M}^{\cdot})\otimes_{W(k)}^{L}k$.
As noted above, we have $\mathcal{H}^{b}((\mathcal{M}^{\cdot}\otimes_{W(k)}^{L}k))\tilde{=}\mathcal{H}^{b}(\mathcal{M}^{\cdot})/p$,
so this triangle implies the isomorphism 
\[
\tau_{\leq b-1}(\mathcal{M}^{\cdot})\otimes_{W(k)}^{L}k\tilde{=}\tau_{\leq b-1}(\mathcal{M}^{\cdot}\otimes_{W(k)}^{L}k)
\]
Further, since $\mathcal{H}^{b}(\mathcal{M}^{\cdot})$ is standard
we have that $\mathcal{H}^{b}(\mathcal{M}^{\cdot})/p$ is rigid; therefore
by \lemref{Basic-Facts-on-Rigid} we have $(\mathcal{H}^{b}(\mathcal{M}^{\cdot})/p)\otimes_{D(k)}^{L}k[f]\tilde{=}(\mathcal{H}^{b}(\mathcal{M}^{\cdot})/p)/v$
is concentrated in a single degree. Thus, the distinguished triangle
\eqref{triangle!} becomes 
\[
(\tau_{\leq b-1}(\mathcal{M}^{\cdot})\otimes_{W(k)}^{L}k)\otimes_{D(k)}^{L}k[f]\to(\mathcal{M}^{\cdot}\otimes_{W(k)}^{L}k)\otimes_{D(k)}^{L}k[f]\to(\mathcal{H}^{b}(\mathcal{M}^{\cdot})/p)/v[-b]
\]
and so we have the isomorphism 
\[
(\tau_{\leq b-1}(\mathcal{M}^{\cdot})\otimes_{W(k)}^{L}k)\otimes_{D(k)}^{L}k[f]\tilde{=}\tau_{\leq b-1}((\mathcal{M}^{\cdot}\otimes_{W(k)}^{L}k)\otimes_{D(k)}^{L}k[f])
\]
Thus the complex $\tau_{\leq b-1}(\mathcal{M}^{\cdot})$ satisfies
the assumption that $(\tau_{\leq b-1}(\mathcal{M}^{\cdot})\otimes_{W(k)}^{L}k)\otimes_{D(k)}^{L}k[f]$
has cohomology sheaves which are $f$-torsion-free, and so the complex
$\tau_{\leq b-1}(\mathcal{M}^{\cdot})$ satisfies all of the assumptions
of the theorem, but has a lesser cohomological length than $\mathcal{M}^{\cdot}$.
So we conclude by induction that $\mathcal{H}^{n}(\mathcal{M}^{\cdot})$
is standard for all $n$. 

For the final part, since standard modules are torsion-free, we see
\[
\mathcal{H}^{n}(\mathcal{M}^{\cdot})/p\tilde{=}\mathcal{H}^{n}(\mathcal{M}^{\cdot}\otimes_{W(k)}^{L}k)
\]
for all $n$, and since each $\mathcal{H}^{n}(\mathcal{M}^{\cdot})/p$
is rigid, the complex $\mathcal{M}^{\cdot}\otimes_{W(k)}^{L}k$ has
cohomology sheaves which are all acyclic for $\otimes_{D(k)}k[f]$
and for $\otimes_{D(k)}k[f]$, by (\lemref{Basic-Facts-on-Rigid});
and the last sentence follows. 
\end{proof}
\begin{rem}
As we shall see below, the condition that each cohomology sheaf of
$((\mathcal{M}^{\cdot}\otimes_{W(k)}^{L}k)\otimes_{D(k)}^{L}k[f])$
is $f$-torsion-free is quite natural; it says that the spectral sequence
associated to the Hodge filtration on $(\mathcal{M}^{\cdot}\otimes_{W(k)}^{L}k)^{\infty}$
degenerates at $E^{1}$; this can be checked using Hodge theory in
many geometric situations. On the other hand, one conclusion of the
theorem is that each cohomology sheaf of $(\mathcal{M}^{\cdot}\otimes_{W(k)}^{L}k)\otimes_{D(k)}^{L}k[v]$
is $v$-torsion-free; this corresponds to degeneration of the conjugate
spectral sequence on $(\mathcal{M}^{\cdot}\otimes_{W(k)}^{L}k)^{-\infty}$.
Over a point, one checks in an elementary way (using the finite dimensionality
of the vector spaces involved) that these two degenerations are equivalent;
this is true irrespective of weather the lift $\mathcal{M}^{\cdot}$
has $p$-torsion-free cohomology groups. This allows one to make various
stronger statements in this case (c.f., e.g., \cite{key-10}, proof
of theorem 8.26). I don't know if this is true over a higher dimensional
base. 
\end{rem}

In most cases of interest, the assumption that $\mathcal{H}^{n}(\mathcal{M}^{\cdot})^{\infty}$
is $p$-torsion-free is actually redundant, more precisely, it is
implied by the assumption that $\mathcal{H}^{n}(\mathcal{M}^{\cdot})^{-\infty}$
is $p$-torsion-free when one has a Frobenius action; c.f. \thmref{F-Mazur}
below. 

\subsection{Filtrations, Rees algebras, and filtered Frobenius descent}

In this section, we consider how the various gradings and filtrations
appearing in this paper (in positive characteristic) relate to the
more usual Hodge and conjugate filtrations in $\mathcal{D}$-module
theory. We start with the basic definitions; as usual $X$ is smooth
over $k$. 
\begin{defn}
\label{def:Hodge-and-Con} The decreasing filtration ${\displaystyle \text{image}(\mathcal{D}^{(0,1),i}\to\mathcal{D}_{X}^{(0)})}:=C^{i}(\mathcal{D}_{X}^{(0)})$
is called the conjugate filtration. The increasing filtration ${\displaystyle \text{image}(\mathcal{D}^{(0,1),i}\to\mathcal{D}_{X}^{(1)})}:=F^{i}(\mathcal{D}_{X}^{(1)})$
is called the Hodge filtration. 

Similarly, for any $\mathcal{M}\in\mathcal{G}(\mathcal{D}_{X}^{(0,1)})$
we may define ${\displaystyle \text{image}(\mathcal{M}^{i}\xrightarrow{v_{\infty}}\mathcal{M}^{-\infty})}:=C^{i}(\mathcal{M}^{-\infty})$
and ${\displaystyle \text{image}(\mathcal{M}^{i}\xrightarrow{f_{\infty}}\mathcal{M}^{\infty})}:=F^{i}(\mathcal{M}^{\infty})$,
the conjugate and Hodge filtrations, respectively. 
\end{defn}

\begin{rem}
\label{rem:Description-of-conjugate}1) From the explicit description
of $v$ given in (the proof of) \lemref{Construction-of-v}, we see
that $C^{i}(\mathcal{D}_{X}^{(0)})=\mathcal{I}^{i}\mathcal{D}_{X}^{(0)}$
where $\mathcal{I}$ is the two-sided ideal of $\mathcal{D}_{X}^{(0)}$
generated by $\mathcal{Z}(\mathcal{D}_{X}^{(0)})^{+}$, the positive
degree elements of the center\footnote{The center is a graded sheaf of algebras via the isomorphism $\mathcal{Z}(\mathcal{D}_{X}^{(0)})\tilde{=}\mathcal{O}_{T^{*}X^{(1)}}$}.
In local coordinates, $\mathcal{I}$ is the just ideal generated by
$\{\partial_{1}^{p},\dots,\partial_{n}^{p}\}$, which matches the
explicit description of the action of $v$ given above. This is the
definition of the conjugate filtration on $\mathcal{D}_{X}^{(0)}$
given, in {[}OV{]} section 3.4, extended to a $\mathbb{Z}$-filtration
by setting $C_{i}(\mathcal{D}_{X}^{(0)})=\mathcal{D}_{X}^{(0)}$ for
all $i\leq0$. 

2) On the other hand, from \thmref{Local-Coords-for-D+}, we see that
$F^{i}(\mathcal{D}_{X}^{(1)})$ a locally free, finite $\overline{\mathcal{D}_{X}^{(0)}}$-module;
in local coordinates it has a basis $\{(\partial_{1}^{[p]})^{j_{1}}\cdots(\partial_{n}^{[p]})^{j_{n}}\}_{0\leq|J|\leq i}$. 

3) If $\mathcal{M}$ is a coherent gauge over $X$, then by \lemref{Basic-v}
the Hodge filtration of $\mathcal{M}^{\infty}$ is exhaustive and
$F^{i}(\mathcal{M}^{\infty})=0$ for $i<<0$, and the conjugate filtration
satisfies $C^{i}(\mathcal{M}^{-\infty})=\mathcal{M}^{-\infty}$ for
all $i<<0$. 
\end{rem}

\begin{defn}
\label{def:Rees-and-Rees-bar}Let $\mathcal{R}(\mathcal{D}_{X}^{(1)})$
denote the Rees algebra of $\mathcal{D}_{X}^{(1)}$ with respect to
the Hodge filtration; and let $\mathcal{\overline{R}}(\mathcal{D}_{X}^{(0)})$
denote the Rees algebra of $\mathcal{D}_{X}^{(0)}$ with respect to
the conjugate filtration. We will denote the Rees parameters (i.e.,
the element $1\in F^{1}(\mathcal{D}_{X}^{(1)})$, respectively $1\in C^{-1}(\mathcal{D}_{X}^{(0)})$)
by $f$ and $v$, respectively. We also let $\mathcal{R}(\mathcal{D}_{X}^{(0)})$
denote the Rees algebra of $\mathcal{D}_{X}^{(0)}$ with respect to
the symbol filtration; here the Rees parameter will also be denoted
$f$. 
\end{defn}

\begin{lem}
We have $\mathcal{D}_{X}^{(0,1)}/v\tilde{=}\mathcal{R}(\mathcal{D}_{X}^{(1)})$
and $\mathcal{D}_{X}^{(0,1)}/f\tilde{=}\mathcal{\overline{R}}(\mathcal{D}_{X}^{(0)})$
as graded rings. 
\end{lem}

\begin{proof}
By \corref{Local-coords-over-A=00005Bf,v=00005D}, we have that $f$
acts injectively on $\mathcal{D}_{X}^{(0,1)}/v$. Since $fv=0$ the
map $f_{\infty}:\mathcal{D}_{X}^{(0,1),i}\to\mathcal{D}_{X}^{(1)}$
factors through a map $f_{\infty}:\mathcal{D}_{X}^{(0,1),i}/v\to\mathcal{D}_{X}^{(1)}$,
which has image equal to $F^{i}(\mathcal{D}_{X}^{(1)})$ (by definition).
The kernel is $0$ since $f$ acts injectively; so $\mathcal{D}_{X}^{(0,1),i}/v\tilde{\to}F^{i}(\mathcal{D}_{X}^{(1)})$
as required. The isomorphism $\mathcal{D}_{X}^{(0,1)}/f\tilde{=}\mathcal{\overline{R}}(\mathcal{D}_{X}^{(0)})$
is proved identically.
\end{proof}
Therefore we have the natural functors 
\[
\mathcal{M}^{\cdot}\to\mathcal{R}(\mathcal{D}_{X}^{(1)})\otimes_{\mathcal{D}_{X}^{(0,1)}}^{L}\mathcal{M}^{\cdot}\tilde{\to}k[f]\otimes_{D(k)}^{L}\mathcal{M}^{\cdot}
\]
from $D(\mathcal{G}(\mathcal{D}_{X}^{(0,1)}))$ to $D(\mathcal{G}(\mathcal{R}(\mathcal{D}_{X}^{(1)})))$
and 
\[
\mathcal{M}^{\cdot}\to\overline{\mathcal{R}}(\mathcal{D}_{X}^{(0)})\otimes_{\mathcal{D}_{X}^{(0,1)}}^{L}\mathcal{M}^{\cdot}\tilde{\to}k[v]\otimes_{D(k)}^{L}\mathcal{M}^{\cdot}
\]
from $D(\mathcal{G}(\mathcal{D}_{X}^{(0,1)}))$ to $D(\mathcal{G}(\mathcal{\overline{R}}(\mathcal{D}_{X}^{(0)})))$. 

We are going to give some basic results on the derived categories
of modules over these rings. As a motivation, we recall general result
of Schapira-Schneiders (\cite{key-47}, theorem 4.20; c.f, also example
4.22)
\begin{thm}
Let $(\mathcal{A},F)$ be a $\mathbb{Z}$-filtered sheaf of rings
on a topological space; let $\mathcal{R}(\mathcal{A})$ denote the
associated Rees algebra. Let $D((\mathcal{A},F)-\text{mod})$ denote
the filtered derived category of modules over $(\mathcal{A},F)$.
Then there is an equivalence of categories 
\[
\mathcal{R}:D((\mathcal{A},F)-\text{mod})\tilde{\to}D(\mathcal{G}(\mathcal{R}(\mathcal{A})))
\]
which preserves the subcategories of bounded, bounded below, and bounded
above complexes. To a filtered module $\mathcal{M}$ (considered as
a complex in degree $0$) this functor attaches the usual Rees module
$\mathcal{R}(\mathcal{M})$. 
\end{thm}

Recall that a filtered complex $\mathcal{M}^{\cdot}$ over $(\mathcal{A},F)$
is said to be strict if each morphism $d:(\mathcal{M}^{i},F)\to(\mathcal{M}^{i+1},F)$
satisfies $d(m)\in F_{j}(\mathcal{M}^{i+1})$ iff $m\in F_{j}(\mathcal{M}^{i})$
(for all local sections $m$). Then $\mathcal{M}^{\cdot}$ is quasi-isomorphic
to a strict complex iff each cohomology sheaf $\mathcal{H}^{i}(\mathcal{R}(\mathcal{M}^{\cdot}))$
is torsion-free with respect to the Rees parameter. If $\mathcal{M}^{\cdot}$
is a bounded complex, for which the filtration is bounded below (i.e.
there is some $j\in\mathbb{Z}$ so that $F_{j}(\mathcal{M}^{i})=0$
for all $i$), then this condition is equivalent to the degeneration
at $E_{1}$ of the spectral sequence associated to the filtration. 

Now we return the discussion to $\mathcal{R}(\mathcal{D}_{X}^{(1)})$
and $\mathcal{\overline{R}}(\mathcal{D}_{X}^{(0)})$. We begin with
the latter; recall that Ogus and Vologodsky in \cite{key-11} have
considered the filtered derived category associated to the conjugate
filtration on $\mathcal{D}_{X}^{(0)}$; by the above theorem\footnote{The careful reader will note that in their work they require filtrations
to be separated; however, this leads to a canonically isomorphic filtered
derived category, as explained in \cite{key-59}, proposition 3.1.22 } this category is equivalent to $\mathcal{G}(\overline{\mathcal{R}}(\mathcal{D}_{X}^{(0)}))$.
After we construct our pushforward on $\overline{\mathcal{R}}(\mathcal{D}_{X}^{(0)})$,
we will show that it is compatible with the one constructed on \cite{key-11},
for now, we will just prove the following basic structure theorem
for $\overline{\mathcal{R}}(\mathcal{D}_{X}^{(0)})$; following \cite{key-3},
theorem 2.2.3: 
\begin{prop}
We have $\mathcal{Z}(\overline{\mathcal{R}}(\mathcal{D}_{X}^{(0)}))\tilde{=}\mathcal{O}_{T^{*}X^{(1)}}[v]$;
this is a graded ring where $\mathcal{O}_{T^{*}X^{(1)}}$ is graded
as usual and $v$ is placed in degree $-1$. The algebra $\overline{\mathcal{R}}(\mathcal{D}_{X}^{(0)})$
is Azumaya over $\mathcal{Z}(\overline{\mathcal{R}}(\mathcal{D}_{X}^{(0)}))$,
of index $p^{\text{dim}(X)}$. In particular, $\overline{\mathcal{R}}(\mathcal{D}_{X}^{(0)})(U)$
has finite homological dimension for each open affine $U$. 
\end{prop}

\begin{proof}
The filtered embedding $\mathcal{O}_{T^{*}X^{(1)}}\to\mathcal{D}_{X}^{(0)}$
induces the map $\mathcal{O}_{T^{*}X^{(1)}}[v]\to\overline{\mathcal{R}}(\mathcal{D}_{X}^{(0)})$,
by the very definition of the conjugate filtration it is a map of
graded rings. To show that this map is an isomorphism onto the center,
note that by \corref{Local-coords-over-A=00005Bf,v=00005D}, after
choosing etale local coordinates we have that a basis for $\overline{\mathcal{R}}(\mathcal{D}_{X}^{(0)})$
over $\mathcal{O}_{X}[v]$ is given by $\{\partial^{I}(\partial^{[p]})^{J}\}$
where each entry of $I$is contained in $\{0,\dots,p-1\}$; and the
formula for the bracket by $\partial_{i}^{[p]}$ (c.f. \thmref{Local-Coords-for-D+})
shows that $(\partial^{[p]})^{J}$ is now central. Thus the center
is given by $\mathcal{O}_{X^{(1)}}[v,\partial_{1}^{[p]},\dots,\partial_{n}^{[p]}]$
which is clearly the (isomorphic) image of the map.

The above local coordinates also show that $\overline{\mathcal{R}}(\mathcal{D}_{X}^{(0)})$
is locally free over $\mathcal{Z}(\overline{\mathcal{R}}(\mathcal{D}_{X}^{(0)}))$,
of rank $p^{2\text{dim}(X)}$. Now we can follow the strategy of \cite{key-3},
to show that $\overline{\mathcal{R}}(\mathcal{D}_{X}^{(0)})$ is Azumaya:
we consider the commutative subalgebra $\mathcal{A}_{X,v}:=\mathcal{O}_{X}\otimes_{\mathcal{O}_{X^{(1)}}}\mathcal{O}_{T^{*}X^{(1)}}[v]$
inside $\overline{\mathcal{R}}(\mathcal{D}_{X}^{(0)})$; it acts by
right multiplication on $\overline{\mathcal{R}}(\mathcal{D}_{X}^{(0)})$,
and $\overline{\mathcal{R}}(\mathcal{D}_{X}^{(0)})$ is a locally
free module over it of rank $p^{\text{dim}(X)}$. We have the action
map 
\[
A:\overline{\mathcal{R}}(\mathcal{D}_{X}^{(0)})\otimes_{\mathcal{O}_{T^{*}X^{(1)}}[v]}\mathcal{A}_{X,v}\to\mathcal{E}nd_{\mathcal{A}_{X,v}}(\overline{\mathcal{R}}(\mathcal{D}_{X}^{(0)}))
\]
which is a morphism of algebras, both of which are locally free modules
of rank $p^{2\text{dim}(X)}$ over $\mathcal{A}_{X,v}$. Since the
left hand side is the pullback of $\overline{\mathcal{R}}(\mathcal{D}_{X}^{(0)})$,
considered as a sheaf of algebras over $T^{*}X^{(1)}\times\mathbb{A}^{1}$,
to the flat cover $X\times_{X^{(1)}}T^{*}X^{(1)}\times\mathbb{A}^{1}$,
we see that $\overline{\mathcal{R}}(\mathcal{D}_{X}^{(0)})$ is Azumaya
if $A$ is an isomorphism. 

To prove that $A$ an isomorphism it suffices to prove it after inverting
$v$ and after setting $v=0$. Upon inverting $v$, we have $\overline{\mathcal{R}}(\mathcal{D}_{X}^{(0)})=\mathcal{D}_{X}^{(0)}[v,v^{-1}]$,
so the map $A$ simply becomes the analogous map for $\mathcal{D}_{X}^{(0)}$
tensored with $k[v,v^{-1}]$; this is shown to be an isomorphism by
\cite{key-3}, proposition 2.2.2. Upon setting $v=0$, we obtain 
\[
A_{0}:\text{gr}(\mathcal{D}_{X}^{(0)})\otimes_{\mathcal{O}_{T^{*}X^{(1)}}}\mathcal{A}_{X}\to\mathcal{E}nd_{\mathcal{A}_{X}}(\text{gr}(\mathcal{D}_{X}^{(0)}))
\]
where $\text{gr}(\mathcal{D}_{X}^{(0)})$ is the associated graded
of $\mathcal{D}_{X}^{(0)}$ with respect to the conjugate filtration;
this is a (split) Azumaya algebra which is easily seen to be isomorphic
to $\overline{\mathcal{D}}_{X}^{(0)}\otimes_{\mathcal{O}_{X^{(1)}}}\mathcal{O}_{T^{*}X^{(1)}}$
(c.f. \cite{key-11}; the discussion below lemma 3.18). Thus the map
$A_{0}$ is again an isomorphism; indeed, we have 
\[
\text{gr}(\mathcal{D}_{X}^{(0)})\otimes_{\mathcal{O}_{T^{*}X^{(1)}}}\mathcal{A}_{X}\tilde{\to}\overline{\mathcal{D}}_{X}^{(0)}\otimes_{\mathcal{O}_{X^{(1)}}}\mathcal{O}_{T^{*}X^{(1)}}\otimes_{\mathcal{O}_{T^{*}X^{(1)}}}\mathcal{A}_{X}
\]
\[
\tilde{\to}\mathcal{E}nd_{\mathcal{O}_{X^{(1)}}}(\mathcal{O}_{X})\otimes_{\mathcal{O}_{X^{(1)}}}\mathcal{A}_{X}
\]
so that each reduction of $\text{gr}(\mathcal{D}_{X}^{(0)})\otimes_{\mathcal{O}_{T^{*}X^{(1)}}}\mathcal{A}_{X}$
to closed point in $X\times_{X^{(1)}}T^{*}X^{(1)}$ is a matrix algebra
of rank $p^{\text{dim}(X)}$, and hence a central simple ring, and
the result follows immediately. 
\end{proof}
Next we turn to the category of modules over $\mathcal{R}(\mathcal{D}_{X}^{(1)})$,
in this case, we can describe them in terms of the familiar filtered
$\mathcal{D}_{X}^{(0)}$-modules (in terms of the symbol filtration).
The key to doing so is a version of Berthelot's Frobenius descent
for filtered $\mathcal{D}_{X}^{(1)}$-modules; while we will consider
the more general Frobenius descent over $\mathfrak{X}$ in the next
subsection, we will give the basic construction on $X$ for now.

To proceed, recall that we have the embedding $\overline{\mathcal{D}_{X}^{(0)}}\subset\mathcal{D}_{X}^{(1)}$
which is simply the image of map $f_{\infty}:\mathcal{D}_{X}^{(0)}\to\mathcal{D}_{X}^{(1)}$.
Let $\mathcal{J}\subset\overline{\mathcal{D}_{X}^{(0)}}$ denote the
annihilator of $1\in\mathcal{O}_{X}$ under the action of $\overline{\mathcal{D}_{X}^{(0)}}$
on $\mathcal{O}_{X}$; we have the left ideal $\mathcal{D}_{X}^{(1)}\cdot\mathcal{J}$. 
\begin{prop}
\label{prop:Basic-F^*-over-k}There is an isomorphism of $\mathcal{O}_{X}$-modules
$\mathcal{D}_{X}^{(1)}/\mathcal{D}_{X}^{(1)}\cdot\mathcal{J}\tilde{\to}F^{*}\mathcal{D}_{X}^{(0)}$,
thus endowing $F^{*}\mathcal{D}_{X}^{(0)}$ with the structure of
a left $\mathcal{D}_{X}^{(1)}$-module; and hence the structure of
a $(\mathcal{D}_{X}^{(1)},\mathcal{D}_{X}^{(0)})$-bimodule. Let $F^{i}(\mathcal{D}_{X}^{(1)}/\mathcal{D}_{X}^{(1)}\cdot\mathcal{J})$
be the filtration induced from the Hodge filtration on $\mathcal{D}_{X}^{(1)}$,
and let $F^{i}(\mathcal{D}_{X}^{(0)})$ be the symbol filtration.
Then we have 
\[
F^{i}(\mathcal{D}_{X}^{(1)}/\mathcal{D}_{X}^{(1)}\cdot\mathcal{J})\cdot F^{j}(\mathcal{D}_{X}^{(0)})=F^{i+j}(\mathcal{D}_{X}^{(1)}/\mathcal{D}_{X}^{(1)}\cdot\mathcal{J})
\]
for all $i,j\geq0$. The induced morphism $\mathcal{D}_{X}^{(1)}\to\mathcal{E}nd_{\mathcal{D}_{X}^{(0),\text{opp}}}(F^{*}\mathcal{D}_{X}^{(0)})$
is an isomorphism of filtered algebras.
\end{prop}

\begin{proof}
We put a right $\mathcal{D}_{X}^{(0)}$-module structure on $\mathcal{D}_{X}^{(1)}/\mathcal{D}_{X}^{(1)}\cdot\mathcal{J}$
as follows: let $\Phi\in\mathcal{D}_{X}^{(1)}/\mathcal{D}_{X}^{(1)}\cdot\mathcal{J}$
be a section, over some open affine subset $U$ which possesses local
coordinates. Let $\partial$ be a derivation over $U$. We may choose
a differential operator $\delta$ of order $p$ on $U$ such that
$\delta(f^{p})=(\partial f)^{p}$ for all $f\in\mathcal{O}_{X}(U)$;
for instance, if $\partial=\sum a_{i}\partial_{i}$ then we may choose
$\delta=\sum a_{i}^{p}\partial_{i}^{[p]}$. If $\delta'$ is another
such differential operator, then $\delta-\delta'$ is a section of
$\overline{\mathcal{D}_{X}^{(0)}}(U)$ which annihilates $\mathcal{O}_{X}(U)^{p}$.
In particular, $\delta-\delta'\in\mathcal{J}$, and so $\Phi\delta=\Phi\delta'$
inside $\mathcal{D}_{X}^{(1)}/\mathcal{D}_{X}^{(1)}\cdot\mathcal{J}$. 

So ,if we set $\Phi\star f=\Phi\cdot f^{p}$ and $\Phi\star\partial=\Phi\cdot\delta$
we obtain a (semilinear) right action of $\mathcal{D}_{X}^{(0)}$
on $\mathcal{D}_{X}^{(1)}/\mathcal{D}_{X}^{(1)}\cdot\mathcal{J}$.
Since $\mathcal{O}_{X}$ acts on $\mathcal{D}_{X}^{(1)}/\mathcal{D}_{X}^{(1)}\cdot\mathcal{J}$
on the left, the map 
\[
(f,\Psi)\to f\star\Psi
\]
induces a morphism $F^{*}\mathcal{D}_{X}^{(0)}\to\mathcal{D}_{X}^{(1)}/\mathcal{D}_{X}^{(1)}\cdot\mathcal{J}$.
To show it is an isomorphism, let us consider filtrations: by \thmref{Local-Coords-for-D+}
we have that $F^{i}(\mathcal{D}_{X}^{(1)})(U)$ is the free $\overline{\mathcal{D}_{X}^{(0)}}(U)$
module on $\{(\partial_{1}^{[p]})^{j_{1}}\cdots(\partial_{n}^{[p]})^{j_{n}}\}_{|J|\leq i}$.
Since $\overline{\mathcal{D}_{X}^{(0)}}/\mathcal{J}\tilde{\to}\mathcal{O}_{X}$,
we see that $F^{i}(\mathcal{D}_{X}^{(1)}/\mathcal{D}_{X}^{(1)}\cdot\mathcal{J})$
is the free $\mathcal{O}_{X}(U)$-module on $\{(\partial_{1}^{[p]})^{j_{1}}\cdots(\partial_{n}^{[p]})^{j_{n}}\}_{|J|\leq i}$.
On the other hand, $F^{i}(\mathcal{D}_{X}^{(0)})(U)$ is the free
$\mathcal{O}_{X}(U)$-module on $\{\partial_{1}^{j_{1}}\cdots\partial_{n}^{j_{n}}\}_{|J|\leq i}$.
Since $1\star\partial_{i}=\partial_{i}^{[p]}$ we deduce $F^{*}F^{i}(\mathcal{D}_{X}^{(0)})=F^{i}(\mathcal{D}_{X}^{(1)}/\mathcal{D}_{X}^{(1)}\cdot\mathcal{J})$
which implies that the map $F^{*}\mathcal{D}_{X}^{(0)}\to\mathcal{D}_{X}^{(1)}/\mathcal{D}_{X}^{(1)}\cdot\mathcal{J}$
is an isomorphism. The same calculation gives 
\[
F^{i}(\mathcal{D}_{X}^{(1)}/\mathcal{D}_{X}^{(1)}\cdot\mathcal{J})\cdot F_{j}(\mathcal{D}_{X}^{(0)})=F^{i+j}(\mathcal{D}_{X}^{(1)}/\mathcal{D}_{X}^{(1)}\cdot\mathcal{J})
\]
Therefore, the map 
\[
\mathcal{D}_{X}^{(1)}\to\mathcal{E}nd_{\mathcal{D}_{X}^{(0),\text{op}}}(F^{*}\mathcal{D}_{X}^{(0)})
\]
is a morphism of filtered algebras, where the filtration on the right
hand side is defined by 
\[
F^{i}(\mathcal{E}nd_{\mathcal{D}_{X}^{(0)}}(F^{*}\mathcal{D}_{X}^{(0)}))=\{\varphi\in\mathcal{E}nd_{\mathcal{D}_{X}^{(0)}}(F^{*}\mathcal{D}_{X}^{(0)})|\varphi(F^{j}(F^{*}\mathcal{D}_{X}^{(0)})\subset F^{i+j}(F^{*}\mathcal{D}_{X}^{(0)})\phantom{i}\text{for all}\phantom{i}j\}
\]
Upon passing to the associated graded, we obtain the morphism 
\[
\text{gr}(\mathcal{D}_{X}^{(1)})\to\text{gr}\mathcal{E}nd_{\mathcal{D}_{X}^{(0),\text{op}}}(F^{*}\mathcal{D}_{X}^{(0)})\tilde{=}\mathcal{E}nd_{\text{gr}(\mathcal{D}_{X}^{(0)})}(\text{gr}(F^{*}\mathcal{D}_{X}^{(0)}))
\]
(the last isomorphism follows from the fact that $F^{*}\mathcal{D}_{X}^{(0)}$
is a locally free filtered module over $\mathcal{D}_{X}^{(0)}$).
Working in local coordinates, we obtain the morphism 
\[
\mathcal{\overline{D}}_{X}^{(0)}[\partial_{1}^{[p]},\dots,\partial_{n}^{[p]}]\to\mathcal{E}nd_{\text{Sym}_{\mathcal{O}_{X}}(\mathcal{T}_{X})}(F^{*}(\text{Sym}_{\mathcal{O}_{X}}(\mathcal{T}_{X})))
\]
where $\partial_{i}^{[p]}$ is sent to $\partial_{i}\in\mathcal{T}_{X}$.
By Cartier descent, there is an isomorphism $\mathcal{\overline{D}}_{X}^{(0)}\tilde{=}\mathcal{E}nd_{\mathcal{O}_{X}}(F^{*}\mathcal{O}_{X})$
(here, the action of $\mathcal{O}_{X}$ on $F^{*}\mathcal{O}_{X}$
is on the right-hand factor in the tensor product; in other words,
it is the action of $\mathcal{O}_{X}$ on itself through the Frobenius);
and so we see that this map is an isomorphism. Thus the map $\mathcal{D}_{X}^{(1)}\to\mathcal{E}nd_{\mathcal{D}_{X}^{(0),\text{op}}}(F^{*}\mathcal{D}_{X}^{(0)})$
is an isomorphism as claimed. 
\end{proof}
This yields a functor $\mathcal{M}\to F^{*}\mathcal{M}:=F^{*}\mathcal{D}_{X}^{(0)}\otimes_{\mathcal{D}_{X}^{(0)}}\mathcal{M}$
(the Frobenius pullback) from $\mathcal{D}_{X}^{(0)}-\text{mod}$
to $\mathcal{D}_{X}^{(1)}-\text{mod}$; from the last part of the
above proposition and standard Morita theory one sees that it is an
equivalence of categories. Further: 
\begin{thm}
\label{thm:Filtered-Frobenius} The Frobenius pullback $F^{*}$ can
be upgraded to an equivalence from $\mathcal{G}(\mathcal{R}(\mathcal{D}_{X}^{(0)}))$
to $\mathcal{G}(\mathcal{R}(\mathcal{D}_{X}^{(1)}))$. Therefore,
the functor $F^{*}$ can also be upgraded to an equivalence of categories
from filtered $\mathcal{D}_{X}^{(0)}$-modules to filtered $\mathcal{D}_{X}^{(1)}$-modules.
In particular, $\mathcal{R}(\mathcal{D}_{X}^{(1)})(U)$ has finite
homological dimension for each open affine $U$. 
\end{thm}

\begin{proof}
In \propref{Basic-F^*-over-k}, we showed that $F^{*}\mathcal{D}_{X}^{(0)}$
is filtered, in a way that strictly respects the filtered action of
both $\mathcal{D}_{X}^{(1)}$ and $\mathcal{D}_{X}^{(0)}$. So, consider
the Rees module $\mathcal{R}(F^{*}\mathcal{D}_{X}^{(0)})$. This is
a graded $(\mathcal{R}(\mathcal{D}_{X}^{(1)}),\mathcal{R}(\mathcal{D}_{X}^{(0)}))$-bimodule;
and the isomorphism $F^{*}F_{i}(\mathcal{D}_{X}^{(0)}\tilde{=}F_{i}(F^{*}\mathcal{D}_{X}^{(0)})$
proved in loc.cit. shows that $\mathcal{R}(F^{*}\mathcal{D}_{X}^{(0)})\tilde{=}F^{*}\mathcal{R}_{X}^{(0)}$
as a right $\mathcal{R}(\mathcal{D}_{X}^{(0)})$-module. Thus the
result will follow if we can show that the action map 
\begin{equation}
\mathcal{R}(\mathcal{D}_{X}^{(1)})\to\mathcal{E}nd_{\mathcal{R}(\mathcal{D}_{X}^{(0)})}(\mathcal{R}(F^{*}\mathcal{D}_{X}^{(0)}))\tilde{=}\underline{\mathcal{E}nd}{}_{\mathcal{R}(\mathcal{D}_{X}^{(0)})}(\mathcal{R}(F^{*}\mathcal{D}_{X}^{(0)}))\label{eq:first-map}
\end{equation}
is an isomorphism (the latter isomorphism follows from the fact that
$\mathcal{R}(F^{*}\mathcal{D}_{X}^{(0)})$ is coherent over $\mathcal{R}(\mathcal{D}_{X}^{(0)})$).
Both sides are therefore positively graded algebras over the ring
$k[f]$; taking reduction mod $f$ we obtain the map $\text{gr}(\mathcal{D}_{X}^{(1)})\to\mathcal{E}nd_{\text{gr}(\mathcal{D}_{X}^{(0)})}(\text{gr}(F^{*}\mathcal{D}_{X}^{(0)}))$
which we already showed to be an isomorphism. Thus by the graded Nakayama
lemma \eqref{first-map} is surjective. As both sides are $f$-torsion
free it follows that it is an isomorphism. Thus the first result of
$1)$ is proved, the second follows by identifying filtered modules
with graded modules over the Rees ring which are torsion-free with
respect to $f$. 
\end{proof}
\begin{rem}
\label{rem:The-inverse-to-F^*}The inverse to the functor $F^{*}$
can be described as follows: via the embedding $\overline{\mathcal{D}_{X}^{(0)}}\subset\mathcal{R}(\mathcal{D}_{X}^{(1)})$,
any module $\mathcal{M}$ over $\mathcal{R}(\mathcal{D}_{X}^{(1)})$
possesses a connection which has $p$-curvature $0$. Apply this to
$\mathcal{R}(\mathcal{D}_{X}^{(1)})/\mathcal{J}\cdot\mathcal{R}(\mathcal{D}_{X}^{(1)})$,
where as above $\mathcal{J}\subset\overline{\mathcal{D}_{X}^{(0)}}$
denotes the annihilator of $1\in\mathcal{O}_{X}$ under the action
of $\overline{\mathcal{D}_{X}^{(0)}}$ on $\mathcal{O}_{X}$. We obtain
from the above argument the isomorphism 
\[
(\mathcal{R}(\mathcal{D}_{X}^{(1)})/\mathcal{J}\cdot\mathcal{R}(\mathcal{D}_{X}^{(1)}))^{\nabla}\tilde{=}\mathcal{R}(\mathcal{D}_{X^{(1)}}^{(0)})
\]
Thus for any the sheaf $\mathcal{M}^{\nabla}:=\text{ker}(\nabla:\mathcal{M}\to\mathcal{M})$
inherits the structure of a module over $\mathcal{R}(\mathcal{D}_{X^{(1)}}^{(0)})$.
As $k$ is perfect we have an isomorphism of schemes $\sigma:X^{(1)}\to X$,
and so composing with this we can obtain from $\mathcal{M}^{\nabla}$
an $\mathcal{R}(\mathcal{D}_{X}^{(0)})$-module; this is the inverse
functor to $F^{*}$.
\end{rem}

To close out this subsection, we'd like to discuss an important tool
for studying $\mathcal{G}(\mathcal{D}_{X}^{(0,1)})$; namely, reducing
statements to their analogues in $\mathcal{R}(\mathcal{D}_{X}^{(1)})$
and $\overline{\mathcal{R}}(\mathcal{D}_{X}^{(0)})$. For any $\mathcal{M}\in\mathcal{G}(\mathcal{D}_{X}^{(0,1)})$,
we have a short exact sequence 
\[
0\to\text{ker}(f)\to\mathcal{M}\to\mathcal{M}/\text{ker}(f)\to0
\]
the module on the left is annihilated by $f$; i.e., it is a module
over $\mathcal{\overline{R}}(\mathcal{D}_{X}^{(0)})$ while the module
on the right is annihilated by $v$; i.e., it is a module over $\mathcal{R}(\mathcal{D}_{X}^{(1)})$.
This allows us to deduce many basic structural properties of $\mathcal{G}(\mathcal{D}_{X}^{(0,1)})$
from properties of $\mathcal{G}(\mathcal{R}(\mathcal{D}_{X}^{(1)}))$
and $\mathcal{G}(\mathcal{\overline{R}}(\mathcal{D}_{X}^{(0)}))$.
We now give the key technical input; to state it, we will abuse notation
slightly, so that if $\mathcal{M}^{\cdot}\in D(\mathcal{G}(\mathcal{R}(\mathcal{D}_{X}^{(1)})))$
(or in $D(\mathcal{G}(\mathcal{\overline{R}}(\mathcal{D}_{X}^{(0)})))$)
we will use the same symbol $\mathcal{M}^{\cdot}$ to denote its image
in $D(\mathcal{G}(\mathcal{D}_{X}^{(0,1)}))$
\begin{prop}
\label{prop:Sandwich!}1) Let $\mathcal{M}^{\cdot}\in D(\mathcal{G}(\mathcal{R}(\mathcal{D}_{X}^{(1)})))$.
Suppose $\mathcal{N}\in\mathcal{G}(\mathcal{D}_{X}^{(0,1),\text{opp}})$
is quasi-rigid. Then 
\[
\mathcal{N}\otimes_{\mathcal{D}_{X}^{(0,1)}}^{L}\mathcal{M}^{\cdot}\tilde{=}\mathcal{N}/v\otimes_{\mathcal{R}(\mathcal{D}_{X}^{(1)})}^{L}\mathcal{M}^{\cdot}
\]
Similarly, if $\mathcal{M}^{\cdot}\in D(\mathcal{G}(\mathcal{\overline{R}}(\mathcal{D}_{X}^{(0)})))$
we have 
\[
\mathcal{N}\otimes_{\mathcal{D}_{X}^{(0,1)}}^{L}\mathcal{M}^{\cdot}\tilde{=}\mathcal{N}/f\otimes_{\overline{\mathcal{R}}(\mathcal{D}_{X}^{(0)})}^{L}\mathcal{M}^{\cdot}
\]
The analogous statement holds for $\mathcal{N}\in\mathcal{G}(\mathcal{D}_{X}^{(0,1)})$
quasi-rigid and $\mathcal{M}^{\cdot}\in D(\mathcal{G}(\mathcal{R}(\mathcal{D}_{X}^{(1)})^{\text{opp}}))$,
resp. $\mathcal{M}^{\cdot}\in D(\mathcal{G}(\mathcal{\overline{R}}(\mathcal{D}_{X}^{(0)})^{\text{opp}}))$. 

2) As above, let $\mathcal{M}^{\cdot}\in D(\mathcal{G}(\mathcal{R}(\mathcal{D}_{X}^{(1)})))$
and suppose $\mathcal{N}$ is quasi-rigid. Then 
\[
R\underline{\mathcal{H}om}_{\mathcal{D}_{X}^{(0,1)}}(\mathcal{M}^{\cdot},\mathcal{N})\tilde{=}R\underline{\mathcal{H}om}_{\mathcal{R}(\mathcal{D}_{X}^{(1)})}(\mathcal{M}^{\cdot},\text{ker}(v:\mathcal{N}\to\mathcal{N}))
\]
Similarly, if $\mathcal{M}^{\cdot}\in D(\mathcal{G}(\mathcal{\overline{R}}(\mathcal{D}_{X}^{(0)})))$
then 
\[
R\underline{\mathcal{H}om}_{\mathcal{D}_{X}^{(0,1)}}(\mathcal{M},\mathcal{N})\tilde{=}R\underline{\mathcal{H}om}_{\overline{\mathcal{R}}(\mathcal{D}_{X}^{(0)})}(\mathcal{M},\text{ker}(f:\mathcal{N}\to\mathcal{N}))
\]
\end{prop}

\begin{proof}
1) Choose a flat resolution $\mathcal{F}^{\cdot}\to\mathcal{N}$ (in
the category of right $\mathcal{D}_{X}^{(0,1)}$-gauges); concretely,
the terms of $\mathcal{F}^{\cdot}$ are direct sums of sheaves of
the form $j_{!}(\mathcal{D}_{X}^{(0,1)}(i)|_{U})$ (where $U\subset X$
is open and $j_{!}$ denotes extension by zero). Then $\mathcal{N}\otimes_{\mathcal{D}_{X}^{(0,1)}}^{L}\mathcal{M}^{\cdot}$
is represented by the complex 
\[
\mathcal{F}^{\cdot}\otimes_{\mathcal{D}_{X}^{(0,1)}}\mathcal{M}^{\cdot}\tilde{=}(\mathcal{F}/v)^{\cdot}\otimes_{\mathcal{R}(\mathcal{D}_{X}^{(1)})}\mathcal{M}^{\cdot}
\]
where the isomorphism follows from the fact that each term of $\mathcal{M}^{\cdot}$
is annihilated by $v$. On the other hand, $(\mathcal{F}/v)^{\cdot}$
is a complex, whose terms are direct sums of sheaves of the form $j_{!}(\mathcal{R}(\mathcal{D}_{X}^{(1)}(i)|_{U})$,
which computes $\mathcal{N}\otimes_{\mathcal{D}_{X}^{(0,1)}}^{L}\mathcal{R}(\mathcal{D}_{X}^{(1)})$.
However, we have $\mathcal{N}\otimes_{\mathcal{D}_{X}^{(0,1)}}^{L}\mathcal{R}(\mathcal{D}_{X}^{(1)})\tilde{=}\mathcal{N}/v$
by the assumption on $\mathcal{N}$ (c.f. \lemref{Basic-Facts-on-Rigid})
Therefore $(\mathcal{F}/v)^{\cdot}$ is a flat resolution (in the
category of graded right $\mathcal{R}(\mathcal{D}_{X}^{(1)})$-modules)
of $\mathcal{N}$, and so 
\[
(\mathcal{F}/v)^{\cdot}\otimes_{\mathcal{R}(\mathcal{D}_{X}^{(1)})}\mathcal{M}^{\cdot}\tilde{=}\mathcal{N}\otimes_{\mathcal{R}(\mathcal{D}_{X}^{(1)})}^{L}\mathcal{M}^{\cdot}
\]
as claimed. The case $\mathcal{M}^{\cdot}\in D(\mathcal{G}(\mathcal{\overline{R}}(\mathcal{D}_{X}^{(0)})))$
is essentially identical. 

2) Choose an injective resolution $\mathcal{N}\to\mathcal{I}^{\cdot}$.
Then we have that $R\underline{\mathcal{H}om}_{\mathcal{D}_{X}^{(0,1)}}(\mathcal{M}^{\cdot},\mathcal{N})$
is represented by
\[
\underline{\mathcal{H}om}_{\mathcal{D}_{X}^{(0,1)}}(\mathcal{M}^{\cdot},\mathcal{I}^{\cdot})=\underline{\mathcal{H}om}_{\mathcal{D}_{X}^{(0,1)}}(\mathcal{M}^{\cdot},\mathcal{I}^{\cdot,v=0})=\underline{\mathcal{H}om}_{\mathcal{R}(\mathcal{D}_{X}^{(1)})}(\mathcal{M}^{\cdot},\mathcal{I}^{\cdot,v=0})
\]
where $\mathcal{I}^{j,v=0}=\{m\in\mathcal{I}^{j}|vm=0\}$. From the
isomorphism $\underline{\mathcal{H}om}_{\mathcal{D}_{X}^{(0,1)}}(\mathcal{M},\mathcal{I}^{\cdot})=\underline{\mathcal{H}om}_{\mathcal{R}(\mathcal{D}_{X}^{(1)})}(\mathcal{M},\mathcal{I}^{\cdot,v=0})$
we see that the the functor $\mathcal{I}\to\mathcal{I}^{v=0}$ takes
injectives in $\mathcal{G}(\mathcal{D}_{X}^{(0,1)})$ to injectives
in $\mathcal{G}(\mathcal{R}(\mathcal{D}_{X}^{(1)}))$. On the other
hand, we have $\mathcal{I}^{j,v=0}=\underline{\mathcal{H}om}_{\mathcal{D}_{X}^{(0,1)}}(\mathcal{R}(\mathcal{D}_{X}^{(1)}),\mathcal{I}^{j,v=0})$.
Thus the functor $\underline{\mathcal{H}om}_{\mathcal{D}_{X}^{(0,1)}}(\mathcal{R}(\mathcal{D}_{X}^{(1)}),)$
takes injectives in $\mathcal{G}(\mathcal{D}_{X}^{(0,1)})$ to injectives
in $\mathcal{R}(\mathcal{D}_{X}^{(1)})$ and so 
\[
R\underline{\mathcal{H}om}_{\mathcal{D}_{X}^{(0,1)}}(\mathcal{M}^{\cdot},\mathcal{N})\tilde{=}R\underline{\mathcal{H}om}_{\mathcal{R}(\mathcal{D}_{X}^{(1)})}(\mathcal{M}^{\cdot},R\underline{\mathcal{H}om}_{\mathcal{D}_{X}^{(0,1)}}(\mathcal{R}(\mathcal{D}_{X}^{(1)}),\mathcal{N}))
\]
On the other hand, using the resolution 
\[
\cdots\to\mathcal{D}_{X}^{(0,1)}(-1)\xrightarrow{v}\mathcal{D}_{X}^{(0,1)}\xrightarrow{f}\mathcal{D}_{X}^{(0,1)}(-1)\xrightarrow{v}\mathcal{D}_{X}^{(0,1)}\to\mathcal{R}(\mathcal{D}_{X}^{(1)})
\]
one deduces 
\[
R\underline{\mathcal{H}om}_{\mathcal{D}_{X}^{(0,1)}}(\mathcal{R}(\mathcal{D}_{X}^{(1)}),\mathcal{N})\tilde{=}\text{ker}(v:\mathcal{N}\to\mathcal{N})
\]
and the first statement in $2)$ follows; the second statement is
proved in an identical fashion. 
\end{proof}
Here is a typical application: 
\begin{prop}
\label{prop:Quasi-rigid=00003Dfinite-homological}A quasicoherent
gauge $\mathcal{N}\in\mathcal{G}_{qcoh}(\mathcal{D}_{X}^{(0,1)})$
is quasi-rigid iff, for each open affine $U\subset X$, $\mathcal{N}(U)$
has finite projective dimension over $\mathcal{D}_{X}^{(0,1)}(U)$. 
\end{prop}

\begin{proof}
Let $\mathcal{N}$ be quasi-rigid. Then for any quasicoherent $\mathcal{M}\in\mathcal{G}(\mathcal{D}_{X}^{(0,1),\text{opp}})$,
we have the short exact sequence 
\[
0\to\text{ker}(f)\to\mathcal{M}\to\mathcal{M}/\text{ker}(f)\to0
\]
which yields the distinguished triangle 
\[
0\to\mathcal{N}\otimes_{\mathcal{D}_{X}^{(0,1)}}^{L}\text{ker}(f)\to\mathcal{N}\otimes_{\mathcal{D}_{X}^{(0,1)}}^{L}\mathcal{M}\to\mathcal{N}\otimes_{\mathcal{D}_{X}^{(0,1)}}^{L}\mathcal{M}/\text{ker}(f)\to0
\]
Applying the previous result; we see that the outer two tensor products
are isomorphic to tensor products over $\mathcal{R}(\mathcal{D}_{X}^{(1)})$
and $\mathcal{\overline{R}}(\mathcal{D}_{X}^{(0)})$, respectively.
As these algebras have finite homological dimension (the dimension
is $2\text{dim}(X)+1$, in fact) over any open affine, we see that
$\mathcal{N}\otimes_{\mathcal{D}_{X}^{(0,1)}}^{L}\mathcal{M}$ is
a bounded complex; since this is true for all $\mathcal{M}$ we obtain
the forward implication. For the reverse, note that by \lemref{Basic-Facts-on-Rigid},
the functor $\mathcal{M}\to\mathcal{R}(\mathcal{D}_{X}^{(1)})\otimes_{\mathcal{D}_{X}^{(0,1)}}\mathcal{M}\tilde{\to}k[f]\otimes_{D(k)}\mathcal{M}$
has infinite homological dimension when $\mathcal{M}$ is not quasi-rigid. 
\end{proof}

\subsection{\label{subsec:Frobenius-Descent,--Gauges}Frobenius Descent, $F^{-1}$-Gauges}

In this section we recall Berthelot's theory of Frobenius descent
for $\mathcal{D}$-modules and give the definition of an $F^{-1}$-gauge
over a higher dimensional base. 

We begin by briefly recalling Berthelot's theory of the Frobenius
action in mixed characteristic. This is developed using the theory
of (mixed) divided powers in \cite{key-2}; for the reader's convenience
we will recall a simple description in the case of interest to us
(this point of view is emphasized in \cite{key-48}). 

First suppose that $\mathfrak{X}$ admits an endomorphism $F$ which
lifts the Frobenius on $X$, and whose restriction to $W(k)$ agrees
with the Witt-vector Frobenius on $W(k)$. This is equivalent to giving
a morphism $\mathfrak{X}\to\mathfrak{X}^{(1)}$ whose composition
with the natural map $\mathfrak{X}^{(1)}\to\mathfrak{X}$ agrees with
$F$ (here, $\mathfrak{X}^{(1)}$ denotes the first Frobenius twist
of $\mathfrak{X}$ over $W(k)$); we will also denote the induced
morphism $\mathfrak{X}\to\mathfrak{X}^{(1)}$ by $F$. On the underlying
topological spaces (namely $X$ and $X^{(1)}$), this map is a bijection,
and we shall consistently consider $\mathcal{O}_{\mathfrak{X}^{(1)}}$
as a sheaf of rings on $X$, equipped with an injective map of sheaves
of algebras $F^{\#}:\mathcal{O}_{\mathfrak{X}^{(1)}}\to\mathcal{O}_{\mathfrak{X}}$
which makes $\mathcal{O}_{\mathfrak{X}}$ into a finite $\mathcal{O}_{\mathfrak{X}^{(1)}}$-module.

Now consider the sheaf $\mathcal{H}om_{W(k)}(\mathcal{O}_{\mathfrak{X}^{(1)}},\mathcal{O}_{\mathfrak{X}})$.
For any $i\geq0$, this is a $(\mathcal{D}_{\mathfrak{X}}^{(i+1)},\mathcal{D}_{\mathfrak{X}^{(1)}}^{(i)})$
bi-module (via the actions of these rings of differential operators
on $\mathcal{O}_{\mathfrak{X}}$ and $\mathcal{O}_{\mathfrak{X}^{(1)}}$,
respectively). Then we have the 
\begin{thm}
\label{thm:Berthelot-Frob}(Berthelot) The $(\mathcal{\widehat{D}}_{\mathfrak{X}}^{(i+1)},\mathcal{\widehat{D}}_{\mathfrak{X}^{(1)}}^{(i)})$
bi-sub-module of $\mathcal{H}om_{W(k)}(\mathcal{O}_{\mathfrak{X}^{(1)}},\mathcal{O}_{\mathfrak{X}})$
locally generated by $F^{\#}$ is isomorphic to $\mathcal{O}_{\mathfrak{X}}\otimes_{\mathcal{O}_{\mathfrak{X}^{(1)}}}\mathcal{\widehat{D}}_{\mathfrak{X}^{(1)}}^{(i)}$,
via the map 
\[
(f,\Phi)\to f\circ F^{\#}\circ\Phi\in\mathcal{H}om_{W(k)}(\mathcal{O}_{\mathfrak{X}^{(1)}},\mathcal{O}_{\mathfrak{X}})
\]
for local sections $f\in\mathcal{O}_{\mathfrak{X}}$ and $\Phi\in\mathcal{\widehat{D}}_{\mathfrak{X}^{(1)}}^{(i)}$.
This gives the sheaf $\mathcal{O}_{\mathfrak{X}}\otimes_{\mathcal{O}_{\mathfrak{X}^{(1)}}}\mathcal{\widehat{D}}_{\mathfrak{X}^{(1)}}^{(i)}=F^{*}\mathcal{\widehat{D}}_{\mathfrak{X}^{(1)}}^{(i)}$
the structure of a $(\mathcal{\widehat{D}}_{\mathfrak{X}}^{(i+1)},\mathcal{\widehat{D}}_{\mathfrak{X}^{(1)}}^{(i)})$
bimodule. The associated functor, denoted $F^{*}$, 
\[
\mathcal{M}\to F^{*}\mathcal{\widehat{D}}_{\mathfrak{X}^{(1)}}^{(i)}\otimes_{\mathcal{D}_{\mathfrak{X}^{(1)}}^{(i)}}\mathcal{M}\tilde{=}F^{*}\mathcal{M}
\]
is an equivalence of categories from $\mathcal{\widehat{D}}_{\mathfrak{X}^{(1)}}^{(i)}-\text{mod}\to\mathcal{\widehat{D}}_{\mathfrak{X}}^{(i+1)}-\text{mod}$
; which induces an equivalence $\text{Coh}(\mathcal{\widehat{D}}_{\mathfrak{X}^{(1)}}^{(i)})\to\text{Coh}(\mathcal{\widehat{D}}_{\mathfrak{X}}^{(i+1)})$.
In particular, the map $\widehat{\mathcal{D}}_{\mathfrak{X}}^{(i+1)}\to\mathcal{E}nd_{\mathcal{\widehat{D}}_{\mathfrak{X}^{(1)}}^{(i),\text{op}}}(F^{*}\mathcal{\widehat{D}}_{\mathfrak{X}^{(1)}}^{(i)})$
is an isomorphism of sheaves of algebras. 

As $W(k)$ is perfect, we have an isomorphism $\mathfrak{X}^{(1)}\tilde{\to}\mathfrak{X}$;
and we may therefore regard $F^{*}$ as being an equivalence of categories
from $\mathcal{\widehat{D}}_{\mathfrak{X}}^{(i)}-\text{mod}$ to $\mathcal{\widehat{D}}_{\mathfrak{X}}^{(i+1)}-\text{mod}$
\end{thm}

This is proved in \cite{key-2}, section 2.3. In fact, in the case
where $\mathfrak{X}=\text{Specf}(\mathcal{A})$ is affine and admits
etale local coordinates, and the map $F$ acts on coordinates $\{t_{i}\}_{i=1}^{n}$
via $F(t_{i})=t_{i}^{p}$, then the first assertion can be proved
quite directly. The second is the theory of \cite{key-2}. Note that
this description implies that the reduction mod $p$ of the bimodule
$F^{*}\mathcal{\widehat{D}}_{\mathfrak{X}}^{(0)}$ agrees with the
bimodule $F^{*}\mathcal{D}_{X}^{(0)}$ constructed in \propref{Basic-F^*-over-k}. 
\begin{rem}
\label{rem:Compare-With-Berthelot}1) Let $\mathcal{D}_{X,\mathbf{Ber}}^{(1)}$
denote Berthelot's ring of divided power differential operators of
level $1$ on $X$. Then the Frobenius descent theory of the previous
theorem gives an isomorphism 
\[
\mathcal{D}_{X,\text{Ber}}^{(1)}\to\mathcal{E}nd_{\mathcal{D}_{X}^{(0),\text{op}}}(F^{*}\mathcal{D}_{X}^{(0)})
\]
It follows that $\mathcal{D}_{X,\mathbf{Ber}}^{(1)}\tilde{=}\mathcal{D}_{X}^{(1)}$
even if $X$ is not liftable to $W(k)$. 

2) The Frobenius descent over $X$ implies the Frobenius descent over
$\mathfrak{X}$, once one constructs the $(\mathcal{\widehat{D}}_{\mathfrak{X}}^{(1)},\mathcal{\widehat{D}}_{\mathfrak{X}}^{(0)})$
bimodule structure on $F^{*}\widehat{\mathcal{D}}_{\mathfrak{X}}^{(0)}$.
Indeed, this structure yields a morphism 
\[
\widehat{\mathcal{D}}_{\mathfrak{X}}^{(1)}\to\mathcal{E}nd_{\mathcal{\widehat{D}}_{\mathfrak{X}}^{(0),\text{op}}}(F^{*}\mathcal{\widehat{D}}_{\mathfrak{X}}^{(0)})
\]
as both sides are $p$-adically complete and $p$-torsion-free, to
check that this map is an isomorphism one simply has to reduce mod
$p$. 
\end{rem}

Now let us return to a general $\mathfrak{X}$. It is an fundamental
fact that Frobenius descent doesn't really depend on the existence
of the lift $F$:
\begin{thm}
(Berthelot) Suppose $F_{1},F_{2}$ are two lifts of Frobenius on $\mathfrak{X}$.
Then there is an isomorphism of bimodules $\sigma_{1,2}:F_{1}^{*}\mathcal{D}_{\mathfrak{X}}^{(i)}\tilde{\to}F_{2}^{*}\mathcal{D}_{\mathfrak{X}}^{(i)}$.
If $F_{3}$ is a third lift, we have $\sigma_{2,3}\circ\sigma_{12}=\sigma_{1,3}$. 
\end{thm}

This is \cite{key-2}, theorem 2.2.5; c.f. also \cite{key-21}, corollary
13.3.8. As lifts of Frobenius always exist locally, this implies that
there is a globally defined bimodule $F^{*}\mathcal{\widehat{D}}_{\mathfrak{X}}^{(i)}$,
which induces an equivalence $F^{*}:\mathcal{\widehat{D}}_{\mathfrak{X}}^{(i)}-\text{mod}\to\mathcal{\widehat{D}}_{\mathfrak{X}}^{(i+1)}-\text{mod}$;
we use the same letter to denote the derived equivalence $F^{*}:D(\mathcal{\widehat{D}}_{\mathfrak{X}}^{(i)}-\text{mod})\to D(\mathcal{\widehat{D}}_{\mathfrak{X}}^{(i+1)}-\text{mod})$. 

The equivalence of categories $F^{*}$ has many remarkable properties,
in particular its compatibility with the push-forward, pullback, and
duality functors for $\mathcal{D}$-modules; we will recall these
properties in the relevant sections below. 

It will also be useful to recall some basic facts about the right-handed
version of the equivalence. Recall that we have equivalences of categories
$\mathcal{M}\to\omega_{\mathfrak{X}}\otimes_{\mathcal{O}_{\mathfrak{X}}}\mathcal{M}$
from $\mathcal{\widehat{D}}_{\mathfrak{X}}^{(i)}-\text{mod}$ to $\mathcal{\widehat{D}}_{\mathfrak{X}}^{(i),\text{op}}-\text{mod}$
for any $i$ (c.f, \cite{key-1}, or \propref{Left-Right-Swap} below).
This implies that there is a functor $\mathcal{M}\to F^{!}\mathcal{M}:=\omega_{\mathfrak{X}}\otimes_{\mathcal{O}_{\mathfrak{X}}}F^{*}(\omega_{\mathfrak{X}}^{-1}\otimes_{\mathcal{O}_{\mathfrak{X}}}\mathcal{M})$
which is an equivalence from $\mathcal{\widehat{D}}_{\mathfrak{X}}^{(i),\text{op}}-\text{mod}$
to $\mathcal{\widehat{D}}_{\mathfrak{X}}^{(i+1),\text{op}}-\text{mod}$.
By basic Grothendieck duality theory (c.f. \cite{key-2}, 2.4.1),
there is an isomorphism 
\[
F^{!}\mathcal{M}\tilde{=}F^{-1}\mathcal{H}om_{\mathcal{O}_{\mathfrak{X}}}(F_{*}\mathcal{O}_{\mathfrak{X}},\mathcal{M})
\]
of sheaves of $\mathcal{O}_{\mathfrak{X}}$-modules (this justifies
the notation). If we put $\mathcal{M}=\mathcal{\widehat{D}}_{\mathfrak{X}}^{(i)}$
this isomorphism exhibits the left $\mathcal{\widehat{D}}_{\mathfrak{X}}^{(i)}$-module
structure on $F^{!}\mathcal{\widehat{D}}_{\mathfrak{X}}^{(i)}$. 
\begin{prop}
\label{prop:F^*F^!}1) The equivalence of categories $F^{!}:\mathcal{\widehat{D}}_{\mathfrak{X}}^{(i),\text{op}}-\text{mod}\to\mathcal{\widehat{D}}_{\mathfrak{X}}^{(i+1),\text{op}}-\text{mod}$
is given by $\mathcal{M}\to\mathcal{M}\otimes_{\mathcal{\widehat{D}}_{\mathfrak{X}}^{(i)}}F^{!}\mathcal{\widehat{D}}_{\mathfrak{X}}^{(i)}$. 

2) There are isomorphisms of $(\mathcal{\widehat{D}}_{\mathfrak{X}}^{(i+!)},\mathcal{\widehat{D}}_{\mathfrak{X}}^{(i+1)})$
bimodules $F^{*}\mathcal{\widehat{D}}_{\mathfrak{X}}^{(i)}\otimes_{\mathcal{\widehat{D}}_{\mathfrak{X}}^{(i)}}F^{!}\mathcal{\widehat{D}}_{\mathfrak{X}}^{(i)}=F^{*}F^{!}\mathcal{\widehat{D}}_{\mathfrak{X}}^{(i)}\tilde{\to}\mathcal{\widehat{D}}_{\mathfrak{X}}^{(i+1)}$
and $\mathcal{\widehat{D}}_{\mathfrak{X}}^{(i+1)}\tilde{\leftarrow}F^{!}F^{*}\mathcal{\widehat{D}}_{\mathfrak{X}}^{(i)}=F^{*}\mathcal{\widehat{D}}_{\mathfrak{X}}^{(i)}\otimes_{\mathcal{\widehat{D}}_{\mathfrak{X}}^{(i)}}F^{!}\mathcal{\widehat{D}}_{\mathfrak{X}}^{(i)}$.
In particular, for a $\mathcal{\widehat{D}}_{\mathfrak{X}}^{(i+1)}$-module
$\mathcal{M}$, we have $\mathcal{M}=F^{*}\mathcal{N}$ iff $F^{!}\mathcal{\widehat{D}}_{\mathfrak{X}}^{(i)}\otimes_{\mathcal{\widehat{D}}_{\mathfrak{X}}^{(i+1)}}\mathcal{M}\tilde{=}\mathcal{N}$. 
\end{prop}

This is proved in \cite{key-2}, 2.5.1 (c.f. also \cite{key-21},
lemma 13.5.1). Further, by applying the Rees functor it directly implies
the analogue for the filtered Frobenius descent of \thmref{Filtered-Frobenius}: 
\begin{cor}
\label{cor:Filtered-right-Frob}There is an equivalence of categories
$F^{!}:\mathcal{G}(\mathcal{R}(\mathcal{D}_{X}^{(0)})^{\text{op}})\to\mathcal{G}(\mathcal{R}(\mathcal{D}_{X}^{(1)})^{\text{op}})$;
which yields a $(\mathcal{R}(\mathcal{D}_{X}^{(0)}),\mathcal{R}(\mathcal{D}_{X}^{(1)}))$
bimodule $F^{!}\mathcal{R}(\mathcal{D}_{X}^{(0)})$. We have isomorphisms
of $(\mathcal{R}(\mathcal{D}_{X}^{(1)}),\mathcal{R}(\mathcal{D}_{X}^{(1)}))$
bimodules 
\[
F^{!}F^{*}\mathcal{R}(\mathcal{D}_{X}^{(0)})\tilde{\to}\mathcal{R}(\mathcal{D}_{X}^{(1)})\leftarrow F^{*}F^{!}\mathcal{R}(\mathcal{D}_{X}^{(0)})
\]
\end{cor}

Now we proceed to the
\begin{defn}
\label{def:Gauge-Defn!}An $F^{-1}$-gauge over $\mathfrak{X}$ is
an object of $\mathcal{G}(\mathcal{\widehat{D}}_{\mathfrak{X}}^{(0,1)})$
equipped with an isomorphism $F^{*}\mathcal{M}^{-\infty}\tilde{\to}\widehat{\mathcal{M}^{\infty}}$
(here $\widehat{?}$ denotes $p$-adic completion). A coherent $F^{-1}$-gauge
is an $F^{-1}$-gauge whose underlying $\mathcal{\widehat{D}}_{\mathfrak{X}}^{(0,1)}$-module
is coherent. We define the category of $F^{-1}$-gauges, $\mathcal{G}_{F^{-1}}(\mathcal{D}_{\mathfrak{X}}^{(0,1)})$
by demanding that morphisms between $F^{-1}$-gauges respect the $F^{-1}$-structure
(as in \defref{F-gauge}), and similarly for the category of coherent
$F^{-1}$-gauges, $\mathcal{G}_{F^{-1},coh}(\mathcal{D}_{\mathfrak{X}}^{(0,1)})$. 

Similarly, An $F^{-1}$-Gauge over $X$ is an object of $\mathcal{G}(\mathcal{D}_{X}^{(0,1)})$
equipped with an isomorphism $F^{*}\mathcal{M}^{-\infty}\tilde{\to}\widehat{\mathcal{M}^{\infty}}$,
and for the category $\mathcal{G}_{F^{-1}}(\mathcal{D}_{X}^{(0,1)})$
we demand that morphisms between $F^{-1}$-gauges respect the $F^{-1}$-structure.
We have the obvious subcategories of quasi-coherent and coherent gauges. 
\end{defn}

In the world of coherent gauges, we have seen in \propref{Completion-for-noeth}
that completion is an exact functor. Therefore, the category of coherent
$F^{-1}$-gauges over $\mathfrak{X}$ is abelian; the same does not
seem to be true for the category of all gauges over $\mathfrak{X}$.
On the other hand, the category of all $F^{-1}$-gauges over $X$
is abelian, as are the categories of coherent and quasicoherent $F^{-1}$-gauges. 

Now let us turn to the derived world:
\begin{defn}
\label{def:F-gauge-for-complexes}A complex $\mathcal{M}^{\cdot}$
in $D(\mathcal{G}(\mathcal{\widehat{D}}_{\mathfrak{X}}^{(0,1)})$
is said to admit the structure of an $F^{-1}$-gauge if there is an
isomorphism $F^{*}(\mathcal{M}^{\cdot})^{-\infty}\tilde{\to}\widehat{(\mathcal{M}^{\cdot})^{\infty}}$
where $\widehat{}$ denotes the cohomological completion. Similarly,
we say that a complex $\mathcal{M}^{\cdot}$ in $D(\mathcal{G}(\mathcal{D}_{X}^{(0,1)}))$
admits the structure of an $F^{-1}$-gauge if there is an isomorphism
$F^{*}(\mathcal{M}^{\cdot})^{-\infty}\tilde{\to}(\mathcal{M}^{\cdot})^{\infty}$.
We will use a subscript $F^{-1}$ to denote the relevant categories;
e.g. $D_{F^{-1}}(\mathcal{G}(\mathcal{\widehat{D}}_{\mathfrak{X}}^{(0,1)})$. 
\end{defn}

These are not triangulated categories in general, though there is
an obvious functor $D^{b}(\mathcal{G}_{F^{-1},coh}(\mathcal{D}_{\mathfrak{X}}^{(0,1)}))\to D_{coh,F^{-1}}^{b}(\mathcal{G}(\mathcal{\widehat{D}}_{\mathfrak{X}}^{(0,1)})$
(and similarly for $X$). To give the correct triangulated analogue
of \defref{Gauge-Defn!} one must use higher homotopy theory; namely,
the glueing of $\infty$-categories along a pair of functors. I intend
to peruse this in a later project. For the purposes of this paper,
\defref{F-gauge-for-complexes} will suffice.
\begin{rem}
\label{rem:Cut-off-for-F-gauges}Suppose $\mathcal{M}^{\cdot}\in D_{coh,F^{-1}}^{b}(\mathcal{G}(\mathcal{\widehat{D}}_{\mathfrak{X}}^{(0,1)})$,
then by \propref{Completion-for-noeth} $\widehat{\mathcal{H}^{i}(\mathcal{M}^{\cdot})^{\infty}}\tilde{=}\mathcal{H}^{i}(\widehat{\mathcal{M}^{\cdot,\infty}})$.
Therefore $\mathcal{H}^{i}(\mathcal{M}^{\cdot})$ admits the structure
of an $F^{-1}$-gauge for each $i$. Further, as both $F^{*}$ and
the completion functor are exact, we have that $\tau_{\leq i}(\mathcal{M}^{\cdot})$
and $\tau_{\geq i}(\mathcal{M}^{\cdot})$ are contained in $\mathcal{M}^{\cdot}\in D_{coh,F^{-1}}^{b}(\mathcal{G}(\mathcal{\widehat{D}}_{\mathfrak{X}}^{(0,1)})$,
where $\tau_{\leq i},\tau_{\geq i}$ are the cut-off functors. 
\end{rem}

Given this, we can give the more refined version of Mazur's theorem
for $F^{-1}$-gauges: 
\begin{thm}
\label{thm:F-Mazur}Let $\mathcal{M}^{\cdot}\in D_{\text{coh},F^{-1}}^{b}(\mathcal{G}(\widehat{\mathcal{D}}_{\mathfrak{X}}^{(0,1)}))$.
Suppose that $\mathcal{H}^{n}(\mathcal{M}^{\cdot})^{-\infty}$ is
$p$-torsion-free for all $n$, and suppose that $\mathcal{H}^{n}((\mathcal{M}^{\cdot}\otimes_{W(k)}^{L}k)\otimes_{D(k)}^{L}k[f])$
is $f$-torsion-free for all $n$. Then $\mathcal{H}^{n}(\mathcal{M}^{\cdot})$
is standard for all $n$. 

In particular, $\mathcal{H}^{n}(\mathcal{M}^{\cdot})$ is $p$-torsion-free,
and $\mathcal{H}^{n}(\mathcal{M}^{\cdot})/p$ is rigid for all $n$.
We have $\mathcal{H}^{n}(\mathcal{M}^{\cdot})/p\tilde{=}\mathcal{H}^{n}(\mathcal{M}^{\cdot}\otimes_{W(k)}^{L}k)$
, $(\mathcal{H}^{n}(\mathcal{M}^{\cdot})/p)/v\tilde{=}\mathcal{H}^{n}((\mathcal{M}^{\cdot}\otimes_{W(k)}^{L}k)\otimes_{D(k)}^{L}k[f])$,
and $(\mathcal{H}^{n}(\mathcal{M}^{\cdot})/p)/f\tilde{=}\mathcal{H}^{n}((\mathcal{M}^{\cdot}\otimes_{W(k)}^{L}k)\otimes_{D(k)}^{L}k[v])$
for all $n$. Further, $(\mathcal{H}^{n}(\mathcal{M}^{\cdot})/p)/f$
is $v$-torsion-free and $(\mathcal{H}^{n}(\mathcal{M}^{\cdot})/p)/v$
is $f$-torion-free for all $n$. 
\end{thm}

\begin{proof}
This follows from \thmref{Mazur!} if we can show that $\mathcal{H}^{n}(\mathcal{M}^{\cdot})^{\infty}\tilde{=}\mathcal{H}^{n}(\mathcal{M}^{\cdot,\infty})$
is also $p$-torsion-free for all $n$. Since $\mathcal{M}^{\cdot}\in D_{coh,F^{-1}}^{b}(\mathcal{G}(\mathcal{\widehat{D}}_{\mathfrak{X}}^{(0,1)})$
, we have that the cohomological completion of the complex $\mathcal{\widehat{M}}^{\cdot,\infty}$
is isomorphic to $F^{*}(\mathcal{M}^{\cdot,-\infty})$; and this complex
has $p$-torsion-free cohomologies by the assumption. Since $\mathcal{M}^{\cdot,\infty}$
is a bounded complex with coherent cohomology sheaves, by \propref{Completion-for-noeth}
we have that $\mathcal{H}^{n}(\mathcal{\widehat{M}}^{\cdot,\infty})\tilde{=}\widehat{\mathcal{H}^{n}(\mathcal{M}^{\cdot,\infty})}$,
where the completion on the right denotes the usual $p$-adic completion.
But the module $\mathcal{H}^{n}(\mathcal{M}^{\cdot,\infty})$, being
coherent, is $p$-torsion-free iff its $p$-adic completion is. Thus
each $\mathcal{H}^{n}(\mathcal{M}^{\cdot,\infty})$ is $p$-torsion-free
as desired. 
\end{proof}
In the case where $\mathfrak{X}=\text{Specf}(W(k))$ is a point, and
$\mathcal{M}^{\cdot}$ is the gauge coming from cohomology of some
smooth proper $\mathfrak{X}$ (this exists by \thmref{=00005BFJ=00005D},
and we'll construct it, in the language of this paper, in \secref{Push-Forward}
below), this is exactly the content of \thmref{(Mazur)}; indeed,
the first assumption is that $\mathbb{H}_{dR}^{i}(\mathfrak{X})$
is $p$-torsion-free for all $i$, and the second assumption is the
degeneration of the Hodge to de Rham spectral sequence. 

\subsection{Examples of Gauges}

We close out this chapter by giving a few important examples of gauges,
beyond $\widehat{\mathcal{D}}_{\mathfrak{X}}^{(0,1)}$ itself. 
\begin{example}
Let $\mathfrak{X}$ be a smooth formal scheme. Then $D(\mathcal{O}_{\mathfrak{X}})\in\mathcal{G}(\widehat{\mathcal{D}}_{\mathfrak{X}}^{(0,1)})$
by the very definition of $\widehat{\mathcal{D}}_{\mathfrak{X}}^{(0,1)}$
; indeed, we have $D(\mathcal{O}_{\mathfrak{X}}){}^{i}=\{g\in\mathcal{O}_{\mathfrak{X}}|p^{i}g\in\mathcal{O}_{\mathfrak{X}}\}$
so that the natural action of $\widehat{\mathcal{D}}_{\mathfrak{X}}^{(0)}[p^{-1}]$
on $\mathcal{O}_{\mathfrak{X}}[p^{-1}]$ induces the action of $\widehat{\mathcal{D}}_{\mathfrak{X}}^{(0,1)}$
on $\mathcal{O}_{\mathfrak{X}}[f,v]$. This is an $F^{-1}$-gauge
via the isomorphism $F^{*}\mathcal{O}_{\mathfrak{X}}\tilde{\to}\mathcal{O}_{\mathfrak{X}}$. 
\end{example}

To generalize this, suppose $\mathfrak{D}\subset\mathfrak{X}$ is
a locally normal crossings divisor. Let $\mathfrak{U}$ be the compliment
of $\mathfrak{D}$. Denote the inclusion map by $j$. We are going
to define a coherent $F^{-1}$-gauge ${\displaystyle j_{\star}D(\mathcal{O}_{\mathfrak{U}})}$,
whose cohomology is the gauge version of the log de Rham cohomology
of $\mathfrak{X}$ with respect to $\mathfrak{D}$. 

To proceed, let $\mathfrak{V}\subset\mathfrak{X}$ be an affine open,
on which there are local coordinates $\{x_{1},\dots,x_{n}\}$ in which
the divisor $\mathfrak{D}$ is given by $\{x_{1}\cdots x_{j}=0\}$.
Then (starting with the action of finite-order differential operators),
we may consider the $D_{\mathfrak{V}}^{(0)}$-submodule of $\mathcal{O}_{\mathfrak{V}}[x_{1}^{-1}\cdots x_{j}^{-1}]$
generated by $x_{1}^{-1}\cdots x_{j}^{-1}$; it is easily seen to
be independent of the choice of coordinates; hence we obtain a well-defined
$D_{\mathfrak{V}}^{(0)}$-module denoted ${\displaystyle (j_{\star}\mathcal{O}_{\mathfrak{U}})}^{\text{fin}}$;
and we define the $\mathcal{\widehat{D}}_{\mathfrak{V}}^{(0)}$ module,
denoted $(j_{\star}\mathcal{O}_{\mathfrak{U}})^{-\infty}|_{\mathfrak{V}}$
to be the $p$-adic completion of ${\displaystyle (j_{\star}\mathcal{O}_{\mathfrak{U}})}^{\text{fin}}$.
By glueing we obtain a coherent $\mathcal{\widehat{D}}_{\mathfrak{X}}^{(0)}$-module
$(j_{\star}\mathcal{O}_{\mathfrak{U}})^{-\infty}$. We have
\begin{lem}
\label{lem:Injectivity-of-completion}For any $\mathfrak{V}$ as above,
the natural map $\text{(\ensuremath{{\displaystyle j_{\star}\mathcal{O}_{\mathfrak{U}}}}})^{-\infty}|_{\mathfrak{V}}\to\widehat{(\mathcal{O}_{\mathfrak{V}}[x_{1}^{-1}\cdots x_{j}^{-1}])}$
(where $\widehat{}$ denotes $p$-adic completion) is injective. 
\end{lem}

We'll give a proof of this rather technical result in \secref{Appendix:-an-Inectivity}.
From this we deduce
\begin{lem}
\label{lem:Hodge-filt-on-log}Let $F$ be a lift of Frobenius satisfying
$F(x_{i})=x_{i}^{p}$ for all $1\leq i\leq n$. Then the natural map
$F^{*}\text{(\ensuremath{{\displaystyle j_{\star}\mathcal{O}_{\mathfrak{U}}}}})^{-\infty}|_{\mathfrak{V}}\to\widehat{(\mathcal{O}_{\mathfrak{V}}[x_{1}^{-1}\cdots x_{j}^{-1}])}$
is injective, and its image is the $\widehat{\mathcal{D}}_{\mathfrak{V}}^{(1)}$-submodule
generated by $x_{1}^{-1}\cdots x_{j}^{-1}$. 
\end{lem}

\begin{proof}
For each $r>0$ we have an isomorphism $F^{*}(\mathcal{O}_{\mathfrak{V}}[x_{1}^{-1}\cdots x_{j}^{-1}]/p^{r})\tilde{\to}\mathcal{O}_{\mathfrak{V}}[x_{1}^{-1}\cdots x_{j}^{-1}]/p^{r}$;
upon taking the inverse limit we obtain $F^{*}\widehat{(\mathcal{O}_{\mathfrak{V}}[x_{1}^{-1}\cdots x_{j}^{-1}])}\tilde{\to}\widehat{(\mathcal{O}_{\mathfrak{V}}[x_{1}^{-1}\cdots x_{j}^{-1}])}$.
Since $F^{*}$ is an exact, conservative functor on $\mathcal{O}_{\mathfrak{V}}-\text{mod}$,
the previous lemma implies that $F^{*}\text{(\ensuremath{{\displaystyle j_{\star}\mathcal{O}_{\mathfrak{U}}}}})^{-\infty}|_{\mathfrak{V}}\to\widehat{(\mathcal{O}_{\mathfrak{V}}[x_{1}^{-1}\cdots x_{j}^{-1}])}$
is injective. Since the image of $({\displaystyle j_{\star}\mathcal{O}_{\mathfrak{U}}})^{-\infty}|_{\mathfrak{V}}\to\widehat{(\mathcal{O}_{\mathfrak{V}}[x_{1}^{-1}\cdots x_{j}^{-1}])}$
is the $\widehat{\mathcal{D}}_{\mathfrak{V}}^{(0)}$-submodule generated
by $x_{1}^{-1}\cdots x_{j}^{-1}$, the image of $F^{*}(j_{\star}{\displaystyle \mathcal{O}_{\mathfrak{U}}})^{-\infty}|_{\mathfrak{V}}\to\widehat{(\mathcal{O}_{\mathfrak{V}}[x_{1}^{-1}\cdots x_{j}^{-1}])}$
is the $\widehat{\mathcal{D}}_{\mathfrak{V}}^{(1)}$-submodule generated
by $F(x_{1}^{-1}\cdots x_{j}^{-1})=x_{1}^{-p}\cdots x_{j}^{-p}$.
But since $\partial_{i}^{[p]}x_{i}^{-1}=-x_{i}^{-p-1}$ we see that
$\widehat{\mathcal{D}}_{\mathfrak{V}}^{(1)}\cdot x_{1}^{-1}\cdots x_{j}^{-1}=\widehat{\mathcal{D}}_{\mathfrak{V}}^{(1)}\cdot x_{1}^{-p}\cdots x_{j}^{-p}$
as claimed.
\end{proof}
We can now construct the full gauge ${\displaystyle j_{\star}\mathcal{O}_{\mathfrak{U}}[f,v]}$
as follows: denote by $\widehat{\text{(\ensuremath{{\displaystyle j_{\star}\mathcal{O}_{\mathfrak{U}}}}})^{\infty}}$
the $\widehat{\mathcal{D}}_{\mathfrak{X}}^{(1)}$-submodule of $\widehat{(\mathcal{O}_{\mathfrak{V}}[x_{1}^{-1}\cdots x_{j}^{-1}])}$
locally generated by $x_{1}^{-1}\cdots x_{j}^{-1}$; as above this
is independent of the choice of coordinates for the divisor $\mathfrak{D}$.
Then we have
\begin{example}
\label{exa:Integral-j} Define $({\displaystyle j_{\star}D(\mathcal{O}_{\mathfrak{U}})})^{i}:=\{m\in\widehat{\text{(\ensuremath{{\displaystyle j_{\star}\mathcal{O}_{\mathfrak{U}}}}})^{\infty}}|p^{i}m\in\text{(\ensuremath{{\displaystyle j_{\star}\mathcal{O}_{\mathfrak{U}}}}})^{-\infty}\}$.
By the above discussion this is an object in $\text{Coh}_{F^{-1}}(\widehat{\mathcal{D}}_{\mathfrak{X}}^{(0,1)})$
via the isomorphism \linebreak{}
$F^{*}\text{(\ensuremath{{\displaystyle j_{\star}\mathcal{O}_{\mathfrak{U}}}}})^{-\infty}\tilde{\to}\widehat{\text{(\ensuremath{{\displaystyle j_{\star}\mathcal{O}_{\mathfrak{U}}}}})^{\infty}}$.
Let ${\displaystyle j_{\star}D(\mathcal{O}_{U}})$ denote the reduction
mod $p$. We claim that the $l$th term of the Hodge filtration on
$({\displaystyle j_{\star}D(\mathcal{O}_{U})})^{\infty}$ is given
by $F^{*}(F^{l}(\mathcal{D}_{X}^{(0)})\cdot(x_{1}^{-1}\cdots x_{j}^{-1}))$,
where $F^{l}\mathcal{D}_{X}^{(0)}$ is the $l$th term of the symbol
filtration. 

To see this, we work again in local coordinates over $\mathfrak{V}$.
One computes that $(\partial_{i}^{[p]})^{l}(x_{i}^{-p})=u\cdot l!x_{i}^{-p(l+1)}$
where $u$ is a unit in $\mathbb{Z}_{p}$. Therefore the module \linebreak{}
$D_{\mathfrak{V}}^{(1)}\cdot x_{1}^{-1}\cdots x_{j}^{-1}=D_{\mathfrak{V}}^{(1)}\cdot x_{1}^{-p}\cdots x_{j}^{-p}$
is spanned over $\mathcal{O}_{\mathfrak{V}}$ by terms of the form
$I!\cdot x_{1}^{-p(i_{1}+1)}\cdots x_{j}^{-p(i_{j}+1)}$; the $p$-adic
completion of this module is ${\displaystyle \widehat{\text{(\ensuremath{{\displaystyle j_{\star}\mathcal{O}_{\mathfrak{U}}}}})^{\infty}}}$. 

For a multi-index $I$, set $\tilde{I}=(pi_{1}+p-1,\dots,pi_{j}+p-1)$.
Then \linebreak{}
$I!\cdot x_{1}^{-p(i_{1}+1)}\cdots x_{j}^{-p(i_{j}+1)}\in({\displaystyle j_{\star}D(\mathcal{O}_{\mathfrak{U}})})^{r}$
iff $p^{r}\cdot I!\cdot x_{1}^{-p(i_{1}+1)}\cdots x_{j}^{-p(i_{j}+1)}\in\mathcal{\widehat{D}}_{\mathfrak{V}}^{(0)}\cdot x_{1}^{-1}\cdots x_{j}^{-1}$
. Furthermore, it is not difficult to see that $p^{r}\cdot I!\cdot x_{1}^{-p(i_{1}+1)}\cdots x_{j}^{-p(i_{j}+1)}\in\mathcal{\widehat{D}}_{\mathfrak{V}}^{(0)}\cdot x_{1}^{-1}\cdots x_{j}^{-1}$
iff $p^{r}\cdot I!\cdot x_{1}^{-p(i_{1}+1)}\cdots x_{j}^{-p(i_{j}+1)}\in\mathcal{D}_{\mathfrak{V}}^{(0)}\cdot x_{1}^{-1}\cdots x_{j}^{-1}$;
in turn, this holds iff $r\geq\text{val}(\tilde{I}!)-\text{val}(I!)$
(since $\mathcal{D}_{\mathfrak{V}}^{(0)}\cdot x_{1}^{-1}\cdots x_{j}^{-1}$
is spanned by terms of the form $I!x_{1}^{-(i_{1}+1)}\cdots x_{j}^{-(i_{j}+1)}$);
here $\text{val}$ denotes the usual $p$-adic valuation; so that
$\text{val}(p)=1$. 

On the other hand one has 
\[
\text{val}((pi+p-1)!)-\text{val}(i!)=i
\]
for all $i\geq0$. So ${\displaystyle \text{val}(\tilde{I}!)-\text{val}(I!)=\sum_{t=1}^{j}i_{t}}$
which implies $I!\cdot x_{1}^{-p(i_{1}+1)}\cdots x_{j}^{-p(i_{j}+1)}\in({\displaystyle j_{\star}D(\mathcal{O}_{\mathfrak{U}})})^{r}$
iff ${\displaystyle r\geq\sum_{t=1}^{j}i_{t}}$. 

On the other hand, $(F^{l}(\mathcal{D}_{\mathfrak{X}}^{(0)})\cdot(x_{1}^{-1}\cdots x_{j}^{-1})$
is spanned over $\mathcal{O}_{\mathfrak{V}}$ by terms of the form
$I!\cdot x_{1}^{-(i_{1}+1)}\cdots x_{j}^{-(i_{j}+1)}$ where ${\displaystyle \sum_{t=1}^{j}i_{t}\leq l}$.
Thus the module $F^{*}(F^{l}(\mathcal{D}_{X}^{(0)})\cdot(x_{1}^{-1}\cdots x_{j}^{-1}))$
is exactly the image in $({\displaystyle j_{\star}\mathcal{O}_{U}[f,v]})^{\infty}$
of $({\displaystyle j_{\star}\mathcal{O}_{\mathfrak{U}}[f,v]})^{r}$,
which is the claim. 
\end{example}

Finally, we end with an example of a standard, coherent gauge which
definitely does not admit an $F^{-1}$-action:
\begin{example}
\label{exa:Exponential!} Let $\mathfrak{X}=\widehat{\mathbb{A}_{W(k)}^{1}}$.
Consider the $\widehat{\mathcal{D}}_{\mathfrak{X}}^{(0)}$-module
$e^{x}$; i.e., the sheaf $\mathcal{O}_{\mathfrak{X}}$ equipped with
the action determined by
\[
\sum_{i=0}^{\infty}a_{i}\partial^{i}\cdot1=\sum_{i=0}^{\infty}a_{i}
\]
(here $a_{i}\to0$ as $i\to\infty$). Then $\mathcal{O}_{\mathfrak{X}}[p^{-1}]$
is a coherent $\widehat{\mathcal{D}}_{\mathfrak{X}}^{(0,1),\infty}$-module
since $\partial^{[p]}\cdot1=(p!)^{-1}$; it has a $\widehat{\mathcal{D}}_{\mathfrak{X}}^{(0)}$-lattice
given by $\mathcal{O}_{\mathfrak{X}}$. Thus by \exaref{Basic-Construction-over-X}
we may define 
\[
(e^{x})^{i}:=\{m\in\mathcal{O}_{\mathfrak{X}}[p^{-1}]|p^{i}m\in\mathcal{O}_{\mathfrak{X}}\}
\]
and we obtain a gauge, also denoted $e^{x}$, such that $(e^{x})^{i}\tilde{=}\mathcal{O}_{\mathfrak{X}}$
for all $i$, and such that $v$ is an isomorphism for all $i$, while
$f$ is given by multiplication by $p$. We have $(e^{x})^{-\infty}\tilde{=}\mathcal{O}_{\mathfrak{X}}$
while $(e^{x})^{\infty}=\mathcal{O}_{\mathfrak{X}}[p^{-1}]$. 
\end{example}

This example indicates that the ``exponential Hodge theory'' appearing,
e.g., in Sabbah's work \cite{key-22}, could also be a part of this
story; this should be interesting to pursue in future work. 

\section{\label{sec:Operations:PullBack}Operations on Gauges: Pull-back}

Let $\varphi:\mathfrak{X}\to\mathfrak{Y}$ be a morphism of smooth
formal schemes over $W(k)$. Let us begin by setting our conventions
on the pullback of $\mathcal{O}$-modules: 
\begin{defn}
\label{def:Correct-Pullback}1) If $\mathcal{M}\in\mathcal{O}_{\mathfrak{Y}}-\text{mod}$,
we set $\varphi^{*}\mathcal{M}:=\mathcal{O}_{\mathfrak{X}}\widehat{\otimes}_{\varphi^{-1}(\mathcal{O}_{\mathfrak{Y}})}\varphi^{-1}(\mathcal{M}^{\cdot})$,
the $p$-adic completion of the naive tensor product. If $\mathcal{M}^{\cdot}\in D(\mathcal{O}_{\mathfrak{Y}})$,
then we define $L\varphi^{*}\mathcal{M}^{\cdot}:=\mathcal{O}_{\mathfrak{X}}\widehat{\otimes}_{\varphi^{-1}(\mathcal{O}_{\mathfrak{Y}})}^{L}\varphi^{-1}(\mathcal{M}^{\cdot})$;
the cohomological completion of the usual derived tensor product. 

2) Consider $D(\mathcal{O}_{\mathfrak{X}})$ and $D(\mathcal{O}_{\mathfrak{Y}})$
as graded sheaves of rings as usual. If $\mathcal{M}^{\cdot}\in D(\mathcal{G}(D(\mathcal{O}_{\mathfrak{Y}})))$,
then we define $L\varphi^{*}\mathcal{M}^{\cdot}:=D(\mathcal{O}_{\mathfrak{X}})\widehat{\otimes}_{\varphi^{-1}(D(\mathcal{O}_{\mathfrak{Y}}))}^{L}\varphi^{-1}(\mathcal{M}^{\cdot})$,
the graded cohomological completion of the usual derived tensor product.
\end{defn}

\begin{rem}
The reader will note several inconsistencies in these notations. First
of all, we do not, in general, have $\mathcal{H}^{0}(L\varphi^{*}\mathcal{M})=\varphi^{*}\mathcal{M}$.
Furthermore, the functor $L\varphi^{*}$ does not commute with the
forgetful functor from graded $\mathcal{O}[f,v]$-modules to $\mathcal{O}$-modules.
However, we will only use the underived $\varphi^{*}$ in a few very
special cases (c.f. the lemma directly below), when in fact the equality
$\mathcal{H}^{0}(L\varphi^{*}\mathcal{M})=\varphi^{*}\mathcal{M}$
does hold. Further, we will only apply the graded functor when working
with a graded module; and this will almost always be the case. Hopefully
this notational scheme does not cause any undue confusion.
\end{rem}

Now we should check that this operation behaves well on the basic
objects of interest in our paper: 
\begin{lem}
\label{lem:phi-pullback-of-D^i}For each $i\in\mathbb{Z}$ we have
\[
L\varphi^{*}(\widehat{\mathcal{D}}_{\mathfrak{Y}}^{(0,1),i})\tilde{=}\varphi^{*}(\widehat{\mathcal{D}}_{\mathfrak{Y}}^{(0,1),i})\tilde{=}\lim_{n}(\mathcal{O}_{\mathfrak{X}_{n}}\otimes_{\varphi^{-1}(\mathcal{O}_{\mathfrak{Y}_{n}})}\varphi^{-1}(\widehat{\mathcal{D}}_{\mathfrak{Y}}^{(0,1),i}/p^{n}))
\]
In particular, we have 
\[
\mathcal{H}^{0}(L\varphi^{*}(\widehat{\mathcal{D}}_{\mathfrak{Y}}^{(0,1),i}))\tilde{=}\varphi^{*}(\widehat{\mathcal{D}}_{\mathfrak{Y}}^{(0,1),i})
\]
 under the conventions of \defref{Correct-Pullback}. The same holds
if we replace $\widehat{\mathcal{D}}_{\mathfrak{Y}}^{(0,1),i}$ by
$\mathcal{\widehat{D}}_{\mathfrak{Y}}^{(j)}$ for any $j\geq0$. 
\end{lem}

\begin{proof}
As this question is local, we can assume $\mathfrak{X}=\text{Specf}(\mathcal{B})$
and $\mathfrak{Y}=\text{Specf}(\mathcal{A})$ where $\mathcal{A}$
possess local coordinates $\{t_{1},\dots,t_{n}\}$. By definition
we have that $\widehat{D}_{\mathcal{A}}^{(0,1),i}$ is the $p$-adic
completion of $D_{\mathcal{A}}^{(0,1),i}$. By \corref{Each-D^(i)-is-free}
we have that $D_{\mathcal{A}}^{(0,1),i}$ is free over $\mathcal{A}$.
In particular, it is $p$-torsion free and $p$-adically separated;
and so by \cite{key-8}, lemma 1.5.4 its cohomological completion
is equal to $\widehat{D}_{\mathcal{A}}^{(0,1),i}$. Therefore we have
the short exact sequence 
\[
D_{\mathcal{A}}^{(0,1),i}\to\widehat{D}_{\mathcal{A}}^{(0,1),i}\to K
\]
where $p$ acts invertibly on $K$. Now we apply the functor $\mathcal{B}\otimes_{\mathcal{A}}^{L}$.
By \cite{key-8}, theorem 1.6.6, we have that $\widehat{D}_{\mathcal{A}}^{(0,1),i}$
is flat over $\mathcal{A}$. Thus we see that $\mathcal{B}\widehat{\otimes}_{\mathcal{A}}^{L}\mathcal{\widehat{D}}_{\mathcal{A}}^{(0,1),i}$,
the cohomological completion of $\mathcal{B}\otimes_{\mathcal{A}}^{L}\widehat{D}_{\mathcal{A}}^{(0,1),i}$,
is isomorphic to the cohomological completion of $\mathcal{B}\otimes_{\mathcal{A}}^{L}\mathcal{D}_{\mathcal{A}}^{(0,1),i}$,
which is just the usual $p$-adic completion since this is a free
$\mathcal{B}$-module, and the statement follows. An identical argument
works for $\mathcal{\widehat{D}}_{\mathfrak{Y}}^{(j)}$. 
\end{proof}
Now let $j\geq0$. We recall that, for each $j\geq0$, Berthelot has
constructed a pullback functor $\varphi^{!,(j)}$ from $\widehat{\mathcal{D}}_{\mathfrak{Y}}^{(j)}-\text{mod}$
to $\widehat{\mathcal{D}}_{\mathfrak{X}}^{(j)}-\text{mod}$. In fact,
in \cite{key-2}, section 3.2, he has shown that $\mathcal{\widehat{D}}_{\mathfrak{X}\to\mathfrak{Y}}^{(j)}:=\varphi^{*}(\mathcal{\widehat{D}}_{\mathfrak{Y}}^{(j)})$
carries the structure of a left $\mathcal{\widehat{D}}_{\mathfrak{X}}^{(j)}$-module.
By definition $\varphi^{*}\mathcal{\widehat{D}}_{\mathfrak{Y}}^{(j)}$
carries the structure of a right $\varphi^{-1}(\mathcal{\widehat{D}}_{\mathfrak{Y}}^{(j)})$-module.
This, in turn allows one to define the functor $\varphi^{*,(j)}$
via 
\[
L\varphi^{*,(j)}(\mathcal{M})=\varphi^{*}\mathcal{\widehat{D}}_{\mathfrak{Y}}^{(j)}\widehat{\otimes}_{\varphi^{-1}(\mathcal{\widehat{D}}_{\mathfrak{Y}})}^{L}\varphi^{-1}(\mathcal{M})\tilde{=}L\varphi^{*}(\mathcal{M})
\]
(where the last isomorphism is as sheaves of $\mathcal{O}_{\mathfrak{X}}$-modules).
One sets $\varphi^{!,(j)}:=L\varphi^{*,(j)}[d_{X/Y}]$ (where $d_{X/Y}=\text{dim}(X)-\text{dim}(Y)$). 

In fact, this is not quite Berthelot's definition, as he does not
use the cohomological completion; rather, he first defines the functor
in the case of a morphism $\varphi:\mathfrak{X}_{n}\to\mathfrak{Y}_{n}$
(the reductions mod $p^{n}$ of $\mathfrak{X}$ and $\mathfrak{Y}$,
respectively), and then applies the $\text{R}\lim$ functor. However,
the two notions agree on bounded complexes of coherent $\widehat{\mathcal{D}}_{\mathfrak{Y}}$-modules;
the version introduced here seems better suited to very general complexes.

In order to upgrade this to gauges, we must upgrade the bimodule $\mathcal{\widehat{D}}_{\mathfrak{X}\to\mathfrak{Y}}^{(0)}$
to a bimodule $\mathcal{\widehat{D}}_{\mathfrak{X}\to\mathfrak{Y}}^{(0,1)}$:
\begin{defn}
\label{def:Transfer-Bimod} We set 
\[
\mathcal{\widehat{D}}_{\mathfrak{X}\to\mathfrak{Y}}^{(0,1)}:=\bigoplus_{i\in\mathbb{Z}}\varphi^{*}\widehat{\mathcal{D}}_{\mathfrak{Y}}^{(0,1),i}
\]
The sheaf ${\displaystyle \bigoplus_{i\in\mathbb{Z}}\mathcal{\widehat{D}}_{\mathfrak{X}\to\mathfrak{Y}}^{(0,1),i}}$
is a graded sheaf of $D(W(k))$-modules; induced from the $D(W(k))$
action on $\mathcal{\widehat{D}}_{\mathfrak{Y}}^{(0,1)}$. Note that
$\mathcal{\widehat{D}}_{\mathfrak{X}\to\mathfrak{Y}}^{(0,1),-\infty}=\varphi^{*}\widehat{\mathcal{D}}_{\mathfrak{Y}}^{(0)}$. 
\end{defn}

Let us analyze this sheaf:
\begin{prop}
\label{prop:Basic-properties-of-the-transfer-module}1) For each $i\in\mathbb{Z}$
, the natural map $\iota:\varphi^{*}\mathcal{\widehat{D}}_{\mathfrak{Y}}^{(0,1),i}\to\varphi^{*}\widehat{\mathcal{D}}_{\mathfrak{Y}}^{(1)}$
(induced from the inclusion $\eta:\mathcal{\widehat{D}}_{\mathfrak{Y}}^{(0,1),i}\to\widehat{\mathcal{D}}_{\mathfrak{Y}}^{(1)}$)
is injective. 

2) The image $\iota(\varphi^{*}\mathcal{\widehat{D}}_{\mathfrak{Y}}^{(0,1),i})$
is equal to the sheaf whose local sections are given by $\{\Psi\in\varphi^{*}\mathcal{\widehat{D}}_{\mathfrak{Y}}^{(1)}|p^{i}\Psi\in\iota(\varphi^{*}\mathcal{\widehat{D}}_{\mathfrak{Y}}^{(0)})\}$.
In particular, $\mathcal{\widehat{D}}_{\mathfrak{X}\to\mathfrak{Y}}^{(0,1)}$
is a standard gauge.

3) The sheaf ${\displaystyle \mathcal{\widehat{D}}_{\mathfrak{X}\to\mathfrak{Y}}^{(0,1)}}$
carries the structure of a graded $(\mathcal{\widehat{D}}_{\mathfrak{X}}^{(0,1)},\varphi^{-1}(\mathcal{\widehat{D}}_{\mathfrak{Y}}^{(0,1)}))$-bimodule
as follows: we have the inclusions $\mathcal{\widehat{D}}_{\mathfrak{X}}^{(0)}\subset\mathcal{\widehat{D}}_{\mathfrak{X}}^{(1)}$,
so if $\Phi\in\mathcal{\widehat{D}}_{\mathfrak{X}}^{(0,1),i}$ and
$\Psi\in\mathcal{\widehat{D}}_{\mathfrak{X}\to\mathfrak{Y}}^{(0,1),j}$
are local sections, then $p^{i+j}(\Phi\cdot\Psi)=(p^{i}\Phi)\cdot(p^{j}\Psi)\in\mathcal{\widehat{D}}_{\mathfrak{X}\to\mathfrak{Y}}^{(0)}$.
Similarly, $\mathcal{\widehat{D}}_{\mathfrak{X}\to\mathfrak{Y}}^{(0,1)}$
becomes a right $\varphi^{-1}(\mathcal{\widehat{D}}_{\mathfrak{Y}}^{(0,1)})$-module
via $\varphi^{-1}\mathcal{\widehat{D}}_{\mathfrak{Y}}^{(0)}\subset\varphi^{-1}\mathcal{\widehat{D}}_{\mathfrak{Y}}^{(1)}$. 
\end{prop}

\begin{proof}
1) As the statement is local, we can suppose $\mathfrak{Y}=\text{Specf}(\mathcal{A})$
and $\mathfrak{X}=\text{Specf}(\mathcal{B})$ where $\mathcal{A}$
and $\mathcal{B}$ admit local coordinates; let the reductions mod
$p$ be $Y=\text{Spec}(A)$ and $X=\text{Spec}(B)$. By \corref{Local-coords-over-A=00005Bf,v=00005D}
we know that $D_{A}^{(0,1)}$ is a free graded $A[f,v]$-module, therefore
$\varphi^{*}D_{A}^{(0,1)}=B\otimes_{A}D_{A}^{(0,1)}$ is a free graded
$B[f,v]$-module; and we have that the kernel of $f_{\infty}:\varphi^{*}D_{A}^{(0,1),i}\to\varphi^{*}D_{A}^{(0,1),\infty}=\varphi^{*}D_{A}^{(1)}$
is exactly the image of $v:\varphi^{*}D_{A}^{(0,1),i+1}\to\varphi^{*}D_{A}^{(0,1),i}$. 

Now consider $m\in\text{ker}(\iota:\varphi^{*}\widehat{D}_{\mathcal{A}}^{(0,1),i}\to\varphi^{*}\widehat{D}_{\mathcal{A}}^{(1)})$.
The reduction mod $p$ of $\iota$ agrees with $f_{\infty}:\varphi^{*}D_{A}^{(0,1),i}\to\varphi^{*}D_{A}^{(1)}$.
Let $\overline{m}$ denote the image of $m$ in $\varphi^{*}D_{A}^{(0,1),i}$.
Then $\overline{m}\in\text{ker}(\varphi^{*}D_{A}^{(0,1),i}\to\varphi^{*}D_{A}^{(1)})=v\cdot\varphi^{*}D_{A}^{(0,1),i+1}$.
So, since $fv=p,$ we have $m\in v\cdot\varphi^{*}\widehat{D}_{\mathcal{A}}^{(0,1),i+1}$;
write $m=vm'$. By definition, the composition $\widehat{D}_{\mathcal{A}}^{(0,1),i+1}\xrightarrow{v}\widehat{D}_{\mathcal{A}}^{(0,1),i}\xrightarrow{\eta}\widehat{D}_{\mathcal{A}}^{(1)}$
is equal to $p\cdot\eta:\widehat{D}_{\mathcal{A}}^{(0,1),i+1}\to\widehat{D}_{\mathcal{A}}^{(1)}$;
thus also $\iota\circ v=p\cdot\iota$ and so $\iota(m)=\iota(vm')=p\iota(m')=0$;
therefore $m'\in\text{ker}(\iota)$ as $\varphi^{*}\widehat{D}_{\mathcal{A}}^{(1)}$
is $p$-torsion-free\footnote{Indeed, it is the inverse limit of the $W_{n}(k)$-flat modules $(\varphi^{*}\widehat{D}_{\mathcal{A}}^{(1)})/p^{n}=(\mathcal{B}/p^{n})\otimes_{\mathcal{A}/p^{n}}(\widehat{D}_{\mathcal{A}}^{(1)}/p^{n})$}.
Iterating the argument, we see that $m\in v^{N}\varphi^{*}\widehat{D}_{\mathcal{A}}^{(0,1),i+N}$
for all $N>0$; reducing mod $p$, this forces $\overline{m}=0$ since
(again by \corref{Local-coords-over-A=00005Bf,v=00005D}) $\varphi^{*}D_{A}^{(0,1)}$
is $v$-adically seperated. Thus $m=pm_{1}$; and then $\iota(m_{1})=0$
since $\varphi^{*}\widehat{D}_{\mathcal{A}}^{(1)}$ is $p$-torsion-free;
continuing in this way we obtain $m\in\bigcap_{n}p^{n}\cdot\varphi^{*}\widehat{D}_{\mathcal{A}}^{(0,1),i}=0$. 

2) For each $i\geq0$ we have a short exact sequence 
\[
0\to\mathcal{\widehat{D}}_{\mathfrak{Y}}^{(0,1),i}\to\mathcal{\widehat{D}}_{\mathfrak{Y}}^{(0,1),i+1}\to\mathcal{F}_{i}\to0
\]
 where $\mathcal{F}_{i}$ is a sheaf which is annihilated by $p$.
By the injectivity just proved (and the equality $L\varphi^{*}\widehat{\mathcal{D}}_{\mathfrak{Y}}^{(0,1),i}=\varphi^{*}\widehat{\mathcal{D}}_{\mathfrak{Y}}^{(0,1),i}$)
we obtain the short exact sequence 
\[
0\to\varphi^{*}\mathcal{\widehat{D}}_{\mathfrak{Y}}^{(0,1),i}\to\varphi^{*}\mathcal{\widehat{D}}_{\mathfrak{Y}}^{(0,1),i+1}\to\mathcal{H}^{0}(L\varphi^{*}\mathcal{F}_{i})\to0
\]
and, since $\mathcal{F}_{i}$ is annihilated by $p$, we have $\mathcal{H}^{0}(L\varphi^{*}\mathcal{F}_{i})=\mathcal{O}_{X}\otimes_{\varphi^{-1}(\mathcal{O}_{Y})}\varphi^{-1}(\mathcal{F}_{i})$.
So we obtain $p\cdot\varphi^{*}\mathcal{\widehat{D}}_{\mathfrak{Y}}^{(0,1),i+1}\subset\varphi^{*}\mathcal{\widehat{D}}_{\mathfrak{Y}}^{(0,1),i}$,
and since $\varphi^{*}\mathcal{\widehat{D}}_{\mathfrak{Y}}^{(0,1),0}=\varphi^{*}\mathcal{\widehat{D}}_{\mathfrak{Y}}^{(0)}$,
we see inductively that $\varphi^{*}\mathcal{\widehat{D}}_{\mathfrak{Y}}^{(0,1),i}\subset\{\Psi\in\varphi^{*}\mathcal{\widehat{D}}_{\mathfrak{Y}}^{(1)}|p^{i}\Psi\in\iota(\varphi^{*}\mathcal{\widehat{D}}_{\mathfrak{Y}}^{(0)})\}$
for all $i$. 

For the converse direction, we work locally and assume $\mathfrak{X}=\text{Specf}(\mathcal{B})$
and $\mathfrak{Y}=\text{Specf}(\mathcal{A})$ where $\mathcal{A}$
possess etale local coordinates $\{t_{1},\dots,t_{n}\}$. Then we
have that $\Gamma(\varphi^{*}\mathcal{\widehat{D}}_{\mathfrak{Y}}^{(1)})=\mathcal{B}\widehat{\otimes}_{\mathcal{A}}\widehat{D}_{\mathcal{A}}^{(1)}\tilde{=}\mathcal{B}\widehat{\otimes}_{\mathcal{A}}D_{\mathcal{A}}^{(1)}$.
As in the proof of \lemref{Basic-structure-of-D_A^(i)}, we will consider
the finite-order analogue first. From (the proof of) that lemma, it
follows that, any element of $\mathcal{B}\otimes_{\mathcal{A}}D_{\mathcal{A}}^{(1)}$
admits a unique expression of the form 
\[
\Psi=\sum_{I,J}b_{I,J}\frac{\partial_{1}^{i_{1}+pj_{1}}\cdots\partial_{n}^{i_{n}+pj_{n}}}{(p!)^{|J|}}
\]
for which $0\leq i_{j}<p$, all $b_{I,J}\in\mathcal{B}$, and the
sum is finite. We have that $p^{i}\Psi\in\mathcal{B}\otimes_{\mathcal{A}}D_{\mathcal{A}}^{(0)}$
iff ${\displaystyle \frac{p^{i}}{p^{|J|}}b_{I,J}}\in\mathcal{B}$.
So, if $|J|>i$ we can conclude (again, as in the proof of \lemref{Basic-structure-of-D_A^(i)})
that
\[
b_{I,J}\frac{\partial_{1}^{i_{1}+pj_{1}}\cdots\partial_{n}^{i_{n}+pj_{n}}}{(p!)^{|J|}}=\tilde{b}_{I,J}\cdot\partial_{1}^{i_{1}+pj'_{1}}\cdots\partial_{n}^{i_{n}+pj'_{n}}\cdot(\partial_{1}^{[p]})^{j''_{1}}\cdots\partial_{n}^{i_{n}}(\partial_{n}^{[p]})^{j''_{n}}
\]
where $\tilde{b}_{I,J}\in\mathcal{B}$, and $j''_{1}+\dots+j_{n}''=i$.
In particular $\Psi$ is contained in the $\mathcal{B}$-submodule
spanned by $\{\partial_{1}^{i_{1}}\cdots\partial_{n}^{i_{n}}\cdot(\partial_{1}^{[p]})^{j_{1}}\cdots(\partial_{n}^{[p]})^{j_{n}}\}$
where $j_{1}+\dots+j_{n}\le i$, which is exactly the image of $\mathcal{B}\otimes_{\mathcal{A}}D_{\mathcal{A}}^{(0,1),i}$
in $\mathcal{B}\otimes_{\mathcal{A}}D_{\mathcal{A}}^{(1)}$. 

Now, if $\Psi\in\mathcal{B}\widehat{\otimes}_{\mathcal{A}}\widehat{D}_{\mathcal{A}}^{(1)}$
is such that $p^{i}\Psi\in\mathcal{B}\widehat{\otimes}_{\mathcal{A}}\widehat{D}_{\mathcal{A}}^{(0)}$,
then we can write ${\displaystyle p^{i}\Psi=\sum_{j=0}^{\infty}p^{j}\Psi_{j}}$
where $\Psi_{j}\in\mathcal{B}\otimes_{\mathcal{A}}D_{\mathcal{A}}^{(0)}$.
Therefore 
\[
\Psi=\sum_{j=0}^{i}p^{j-i}\Psi_{j}+\sum_{j=i+1}^{\infty}p^{j-i}\Psi_{j}
\]
where, by the previous paragraph, the first sum is contained in the
$\mathcal{B}$-submodule spanned by $\{\partial_{1}^{i_{1}}\cdots\partial_{n}^{i_{n}}\cdot(\partial_{1}^{[p]})^{j_{1}}\cdots(\partial_{n}^{[p]})^{j_{n}}\}$
where $j_{1}+\dots+j_{n}\le i$, and the second sum is contained in
$\mathcal{B}\widehat{\otimes}_{\mathcal{A}}\widehat{D}_{\mathcal{A}}^{(0)}$.
Thus $\Psi$ is in the image of $\mathcal{B}\widehat{\otimes}_{\mathcal{A}}\widehat{D}_{\mathcal{A}}^{(0,1),i}$
as required. It follows directly from the definition that $\mathcal{\widehat{D}}_{\mathfrak{X}\to\mathfrak{Y}}^{(0,1)}$
is standard. Part $3)$ of the proposition follows immediately.
\end{proof}
\begin{rem}
\label{rem:Direct-defn-of-transfer-bimodule}Combining the previous
proposition with \lemref{phi-pullback-of-D^i}, we also obtain the
description 
\[
\mathcal{\widehat{D}}_{\mathfrak{X}\to\mathfrak{Y}}^{(0,1)}\tilde{=}L\varphi^{*}(\mathcal{\widehat{D}}_{\mathfrak{Y}}^{(0,1)})=D(\mathcal{O}_{\mathfrak{X}})\widehat{\otimes}_{\varphi^{-1}(D(\mathcal{O}_{\mathfrak{Y}}))}^{L}\varphi^{-1}(\mathcal{\widehat{D}}_{\mathfrak{Y}}^{(0,1)})
\]
in the category $D_{cc}(\mathcal{G}(D(\mathcal{O}_{\mathfrak{X}}))$. 
\end{rem}

This leads to the 
\begin{defn}
\label{def:Pullback!}Let $\mathcal{M}^{\cdot}\in D_{cc}(\mathcal{G}(\mathcal{\widehat{D}}_{\mathfrak{Y}}^{(0,1)}))$.
Then we define 
\[
L\varphi^{*}(\mathcal{M}^{\cdot}):=\mathcal{\widehat{D}}_{\mathfrak{X}\to\mathfrak{Y}}^{(0,1)}\widehat{\otimes}_{\varphi^{-1}(\mathcal{D}_{\mathfrak{Y}}^{(0,1)})}^{L}\varphi^{-1}(\mathcal{M}^{\cdot})\in\mathcal{M}^{\cdot}\in D_{cc}(\mathcal{G}(\mathcal{\widehat{D}}_{\mathfrak{X}}^{(0,1)}))
\]
 where, as usual $\widehat{?}$ denotes graded derived completion.
The induced left action of $\mathcal{D}_{\mathfrak{X}}^{(0,1)}$ given
by the above definition; set $\varphi^{!}:=L\varphi^{*}[d_{X/Y}]$. 
\end{defn}

In order to study this definition, we shall use the corresponding
mod $p$ theory; as usual this can be defined by reduction mod $p$
when the schemes $X$ and $Y$ are liftable, but it actually exists
for all $\varphi:X\to Y$. This is contained in the the following
\begin{prop}
\label{prop:pull-back-in-pos-char}Let $\varphi:X\to Y$ be a morphism
of smooth varieties over $k$. 

1) There is a map of sheaves $\alpha:\mathfrak{l}_{X}\to\varphi^{*}\mathcal{D}_{Y}^{(0,1),1}$
(where $\mathfrak{l}_{X}$ is defined in \defref{L}). 

2) Let $\beta:\mathcal{T}_{X}\to\varphi^{*}\mathcal{D}_{Y}^{(0,1),0}=\varphi^{*}\mathcal{D}_{Y}^{(0)}$
denote the natural map. There is a left action of $\mathcal{D}_{X}^{(0,1)}$
on $\varphi^{*}\mathcal{D}_{Y}^{(0,1)}$ satisfying $\partial\cdot(1\otimes1)=\beta(\partial)$
for all $\partial\in\mathcal{T}_{X}$ and $\delta\cdot(1\otimes1)=\alpha(\delta)$
for all $\delta\in\mathfrak{l}_{X}$. This action commutes with the
right action of $\varphi^{-1}(\mathcal{D}_{Y}^{(0,1)})$ on $\varphi^{*}\mathcal{D}_{Y}^{(0,1)}$. 
\end{prop}

\begin{proof}
1) Let $\Phi$ be a local section of $\mathfrak{l}_{X}$. Composing
the map $\varphi^{\#}:\varphi^{-1}(\mathcal{O}_{Y})\to\mathcal{O}_{X}$
with $\Phi$ gives a differential operator from $\varphi^{-1}(\mathcal{O}_{Y})$
to $\mathcal{O}_{X}$; call this operator $\Phi'$. We claim $\Phi'\in\mathcal{O}_{X}\otimes_{\varphi^{-1}(\mathcal{O}_{Y})}\varphi^{-1}\mathfrak{l}_{Y}$
(here, we are using the fact that the sheaf $\mathfrak{l}_{Y}$ is
a subsheaf of $\mathcal{D}iff_{Y}$ and that $\mathcal{O}_{X}\otimes_{\varphi^{-1}(\mathcal{O}_{Y})}\varphi^{-1}(\mathcal{D}iff_{Y})\tilde{=}\mathcal{D}iff(\varphi^{-1}(\mathcal{O}_{Y}),\mathcal{O}_{X})$). 

Let $U\subset X$ and $V\subset Y$ be open subsets which possess
local coordinates, such that $\varphi(U)\subset V$. As in \lemref{O^p-action}
write 
\[
\Phi=\sum_{i=1}^{n}a_{i}^{p}\partial_{i}^{[p]}+\sum_{I}a_{I}\partial^{I}
\]
where $a_{i},a_{I}\in\mathcal{O}_{X}(U)$. The map $(\sum_{I}a_{I}\partial^{I})\circ\varphi^{\#}:\varphi^{-1}(\mathcal{O}_{V})\to\mathcal{O}_{U}$
is a differential operator which satisfies $((\sum_{I}a_{I}\partial^{I})\circ\varphi^{\#})(g^{p}\cdot h)=\varphi^{\#}(g^{p})\cdot((\sum_{I}a_{I}\partial^{I})\circ\varphi^{\#})(h)$
for all $g,h\in\mathcal{O}_{V}$. From this we conclude 
\[
(\sum_{I}a_{I}\partial^{I})\circ\varphi^{\#}=\sum b_{J}\partial^{J}
\]
where $b_{J}\in\mathcal{O}_{X}(U)$ and now $\partial^{J}=\partial_{1}^{j_{1}}\cdots\partial_{r}^{j_{r}}$
are coordinate derivations on $V$ (to prove this, write the differential
operator $(\sum_{I}a_{I}\partial^{I})\circ\varphi^{\#}$ in terms
of $\partial_{1}^{[j_{1}]}\cdots\partial_{r}^{[j_{r}]}$ and then
use the linearity over $\varphi^{\#}(g^{p})$ to deduce that there
are no terms with any $j_{i}\geq p$). 

Similarly, the map ${\displaystyle \sum_{i=1}^{n}a_{i}^{p}\partial_{i}^{[p]}}\circ\varphi^{\#}:\varphi^{-1}(\mathcal{O}_{V})\to\mathcal{O}_{U}$
is a differential operator of order $\leq p$, whose action on any
$p$th power in $\varphi^{-1}(\mathcal{O}_{V})$ is a $p$th power
in $\mathcal{O}_{U}$. From this one easily sees 
\[
(\sum_{i=1}^{n}a_{i}^{p}\partial_{i}^{[p]})\circ\varphi^{\#}=\sum_{j=1}^{r}b_{j}^{p}\partial_{j}^{[p]}+\sum_{J}b_{J}\partial^{J}
\]
for some $b_{j},b_{J}\in\mathcal{O}_{U}$. So we conclude $\Phi'\in\mathcal{O}_{X}\otimes_{\varphi^{-1}(\mathcal{O}_{Y})}\varphi^{-1}\mathfrak{l}_{Y}$
as desired. Further, since $\mathfrak{l}_{Y}\subset\mathcal{D}_{Y}^{(0,1),1}$
we obtain $\varphi^{-1}(\mathfrak{l}_{Y})\subset\varphi^{-1}(\mathcal{D}_{Y}^{(0,1),1})$
and therefore a map $\mathcal{O}_{X}\otimes_{\varphi^{-1}(\mathcal{O}_{Y})}\varphi^{-1}(\mathfrak{l}_{Y})\to\mathcal{O}_{X}\otimes_{\varphi^{-1}(\mathcal{O}_{Y})}\varphi^{-1}(\mathcal{D}_{Y}^{(0,1),1})$;
we can now define $\alpha$ as the composition. 

2) It suffices to check this locally. Restrict to an open affine $U\subset X$
which posses etale local coordinates, and we may suppose $\varphi(U)\subset V$,
where $V$ also possesses etale local coordinates. Writing $U=\text{Spec}(A)$
and $V=\text{Spec}(B)$, we let $\mathcal{A}$ and $\mathcal{B}$
be flat lifts of $A$ and $B$ to $W(k)$, as in the proof of \lemref{linear-independance-over-D_0-bar}
above. Let $\varphi^{\#}:\mathcal{B}\to\mathcal{A}$ be a lift of
$\varphi^{\#}:B\to A$ (these always exist for affine neighborhoods
which posses local coordinates, by the infinitesimal lifting property).
Then the construction of \defref{Transfer-Bimod} provides an action
of $D_{\mathcal{B}}^{(0,1)}$ on $\varphi^{*}(D_{\mathcal{A}}^{(0,1)})$
which commutes with the obvious right action of $D_{\mathcal{A}}^{(0,1)}$.
The reduction mod $p$ of this action, when restricted to $\mathcal{T}_{X}\subset\mathcal{D}_{X}^{(0)}$
and $\mathfrak{l}_{X}\subset\mathcal{D}_{X}^{(0,1),1}$ clearly agrees
with the map described above. Thus the map extends (uniquely) to an
action, as claimed. 
\end{proof}
Thus we have
\begin{defn}
Let $\mathcal{D}_{X\to Y}^{(0,1)}:=\varphi^{*}\mathcal{D}_{Y}^{(0,1)}$,
equipped with the structure of a graded $(\mathcal{D}_{X}^{(0,1)},\varphi^{-1}(\mathcal{D}_{Y}^{(0,1)}))$-bimoddule
as above. Let $\mathcal{M}^{\cdot}\in D(\mathcal{G}(\mathcal{D}_{Y}^{(0,1)}))$.
Then we define $L\varphi^{*}(\mathcal{M}^{\cdot}):=\mathcal{D}_{X\to Y}^{(0,1)}\otimes_{\varphi^{-1}(\mathcal{D}_{Y}^{(0,1)})}^{L}\varphi^{-1}(\mathcal{M}^{\cdot})$
with the induced left action of $\mathcal{D}_{X}^{(0,1)}$ given by
the above. Set $\varphi^{!}=L\varphi^{*}[d_{X/Y}]$. The functor $L\varphi^{*}$
takes $D_{\text{qcoh}}(\mathcal{G}(\mathcal{D}_{Y}^{(0,1)}))$ to
$D_{\text{qcoh}}(\mathcal{G}(\mathcal{D}_{X}^{(0,1)}))$. 
\end{defn}

\begin{rem}
In fact, as an object in $D(\mathcal{G}(D(\mathcal{O}_{X})))$, we
have that $L\varphi^{*}(\mathcal{M}^{\cdot})$ agrees with the usual
pullback of $\mathcal{O}$-modules. This follows directly from the
isomorphism $\varphi^{*}\mathcal{D}_{Y}^{(0,1)}\tilde{=}\mathcal{O}_{X}\otimes_{\varphi^{-1}(\mathcal{O}_{Y})}\varphi^{-1}(\mathcal{D}_{Y}^{(0,1)})$,
and the fact that $\mathcal{D}_{Y}^{(0,1)}$ is flat over $\mathcal{O}_{Y}$.
The analogous fact is also true for $\varphi:\mathfrak{X}\to\mathfrak{Y}$;
making use of \remref{Direct-defn-of-transfer-bimodule}. It follows
that $L\varphi^{*}$ has finite homological dimension.
\end{rem}

Now we record some basic properties of these functors: 
\begin{lem}
\label{lem:composition-of-pullbacks}If $\psi:\mathfrak{Y}\to\mathfrak{Z}$,
there is an isomorphism of functors \linebreak{}
$L\varphi^{*}\circ L\psi^{*}\tilde{=}L(\psi\circ\varphi)^{*}$. The
same result holds for $\varphi:X\to Y$ and $\psi:Y\to Z$. 
\end{lem}

\begin{proof}
(compare \cite{key-49}, proposition 1.5.11) We have, by \remref{Direct-defn-of-transfer-bimodule},
\[
\mathcal{\widehat{D}}_{\mathfrak{X}\to\mathfrak{Y}}^{(0,1)}\widehat{\otimes}_{\varphi^{-1}(\mathcal{D}_{\mathfrak{Y}}^{(0,1)})}^{L}\varphi^{-1}(\mathcal{D}_{\mathfrak{Y\to\mathfrak{Z}}}^{(0,1)})
\]
\[
=(D(\mathcal{O}_{\mathfrak{X}})\widehat{\otimes}_{\varphi^{-1}(D(\mathcal{O}_{\mathfrak{Y}}))}^{L}\varphi^{-1}(\mathcal{\widehat{D}}_{\mathfrak{Y}}^{(0,1)}))\widehat{\otimes}_{\varphi^{-1}(\mathcal{D}_{\mathfrak{Y}}^{(0,1)})}^{L}\varphi^{-1}(\mathcal{O}_{\mathfrak{Y}}[f,v]\widehat{\otimes}_{\psi^{-1}(D(\mathcal{O}_{\mathfrak{Z}}))}^{L}\psi^{-1}(\mathcal{\widehat{D}}_{\mathfrak{Z}}^{(0,1)}))
\]
\[
\tilde{=}D(\mathcal{O}_{\mathfrak{X}})\widehat{\otimes}_{\varphi^{-1}(D(\mathcal{O}_{\mathfrak{Y}}))}^{L}(\varphi^{-1}D(\mathcal{O}_{\mathfrak{Y}})\widehat{\otimes}_{(\psi\circ\varphi)^{-1}(D(\mathcal{O}_{\mathfrak{Z}}))}^{L}(\psi\circ\varphi)^{-1}(\mathcal{\widehat{D}}_{\mathfrak{Z}}^{(0,1)}))
\]
\[
\tilde{=}D(\mathcal{O}_{\mathfrak{X}})\widehat{\otimes}_{(\psi\circ\varphi)^{-1}(D(\mathcal{O}_{\mathfrak{Z}}))}^{L}(\psi\circ\varphi)^{-1}(\mathcal{\widehat{D}}_{\mathfrak{Z}}^{(0,1)})=\mathcal{\widehat{D}}_{\mathfrak{X}\to\mathfrak{Z}}^{(0,1)}
\]
as $(\mathcal{\widehat{D}}_{\mathfrak{X}}^{(0,1)},(\psi\circ\varphi)^{-1}\mathcal{\widehat{D}}_{\mathfrak{Z}}^{(0,1)})$-bimodules.
This yields 
\[
L(\psi\circ\varphi)^{*}(\mathcal{M}^{\cdot})=\mathcal{\widehat{D}}_{\mathfrak{X}\to\mathfrak{Z}}^{(0,1)}\widehat{\otimes}_{(\psi\circ\varphi)^{-1}(\mathcal{D}_{\mathfrak{Z}}^{(0,1)})}^{L}(\psi\circ\varphi)^{-1}(\mathcal{M}^{\cdot})
\]
\[
\tilde{=}(\mathcal{\widehat{D}}_{\mathfrak{X}\to\mathfrak{Y}}^{(0,1)}\widehat{\otimes}_{\varphi^{-1}(\mathcal{D}_{\mathfrak{Y}}^{(0,1)})}^{L}\varphi^{-1}(\mathcal{D}_{\mathfrak{Y\to\mathfrak{Z}}}^{(0,1)}))\widehat{\otimes}_{(\psi\circ\varphi)^{-1}(\mathcal{D}_{\mathfrak{Z}}^{(0,1)})}^{L}(\psi\circ\varphi)^{-1}(\mathcal{M}^{\cdot})
\]
\[
\tilde{=}\mathcal{\widehat{D}}_{\mathfrak{X}\to\mathfrak{Y}}^{(0,1)}\widehat{\otimes}_{\varphi^{-1}(\mathcal{D}_{\mathfrak{Y}}^{(0,1)})}^{L}\varphi^{-1}((\mathcal{D}_{\mathfrak{Y\to\mathfrak{Z}}}^{(0,1)})\widehat{\otimes}_{\psi^{-1}(\mathcal{D}_{\mathfrak{Z}}^{(0,1)})}^{L}\psi^{-1}(\mathcal{M}^{\cdot}))=L\varphi^{*}(L\psi^{*}\mathcal{M}^{\cdot})
\]
An identical argument works for $\varphi:X\to Y$ and $\psi:Y\to Z$. 
\end{proof}
Next, we have
\begin{prop}
\label{prop:Basic-base-change-for-pullback}1) Let $\mathcal{M}^{\cdot}\in D_{cc}(\mathcal{G}(\mathcal{D}_{\mathfrak{Y}}^{(0,1)}))$.
Then $L\varphi^{*}(\mathcal{M}^{\cdot})^{-\infty}\tilde{\to}L\varphi^{*,(0)}(\mathcal{M}^{\cdot,-\infty})$
and $\widehat{L\varphi^{*}(\mathcal{M}^{\cdot})^{\infty}}\tilde{\to}L\varphi^{*,(1)}(\widehat{\mathcal{M}^{\cdot,\infty}})$.
The analogous result holds for $\varphi:X\to Y$. 

2) Let $\mathcal{M}^{\cdot}\in D_{cc}(\mathcal{G}(\mathcal{D}_{\mathfrak{Y}}^{(0,1)}))$.
Then $L\varphi^{*}(\mathcal{M}^{\cdot})\otimes_{W(k)}^{L}k\tilde{\to}L\varphi^{*,(0)}(\mathcal{M}^{\cdot}\otimes_{W(k)}^{L}k)$.
\end{prop}

\begin{proof}
1) By construction we have 
\[
(\mathcal{\widehat{D}}_{\mathfrak{X}\to\mathfrak{Y}}^{(0,1)})\widehat{\otimes}_{D(W(k))}^{L}W(k)[f,v]/(f-1)\tilde{=}\mathcal{\widehat{D}}_{\mathfrak{X}\to\mathfrak{Y}}^{(1)}
\]
and 
\[
(\mathcal{\widehat{D}}_{\mathfrak{X}\to\mathfrak{Y}}^{(0)})\widehat{\otimes}_{D(W(k))}^{L}W(k)[f,v]/(v-1)\tilde{=}\mathcal{\widehat{D}}_{\mathfrak{X}\to\mathfrak{Y}}^{(0)}
\]
from which the result follows directly. Similarly, for part $2)$
one uses 
\[
(\mathcal{\widehat{D}}_{\mathfrak{X}\to\mathfrak{Y}}^{(0,1)})\otimes_{W(k)}^{L}k\tilde{=}\mathcal{\widehat{D}}_{X\to Y}^{(0,1)}
\]
\end{proof}
Specializing to the case of positive characteristic, it is also useful
to have comparisons with the pullbacks of $\mathcal{R}(\mathcal{D}_{X}^{(1)})$
and $\overline{\mathcal{R}}(\mathcal{D}_{X}^{(0)})$-modules. First,
we need to give the relevant definitions: 
\begin{defn}
Suppose $\varphi:X\to Y$. We let $\mathcal{R}_{X\to Y}^{(1)}:=\varphi^{*}\mathcal{D}_{Y}^{(0,1)}/(v)$
and $\mathcal{\overline{R}}_{X\to Y}^{(0)}:=\varphi^{*}\mathcal{D}_{Y}^{(0,1)}/(f)$;
considered as a graded $(\mathcal{R}(\mathcal{D}_{X}^{(1)}),\varphi^{-1}(\mathcal{R}(\mathcal{D}_{Y}^{(1)}))$
bimodule (respectively a graded $(\mathcal{\overline{R}}(\mathcal{D}_{X}^{(0)}),\varphi^{-1}(\overline{\mathcal{R}}(\mathcal{D}_{Y}^{(0)}))$
bimodule). Let $\mathcal{M}^{\cdot}\in D(\mathcal{G}(\mathcal{R}(\mathcal{D}_{Y}^{(1)})))$.
Then we define $L\varphi^{*,(1)}(\mathcal{M}^{\cdot}):=\mathcal{R}_{X\to Y}^{(1)}\otimes_{\varphi^{-1}(\mathcal{R}(\mathcal{D}_{Y}^{(1)})}^{L}\varphi^{-1}(\mathcal{M}^{\cdot})$
with the induced left action of $\mathcal{R}(\mathcal{D}_{X}^{(1)})$
given by the bimodule structure.. Set $\varphi^{\dagger,(1)}=L\varphi^{*,(1)}[d_{X/Y}]$. 

We make the analogous definition for $\overline{\mathcal{R}}(\mathcal{D}_{Y}^{(0)})$-modules;
and denote the corresponding functors $L\varphi^{*,(0)}$ and $\varphi^{\dagger,(1)}$. 
\end{defn}

We note that the functor $L\varphi^{*,(1)}$ takes $D_{qc}(\mathcal{G}(\mathcal{R}(\mathcal{D}_{Y}^{(1)}))$
to $D_{qc}(\mathcal{G}(\mathcal{R}(\mathcal{D}_{X}^{(1)}))$; and
similarly for $\overline{\mathcal{R}}(\mathcal{D}_{Y}^{(0)})$. Then
we have the 
\begin{prop}
\label{prop:pullback-and-R}Let $\mathcal{M}^{\cdot}\in D(\mathcal{G}(\mathcal{D}_{Y}^{(0,1)}))$.
There is an isomorphism of functors 
\[
\mathcal{R}(\mathcal{D}_{X}^{(1)})\otimes_{\mathcal{D}_{X}^{(0,1)}}^{L}\varphi^{\dagger}\mathcal{M}^{\cdot}\tilde{=}\varphi^{\dagger,(1)}(\mathcal{R}(\mathcal{D}_{Y}^{(1)})\otimes_{\mathcal{D}_{Y}^{(0,1)}}^{L}\mathcal{M}^{\cdot})
\]
and similarly for $\varphi_{0}^{\dagger}$. 
\end{prop}

\begin{proof}
We have
\[
\varphi^{\dagger,(1)}(\mathcal{R}(\mathcal{D}_{Y}^{(1)})\otimes_{\mathcal{D}_{Y}^{(0,1)}}^{L}\mathcal{M}^{\cdot})=\mathcal{R}_{X\to Y}^{(1)}\otimes_{\varphi^{-1}(\mathcal{R}(\mathcal{D}_{Y}^{(1)}))}^{L}\varphi^{-1}(\mathcal{R}(\mathcal{D}_{Y}^{(1)})\otimes_{\mathcal{D}_{Y}^{(0,1)}}^{L}\mathcal{M}^{\cdot})[d_{X/Y}]
\]
\[
\tilde{=}\mathcal{R}_{X\to Y}^{(1)}\otimes_{\varphi^{-1}(\mathcal{R}(\mathcal{D}_{Y}^{(1)}))}^{L}\varphi^{-1}(\mathcal{R}(\mathcal{D}_{Y}^{(1)}))\otimes_{\varphi^{-1}(\mathcal{D}_{Y}^{(0,1)})}^{L}\varphi^{-1}(\mathcal{M}^{\cdot})[d_{X/Y}]
\]
\[
\tilde{=}\mathcal{R}_{X\to Y}^{(1)}\otimes_{\varphi^{-1}(\mathcal{D}_{Y}^{(0,1)})}^{L}\varphi^{-1}(\mathcal{M}^{\cdot})[d_{X/Y}]
\]
Now, by definition, the module $\mathcal{D}_{X\to Y}^{(0,1)}$, admits,
locally on $X$ and $Y$, a lift $\widehat{\mathcal{D}}_{\mathfrak{X}\to\mathfrak{Y}}^{(0,1)}$
which we have constructed above in \defref{Transfer-Bimod}. This
lift is a standard gauge, and so $\mathcal{D}_{X\to Y}^{(0,1)}$ is
quasi-rigid. So, using the resolution (c.f. \lemref{Basic-Facts-on-Rigid})
\[
\cdots\to\mathcal{D}_{X}^{(0,1)}(-1)\xrightarrow{v}\mathcal{D}_{X}^{(0,1)}\xrightarrow{f}\mathcal{D}_{X}^{(0,1)}(-1)\xrightarrow{v}\mathcal{D}_{X}^{(0,1)}\to\mathcal{R}(\mathcal{D}_{X}^{(1)})
\]
for $\mathcal{R}(\mathcal{D}_{X}^{(1)})$ over $\mathcal{D}_{X}^{(0,1)}$,
this tell us that 
\begin{equation}
\mathcal{R}(\mathcal{D}_{X}^{(1)})\otimes_{\mathcal{D}_{X}^{(0,1)}}^{L}\mathcal{D}_{X\to Y}^{(0,1)}\tilde{=}\mathcal{D}_{X\to Y}^{(0,1)}/v=\mathcal{R}_{X\to Y}^{(1)}\label{eq:transfer-iso-1}
\end{equation}
i.e., this complex is concentrated in degree $0$ and is equal to
$\mathcal{R}_{X\to Y}^{(1)}$ there. Thus 
\[
\mathcal{R}_{X\to Y}^{(1)}\otimes_{\varphi^{-1}(\mathcal{D}_{Y}^{(0,1)})}^{L}\varphi^{-1}(\mathcal{M}^{\cdot})[d_{X/Y}]
\]
\[
\tilde{=}\mathcal{R}(\mathcal{D}_{X}^{(1)})\otimes_{\mathcal{D}_{X}^{(0,1)}}^{L}\mathcal{D}_{X\to Y}^{(0,1)}\otimes_{\varphi^{-1}(\mathcal{D}_{Y}^{(0,1)})}^{L}\varphi^{-1}(\mathcal{M}^{\cdot})[d_{X/Y}]=\mathcal{R}(\mathcal{D}_{X}^{(1)})\otimes_{\mathcal{D}_{X}^{(0,1)}}^{L}\varphi^{\dagger}\mathcal{M}^{\cdot}
\]
as desired. The case of $\overline{\mathcal{R}}(\mathcal{D}_{Y}^{(0)})$-modules
is essentially identical.
\end{proof}
Finally, we also have 
\begin{prop}
\label{prop:Smooth-pullback-preserves-coh}If $\varphi$ is smooth,
then $L\varphi^{*}$ takes $D_{coh}^{b}(\mathcal{G}(\mathcal{D}_{\mathfrak{Y}}^{(0,1)}))$
to $D_{coh}^{b}(\mathcal{G}(\mathcal{D}_{\mathfrak{X}}^{(0,1)}))$.
The same holds for a smooth morphism $\varphi:X\to Y$. 
\end{prop}

\begin{proof}
By part $2)$ of \propref{Basic-base-change-for-pullback}, as well
as \propref{coh-to-coh}, the first statement reduces to the second.
We may assume that $X=\text{Spec}(B)$ and $Y=\text{Spec}(A)$ both
possess local coordinates. After further localizing if necessary we
can suppose that there are local coordinates $\{\partial_{1},\dots,\partial_{n}\}$
on $B$ such that the $A$-linear derivations of $B$ are $\{\partial_{1},\dots,\partial_{d}\}$.
In this case, if we let $J\subset\mathcal{D}_{B}^{(0,1)}$ be the
ideal generated by $\{\partial_{1},\dots,\partial_{d},\partial_{1}^{[p]},\dots,\partial_{d}^{[p]}\}$,
then we have 
\[
\mathcal{D}_{B}^{(0,1)}/J\tilde{=}B\otimes_{A}\mathcal{D}_{A}^{(0,1)}
\]
which shows that $B\otimes_{A}\mathcal{D}_{A}^{(0,1)}=\varphi^{*}\mathcal{D}_{A}^{(0,1)}$
is a coherent $\mathcal{D}_{B}^{(0,1)}$-module, which is flat as
a module over $\mathcal{D}_{A}^{(0,1),\text{opp}}$. This shows that
$\varphi^{*}$ is exact; and the coherence of the pullback for an
arbitrary coherent $\mathcal{D}_{A}^{(0,1)}$-module $\mathcal{M}$
follows by taking a finite presentation for $\mathcal{M}$. 
\end{proof}

\section{\label{sec:Operations:Swap-Tensor}Operations on Gauges: Left-Right
Interchange and tensor Product}

The first goal of this subsection is to prove
\begin{prop}
\label{prop:Left-Right-Swap} Let $\mathcal{M}\in\mathcal{G}(\mathcal{\widehat{D}}_{\mathfrak{X}}^{(0,1)})$.
Then $\omega_{\mathfrak{X}}\otimes_{\mathcal{O}_{\mathfrak{X}}}\mathcal{M}$
carries the structure of a right graded $\mathcal{\widehat{D}}_{\mathfrak{X}}^{(0,1)}$-module.
This functor defines an equivalence of categories, which preserves
coherent modules. The derived functor preserves the subcategories
of derived complete complexes. 

The analogous result holds for $X$ (i.e., in positive characteristic);
there, the functor preserves the category of quasi-coherent sheaves
as well.
\end{prop}

In order to prove this, we first recall that $\omega_{\mathfrak{X}}$
naturally carries the structure of a right $\mathcal{D}_{\mathfrak{X}}^{(i)}$-module
for all $i\geq0$; indeed, $\omega_{\mathfrak{X}}[p^{-1}]$ carries
a right $\mathcal{D}_{\mathfrak{X}}^{(i)}[p^{-1}]=\mathcal{D}_{\mathfrak{X}}^{(0)}[p^{-1}]$
structure via the Lie derivative (c.f., e.g. \cite{key-4}, page 8).
In local coordinates, this action is simply given by
\[
(gdx_{1}\wedge\cdots\wedge dx_{n})\partial=-\partial(g)dx_{1}\wedge\cdots\wedge dx_{n}
\]
for any derivation $\partial$. It follows that $\mathcal{D}_{\mathfrak{X}}^{(i)}$
preserves $\omega_{\mathfrak{X}}$ (for all $i$). As $\omega_{\mathfrak{X}}$
is $p$-adically complete, we see that it also inherits a right $\mathcal{\widehat{D}}_{\mathfrak{X}}^{(i)}$-module
structure. 
\begin{lem}
Let $D(\omega_{\mathfrak{X}})=\omega_{\mathfrak{X}}\otimes_{\mathcal{O}_{\mathfrak{X}}}D(\mathcal{O}_{\mathfrak{X}})$.
Then $D(\omega_{\mathfrak{X}})$ has a natural right graded $\mathcal{\widehat{D}}_{\mathfrak{X}}^{(0,1)}$-module
structure. Similarly, $D(\omega_{X})$ admits a right graded $\mathcal{D}_{X}^{(0,1)}$-module
structure, for any smooth $X$ over $k$. 
\end{lem}

\begin{proof}
We note that $(\omega_{\mathfrak{X}}[f,v])^{i}=\{m\in\omega_{\mathfrak{X}}[p^{-1}]|p^{i}m\in\omega_{\mathfrak{X}}\}$.
Thus the first result follows by using the right $\mathcal{\widehat{D}}_{\mathfrak{X}}^{(1)}$-module
structure on $\omega_{\mathfrak{X}}$. To prove the second result,
we choose on open affine $\text{Spec}(A)\subset X$ which possesses
etale local coordinates. In coordinates, the required action is given
by 
\[
(gdx_{1}\wedge\cdots\wedge dx_{n})\partial=-\partial(g)dx_{1}\wedge\cdots\wedge dx_{n}
\]
and 
\[
(gdx_{1}\wedge\cdots\wedge dx_{n})\partial^{[p]}=-f\cdot\partial^{[p]}(g)dx_{1}\wedge\cdots\wedge dx_{n}
\]
for any $g\in D(\mathcal{O}_{\mathfrak{X}})$. If we choose a lift
$\mathcal{A}$ of $A$, then, after lifting the coordinates, we see
that this action is the reduction mod $p$ of the action just defined;
in particular it is actually independent of the choice of coordinates
and therefore glues to define an action on all of $X$. 
\end{proof}
Now we recall a very general construction from \cite{key-4}, section
1.4b 
\begin{lem}
Let $\mathcal{L}$ be any line bundle on $\mathfrak{X}$. Placing
$\mathcal{L}$ and $\mathcal{L}^{-1}$ in degree $0$, the sheaf $\mathcal{\widehat{D}}_{\mathfrak{X},\mathcal{L}}^{(0,1)}:=\mathcal{L}\otimes_{\mathcal{O}_{\mathfrak{X}}}\mathcal{\widehat{D}}_{\mathfrak{X}}^{(0,1)}\otimes_{\mathcal{O}_{\mathfrak{X}}}\mathcal{L}^{-1}$
carries the structure of a graded algebra on $\mathfrak{X}$, via
the multiplication 
\[
(s_{1}\otimes\Phi_{1}\otimes t_{1})\cdot(s_{2}\otimes\Phi_{2}\otimes t_{2})=s_{1}\otimes\Phi_{1}<t_{1},s_{1}>\Phi_{2}\otimes t_{2}
\]
There is a functor $\mathcal{G}(\mathcal{\widehat{D}}_{\mathfrak{X}}^{(0,1)})\to\mathcal{G}(\mathcal{\widehat{D}}_{\mathfrak{X},\mathcal{L}}^{(0,1)})$
given by $\mathcal{M}\to\mathcal{L}\otimes_{\mathcal{O}_{\mathfrak{X}}}\mathcal{M}$;
the action of $\mathcal{\widehat{D}}_{\mathfrak{X},\mathcal{L}}^{(0,1)}$
on $\mathcal{L}\otimes_{\mathcal{O}_{\mathfrak{X}}}\mathcal{M}$ is
defined by 
\[
(s\otimes\Phi\otimes t)\cdot(s_{1}\otimes m)=s\otimes\Phi_{1}<t,s_{1}>m
\]
This functor is an equivalence of categories, whose inverse is given
by $\mathcal{N}\to\mathcal{L}^{-1}\otimes_{\mathcal{O}_{\mathfrak{X}}}\mathcal{N}$. 
\end{lem}

So, \propref{Left-Right-Swap} follows directly from 
\begin{lem}
There is an isomorphism of algebras $\mathcal{\widehat{D}}_{\mathfrak{X},\omega_{\mathfrak{X}}}^{(0,1)}\tilde{=}\mathcal{\widehat{D}}_{\mathfrak{X}}^{(0,1),\text{op}}$.
The same is true over $X$. 
\end{lem}

\begin{proof}
We have the isomorphism $\mathcal{\widehat{D}}_{\mathfrak{X},\omega_{\mathfrak{X}}}^{(0,1)}\tilde{=}D(\omega_{\mathfrak{X}})\otimes_{D(\mathcal{O}_{\mathfrak{X}})}\mathcal{\widehat{D}}_{\mathfrak{X}}^{(0,1)}\otimes_{D(\mathcal{O}_{\mathfrak{X}})}D(\mathcal{\omega}_{\mathfrak{X}}^{-1})$.
This yields a left action of $\mathcal{\widehat{D}}_{\mathfrak{X},\omega_{\mathfrak{X}}}^{(0,1)}$
on $\omega_{\mathfrak{X}}[f,v]$, given by 
\[
(s\otimes\Phi\otimes t)\cdot s_{1}=s\otimes\Phi\cdot<t,s_{1}>
\]
where $<,>$ refers to the pairing $D(\mathcal{\omega}_{\mathfrak{X}})\otimes_{D(\mathcal{O}_{\mathfrak{X}})}D(\mathcal{\omega}_{\mathfrak{X}}^{-1})\to D(\mathcal{O}_{\mathfrak{X}})$.
Computing in local coordinates, one sees that the image of $\mathcal{\widehat{D}}_{\mathfrak{X},\omega_{\mathfrak{X}}}^{(0,1)}$
in $\mathcal{E}nd_{W(k)}(D(\omega_{\mathfrak{X}}))$ is the same as
the image of $\mathcal{\widehat{D}}_{\mathfrak{X}}^{(0,1),\text{op}}$
in $\mathcal{E}nd_{W(k)}(D(\omega_{\mathfrak{X}}))$ via the right
action defined above. This yields the isomorphism over $\mathfrak{X}$.
To deal with $X$, one first obtains the isomorphism locally (via
a local lifting of the variety), and then shows that the resulting
isomorphism is independent of the choice of coordinates (as in the
proof of the previous lemma). 
\end{proof}
Next, we define tensor products of (left) $\widehat{\mathcal{D}}_{\mathfrak{X}}^{(0,1)}$-modules.
The first step is to define the external product of sheaves:
\begin{defn}
1) Let $\mathfrak{X}$ and $\mathfrak{Y}$ be smooth formal schemes,
and let $\mathcal{M}^{\cdot}\in D(\mathcal{G}(D(\mathcal{O}_{\mathfrak{X}})))$,
$\mathcal{N}^{\cdot}\in D(\mathcal{G}(D(\mathcal{O}_{\mathfrak{Y}})))$.
Then we define 
\[
\mathcal{M}^{\cdot}\boxtimes\mathcal{N}^{\cdot}:=Lp_{1}^{*}(\mathcal{M}^{\cdot})\widehat{\otimes}_{D(\mathcal{O}_{\mathfrak{X}\times\mathfrak{Y}})}^{L}Lp_{2}^{*}(\mathcal{N}^{\cdot})\in D_{cc}(\mathcal{G}(D(\mathcal{O}_{\mathfrak{X}\times\mathfrak{Y}})))
\]
where $p_{i}$ ($i\in\{1,2\}$) are the projections and $Lp_{1}^{*},Lp_{2}^{*}$
are defined as in \defref{Correct-Pullback}. 

2) Let $X$ and $Y$ be smooth schemes over $k$. Then for $\mathcal{M}^{\cdot}\in D(\mathcal{G}(D(\mathcal{O}_{X})))$,
$\mathcal{N}^{\cdot}\in D(\mathcal{G}(D(\mathcal{O}_{Y})))$. Then
we define 
\[
\mathcal{M}^{\cdot}\boxtimes\mathcal{N}^{\cdot}:=Lp_{1}^{*}(\mathcal{M}^{\cdot})\otimes_{D(\mathcal{O}_{X\times Y})}^{L}Lp_{2}^{*}(\mathcal{N}^{\cdot})\in D(\mathcal{G}(D(\mathcal{O}_{X\times Y})))
\]
where for $\mathcal{M}^{\cdot}\in D(\mathcal{G}(D(\mathcal{O}_{X})))$
we have $Lp_{1}^{*}\mathcal{M}^{\cdot}=D(\mathcal{O}_{X\times Y})\otimes_{p_{1}^{-1}(D(\mathcal{O}_{X}))}^{L}\mathcal{M}^{\cdot}\in D(\mathcal{G}(D(\mathcal{O}_{X\times Y})))$
( and similarly for $p_{2}$). 
\end{defn}

The relationship with $\mathcal{D}$-modules is the following: 
\begin{lem}
1) There is an isomorphism 
\[
\widehat{\mathcal{D}}_{\mathfrak{X}}^{(0,1)}\boxtimes\widehat{\mathcal{D}}_{\mathfrak{Y}}^{(0,1)}\tilde{=}\widehat{\mathcal{D}}_{\mathfrak{X}\times\mathfrak{Y}}^{(0,1)}
\]
of sheaves of algebras on $\mathfrak{X}\times\mathfrak{Y}$. 

2) There is an isomorphism 
\[
\mathcal{D}_{X}^{(0,1)}\boxtimes\mathcal{D}_{Y}^{(0,1)}\tilde{=}\mathcal{D}_{X\times Y}^{(0,1)}
\]
of sheaves of algebras on $X\times Y$.
\end{lem}

\begin{proof}
First suppose $\mathfrak{X}=\text{Specf}(\mathcal{A})$ and $\mathfrak{Y}=\text{Specf}(\mathcal{B})$.
Then there is a morphism $\mathcal{D}_{\mathcal{A}}^{(\infty)}\otimes_{W(k)}\mathcal{D}_{\mathcal{B}}^{(\infty)}\to\mathcal{D}_{\mathcal{A}\widehat{\otimes}_{W(k)}\mathcal{B}}^{(\infty)}$
defined as follows: for sections $a\in\mathcal{A}$ and $b\in\mathcal{B}$,
we set
\[
(\Phi_{1}\otimes\Phi_{2})(a\otimes b)=\Phi_{1}(a)\otimes\Phi_{2}(b)
\]
and we extend to $\mathcal{A}\widehat{\otimes}_{W(k)}\mathcal{B}$
by linearity and continuity. For a fixed integer $j\geq0$, this yields
a map $\mathcal{D}_{\mathcal{A}}^{(j)}\otimes_{W(k)}\mathcal{D}_{\mathcal{B}}^{(j)}\to\mathcal{D}_{\mathcal{A}\widehat{\otimes}_{W(k)}\mathcal{B}}^{(j)}$;
these maps are compatible with localization at any element of $\mathcal{A}$
or $\mathcal{B}$. After $p$-adically completing we get a map $\widehat{\mathcal{D}}_{\mathcal{A}}^{(j)}\widehat{\otimes}_{W(k)}\widehat{\mathcal{D}}_{\mathcal{B}}^{(j)}\to\widehat{\mathcal{D}}_{\mathcal{A}\widehat{\otimes}_{W(k)}\mathcal{B}}^{(j)}$,
and these maps sheafifiy to a map \linebreak{}
$p_{1}^{-1}(\widehat{\mathcal{D}}_{\mathfrak{X}}^{(j)})\widehat{\otimes}_{W(k)}p_{2}^{-1}(\widehat{\mathcal{D}}_{\mathfrak{Y}}^{(j)})\to\widehat{\mathcal{D}}_{\mathfrak{X}\times\mathfrak{Y}}^{(j)}$.
Note that since $\mathcal{D}_{\mathcal{A}}^{(j)}\otimes_{W(k)}\mathcal{D}_{\mathcal{B}}^{(j)}$
is $p$-torsion-free (as is $\mathcal{D}_{\mathcal{A}\widehat{\otimes}_{W(k)}\mathcal{B}}^{(j)}$),
the usual $p$-adic completion of these sheaves agrees with the cohomological
completion. It follows that $p_{1}^{-1}(\widehat{\mathcal{D}}_{\mathfrak{X}}^{(j)})\widehat{\otimes}_{W(k)}p_{2}^{-1}(\widehat{\mathcal{D}}_{\mathfrak{Y}}^{(j)})\tilde{=}p_{1}^{*}(\widehat{\mathcal{D}}_{\mathfrak{X}}^{(j)})\widehat{\otimes}_{\mathcal{O}_{\mathfrak{X}\times\mathfrak{Y}}}p_{2}^{*}(\widehat{\mathcal{D}}_{\mathfrak{Y}}^{(j)}))$. 

1) We claim that the map 
\[
p_{1}^{*}(\widehat{\mathcal{D}}_{\mathfrak{X}}^{(j)})\widehat{\otimes}_{\mathcal{O}_{\mathfrak{X}\times\mathfrak{Y}}}p_{2}^{*}(\widehat{\mathcal{D}}_{\mathfrak{Y}}^{(j)}))\to\widehat{\mathcal{D}}_{\mathfrak{X}\times\mathfrak{Y}}^{(j)}
\]
is an isomorphism; indeed, both sides are $p$-adically complete and
$p$-torsion-free, so it suffices to check this after reduction mod
$p$, where it becomes an easy computation in local coordinates. Thus
we obtain isomorphisms 
\[
\{\Phi\in p_{1}^{*}(\widehat{\mathcal{D}}_{\mathfrak{X}}^{(1)})\widehat{\otimes}_{\mathcal{O}_{\mathfrak{X}\times\mathfrak{Y}}}p_{2}^{*}(\widehat{\mathcal{D}}_{\mathfrak{Y}}^{(1)})|p^{i}\Phi\in p_{1}^{*}(\widehat{\mathcal{D}}_{\mathfrak{X}}^{(0)})\widehat{\otimes}_{\mathcal{O}_{\mathfrak{X}\times\mathfrak{Y}}}p_{2}^{*}(\widehat{\mathcal{D}}_{\mathfrak{Y}}^{(0)})\}
\]
\[
\tilde{\to}\{\Phi\in\widehat{\mathcal{D}}_{\mathfrak{X}\times\mathfrak{Y}}^{(1)}|p^{i}\Phi\in\widehat{\mathcal{D}}_{\mathfrak{X}\times\mathfrak{Y}}^{(0)}\}=\widehat{\mathcal{D}}_{\mathfrak{X}\times\mathfrak{Y}}^{(0,1),i}
\]
for each $i\in\mathbb{Z}$. 

On the other hand, we claim that there is an isomorphism 
\[
(\widehat{\mathcal{D}}_{\mathfrak{X}}^{(0,1)}\boxtimes\widehat{\mathcal{D}}_{\mathfrak{Y}}^{(0,1)})^{i}\tilde{\to}\{\Phi\in p_{1}^{*}(\widehat{\mathcal{D}}_{\mathfrak{X}}^{(1)})\widehat{\otimes}_{\mathcal{O}_{\mathfrak{X}\times\mathfrak{Y}}}p_{2}^{*}(\widehat{\mathcal{D}}_{\mathfrak{Y}}^{(1)})|p^{i}\Phi\in p_{1}^{*}(\widehat{\mathcal{D}}_{\mathfrak{X}}^{(0)})\widehat{\otimes}_{\mathcal{O}_{\mathfrak{X}\times\mathfrak{Y}}}p_{2}^{*}(\widehat{\mathcal{D}}_{\mathfrak{Y}}^{(0)})\}
\]
Combined with the above, this proves $(\widehat{\mathcal{D}}_{\mathfrak{X}}^{(0,1)}\boxtimes\widehat{\mathcal{D}}_{\mathfrak{Y}}^{(0,1)})^{i}\tilde{\to}\widehat{\mathcal{D}}_{\mathfrak{X}\times\mathfrak{Y}}^{(0,1),i}$
as required. To see it, note that we have the map 
\[
f_{\infty}:(\widehat{\mathcal{D}}_{\mathfrak{X}}^{(0,1)}\boxtimes\widehat{\mathcal{D}}_{\mathfrak{Y}}^{(0,1)})^{i}\to(\widehat{\mathcal{D}}_{\mathfrak{X}}^{(0,1)}\boxtimes\widehat{\mathcal{D}}_{\mathfrak{Y}}^{(0,1)})^{\infty}
\]
The completion of the right hand side is $p_{1}^{*}(\widehat{\mathcal{D}}_{\mathfrak{X}}^{(1)})\widehat{\otimes}_{\mathcal{O}_{\mathfrak{X}\times\mathfrak{Y}}}p_{2}^{*}(\widehat{\mathcal{D}}_{\mathfrak{Y}}^{(1)})$;
so we obtain a map
\[
(\widehat{\mathcal{D}}_{\mathfrak{X}}^{(0,1)}\boxtimes\widehat{\mathcal{D}}_{\mathfrak{Y}}^{(0,1)})^{i}\to\{\Phi\in p_{1}^{*}(\widehat{\mathcal{D}}_{\mathfrak{X}}^{(1)})\widehat{\otimes}_{\mathcal{O}_{\mathfrak{X}\times\mathfrak{Y}}}p_{2}^{*}(\widehat{\mathcal{D}}_{\mathfrak{Y}}^{(1)})|p^{i}\Phi\in p_{1}^{*}(\widehat{\mathcal{D}}_{\mathfrak{X}}^{(0)})\widehat{\otimes}_{\mathcal{O}_{\mathfrak{X}\times\mathfrak{Y}}}p_{2}^{*}(\widehat{\mathcal{D}}_{\mathfrak{Y}}^{(0)})\}
\]
and to see that it is an isomorphism, one may check it after reduction
mod $p$; then it follows from the result of part $2)$ proved directly
below.

2) As above we have the map $p_{1}^{-1}\mathcal{D}_{X}^{(\infty)}\otimes_{k}p_{2}^{-1}\mathcal{D}_{Y}^{(\infty)}\to\mathcal{D}_{X\times Y}^{(\infty)}$.
Restricting to $\mathcal{T}_{X}$ and $\mathfrak{l}_{X}$ (a defined
in \defref{L} above)we get maps $p_{1}^{\#}:p_{1}^{-1}(\mathcal{T}_{X})\to\mathcal{T}_{X\times Y}$
and $p_{1}^{\#}:p_{1}^{-1}(\mathfrak{l}_{X})\to\mathfrak{l}_{X\times Y}$;
and similarly for $p_{2}$. Thus we get a map 
\[
A:(\mathcal{T}_{X}\boxtimes1)\oplus(1\boxtimes\mathcal{T}_{Y})\oplus(\mathfrak{l}_{X}\boxtimes1)\oplus(1\boxtimes\mathfrak{l}_{Y})\to\mathcal{D}_{X\times Y}^{(0,1)}
\]
defined by 
\[
A(\partial_{1}\boxtimes1+1\boxtimes\partial_{2}+\delta_{1}\boxtimes1+1\boxtimes\delta_{2})=p_{1}^{\#}(\partial_{1})+p_{2}^{\#}(\partial_{2})+p_{1}^{\#}(\delta_{1})+p_{2}^{\#}(\delta_{2})
\]
On the other hand, the sheaf $(\mathcal{T}_{X}\boxtimes1)\oplus(1\boxtimes\mathcal{T}_{Y})\oplus(\mathfrak{l}_{X}\boxtimes1)\oplus(1\boxtimes\mathfrak{l}_{Y})$
generates $\mathcal{D}_{X}^{(0,1)}\boxtimes\mathcal{D}_{Y}^{(0,1)}$
as a sheaf of algebras over $\mathcal{O}_{X\times Y}[f,v]$. Thus
to show that $A$ extends (necessarily uniquely) to an isomorphism
of algebras, we can so do locally.

So, let $\{x_{1},\dots,x_{n}\}$ and $\{y_{1},\dots,y_{m}\}$ be local
coordinates on $X$ and $Y$, respectively, with associated derivations
$\{\partial_{x_{1}},\dots,\partial_{x_{n}}\}$ and $\{\partial_{y_{1}},\dots,\partial_{y_{m}}\}$.
Then by \corref{Local-coords-over-A=00005Bf,v=00005D} an $D(\mathcal{O}_{X})$-basis
for $\mathcal{D}_{X}^{(0,1)}$ is given by the set $\{\partial_{x}^{I}(\partial_{x}^{[p]})^{J}\}$
for multi-indices $I,J$ such that each entry of $I$ is contained
in $\{0,1,\dots,p-1\}$; the analogous statement holds over $Y$.
Therefore the set $\{\partial_{x}^{I_{1}}(\partial_{x}^{[p]})^{J_{1}}\otimes\partial_{y}^{I_{2}}(\partial_{y}^{[p]})^{J_{2}}\}$
is an $\mathcal{O}_{X\times Y}[f,v]$-basis for $\mathcal{D}_{X}^{(0,1)}\boxtimes\mathcal{D}_{Y}^{(0,1)}$;
but also $\{\partial_{x}^{I_{1}}\partial_{y}^{I_{2}}(\partial_{x}^{[p]})^{J_{1}}(\partial_{y}^{[p]})^{J_{2}}\}$
is certainly an $D(\mathcal{O}_{X\times Y})$-basis for $\mathcal{D}_{X\times Y}^{(0,1)}$
and so the result follows immediately. 
\end{proof}
Now we can define the tensor product: 
\begin{defn}
Let $\Delta:\mathfrak{X}\to\mathfrak{X}\times\mathfrak{X}$ denote
the diagonal morphism. 

1) Then for $\mathcal{M}^{\cdot},\mathcal{N}^{\cdot}\in D(\mathcal{G}(\widehat{\mathcal{D}}_{\mathfrak{X}}^{(0,1)}))$
we define $\mathcal{M}^{\cdot}\widehat{\otimes}_{D(\mathcal{O}_{\mathfrak{X}})}^{L}\mathcal{N}^{\cdot}:=L\Delta^{*}(\mathcal{M}^{\cdot}\boxtimes\mathcal{N}^{\cdot})\in D_{cc}(\mathcal{G}(\widehat{\mathcal{D}}_{\mathfrak{X}}^{(0,1)}))$,
where $\mathcal{M}^{\cdot}\boxtimes\mathcal{N}^{\cdot}$ is regarded
as an element of $D_{cc}(\mathcal{G}(\widehat{\mathcal{D}}_{\mathfrak{X}\times\mathfrak{X}}^{(0,1)}))$
via the isomorphism $\widehat{\mathcal{D}}_{\mathfrak{X}}^{(0,1)}\boxtimes\widehat{\mathcal{D}}_{\mathfrak{X}}^{(0,1)}\tilde{=}\widehat{\mathcal{D}}_{\mathfrak{X}\times\mathfrak{X}}^{(0,1)}$. 

2) For $\mathcal{M}^{\cdot}\in D(\mathcal{G}(\widehat{\mathcal{D}}_{\mathfrak{X}}^{(0,1),\text{op}}))$
and $\mathcal{N}^{\cdot}\in D(\mathcal{G}(\widehat{\mathcal{D}}_{\mathfrak{X}}^{(0,1)}))$,
we define $\mathcal{M}^{\cdot}\widehat{\otimes}_{D(\mathcal{O}_{\mathfrak{X}})}^{L}\mathcal{N}^{\cdot}:=\omega_{\mathfrak{X}}\otimes_{\mathcal{O}_{\mathfrak{X}}}((\omega_{\mathfrak{X}}^{-1}\otimes_{\mathcal{O}_{\mathfrak{X}}}\mathcal{M}^{\cdot})\widehat{\otimes}_{D(\mathcal{O}_{\mathfrak{X}})}^{L}\mathcal{N}^{\cdot})\in D(\mathcal{G}(\widehat{\mathcal{D}}_{\mathfrak{X}}^{(0,1),\text{op}}))$

One has the analogous constructions for a smooth $X$ over $k$. 
\end{defn}

From the construction, one sees directly that, as an $D(\mathcal{O}_{\mathfrak{X}})$-module,
the module $\mathcal{M}^{\cdot}\widehat{\otimes}_{D(\mathcal{O}_{\mathfrak{X}})}^{L}\mathcal{N}^{\cdot}$
agrees with the $D(\mathcal{O}_{\mathfrak{X}})$-module denoted in
the same way. The issue that this construction resolves is how to
put a $\widehat{\mathcal{D}}_{\mathfrak{X}}^{(0,1)}$-module structure
on this object. 

To proceed further, it is useful to note some explicit formulas in
coordinates: 
\begin{rem}
\label{rem:Two-actions-agree}Suppose we have local coordinates $\{x_{i}\}_{i=1}^{n}$
and $\{\partial_{i}\}_{i=1}^{n}$ on $\mathfrak{X}$. Then for modules
$\mathcal{M},\mathcal{N}\in\mathcal{G}(\widehat{\mathcal{D}}_{\mathfrak{X}}^{(0,1)})$
we can put an action of $\widehat{\mathcal{D}}_{\mathfrak{X}}^{(0,1)}$
on $\mathcal{M}\otimes_{D(\mathcal{O}_{\mathfrak{X}})}\mathcal{N}$
via the following formulas: 
\[
\partial_{i}(m\otimes n)=\partial_{i}m\otimes n+m\otimes\partial_{i}n
\]
and 
\[
\partial_{i}^{([p]}(m\otimes n)=f\sum_{j=1}^{p-1}\partial^{[j]}(m)\otimes\partial^{[p-j]}(m)+\partial^{[p]}(m)\otimes n+m\otimes\partial^{[p]}(n)
\]
Taking a flat resolution of $\mathcal{N}$, this gives $\mathcal{M}\otimes_{D(\mathcal{O}_{\mathfrak{X}})}^{L}\mathcal{N}$
the structure of an element of $D(\mathcal{G}(\widehat{\mathcal{D}}_{\mathfrak{X}}^{(0,1)}))$,
which means that $\mathcal{M}\widehat{\otimes}_{D(\mathcal{O}_{\mathfrak{X}})}^{L}\mathcal{N}$
belongs to $D_{cc}(\mathcal{G}(\widehat{\mathcal{D}}_{\mathfrak{X}}^{(0,1)}))$.
This object is isomorphic to the tensor product defined above. Indeed,
in local coordinates the action of $\widehat{\mathcal{D}}_{\mathfrak{X}}^{(0,1)}$
on $\Delta^{*}(\widehat{\mathcal{D}}_{\mathfrak{X}\times\mathfrak{X}}^{(0,1)})$
is given as follows: let $\{\partial_{i},\partial'_{i}\}_{i=1}^{n}$
be local coordinate derivations on $\mathfrak{X}\times\mathfrak{X}$.
Then the action is given by $\partial_{i}\cdot1=\partial_{i}+\partial_{i}'$
and $\partial_{i}^{[p]}\cdot1=f\sum_{j=1}^{p-1}\partial_{i}^{[j]}\cdot(\partial'_{i})^{[p-j]}+\partial_{i}^{[p]}+(\partial'_{i})^{[p]}$,
which agrees with the above formula. 
\end{rem}

This allows us to prove the following useful 
\begin{lem}
\label{lem:Juggle}(Compare \cite{key-50}, lemma 2.2.5) Let $\mathcal{M}^{\cdot},\mathcal{P}^{\cdot}$
be elements of $D(\mathcal{G}(\widehat{\mathcal{D}}_{\mathfrak{X}}^{(0,1)}))$
and $\mathcal{N}^{\cdot}\in D(\mathcal{G}(\widehat{\mathcal{D}}_{\mathfrak{X}}^{(0,1),\text{opp}}))$.
Then there is an isomorphism 
\[
\mathcal{N}^{\cdot}\widehat{\otimes}_{\widehat{\mathcal{D}}_{\mathfrak{X}}^{(0,1)}}^{L}(\mathcal{M}^{\cdot}\widehat{\otimes}_{D(\mathcal{O}_{\mathfrak{X}})}^{L}\mathcal{P}^{\cdot})\tilde{\to}(\mathcal{N}^{\cdot}\widehat{\otimes}_{D(\mathcal{O}_{\mathfrak{X}})}^{L}\mathcal{M}^{\cdot})\widehat{\otimes}_{\widehat{\mathcal{D}}_{\mathfrak{X}}^{(0,1)}}^{L}\mathcal{P}^{\cdot}
\]
\end{lem}

\begin{proof}
Let $\mathcal{M},\mathcal{P}\in\mathcal{G}(\widehat{\mathcal{D}}_{\mathfrak{X}}^{(0,1)})$
and $\mathcal{N}\in\mathcal{G}(\widehat{\mathcal{D}}_{\mathfrak{X}}^{(0,1),\text{opp}})$.
We have a map of $D(\mathcal{O}_{\mathfrak{X}})$-modules 
\[
\mathcal{N}\otimes_{D(\mathcal{O}_{\mathfrak{X}})}(\mathcal{M}\otimes_{D(\mathcal{O}_{\mathfrak{X}})}\mathcal{P})\to(\mathcal{N}\otimes_{D(\mathcal{O}_{\mathfrak{X}})}\mathcal{M})\otimes_{\widehat{\mathcal{D}}_{\mathfrak{X}}^{(0,1)}}\mathcal{P}
\]
simply because $D(\mathcal{O}_{\mathfrak{X}})$ is a sub-algebra of
$\widehat{\mathcal{D}}_{\mathfrak{X}}^{(0,1)}$. Using the local description
of the $\widehat{\mathcal{D}}_{\mathfrak{X}}^{(0,1)}$-module action
on $\mathcal{N}\otimes_{D(\mathcal{O}_{\mathfrak{X}})}\mathcal{M}$
given by \remref{Two-actions-agree}, one sees that this map factors
through $\mathcal{N}\otimes_{\widehat{\mathcal{D}}_{\mathfrak{X}}^{(0,1)}}(\mathcal{M}\otimes_{D(\mathcal{O}_{\mathfrak{X}})}\mathcal{P})$
and we obtain a morphism
\[
\mathcal{N}\otimes_{\widehat{\mathcal{D}}_{\mathfrak{X}}^{(0,1)}}(\mathcal{M}\otimes_{D(\mathcal{O}_{\mathfrak{X}})}\mathcal{P})\to(\mathcal{N}\otimes_{D(\mathcal{O}_{\mathfrak{X}})}\mathcal{M})\otimes_{\widehat{\mathcal{D}}_{\mathfrak{X}}^{(0,1)}}\mathcal{P}
\]
Since $\widehat{\mathcal{D}}_{\mathfrak{X}}^{(0,1)}$ is flat over
$D(\mathcal{O}_{\mathfrak{X}})$, we can compute the associated derived
functors using K-flat resolutions over $\widehat{\mathcal{D}}_{\mathfrak{X}}^{(0,1)}$
of $\mathcal{N}$, and $\mathcal{P}$, respectively. Doing so gives
a map in the derived category 
\[
\mathcal{N}^{\cdot}\otimes_{\widehat{\mathcal{D}}_{\mathfrak{X}}^{(0,1)}}^{L}(\mathcal{M}^{\cdot}\otimes_{D(\mathcal{O}_{\mathfrak{X}})}^{L}\mathcal{P}^{\cdot})\to(\mathcal{N}^{\cdot}\otimes_{D(\mathcal{O}_{\mathfrak{X}})}^{L}\mathcal{M}^{\cdot})\otimes_{\widehat{\mathcal{D}}_{\mathfrak{X}}^{(0,1)}}^{L}\mathcal{P}^{\cdot}
\]
and passing to the derived completions gives the map in the statement
of the lemma; to show it is an isomorphism we may reduce mod $p$
and, taking K-flat resolutions, assume that each term of both $\mathcal{N}^{\cdot}$
and $\mathcal{P}^{\cdot}$ is stalk-wise free over $\mathcal{D}_{X}^{(0,1)}$;
thus the statement comes down to the claim that 
\[
\mathcal{D}_{X}^{(0,1)}\otimes_{\mathcal{D}_{X}^{(0,1)}}(\mathcal{M}\otimes_{D(\mathcal{O}_{X})}\mathcal{D}_{X}^{(0,1)})\tilde{\to}(\mathcal{D}_{X}^{(0,1)}\otimes_{D(\mathcal{O}_{X})}\mathcal{M})\otimes_{\mathcal{D}_{X}^{(0,1)}}\mathcal{D}_{X}^{(0,1)}
\]
which is immediate. 
\end{proof}
Finally, we note the following compatibility of tensor product and
pull-back, which follows directly from unpacking the definitions. 
\begin{lem}
\label{lem:Tensor-and-pull}Let $\varphi:\mathfrak{X}\to\mathfrak{Y}$
be a morphism. Then there is a canonical isomorphism $L\varphi^{*}(\mathcal{M}^{\cdot}\widehat{\otimes}_{D(\mathcal{O}_{\mathfrak{Y}})}^{L}\mathcal{N}^{\cdot})\tilde{\to}L\varphi^{*}(\mathcal{M}^{\cdot})\widehat{\otimes}_{D(\mathcal{O}_{\mathfrak{X}})}^{L}L\varphi^{*}(\mathcal{N}^{\cdot})$.
The analogous statement holds for a morphism of smooth $k$-schemes
$\varphi:X\to Y$. 
\end{lem}

\section{\label{sec:Push-Forward}Operations on Gauges: Push-Forward}

As above let $\varphi:\mathfrak{X}\to\mathfrak{Y}$. Now that we have
both the pull-back and the left-right swap, we can define the push-forward. 

We start by noting that $\mathcal{\widehat{D}}_{\mathfrak{Y}}^{(0,1)}$
carries a natural right module structure over itself (by right multiplication).
Therefore, by \propref{Left-Right-Swap} there is a natural left $\mathcal{\widehat{D}}_{\mathfrak{Y}}^{(0,1)}$
gauge structure on $\mathcal{\widehat{D}}_{\mathfrak{Y}}^{(0,1)}\otimes\omega_{\mathfrak{Y}}^{-1}$.
By \defref{Pullback!} there is a natural left $\mathcal{\widehat{D}}_{\mathfrak{X}}^{(0,1)}$-module
structure on $\varphi^{*}(\mathcal{\widehat{D}}_{\mathfrak{Y}}^{(0,1)}\otimes\omega_{\mathfrak{Y}}^{-1})=L\varphi^{*}(\mathcal{\widehat{D}}_{\mathfrak{X}}^{(0,1)}\otimes\omega_{\mathfrak{Y}}^{-1})$. 
\begin{defn}
\label{def:Push!}1) Define the $(\varphi^{-1}(\mathcal{\widehat{D}}_{\mathfrak{Y}}^{(0,1)}),\mathcal{\widehat{D}}_{\mathfrak{X}}^{(0,1)})$
bimodule $\mathcal{\widehat{D}}_{\mathfrak{Y}\leftarrow\mathfrak{X}}^{(0,1)}:=\varphi^{*}(\mathcal{\widehat{D}}_{\mathfrak{Y}}^{(0,1)}\otimes\omega_{\mathfrak{Y}}^{-1})\otimes\omega_{\mathfrak{X}}$;
here, the right $\mathcal{\widehat{D}}_{\mathfrak{X}}^{(0,1)}$-module
structure comes from the left $\mathcal{\widehat{D}}_{\mathfrak{X}}^{(0,1)}$-module
structure on $\varphi^{*}(\mathcal{\widehat{D}}_{\mathfrak{Y}}^{(0,1)}\otimes\omega_{\mathfrak{Y}}^{-1})$;
the left $\varphi^{-1}(\mathcal{\widehat{D}}_{\mathfrak{Y}}^{(0,1)})$-structure
comes from the left multiplication of $\varphi^{-1}(\mathcal{\widehat{D}}_{\mathfrak{Y}}^{(0,1)})$
on $\varphi^{*}(\mathcal{\widehat{D}}_{\mathfrak{Y}}^{(0,1)}\otimes\omega_{\mathfrak{Y}}^{-1})$. 

2) Let $\mathcal{M}^{\cdot}\in D(\mathcal{G}(\mathcal{\widehat{D}}_{\mathfrak{X}}^{(0,1)}))$.
Then we define ${\displaystyle \int_{\varphi}\mathcal{M}^{\cdot}:=R\varphi_{*}(\mathcal{\widehat{D}}_{\mathfrak{Y}\leftarrow\mathfrak{X}}^{(0,1)}\widehat{\otimes}_{\mathcal{\widehat{D}}_{\mathfrak{X}}^{(0,1)}}^{L}\mathcal{M}^{\cdot})}\in D(\mathcal{G}(\mathcal{\widehat{D}}_{\mathfrak{Y}}^{(0,1)}))$. 

3) If we instead have $\varphi:X\to Y$ over $k$; then for $\mathcal{M}^{\cdot}\in D(\mathcal{G}(\mathcal{D}_{X}^{(0,1)}))$
we define ${\displaystyle \int_{\varphi}\mathcal{M}^{\cdot}:=R\varphi_{*}(\mathcal{D}_{Y\leftarrow X}^{(0,1)}\otimes_{\mathcal{D}_{X}^{(0,1)}}^{L}\mathcal{M}^{\cdot})}\in D(\mathcal{G}(\mathcal{D}_{Y}^{(0,1)}))$
where $\mathcal{D}_{Y\leftarrow X}^{(0,1)}$ is defined analogously
to $\mathcal{\widehat{D}}_{\mathfrak{Y}\leftarrow\mathfrak{X}}^{(0,1)}$. 

4) if $\mathfrak{Y}=\text{Specf}(W(k))$, then we denote ${\displaystyle \mathbb{H}_{\mathcal{G}}^{\cdot}(\mathcal{M}^{\cdot}):=\int_{\varphi}\mathcal{M}^{\cdot}}$
for any $\mathcal{M}^{\cdot}\in D_{cc}(\mathcal{G}(\mathcal{\widehat{D}}_{\mathfrak{X}}^{(0,1)}))$. 

Similarly, there are push-forwards in the category of right $\mathcal{\widehat{D}}^{(0,1)}$-modules
defined by ${\displaystyle \int_{\varphi}\mathcal{M}_{r}^{\cdot}:=R\varphi_{*}(\mathcal{M}_{r}^{\cdot}\widehat{\otimes}_{\mathcal{\widehat{D}}_{\mathfrak{X}}^{(0,1)}}^{L}\mathcal{\widehat{D}}_{\mathfrak{X}\to\mathfrak{Y}}^{(0,1)})}$
for $\mathcal{M}_{r}^{\cdot}\in D(\mathcal{G}(\mathcal{\widehat{D}}_{\mathfrak{X}}^{(0,1)}))^{\text{op}}$;
clearly the left-right interchange intertwines the two pushforwards.
Similar remarks apply to a morphism $\varphi:X\to Y$ over $k$. 
\end{defn}

We begin by recording some basic compatibilities; for these note that
we have the transfer bimodule $\mathcal{\widehat{D}}_{\mathfrak{Y}\leftarrow\mathfrak{X}}^{(0)}:=\mathcal{\widehat{D}}_{\mathfrak{Y}\leftarrow\mathfrak{X}}^{(0,1)}/(v-1)$
in the category of $\widehat{\mathcal{D}}_{\mathfrak{X}}^{(0)}$-modules,
and $\mathcal{\widehat{D}}_{\mathfrak{Y}\leftarrow\mathfrak{X}}^{(1)}:=(\mathcal{\widehat{D}}_{\mathfrak{Y}\leftarrow\mathfrak{X}}^{(0,1)}/(f-1))^{\widehat{}}$
(here the $()^{\widehat{}}$ denotes $p$-adic completion, which is
the same as cohomological completion in this case by \propref{Basic-properties-of-the-transfer-module}).
One may therefore define ${\displaystyle \int_{\varphi,0}\mathcal{M}^{\cdot}:=R\varphi_{*}(\mathcal{\widehat{D}}_{\mathfrak{Y}\leftarrow\mathfrak{X}}^{(0)}\widehat{\otimes}_{\mathcal{\widehat{D}}_{\mathfrak{X}}^{(0)}}^{L}\mathcal{M}^{\cdot})}$
for $\mathcal{M}^{\cdot}\in\widehat{\mathcal{D}}_{\mathfrak{X}}^{(0)}$
and ${\displaystyle \int_{\varphi,1}\mathcal{M}^{\cdot}:=R\varphi_{*}(\mathcal{\widehat{D}}_{\mathfrak{Y}\leftarrow\mathfrak{X}}^{(1)}\widehat{\otimes}_{\mathcal{\widehat{D}}_{\mathfrak{X}}^{(0)}}^{L}\mathcal{M}^{\cdot})}$
for $\mathcal{M}^{\cdot}\in\widehat{\mathcal{D}}_{\mathfrak{X}}^{(1)}$.
As in the case of the pullback, this is not quite Berthelot's definition
of these functors; because he uses the more traditional $\text{R}\lim$.
However, they do agree in important cases, such as when $\varphi$
is proper and $\mathcal{M}^{\cdot}$ is coherent. 

We have
\begin{prop}
\label{prop:push-and-complete-for-D} Let $\mathcal{M}^{\cdot}\in D(\mathcal{G}(\mathcal{\widehat{D}}_{\mathfrak{X}}^{(0,1)}))$. 

1) ${\displaystyle (\int_{\varphi}\mathcal{M}^{\cdot})\otimes_{W(k)}^{L}k\tilde{=}\int_{\varphi}(\mathcal{M}^{\cdot}\otimes_{W(k)}^{L}k)}$
in the category $D(\mathcal{G}(\mathcal{D}_{X}^{(0,1)}))$. 

2) $(\int_{\varphi}\mathcal{M}^{\cdot})^{-\infty}\tilde{=}(\int_{\varphi,0}\mathcal{M}^{\cdot,-\infty})$
where the pushforward on the right is defined as $R\varphi_{*}(\mathcal{\widehat{D}}_{\mathfrak{Y}\leftarrow\mathfrak{X}}^{(0)}\widehat{\otimes}_{\mathcal{\widehat{D}}_{\mathfrak{X}}^{(0)}}^{L}\mathcal{M}^{\cdot,-\infty})$.

3) If $\mathcal{M}^{\cdot}\in D_{coh}^{b}(\mathcal{G}(\mathcal{\widehat{D}}_{\mathfrak{X}}^{(0,1)}))$,
then $\widehat{((\int_{\varphi}\mathcal{M}^{\cdot})^{\infty})}\tilde{=}\int_{\varphi,1}\widehat{(\mathcal{M}^{\cdot,\infty})}$
where both uses of $\widehat{}$ denote derived completion.
\end{prop}

\begin{proof}
1) We have 
\[
\int_{\varphi}\mathcal{M}^{\cdot}\otimes_{W(k)}^{L}k=R\varphi_{*}(\mathcal{\widehat{D}}_{\mathfrak{Y}\leftarrow\mathfrak{X}}^{(0,1)}\widehat{\otimes}_{\mathcal{\widehat{D}}_{\mathfrak{X}}^{(0,1)}}^{L}\mathcal{M}^{\cdot})\otimes_{W(k)}^{L}k
\]
\[
\tilde{=}R\varphi_{*}((\mathcal{\widehat{D}}_{\mathfrak{Y}\leftarrow\mathfrak{X}}^{(0,1)}\widehat{\otimes}_{\mathcal{\widehat{D}}_{\mathfrak{X}}^{(0,1)}}^{L}\mathcal{M}^{\cdot})\otimes_{W(k)}^{L}k)
\]
(since $k$ is a perfect complex over $W(k)$, this is a special case
of the projection formula where we consider $X$ and $Y$ as ringed
spaces with the locally constant sheaf of rings $W(k)$; c.f. {[}Stacks{]},
tag 0B54). We have the isomorphism 
\[
k\otimes_{W(k)}^{L}\mathcal{\widehat{D}}_{\mathfrak{Y}\leftarrow\mathfrak{X}}^{(0,1)}\tilde{=}\mathcal{D}_{Y\leftarrow X}^{(0,1)}
\]
since $\mathcal{\widehat{D}}_{\mathfrak{Y}\leftarrow\mathfrak{X}}^{(0,1)}$
is a $p$-torsion-free sheaf; and so
\[
R\varphi_{*}((\mathcal{\widehat{D}}_{\mathfrak{Y}\leftarrow\mathfrak{X}}^{(0,1)}\widehat{\otimes}_{\mathcal{\widehat{D}}_{\mathfrak{X}}^{(0,1)}}^{L}\mathcal{M}^{\cdot})\otimes_{W(k)}^{L}k)\tilde{=}R\varphi_{*}(k\otimes_{W(k)}^{L}(\mathcal{\widehat{D}}_{\mathfrak{Y}\leftarrow\mathfrak{X}}^{(0,1)}\widehat{\otimes}_{\mathcal{\widehat{D}}_{\mathfrak{X}}^{(0,1)}}^{L}\mathcal{M}^{\cdot}))
\]
\[
R\varphi_{*}(k\otimes_{W(k)}^{L}(\mathcal{\widehat{D}}_{\mathfrak{Y}\leftarrow\mathfrak{X}}^{(0,1)}\otimes_{\mathcal{\widehat{D}}_{\mathfrak{X}}^{(0,1)}}^{L}\mathcal{M}^{\cdot}))\tilde{=}R\varphi_{*}(\mathcal{D}_{Y\leftarrow X}^{(0,1)}\otimes_{\mathcal{\widehat{D}}_{\mathfrak{X}}^{(0,1)}}^{L}\mathcal{M}^{\cdot})
\]
where we used that for any complex $\mathcal{N}^{\cdot}$ we have
$k\otimes_{W(k)}^{L}\mathcal{N}^{\cdot}\tilde{=}k\otimes_{W(k)}^{L}\widehat{\mathcal{N}^{\cdot}}$
(c.f. \lemref{reduction-of-completion}). Now, we have
\[
R\varphi_{*}(\mathcal{D}_{Y\leftarrow X}^{(0,1)}\otimes_{\mathcal{\widehat{D}}_{\mathfrak{X}}^{(0,1)}}^{L}\mathcal{M}^{\cdot})\tilde{=}\int_{\varphi}\mathcal{D}_{Y\leftarrow X}^{(0,1)}\otimes_{\mathcal{D}_{X}^{(0,1)}}^{L}(\mathcal{D}_{X}^{(0,1)}\otimes_{\mathcal{\widehat{D}}_{\mathfrak{X}}^{(0,1)}}^{L}\mathcal{M}^{\cdot})
\]
But since $\mathcal{D}_{X}^{(0,1)}=\mathcal{\widehat{D}}_{\mathfrak{X}}^{(0,1)}/p$
we have $\mathcal{D}_{X}^{(0,1)}\otimes_{\mathcal{\widehat{D}}_{\mathfrak{X}}^{(0,1)}}^{L}\mathcal{M}^{\cdot}\tilde{=}\mathcal{M}^{\cdot}\otimes_{W(k)}^{L}k$
and the result follows.

2) For any complex we have $\mathcal{M}^{\cdot,-\infty}=\mathcal{M}^{\cdot}\otimes_{D(W(k))}^{L}(D(W(k))/(v-1))$.
Thus the proof is an easier variant of that of $1)$, replacing $\otimes_{W(k)}^{L}k$
with $\otimes_{D(W(k))}^{L}D(W(k))/(v-1)$. 

3) We have 
\[
\mathcal{K}^{\cdot}\to\mathcal{\widehat{D}}_{\mathfrak{Y}\leftarrow\mathfrak{X}}^{(0,1),\infty}\to\mathcal{\widehat{D}}_{\mathfrak{Y}\leftarrow\mathfrak{X}}^{(1)}
\]
where $\mathcal{K}^{\cdot}$ is a complex of $\widehat{\mathcal{D}}_{\mathfrak{X}}^{(0,1),\infty}[p^{-1}]$-modules;
indeed, $\mathcal{\widehat{D}}_{\mathfrak{Y}\leftarrow\mathfrak{X}}^{(0,1),\infty}$
is a $p$-torsion-free sheaf whose completion is exactly $\mathcal{\widehat{D}}_{\mathfrak{Y}\leftarrow\mathfrak{X}}^{(1)}$
(c.f. \propref{Basic-properties-of-the-transfer-module} and \lemref{Basic-Structure-of-D^(1)}).Thus
there is a distinguished triangle 
\[
\mathcal{K}^{\cdot}\otimes_{\mathcal{\widehat{D}}_{\mathfrak{X}}^{(0,1),\infty}}^{L}\mathcal{M}^{\cdot,\infty}\to\mathcal{\widehat{D}}_{\mathfrak{Y}\leftarrow\mathfrak{X}}^{(0,1),\infty}\otimes_{\mathcal{\widehat{D}}_{\mathfrak{X}}^{(0,1),\infty}}^{L}\mathcal{M}^{\cdot,\infty}\to\mathcal{\widehat{D}}_{\mathfrak{Y}\leftarrow\mathfrak{X}}^{(1)}\otimes_{\mathcal{\widehat{D}}_{\mathfrak{X}}^{(0,1),\infty}}^{L}\mathcal{M}^{\cdot,\infty}
\]
and the term on the left is a complex of $\widehat{\mathcal{D}}_{\mathfrak{X}}^{(0,1),\infty}[p^{-1}]$-modules.
Thus the derived completion of $\mathcal{\widehat{D}}_{\mathfrak{Y}\leftarrow\mathfrak{X}}^{(0,1),\infty}\otimes_{\mathcal{\widehat{D}}_{\mathfrak{X}}^{(0,1),\infty}}^{L}\mathcal{M}^{\cdot,\infty}$
is isomorphic to the derived completion of $\mathcal{\widehat{D}}_{\mathfrak{Y}\leftarrow\mathfrak{X}}^{(1)}\otimes_{\mathcal{\widehat{D}}_{\mathfrak{X}}^{(0,1),\infty}}^{L}\mathcal{M}^{\cdot,\infty}$. 

Further, we have 
\[
\mathcal{\widehat{D}}_{\mathfrak{Y}\leftarrow\mathfrak{X}}^{(1)}\otimes_{\mathcal{\widehat{D}}_{\mathfrak{X}}^{(0,1),\infty}}^{L}\mathcal{M}^{\cdot,\infty}\tilde{=}\mathcal{\widehat{D}}_{\mathfrak{Y}\leftarrow\mathfrak{X}}^{(1)}\otimes_{\widehat{\mathcal{D}}_{\mathfrak{X}}^{(1)}}^{L}(\widehat{\mathcal{D}}_{\mathfrak{X}}^{(1)}\otimes_{\mathcal{\widehat{D}}_{\mathfrak{X}}^{(0,1),\infty}}^{L}\mathcal{M}^{\cdot,\infty})
\]
And, since $\mathcal{M}^{\cdot}\in D_{coh}^{b}(\mathcal{G}(\mathcal{\widehat{D}}_{\mathfrak{X}}^{(0,1)}))$,
we have (by \propref{Completion-for-noeth}) that $\widehat{\mathcal{M}^{\cdot,\infty}}\tilde{=}\widehat{\mathcal{D}}_{\mathfrak{X}}^{(1)}\otimes_{\mathcal{\widehat{D}}_{\mathfrak{X}}^{(0,1),\infty}}^{L}\mathcal{M}^{\cdot,\infty}$
as modules over $\widehat{\mathcal{D}}_{\mathfrak{X}}^{(1)}$. Therefore
we obtain 
\[
\mathcal{\widehat{D}}_{\mathfrak{Y}\leftarrow\mathfrak{X}}^{(0,1),\infty}\widehat{\otimes}_{\mathcal{\widehat{D}}_{\mathfrak{X}}^{(0,1),\infty}}^{L}\mathcal{M}^{\cdot,\infty}\tilde{=}\mathcal{\widehat{D}}_{\mathfrak{Y}\leftarrow\mathfrak{X}}^{(1)}\widehat{\otimes}_{\widehat{\mathcal{D}}_{\mathfrak{X}}^{(1)}}^{L}\widehat{(\mathcal{M}^{\cdot,\infty})}
\]
and so, taking $R\varphi_{*}$ yields 
\[
R\varphi_{*}(\mathcal{\widehat{D}}_{\mathfrak{Y}\leftarrow\mathfrak{X}}^{(0,1),\infty}\widehat{\otimes}_{\mathcal{\widehat{D}}_{\mathfrak{X}}^{(0,1),\infty}}^{L}\mathcal{M}^{\cdot,\infty})\tilde{\to}\int_{\varphi,1}\widehat{(\mathcal{M}^{\cdot,\infty})}
\]
But the term on the left is isomorphic to the derived completion of
${\displaystyle \int_{\varphi}\mathcal{M}^{\cdot,\infty}}$ by \propref{Push-and-complete}. 
\end{proof}
Now we will discuss the relationship between the $\mathcal{D}_{X}^{(0,1)}$
pushforward and the push-forwards over $\mathcal{R}(\mathcal{D}_{X}^{(1)})$
and $\overline{\mathcal{R}}(\mathcal{D}_{X}^{(0)})$. As usual we'll
work with the functors $\text{\ensuremath{\mathcal{M}}}^{\cdot}\to\mathcal{R}(\mathcal{D}_{X}^{(1)})\otimes_{\mathcal{D}_{X}^{(0,1)}}^{L}\mathcal{M}^{\cdot}\tilde{=}k[f]\otimes_{D(k)}^{L}\mathcal{M}^{\cdot}$
and $\text{\ensuremath{\mathcal{M}}}^{\cdot}\to\mathcal{\overline{R}}(\mathcal{D}_{X}^{(0)})\otimes_{\mathcal{D}_{X}^{(0,1)}}^{L}\mathcal{M}^{\cdot}\tilde{=}k[v]\otimes_{D(k)}^{L}\mathcal{M}^{\cdot}$
which take $D(\mathcal{G}(\mathcal{D}_{X}^{(0,1)}))$ to $D(\mathcal{G}(\mathcal{R}(\mathcal{D}_{X}^{(1)})))$
and $D(\mathcal{G}(\overline{\mathcal{R}}(\mathcal{D}_{X}^{(0)})))$,
respectively (as in \propref{Quasi-rigid=00003Dfinite-homological}). 

Both of the algebras $\overline{\mathcal{R}}(\mathcal{D}_{X}^{(0)})$
and $\mathcal{R}(\mathcal{D}_{X}^{(0)})$ possess transfer bimodules
associated to any morphism $\varphi:X\to Y$, and hence are equipped
with a push-pull formalism. In the case of $\mathcal{R}(\mathcal{D}_{X}^{(0)})$
this is well known (c.f., e.g \cite{key-22}, chapter $1$), while
in the case of $\overline{\mathcal{R}}(\mathcal{D}_{X}^{(0)})$ this
theory is developed in \cite{key-11}, in the language of filtered
derived categories. We shall proceed using the push-pull formalism
for $\mathcal{D}_{X}^{(0,1)}$-modules that we have already developed,
and discuss the relations with the other theories in section \subsecref{Hodge-and-Conjugate}
below.
\begin{defn}
Let $\varphi:X\to Y$ be a morphism. We define a $(\varphi^{-1}\mathcal{R}(\mathcal{D}_{Y}^{(1)}),\mathcal{R}(\mathcal{D}_{X}^{(1)}))$
bimodule $\mathcal{R}_{Y\leftarrow X}^{(1)}:=\mathcal{D}_{Y\leftarrow X}^{(0,1)}/v$.
Define a $(\varphi^{-1}\mathcal{\overline{R}}(\mathcal{D}_{Y}^{(1)}),\overline{\mathcal{R}}(\mathcal{D}_{X}^{(0)}))$
bimodule $\mathcal{R}_{Y\leftarrow X}^{(1)}:=\mathcal{D}_{Y\leftarrow X}^{(0,1)}/f$.
Define ${\displaystyle \int_{\varphi,1}}\mathcal{M}^{\cdot}=R\varphi_{*}(\mathcal{R}_{Y\leftarrow X}^{(1)}\otimes_{\mathcal{R}(\mathcal{D}_{X}^{(1)})}^{L}\mathcal{M}^{\cdot})$
on the category $\mathcal{G}(\mathcal{R}(\mathcal{D}_{X}^{(1)}))$,
and analogously ${\displaystyle \int_{\varphi,0}}$ for $\mathcal{\overline{R}}(\mathcal{D}_{X}^{(0)})$-modules.
As above, there is also a push-forward for right modules defined by
${\displaystyle \int_{\varphi,1}}\mathcal{M}_{r}^{\cdot}=R\varphi_{*}(\mathcal{M}_{r}^{\cdot}\otimes_{\mathcal{R}(\mathcal{D}_{X}^{(1)})}^{L}\mathcal{R}_{X\to Y}^{(1)})$,
and analogously for right $\mathcal{\overline{R}}(\mathcal{D}_{X}^{(0)})$-modules. 

We have the basic compatibility:
\end{defn}

\begin{prop}
If $\mathcal{M}^{\cdot}\in D_{qcoh}(\mathcal{G}(\mathcal{D}_{X}^{(0,1)}))$,
then we have 
\[
{\displaystyle \mathcal{R}(\mathcal{D}_{Y}^{(1)})\otimes_{\mathcal{D}_{Y}^{(0,1)}}^{L}\int_{\varphi}\mathcal{M}^{\cdot}\tilde{=}\int_{\varphi,1}(\mathcal{R}(\mathcal{D}_{X}^{(1)})\otimes_{\mathcal{D}_{X}^{(0,1)}}^{L}\mathcal{M}^{\cdot})}
\]
The analogous result holds for $\mathcal{\overline{R}}(\mathcal{D}_{X}^{(0)})$. 
\end{prop}

\begin{proof}
We have 
\[
\mathcal{R}(\mathcal{D}_{Y}^{(1)})\otimes_{\mathcal{D}_{Y}^{(0,1)}}^{L}\int_{\varphi}\mathcal{M}^{\cdot}=\mathcal{R}(\mathcal{D}_{Y}^{(1)})\otimes_{\mathcal{D}_{Y}^{(0,1)}}^{L}R\varphi_{*}(\mathcal{D}_{Y\leftarrow X}^{(0,1)}\otimes_{\mathcal{D}_{X}^{(0,1)}}^{L}\mathcal{M}^{\cdot})
\]
\[
\tilde{=}R\varphi_{*}(\varphi^{-1}(\mathcal{R}(\mathcal{D}_{Y}^{(1)}))\otimes_{\varphi^{-1}(\mathcal{D}_{Y}^{(0,1)})}^{L}(\mathcal{D}_{Y\leftarrow X}^{(0,1)}\otimes_{\mathcal{D}_{X}^{(0,1)}}^{L}\mathcal{M}^{\cdot}))
\]
(we will prove this last isomorphism in the lemma directly below).
We have the isomorphism 
\[
\varphi^{-1}(\mathcal{R}(\mathcal{D}_{Y}^{(1)}))\otimes_{\varphi^{-1}(\mathcal{D}_{Y}^{(0,1)})}^{L}\mathcal{D}_{Y\leftarrow X}^{(0,1)}\tilde{=}\mathcal{R}_{Y\leftarrow X}^{(1)}
\]
which is proved in the same way as \eqref{transfer-iso-1} above.
Therefore 
\[
R\varphi_{*}(\varphi^{-1}(\mathcal{R}(\mathcal{D}_{Y}^{(1)}))\otimes_{\varphi^{-1}(\mathcal{D}_{Y}^{(0,1)})}^{L}(\mathcal{D}_{Y\leftarrow X}^{(0,1)}\otimes_{\mathcal{D}_{X}^{(0,1)}}^{L}\mathcal{M}^{\cdot}))\tilde{=}R\varphi_{*}(\mathcal{R}_{Y\leftarrow X}^{(1)}\otimes_{\mathcal{D}_{X}^{(0,1)}}^{L}\mathcal{M}^{\cdot})
\]
\[
\tilde{=}R\varphi_{*}(\mathcal{R}_{Y\leftarrow X}^{(1)}\otimes_{\mathcal{R}(\mathcal{D}_{X}^{(1)})}^{L}\mathcal{R}(\mathcal{D}_{X}^{(1)})\otimes_{\mathcal{D}_{X}^{(0,1)}}^{L}\mathcal{M}^{\cdot})=\int_{\varphi,1}(\mathcal{R}(\mathcal{D}_{X}^{(1)})\otimes_{\mathcal{D}_{X}^{(0,1)}}^{L}\mathcal{M}^{\cdot})
\]
as claimed. The proof for the case of $\overline{\mathcal{R}}(\mathcal{D}_{X}^{(0)})$
is essentially identical. 
\end{proof}
In the previous proof we used the 
\begin{lem}
\label{lem:baby-projection-1}Let $\mathcal{M}^{\cdot}\in D_{qcoh}(\mathcal{G}(\mathcal{D}_{X}^{(0,1)}))$,
and $\mathcal{N}^{\cdot}\in D_{qcoh}(\mathcal{G}(\mathcal{D}_{Y}^{(0,1)})^{\text{opp}})$.
Then there is an isomorphism 
\[
\mathcal{N}^{\cdot}\otimes_{\mathcal{D}_{Y}^{(0,1)}}^{L}R\varphi_{*}(\mathcal{D}_{Y\leftarrow X}^{(0,1)}\otimes_{\mathcal{D}_{X}^{(0,1)}}^{L}\mathcal{M}^{\cdot})\tilde{=}R\varphi_{*}(\varphi^{-1}(\mathcal{N}^{\cdot})\otimes_{\varphi^{-1}(\mathcal{D}_{Y}^{(0,1)})}^{L}(\mathcal{D}_{Y\leftarrow X}^{(0,1)}\otimes_{\mathcal{D}_{X}^{(0,1)}}^{L}\mathcal{M}^{\cdot}))
\]
\end{lem}

\begin{proof}
(c.f. the proof of \cite{key-17}, proposition 5.3). First, we construct
a canonical map 
\[
\mathcal{N}^{\cdot}\otimes_{\mathcal{D}_{Y}^{(0,1)}}^{L}R\varphi_{*}(\mathcal{D}_{Y\leftarrow X}^{(0,1)}\otimes_{\mathcal{D}_{X}^{(0,1)}}^{L}\mathcal{M}^{\cdot})\to R\varphi_{*}(\varphi^{-1}(\mathcal{N}^{\cdot})\otimes_{\varphi^{-1}(\mathcal{D}_{Y}^{(0,1)})}^{L}(\mathcal{D}_{Y\leftarrow X}^{(0,1)}\otimes_{\mathcal{D}_{X}^{(0,1)}}^{L}\mathcal{M}^{\cdot}))
\]
as follows: one may replace $\mathcal{N}^{\cdot}$ with a complex
of $K$-flat graded $\mathcal{D}_{Y}^{(0,1)}$-modules, $\mathcal{F}^{\cdot}$.
Choosing a quasi-isomorphism $\mathcal{D}_{Y\leftarrow X}^{(0,1)}\otimes_{\mathcal{D}_{X}^{(0,1)}}^{L}\mathcal{M}^{\cdot}\tilde{\to}\mathcal{I}^{\cdot}$,
a $K$-injective complex of graded $\varphi^{-1}(\mathcal{D}_{Y}^{(0,1)})$-modules,
one obtains the quasi-isomorphism 
\[
\mathcal{N}^{\cdot}\otimes_{\mathcal{D}_{Y}^{(0,1)}}^{L}R\varphi_{*}(\mathcal{D}_{Y\leftarrow X}^{(0,1)}\otimes_{\mathcal{D}_{X}^{(0,1)}}^{L}\mathcal{M}^{\cdot})\tilde{\to}\mathcal{F}^{\cdot}\otimes_{\mathcal{D}_{Y}^{(0,1)}}\varphi_{*}\mathcal{I}^{\cdot}
\]
Then there is the obvious isomorphism 
\[
\mathcal{F}^{\cdot}\otimes_{\mathcal{D}_{Y}^{(0,1)}}\varphi_{*}\mathcal{I}^{\cdot}\tilde{\to}\varphi_{*}(\varphi^{-1}(\mathcal{F}^{\cdot})\otimes_{\varphi^{-1}(\mathcal{D}_{Y}^{(0,1)})}\mathcal{I}^{\cdot})
\]
and a canonical map 
\[
\varphi_{*}(\varphi^{-1}(\mathcal{F}^{\cdot})\otimes_{\varphi^{-1}(\mathcal{D}_{Y}^{(0,1)})}\mathcal{I}^{\cdot})\to R\varphi_{*}((\varphi^{-1}(\mathcal{F}^{\cdot})\otimes_{\varphi^{-1}(\mathcal{D}_{Y}^{(0,1)})}\mathcal{I}^{\cdot}))
\]
\[
\tilde{\to}R\varphi_{*}(\varphi^{-1}(\mathcal{N}^{\cdot})\otimes_{\varphi^{-1}(\mathcal{D}_{Y}^{(0,1)})}^{L}(\mathcal{D}_{Y\leftarrow X}^{(0,1)}\otimes_{\mathcal{D}_{X}^{(0,1)}}^{L}\mathcal{M}^{\cdot}))
\]
Thus we obtain the canonical map 
\[
\mathcal{N}^{\cdot}\otimes_{\mathcal{D}_{Y}^{(0,1)}}^{L}R\varphi_{*}(\mathcal{D}_{Y\leftarrow X}^{(0,1)}\otimes_{\mathcal{D}_{X}^{(0,1)}}^{L}\mathcal{M}^{\cdot})\to R\varphi_{*}(\varphi^{-1}(\mathcal{N}^{\cdot})\otimes_{\varphi^{-1}(\mathcal{D}_{Y}^{(0,1)})}^{L}(\mathcal{D}_{Y\leftarrow X}^{(0,1)}\otimes_{\mathcal{D}_{X}^{(0,1)}}^{L}\mathcal{M}^{\cdot}))
\]
this map exists for all $\mathcal{N}^{\cdot}\in D(\mathcal{G}(\mathcal{D}_{Y}^{(0,1)})^{\text{op}})$
and $\mathcal{M}^{\cdot}\in D(\mathcal{G}(\mathcal{D}_{X}^{(0,1)}))$.
To check whether it is an isomorphism, we may work locally on $Y$
and suppose that $Y$ is affine from now on.

To prove this, we proceed in a similar manner to the proof of the
projection formula for quasi-coherent sheaves, in the general version
of \cite{key-17}, proposition 5.3. Fix $\mathcal{M}^{\cdot}\in D_{qcoh}(\mathcal{G}(\mathcal{D}_{X}^{(0,1)}))$.
For any $\mathcal{N}^{\cdot}\in D_{qcoh}(\mathcal{G}(\mathcal{D}_{Y}^{(0,1)})^{\text{opp}})$,
we claim that $\varphi^{-1}(\mathcal{N}^{\cdot})\otimes_{\varphi^{-1}(\mathcal{D}_{Y}^{(0,1)})}^{L}(\mathcal{D}_{Y\leftarrow X}^{(0,1)}\otimes_{\mathcal{D}_{X}^{(0,1)}}^{L}\mathcal{M}^{\cdot})$
is quasi-isomorphic to a complex in $D_{qcoh}(D(\mathcal{O}_{X}))$.
To see this, we observe that any quasicoherent $\mathcal{D}_{X}^{(0,1)}$-module
$\mathcal{M}$ is a quotient of the $\mathcal{D}_{X}^{(0,1)}$-module
$\mathcal{D}_{X}^{(0,1)}\otimes_{D(\mathcal{O}_{X})}\mathcal{M}$
(where the $\mathcal{D}_{X}^{(0,1)}$-module is via the action on
the left hand factor on the tensor product). It follows that any bounded-above
complex in $D_{qcoh}(\mathcal{G}(\mathcal{D}_{X}^{(0,1)}))$ is quasi-isomorphic
to a complex, whose terms are of the form $\mathcal{D}_{X}^{(0,1)}\otimes_{D(\mathcal{O}_{X})}\mathcal{M}$
for quasi-coherent $\mathcal{M}$. Therefore any complex in $D_{qcoh}(\mathcal{G}(\mathcal{D}_{X}^{(0,1)}))$
is a homotopy colimit of such complexes. Therefore $\mathcal{D}_{Y\leftarrow X}^{(0,1)}\otimes_{\mathcal{D}_{X}^{(0,1)}}^{L}\mathcal{M}^{\cdot}$
is quasi-isomorphic to a complex of quasicoherent $D(\mathcal{O}_{X})$-modules.
In addition, since $Y$ is affine, $\mathcal{N}^{\cdot}$ is quasi-isomorphic
to a $K$-projective complex of $\mathcal{D}_{Y}^{(0,1)}$-modules;
in particular, a complex whose terms are projective $\mathcal{D}_{Y}^{(0,1)}$-modules.
It follows that $\varphi^{-1}(\mathcal{N}^{\cdot})\otimes_{\varphi^{-1}(\mathcal{D}_{Y}^{(0,1)})}^{L}(\mathcal{D}_{Y\leftarrow X}^{(0,1)}\otimes_{\mathcal{D}_{X}^{(0,1)}}^{L}\mathcal{M}^{\cdot})$
is quasi-isomorphic to a complex in $D_{qcoh}(D(\mathcal{O}_{X}))$
as claimed. 

Now, since $R\varphi_{*}$ commutes with arbitrary direct sums on
$D_{qcoh}(D(\mathcal{O}_{X}))$ (by \cite{key-17}, lemma 1.4), we
see that both sides of arrow commute with arbitrary direct sums (over
objects in $D_{qcoh}(\mathcal{G}(\mathcal{D}_{Y}^{(0,1)})^{\text{opp}})$);
so the set of objects on which the arrow is an isomorphism is closed
under arbitrary direct sums. Since $Y$ is affine, the category $D_{qcoh}(\mathcal{G}(\mathcal{D}_{Y}^{(0,1)})^{\text{op}})$
is generated by the compact objects $\{\mathcal{D}_{Y}^{(0,1)}[i]\}_{i\in\mathbb{Z}}$;
therefore (as in the proof of \cite{key-17}, lemma 5.3), it actually
suffices to show that the arrow is an isomorphism on $\mathcal{D}_{Y}^{(0,1)}$
itself, but this is obvious. 
\end{proof}
This type of projection formula is so useful that we will record here
a minor variant: 
\begin{lem}
\label{lem:proj-over-D}Let $\varphi:\mathfrak{X}\to\mathfrak{Y}$
be a morphism. Let $\mathcal{M}^{\cdot}\in D_{cc}(\mathcal{G}(\widehat{\mathcal{D}}_{\mathfrak{Y}}^{(0,1)}))$
and $\mathcal{N}^{\cdot}\in D_{cc}(\mathcal{G}(\widehat{\mathcal{D}}_{\mathfrak{X}}^{(0,1),\text{opp}}))$,
such that $\mathcal{M}^{\cdot}\otimes_{W(k)}^{L}k\in D_{qcoh}(\mathcal{G}(\mathcal{D}_{Y}^{(0,1)}))$
and $\mathcal{N}^{\cdot}\otimes_{W(k)}^{L}k\in D_{qcoh}(\mathcal{G}(\mathcal{D}_{X}^{(0,1)})^{\text{opp}})$.
Then we have 
\[
(\int_{\varphi}\mathcal{N}^{\cdot})\widehat{\otimes}_{\widehat{\mathcal{D}}_{\mathfrak{Y}}^{(0,1)}}^{L}\mathcal{M}^{\cdot}\tilde{\to}R\varphi_{*}(\mathcal{N}^{\cdot}\widehat{\otimes}_{\widehat{\mathcal{D}}_{\mathfrak{X}}^{(0,1)}}^{L}L\varphi^{*}\mathcal{M}^{\cdot})
\]
The analogous statement holds for $\mathcal{M}^{\cdot}\in D_{qcoh}(\mathcal{G}(\mathcal{D}_{Y}^{(0,1)}))$
and $\mathcal{N}^{\cdot}\in D_{coh}^{b}(\mathcal{G}(\mathcal{D}_{X}^{(0,1),\text{opp}}))$;
as well as for the Rees algebras $\mathcal{R}(\mathcal{D}^{(1)})$
and $\mathcal{\overline{R}}(\mathcal{D}_{X}^{(0)})$.
\end{lem}

\begin{proof}
We have that ${\displaystyle \int_{\varphi}\mathcal{N}^{\cdot}=R\varphi_{*}(\mathcal{N}^{\cdot}\widehat{\otimes}_{\widehat{\mathcal{D}}_{\mathfrak{X}}^{(0,1)}}^{L}\mathcal{\widehat{D}}_{\mathfrak{X}\to\mathfrak{Y}}^{(0,1)})}$.
As in the proof of \lemref{baby-projection-1}, there is a morphism
\begin{equation}
R\varphi_{*}(\mathcal{N}^{\cdot}\widehat{\otimes}_{\widehat{\mathcal{D}}_{\mathfrak{X}}^{(0,1)}}^{L}\mathcal{\widehat{D}}_{\mathfrak{X}\to\mathfrak{Y}}^{(0,1)})\widehat{\otimes}_{\widehat{\mathcal{D}}_{\mathfrak{Y}}^{(0,1)}}^{L}\mathcal{M}^{\cdot}\to R\varphi_{*}(\mathcal{N}^{\cdot}\widehat{\otimes}_{\widehat{\mathcal{D}}_{\mathfrak{X}}^{(0,1)}}^{L}\mathcal{\widehat{D}}_{\mathfrak{X}\to\mathfrak{Y}}^{(0,1)}\widehat{\otimes}_{\varphi^{-1}(\widehat{\mathcal{D}}_{\mathfrak{Y}}^{(0,1)})}^{L}\varphi^{-1}(\mathcal{M}^{\cdot}))\label{eq:adunction}
\end{equation}
Indeed, one constructs the map 
\[
R\varphi_{*}(\mathcal{N}^{\cdot}\otimes_{\widehat{\mathcal{D}}_{\mathfrak{X}}^{(0,1)}}^{L}\mathcal{\widehat{D}}_{\mathfrak{X}\to\mathfrak{Y}}^{(0,1)})\otimes_{\widehat{\mathcal{D}}_{\mathfrak{Y}}^{(0,1)}}^{L}\mathcal{M}^{\cdot}\to R\varphi_{*}(\mathcal{N}^{\cdot}\otimes_{\widehat{\mathcal{D}}_{\mathfrak{X}}^{(0,1)}}^{L}\mathcal{\widehat{D}}_{\mathfrak{X}\to\mathfrak{Y}}^{(0,1)}\otimes_{\varphi^{-1}(\widehat{\mathcal{D}}_{\mathfrak{Y}}^{(0,1)})}^{L}\varphi^{-1}(\mathcal{M}^{\cdot}))
\]
exactly as above; and then passes to the cohomological completion.

Since $L\varphi^{*}\mathcal{M}^{\cdot}=\mathcal{\widehat{D}}_{\mathfrak{X}\to\mathfrak{Y}}^{(0,1)}\widehat{\otimes}_{\varphi^{-1}(\widehat{\mathcal{D}}_{\mathfrak{Y}}^{(0,1)})}^{L}\varphi^{-1}(\mathcal{M}^{\cdot})$
by definition, the result will follow if \eqref{adunction} is an
isomorphism. To prove that, apply $\otimes_{W(k)}^{L}k$ and quote
the previous result. The proof in the case of the Rees algebras is
completely analogous. 
\end{proof}
Here is an important application of these ideas:
\begin{lem}
\label{lem:Composition-of-pushforwards}Let $\varphi:\mathfrak{X}\to\mathfrak{Y}$
and $\psi:\mathfrak{Y}\to\mathfrak{Z}$ be morphisms. There is a canonical
map 
\[
\int_{\psi}\circ\int_{\varphi}\mathcal{M}^{\cdot}\to\int_{\psi\circ\varphi}\mathcal{M}^{\cdot}
\]
for any $\mathcal{M}^{\cdot}\in D_{cc}(\mathcal{G}(\widehat{\mathcal{D}}_{\mathfrak{X}}^{(0,1)}))$,
which is an isomorphism if $\mathcal{M}^{\cdot}\otimes_{W(k)}^{L}k\in D_{qoh}(\mathcal{G}(\mathcal{D}_{X}^{(0,1)}))$.
If $\varphi:X\to Y$ and $\psi:Y\to Z$ are morphisms, we have the
analogous statements in $D(\mathcal{G}(\mathcal{D}_{Z}^{(0,1)}))$. 
\end{lem}

\begin{proof}
As in \lemref{composition-of-pullbacks}, we have an isomorphism 
\[
\varphi^{-1}(\mathcal{D}_{\mathfrak{Z\leftarrow\mathfrak{Y}}}^{(0,1)})\widehat{\otimes}_{\varphi^{-1}(\mathcal{D}_{\mathfrak{Y}}^{(0,1)})}^{L}\mathcal{\widehat{D}}_{\mathfrak{Y}\leftarrow\mathfrak{X}}^{(0,1)}\tilde{=}\mathcal{\widehat{D}}_{\mathfrak{Z}\leftarrow\mathfrak{X}}^{(0,1)}
\]
as $((\psi\circ\varphi)^{-1}(\widehat{\mathcal{D}}_{\mathfrak{Z}}^{(0,1)}),\widehat{\mathcal{D}}_{\mathfrak{X}}^{(0,1)})$
bimodules. Then we have 
\[
\int_{\psi}\circ\int_{\varphi}\mathcal{M}^{\cdot}=R\psi_{*}(\mathcal{D}_{\mathfrak{Z\leftarrow\mathfrak{Y}}}^{(0,1)}\widehat{\otimes}_{\widehat{\mathcal{D}}_{\mathfrak{Y}}^{(0,1)}}^{L}R\varphi_{*}(\mathcal{D}_{\mathfrak{Y\leftarrow\mathfrak{X}}}^{(0,1)}\widehat{\otimes}_{\widehat{\mathcal{D}}_{\mathfrak{X}}^{(0,1)}}^{L}\mathcal{M}^{\cdot}))
\]
\[
\to R\psi_{*}R\varphi_{*}(\varphi^{-1}(\mathcal{D}_{\mathfrak{Z\leftarrow\mathfrak{Y}}}^{(0,1)})\widehat{\otimes}_{\varphi^{-1}(\widehat{\mathcal{D}}_{\mathfrak{Y}}^{(0,1)})}^{L}\mathcal{D}_{\mathfrak{Y\leftarrow\mathfrak{X}}}^{(0,1)}\widehat{\otimes}_{\widehat{\mathcal{D}}_{\mathfrak{X}}^{(0,1)}}^{L}\mathcal{M}^{\cdot})
\]
\[
\tilde{\to}R(\psi\circ\varphi)_{*}(\mathcal{\widehat{D}}_{\mathfrak{Z}\leftarrow\mathfrak{X}}^{(0,1)}\widehat{\otimes}_{\widehat{\mathcal{D}}_{\mathfrak{X}}^{(0,1)}}^{L}\mathcal{M}^{\cdot})=\int_{\psi\circ\varphi}\mathcal{M}^{\cdot}
\]
where the first arrow is constructed as in \lemref{proj-over-D} and
the second isomorphism is given above. Applying the functor $\otimes_{W(k)}^{L}k$
and using \propref{Push-and-complete}, part $1)$, we reduce to proving
the analogous statement for $\varphi:X\to Y$ and $\psi:Y\to Z$;
where it follows exactly as in \lemref{baby-projection-1}. 
\end{proof}
We shall also need results relating the pushforwards when $\mathcal{M}\in\mathcal{G}(\mathcal{D}_{X}^{(0,1)})$
is already annihilated by $f$ (or $v$): 
\begin{prop}
\label{prop:Sandwich-push}Suppose $\mathcal{M}\in\mathcal{G}(\mathcal{D}_{X}^{(0,1)})$
satisfies $v\mathcal{M}=0$. Then ${\displaystyle \int_{\varphi}\mathcal{M}}$
is contained in the image of the functor $D(\mathcal{R}(\mathcal{D}_{X}^{(1)})-\text{mod})\to D(\mathcal{G}(\mathcal{D}_{X}^{(0,1)}))$.
In fact, there is an isomorphism of graded sheaves of $\mathcal{O}_{X}[f,v]$-modules
\[
R\varphi_{*}(\mathcal{R}_{Y\leftarrow X}^{(1)}\otimes_{\mathcal{R}(\mathcal{D}_{X}^{(1)})}^{L}\mathcal{M})\tilde{=}R\varphi_{*}(\mathcal{D}_{Y\leftarrow X}^{(0,1)}\otimes_{\mathcal{D}_{X}^{(0,1)}}^{L}\mathcal{M})
\]
In other words, the pushforward of $\mathcal{M}$, regarded as a module
over $\mathcal{R}(\mathcal{D}_{X}^{(1)})$, agrees with its pushforward
as a $\mathcal{D}_{X}^{(0,1)}$-module. The analogous result hold
when $f\mathcal{M}=0$. 
\end{prop}

\begin{proof}
This is an immediate consequence of \propref{Sandwich!}
\end{proof}
As a consequence of these results, we obtain: 
\begin{thm}
\label{thm:phi-push-is-bounded}Let $\varphi:X\to Y$ be a morphism.
Then, for each of the algebras $\mathcal{R}(\mathcal{D}_{X}^{(1)})$,
$\overline{\mathcal{R}}(\mathcal{D}_{X}^{(0)})$, and $\mathcal{D}_{X}^{(0,1)}$,
the pushforward along $\varphi$ takes $D_{qcoh}^{b}$ to $D_{qcoh}^{b}$.
If $\varphi$ is proper, then the pushforward along $\varphi$ takes
$D_{coh}^{b}$ to $D_{coh}^{b}$.
\end{thm}

\begin{proof}
Let us start with the statement, that the pushforward takes $D_{qcoh}$
to $D_{qcoh}$ in all of these cases. For this, we can argue as in
the proof of \lemref{baby-projection-1}: namely, one may assume $Y$
is affine, and then if $\mathcal{M}^{\cdot}\in D_{qcoh}(\mathcal{G}(\mathcal{D}_{X}^{(0,1)}))$,
we may replace $\mathcal{M}^{\cdot}$ by a homotopy colimit of $\mathcal{D}_{X}^{(0,1)}$-modules
of the form $\mathcal{D}_{X}^{(0,1)}\otimes_{\mathcal{O}_{X}[f,v]}\mathcal{M}$,
for quasi-coherent $\mathcal{M}$. Therefore $\mathcal{D}_{Y\leftarrow X}^{(0,1)}\otimes_{\mathcal{D}_{X}^{(0,1)}}^{L}\mathcal{M}^{\cdot}$
is quasi-isomorphic to a complex of quasicoherent $\mathcal{O}_{X}[f,v]$-modules,
which implies that the cohomology sheaves of its pushforward are quasi-coherent
$\mathcal{O}_{Y}[f,v]$-modules; and therefore quasi-coherent $\mathcal{D}_{Y}^{(0,1)}$-modules.
The same argument works for $\mathcal{R}(\mathcal{D}_{X}^{(1)})$
and $\overline{\mathcal{R}}(\mathcal{D}_{X}^{(0)})$. 

To prove the boundedness, we can factor $\varphi$ as a closed immersion
(the graph $X\to X\times Y$) followed by the projection $X\times Y\to Y$,
and, applying \lemref{Composition-of-pushforwards}, we see that it
suffices to consider separately the case of a closed immersion and
the case of a smooth morphism. For a closed immersion $\iota:X\to Y$
we have that the bimodule $\mathcal{D}_{X\to Y}^{(0)}$ is locally
free over $\mathcal{D}_{X}^{(0,1)}$ (this elementary fact will be
checked below in \lemref{transfer-is-locally-free}) and so the tensor
product $\otimes_{\mathcal{D}_{X}^{(0,1)}}\mathcal{D}_{X\to Y}^{(0)}$
takes quasicoherent sheaves to quasicoherent sheaves. 

Now, if $X\to Y$ is smooth, we have by (the proof of) \propref{Smooth-pullback-preserves-coh},
that $\mathcal{D}_{Y\leftarrow X}^{(0,1)}$ is a coherent $\mathcal{D}_{X}^{(0,1),\text{opp}}$-module.
Further, since it is locally the reduction mod $p$ of a standard
module, it is rigid, so that by \propref{Quasi-rigid=00003Dfinite-homological}
it is locally of finite homological dimension; and the result follows
directly. Thus we see that ${\displaystyle \int_{\varphi}}$ is bounded
on $D_{qcoh}(\mathcal{G}(\mathcal{D}_{X}^{(0,1)}))$, the same holds
for the pushforward on $D_{qcoh}(\mathcal{G}(\mathcal{R}(\mathcal{D}_{X}^{(1)})))$
and $D_{qcoh}(\mathcal{G}(\mathcal{\overline{R}}(\mathcal{D}_{X}^{(0)})))$
by the previous proposition. 

Now suppose $\varphi$ is proper. Let us say that a right $\mathcal{D}_{X}^{(0,1)}$-module
is induced if it is of the form $\mathcal{F}\otimes_{D(\mathcal{O}_{X})}\mathcal{D}_{X}^{(0,1)}$
for some coherent $\mathcal{F}$ over $D(\mathcal{O}_{X})$. In this
case we have 
\[
\int_{\varphi}\mathcal{F}\otimes_{D(\mathcal{O}_{X})}\mathcal{D}_{X}^{(0,1)}=R\varphi_{*}(\mathcal{F}\otimes_{D(\mathcal{O}_{X})}\mathcal{D}_{X}^{(0,1)}\otimes_{\mathcal{D}_{X}^{(0,1)}}^{L}\mathcal{D}_{X\to Y}^{(0,1)})
\]
\[
\tilde{=}R\varphi_{*}(\mathcal{F}\otimes_{D(\mathcal{O}_{X})}^{L}\mathcal{D}_{X}^{(0,1)}\otimes_{\mathcal{D}_{X}^{(0,1)}}^{L}\varphi^{*}(\mathcal{D}_{Y}^{(0,1)}))\tilde{\to}R\varphi_{*}(\mathcal{F}\otimes_{D(\mathcal{O}_{X})}^{L}\varphi^{*}(\mathcal{D}_{Y}^{(0,1)}))
\]
\[
\tilde{\to}R\varphi_{*}(\mathcal{F})\otimes_{D(O_{Y})}^{L}\mathcal{D}_{Y}^{(0,1)}
\]
Thus the result is true for any induced module. If $\mathcal{M}$
is an arbitrary coherent right $\mathcal{D}_{X}^{(0,1)}$-module,
then, as a quasicoherent sheaf over $\mathcal{O}_{X}[f,v]$, it is
the union of its $\mathcal{O}_{X}[f,v]$ coherent sub-sheaves. Selecting
such a subsheaf which generates $\mathcal{M}$ as a $\mathcal{D}_{X}^{(0,1)}$-module,
we obtain a short exact sequence 
\[
0\to\mathcal{K}\to\mathcal{F}\otimes_{D(\mathcal{O}_{X})}\mathcal{D}_{X}^{(0,1)}\to\mathcal{M}\to0
\]
where $\mathcal{K}$ is also coherent. Since the functor ${\displaystyle \int_{\varphi}}$
is concentrated in homological degrees $\leq d_{X/Y}$ for all coherent
$\mathcal{D}_{X}^{(0,1)}$-modules, we can now deduce the coherence
of ${\displaystyle \mathcal{H}^{i}(\int_{\varphi}\mathcal{M})}$ by
descending induction on $i$. This proves the result for $\mathcal{D}_{X}^{(0,1)}$-modules,
and we can deduce the result for $\mathcal{R}(\mathcal{D}_{X}^{(1)})$,
$\overline{\mathcal{R}}(\mathcal{D}_{X}^{(0)})$-modules by again
invoking \propref{Sandwich-push}. 
\end{proof}
From this and the formalism of cohomological completion (specifically,
\propref{coh-to-coh}), we deduce 
\begin{cor}
\label{cor:proper-push-over-W(k)}Let $\varphi:\mathfrak{X}\to\mathfrak{Y}$
be proper. Then ${\displaystyle \int_{\varphi}}$ takes $D_{coh}^{b}(\mathcal{G}(\widehat{\mathcal{D}}_{\mathfrak{X}}^{(0,1)}))$
to $D_{coh}^{b}(\mathcal{G}(\widehat{\mathcal{D}}_{\mathfrak{Y}}^{(0,1)}))$.
If $\mathcal{M}^{\cdot}\in D_{coh,F^{-1}}^{b}(\mathcal{G}(\widehat{\mathcal{D}}_{\mathfrak{X}}^{(0,1)}))$
then ${\displaystyle \int_{\varphi}\mathcal{M}}\in D_{coh,F^{-1}}^{b}(\mathcal{G}(\widehat{\mathcal{D}}_{\mathfrak{Y}}^{(0,1)}))$. 
\end{cor}

\begin{proof}
The first part follows immediately from the proceeding theorem by
applying $\otimes_{W(k)}^{L}k$. The second part follows from \propref{push-and-complete-for-D}
(part $3)$), as well as Berthelot's theorem that ${\displaystyle \int_{\varphi,0}F^{*}\tilde{\to}F^{*}\int_{\varphi,1}}$
(c.f. \cite{key-2}, section 3.4, and also \thmref{Hodge-Filtered-Push}
below). 
\end{proof}

\subsection{\label{subsec:Hodge-and-Conjugate}Push-forwards for $\mathcal{R}(\mathcal{D}_{X}^{(1)})$
and $\overline{\mathcal{R}}(\mathcal{D}_{X}^{(0)})$. }

In this section we take a close look at the theory over $k$. In particular,
we study the pushforwards of modules over $\mathcal{R}(\mathcal{D}_{X}^{(1)})$
and $\overline{\mathcal{R}}(\mathcal{D}_{X}^{(0)})$, and compare
them with more traditional filtered pushforwards found in the literature.
For $\mathcal{R}(\mathcal{D}_{X}^{(1)})$ and $\overline{\mathcal{R}}(\mathcal{D}_{X}^{(0)})$
modules themselves, we will construct the analogue of the relative
de Rham resolution. This will allow us to exhibit an adjunction between
${\displaystyle \int_{\varphi}}$ and $\varphi^{\dagger}$ when $\varphi$
is smooth. 

We begin with $\mathcal{R}(\mathcal{D}_{X}^{(1)})$, where we can
reduce everything to the more familiar situation of $\mathcal{R}(\mathcal{D}_{X}^{(0)})$-modules
using the fact that $\mathcal{R}(\mathcal{D}_{X}^{(1)})$ is Morita
equivalent to $\mathcal{R}(\mathcal{D}_{X}^{(0)})$ (c.f. \thmref{Filtered-Frobenius}). 

Let $\varphi:X\to Y$. Recall that Laumon constructed in \cite{key-19}
the push-forward in the filtered derived category of $\mathcal{D}_{X}^{(0)}$-modules
(with respect to the symbol filtration); essentially, his work upgrades
the bimodule $\mathcal{D}_{Y\leftarrow X}^{(0)}$ to a filtered $(\varphi^{-1}(\mathcal{D}_{Y}^{(0)}),\mathcal{D}_{X}^{(0)})$-bimodule
via 
\[
F_{i}(\mathcal{D}_{Y\leftarrow X}^{(0)}):=\varphi^{-1}(F_{i}(\mathcal{D}_{Y}^{(0)})\otimes_{\mathcal{O}_{Y}}\omega_{Y}^{-1})\otimes_{\varphi^{-1}(\mathcal{O}_{Y})}\omega_{X}
\]
(c.f. \cite{key-19}, formula 5.1.3); then one may define ${\displaystyle \int_{\varphi}}$
via the usual formula, but using the tensor product and push-forward
in the filtered derived categories. On the other hand, we can apply
the Rees construction to the above filtered bimodule to obtain $\mathcal{R}(\mathcal{D}_{Y\leftarrow X}^{(0)})$,
a graded $(\varphi^{-1}(\mathcal{R}(\mathcal{D}_{Y}^{(0)})),\mathcal{R}(\mathcal{D}_{X}^{(0)}))$
bimodule, which (again by the usual formula) yields a push-forward
functor ${\displaystyle \int_{\varphi}}:D(\mathcal{G}(\mathcal{R}(\mathcal{D}_{X}^{(0)})))\to D(\mathcal{G}(\mathcal{R}(\mathcal{D}_{Y}^{(0)})))$,
and we have the following evident compatibility: 
\begin{lem}
Let $\mathcal{M}^{\cdot}\in D((\mathcal{D}_{X}^{(0)},F)-\text{mod})$.
Then we have 
\[
\mathcal{R}(\int_{\varphi}\mathcal{M}^{\cdot})\tilde{\to}\int_{\varphi}\mathcal{R}(\mathcal{M}^{\cdot})
\]
In particular, the Hodge-to-deRham spectral sequence for ${\displaystyle \int_{\varphi}\mathcal{M}^{\cdot}}$
degenerates at $E_{1}$ iff each of the sheaves ${\displaystyle \mathcal{H}^{i}(\int_{\varphi}\mathcal{R}(\mathcal{M}^{\cdot}))}$
is torsion-free over the Rees parameter $f$. 
\end{lem}

Next, we relate this to the pull-back and push-forward for $\mathcal{R}(\mathcal{D}_{X}^{(1)})$
modules; starting with the analogous statement for pull-back: 
\begin{lem}
\label{lem:Hodge-Filtered-Pull}Let $\varphi:X\to Y$ and suppose
$\mathcal{M}^{\cdot}\in D(\mathcal{G}(\mathcal{R}(\mathcal{D}_{X}^{(0)})))$.
Then $L\varphi^{*}\circ F_{Y}^{*}\mathcal{M}^{\cdot}\tilde{=}F_{X}^{*}\circ L\varphi^{*}\mathcal{M}^{\cdot}$.
Here, the pullback on the left is in the category of $\mathcal{R}(\mathcal{D}^{(1)})$-modules,
while the pullback on the right is in the category of $\mathcal{R}(\mathcal{D}^{(0)})$-modules. 
\end{lem}

\begin{proof}
Since $\varphi\circ F_{X}=F_{Y}\circ\varphi$ we have 
\[
L\varphi^{*}\circ F_{Y}^{*}\mathcal{M}^{\cdot}\tilde{\to}F_{X}^{*}\circ L\varphi^{*}\mathcal{M}^{\cdot}
\]
as (graded) $\mathcal{O}_{X}$-modules; we need to check that this
map preserves the $\mathcal{R}(\mathcal{D}_{X}^{(1)})$-module structures
on both sides.This question is local, so we may suppose $X=\text{Spec}(B)$
and $Y=\text{Spec}(A)$ both posses local coordinates. Further, by
taking a K-flat resolution of $\mathcal{M}^{\cdot}$ we may suppose
that $\mathcal{M}^{\cdot}=\mathcal{M}$ is concentrated in a single
degree. Now, as an $\mathcal{R}(\mathcal{D}_{X}^{(1)})$-module, $F_{Y}^{*}\mathcal{M}$
possesses the structure of a connection with $p$-curvature $0$,
and so the induced connection on $\varphi^{*}F_{Y}^{*}\mathcal{M}$
also has $p$-curvature $0$; and the kernel of this connection is
equal to $(\varphi^{(1)})^{*}\mathcal{M}^{(1)}\subset\varphi^{*}F_{Y}^{*}\mathcal{M}$
(here $\mathcal{M}^{(1)}$ denotes $\sigma^{*}\mathcal{M}$ where
$\sigma:X^{(1)}\to X$ is the natural isomorphism of schemes). Note
that $\mathcal{M}^{(1)}$ possesses the action of $\mathcal{R}(\mathcal{D}_{Y^{(1)}}^{(0)})$
(c.f. \remref{The-inverse-to-F^*}). 

Let $\{\partial_{i}\}_{i=1}^{n}$ be coordinate derivations on $X$.
Then the action of $\partial_{i}^{[p]}$ on $\varphi^{*}F_{Y}^{*}\mathcal{M}$
is given (by \propref{pull-back-in-pos-char}) by first restricting
$\partial_{i}^{[p]}$ to a differential operator $\varphi^{-1}(\mathcal{O}_{Y})\to\mathcal{O}_{X}$,
writing the resulting operator as 
\[
\sum_{j=1}^{r}b_{j}^{p}\partial_{j}^{[p]}+\sum_{J}b_{J}\partial^{J}
\]
(where $\{\partial_{j}\}_{j=1}^{r}$ are coordinate derivations on
$Y$, and $b_{j},b_{J}\in B$) and then letting $\partial_{i}^{[p]}$
act as 
\[
\sum_{j=1}^{r}b_{j}^{p}\partial_{j}^{[p]}+\sum_{J}b_{J}\partial^{J}
\]
therefore, the action of $\partial_{i}^{[p]}$ preserves $(\varphi^{(1)})^{*}\mathcal{M}^{(1)}$
and it acts there as ${\displaystyle \partial_{i}^{[p]}(1\otimes m)=\sum_{j=1}^{r}b_{j}^{p}\cdot\partial_{j}^{[p]}(m)}$.
But the action of $\{\partial_{j}^{[p]}\}$ on $\mathcal{M}^{(1)}$
defines the action of $\mathcal{R}(\mathcal{D}_{X^{(1)}}^{(0)})$
on $\mathcal{M}^{(1)}$, and this formula simply defines the pullback
from $\mathcal{R}(\mathcal{D}_{Y^{(1)}}^{(0)})$ to $\mathcal{R}(\mathcal{D}_{X^{(1)}}^{(0)})$-modules;
in other words, $(\varphi^{(1)})^{*}\mathcal{M}^{(1)}=((\varphi)^{*}\mathcal{M})^{(1)}$
where $\varphi^{*}\mathcal{M}$ is the usual pullback of $\mathcal{R}(\mathcal{D}^{(0)})$-modules.
Thus we see that $\varphi^{*}F_{Y}^{*}\mathcal{M}=F_{X}^{*}((\varphi)^{*}\mathcal{M})$
as $\mathcal{R}(\mathcal{D}_{X}^{(1)})$-modules, as desired. 
\end{proof}
Now we discuss push-forward:
\begin{thm}
\label{thm:Hodge-Filtered-Push}Let $\mathcal{M}^{\cdot}$ be a complex
of graded $\mathcal{R}(\mathcal{D}_{X}^{(1)})$-modules, and via \thmref{Filtered-Frobenius}
write $\mathcal{M}^{\cdot}\tilde{=}F_{X}^{*}\mathcal{N}^{\cdot}$,
where $\mathcal{N}^{\cdot}$ is a complex of graded $\mathcal{R}(\mathcal{D}_{X}^{(0)})$-modules.
There is an isomorphism
\[
\int_{\varphi,1}\mathcal{M}^{\cdot}\tilde{\to}F_{X}^{*}\int_{\varphi}\mathcal{N}^{\cdot}
\]
where ${\displaystyle \int_{\varphi}\mathcal{N}}^{\cdot}$ is the
pushforward of $\mathcal{N}^{\cdot}$ over $\mathcal{R}(\mathcal{D}_{X}^{(0)})$. 
\end{thm}

\begin{proof}
(following \cite{key-2}, theoreme 3.4.4) By left-right interchange
it is equivalent to prove the right-handed version
\[
\int_{\varphi,1}F_{X}^{!}\mathcal{N}^{\cdot}\tilde{\to}F_{Y}^{!}\int_{\varphi}\mathcal{N}^{\cdot}
\]
for any $\mathcal{N}^{\cdot}\in D(\mathcal{G}(\mathcal{R}(\mathcal{D}_{X}^{(0)})^{\text{opp}}))$. 

We have 
\[
\int_{\varphi,1}F_{X}^{!}(\mathcal{N}^{\cdot})=R\varphi_{*}(F_{X}^{!}(\mathcal{N}^{\cdot})\otimes_{\mathcal{R}(\mathcal{D}_{X}^{(1)})}^{L}\mathcal{R}_{X\to Y}^{(1)})=R\varphi_{*}(\mathcal{N}^{\cdot}\otimes_{\mathcal{R}(\mathcal{D}_{X}^{(0)})}^{L}F_{X}^{!}(\mathcal{R}(\mathcal{D}_{X}^{(0)}))\otimes_{\mathcal{R}(\mathcal{D}_{X}^{(1)})}^{L}\mathcal{R}_{X\to Y}^{(1)})
\]
Now, recall 
\[
\mathcal{R}_{X\to Y}^{(1)}=\varphi^{*}\mathcal{R}(\mathcal{D}_{X}^{(1)})\tilde{=}\varphi^{*}F_{Y}^{*}F_{Y}^{!}\mathcal{R}(\mathcal{D}_{Y}^{(1)})\tilde{=}F_{X}^{*}\varphi^{*}F_{Y}^{!}\mathcal{R}(\mathcal{D}_{Y}^{(1)})
\]
where the second isomorphism is \corref{Filtered-right-Frob}, and
the third is by the lemma above; note that this isomorphism preserves
the natural right $\varphi^{-1}(\mathcal{R}(\mathcal{D}_{Y}^{(1)}))$
-module structures on both sides. It follows (c.f. \propref{F^*F^!},
part $2)$, and \corref{Filtered-right-Frob}) that 
\[
F_{X}^{!}(\mathcal{R}(\mathcal{D}_{X}^{(0)}))\otimes_{\mathcal{R}(\mathcal{D}_{X}^{(1)})}^{L}\mathcal{R}_{X\to Y}^{(1)}\tilde{=}\varphi^{*}F_{Y}^{!}\mathcal{R}(\mathcal{D}_{Y}^{(0)})
\]
(as $(\mathcal{R}(\mathcal{D}_{X}^{(1)}),\varphi^{-1}(\mathcal{R}(\mathcal{D}_{Y}^{(1)}))$
bimodules). Therefore 
\[
R\varphi_{*}(\mathcal{N}^{\cdot}\otimes_{\mathcal{R}(\mathcal{D}_{X}^{(0)})}^{L}F^{!}(\mathcal{R}(\mathcal{D}_{X}^{(0)}))\otimes_{\mathcal{R}(\mathcal{D}_{X}^{(1)})}^{L}\mathcal{R}_{X\to Y}^{(1)})\tilde{=}R\varphi_{*}(\mathcal{N}^{\cdot}\otimes_{\mathcal{R}(\mathcal{D}_{X}^{(0)})}^{L}\varphi^{*}F_{Y}^{!}\mathcal{R}(\mathcal{D}_{Y}^{(0)}))
\]
\[
\tilde{=}\int_{\varphi}\mathcal{N}^{\cdot}\otimes_{\mathcal{R}(\mathcal{D}_{Y}^{(0)})}^{L}F_{Y}^{!}\mathcal{R}(\mathcal{D}_{Y}^{(0)}))
\]
where the last line is \lemref{proj-over-D}. However, 
\[
\int_{\varphi}\mathcal{N}^{\cdot}\otimes_{\mathcal{R}(\mathcal{D}_{Y}^{(0)})}^{L}F_{Y}^{!}\mathcal{R}(\mathcal{D}_{Y}^{(0)}))=F_{Y}^{!}\int_{\varphi}\mathcal{N}^{\cdot}
\]
whence the result. 
\end{proof}
To exploit this result, we recall that the formalism of de Rham cohomology
applies to $\mathcal{R}(\mathcal{D}_{X}^{(0)})$: 
\begin{prop}
Let $\varphi:X\to Y$ be smooth of relative dimension $d$. Then the
induced connection $\nabla:\mathcal{D}_{X}^{(0)}\to\mathcal{D}_{X}^{(0)}\otimes\Omega_{X/Y}^{1}(1)$
is a morphism of filtered right $\mathcal{D}_{X}^{(0)}$-modules (with
respect to the symbol filtration; the symbol $(1)$ denotes a shift
in the filtration). The associated de Rham complex 
\[
\mathcal{D}_{X}^{(0)}\to\mathcal{D}_{X}^{(0)}\otimes\Omega_{X/Y}^{1}(1)\to\mathcal{D}_{X}^{(0)}\otimes\Omega_{X/Y}^{2}(2)\to\dots\to\mathcal{D}_{X}^{(0)}\otimes\Omega_{X/Y}^{d}(d)
\]
is exact except at the right most term, where the cokernel is $\mathcal{D}_{Y\leftarrow X}^{(0)}(d)$
(as a filtered module).

After applying the left-right swap and a shift in the filtration,
we obtain the Spencer complex
\[
\mathcal{D}_{X}^{(0)}\otimes\mathcal{T}_{X/Y}^{d}(-d)\to\mathcal{D}_{X}^{(0)}\otimes\mathcal{T}_{X/Y}^{d-1}(-d+1)\to\dots\to\mathcal{D}_{X}^{(0)}\otimes\mathcal{T}_{X/Y}(-1)\to\mathcal{D}_{X}^{(0)}
\]
of left filtered $\mathcal{D}_{X}^{(0)}$-modules, which is exact
except at the right most term, and the cokernel is $\mathcal{D}_{X\to Y}^{(0)}$
(as a filtered module). Applying the Rees functor yields a complex
\[
\mathcal{R}(\mathcal{D}_{X}^{(0)})\otimes\mathcal{T}_{X/Y}^{d}(-d)\to\mathcal{R}(\mathcal{D}_{X}^{(0)})\otimes\mathcal{T}_{X/Y}^{d-1}(-d+1)\to\dots\to\mathcal{R}(\mathcal{D}_{X}^{(0)})\otimes\mathcal{T}_{X/Y}(-1)\to\mathcal{R}(\mathcal{D}_{X}^{(0)})
\]
which is exact except at the right most term, and the cokernel is
$\mathcal{R}(\mathcal{D}_{X\to Y}^{(0)})$. 
\end{prop}

The proof of this is identical to that of the corresponding result
in characteristic zero (\cite{key-4}, proposition 4.2); one notes
that the associated graded is a Koszul resolution. Applying this resolution
in the definition of the filtered push-forward, one deduces 
\begin{cor}
Let $\varphi:X\to Y$ be smooth of relative dimension $d$. Let $\mathcal{M}$
be a filtered $\mathcal{D}_{X}^{(0)}$-module (with respect to the
symbol filtration). Then there is an isomorphism 
\[
\int_{\varphi}\mathcal{M}[-d]\tilde{=}R\varphi_{*}(\mathcal{M}(-d)\xrightarrow{\nabla}\mathcal{M}\otimes\Omega_{X/Y}^{1}(1-d)\xrightarrow{\nabla}\mathcal{M}\otimes\Omega_{X/Y}^{2}(2)\xrightarrow{\nabla}\dots\xrightarrow{\nabla}\mathcal{M}\otimes\Omega_{X/Y}^{d})
\]
in the filtered derived category of $\mathcal{O}_{Y}$-modules. 
\end{cor}

In fact, with a little more work, one can show that, for any $i$,
the $\mathcal{D}_{Y}^{(0)}$-module structure on the sheaf ${\displaystyle \mathcal{H}^{i}(\int_{\varphi}\mathcal{M})}$
is given by the Gauss-Manin connection (c.f., e.g. \cite{key-51},
proposition 1.4). Thus the push-forward for $\mathcal{R}(\mathcal{D}_{X}^{(1)})$-modules
is exactly the ``Frobenius pullback of Gauss-Manin.''

As another corollary, we have
\begin{cor}
\label{cor:sm-adunction-for-filtered-D}Let $\varphi:X\to Y$ be smooth
of relative dimension $d$. 

1) There is an isomorphism $R\underline{\mathcal{H}om}{}_{\mathcal{R}(\mathcal{D}_{X}^{(0})}(\mathcal{R}(\mathcal{D}_{X\to Y}^{(0)}),\mathcal{R}(\mathcal{D}_{X}^{(0)}))\tilde{=}\mathcal{R}(\mathcal{D}_{Y\leftarrow X}^{(0)})(d)[-d]$
as $(\varphi^{-1}(\mathcal{R}(\mathcal{D}_{Y}^{(0)})),\mathcal{R}(\mathcal{D}_{X}^{(0)}))$
bimodules.

2) There is an isomorphism of functors 
\[
R\varphi_{*}R\underline{\mathcal{H}om}_{\mathcal{R}(\mathcal{D}_{X}^{(0)})}(\varphi^{\dagger}\mathcal{N}^{\cdot},\mathcal{M}^{\cdot})\tilde{\to}R\underline{\mathcal{H}om}{}_{\mathcal{R}(\mathcal{D}_{Y}^{(0)})}(\mathcal{N}^{\cdot},\int_{\varphi}\mathcal{M}^{\cdot}(d))
\]
for any $\mathcal{N}^{\cdot}\in D(\mathcal{G}(\mathcal{R}(\mathcal{D}_{Y}^{(0)})))$
and any $\mathcal{M}^{\cdot}\in D(\mathcal{G}(\mathcal{R}(\mathcal{D}_{X}^{(0)})))$.
The analogous isomorphism holds for $\mathcal{R}(\mathcal{D}_{X}^{(1)})$-modules. 
\end{cor}

\begin{proof}
Part $1)$ follows directly from the previous proposition; compare
\cite{key-4}, propositions 4.2 and 4.19. Then $2)$ follows from
$1)$, as in \cite{key-4}, Theorem 4.40 (we'll give the argument
below in a slightly different context in \corref{smooth-adjunction})
Finally, the statement for $\mathcal{R}(\mathcal{D}_{X}^{(1)})$-modules
follows from the Frobenius descent (\lemref{Hodge-Filtered-Pull}
and \thmref{Hodge-Filtered-Push}). 
\end{proof}
Next we are going to give the analogue of these results for the push-forward
of $\overline{\mathcal{R}}(\mathcal{D}_{X}^{(0)})$-modules; and compare
with the constructions of \cite{key-11}, section 3.4. We start with
the analogues of the de Rham resolution and the adjunction for smooth
morphisms. Although $\overline{\mathcal{R}}(\mathcal{D}_{X}^{(0)})$
does possess a canonical flat connection, the resulting (relative)
de Rham complex is not a resolution of a transfer bimodule. Instead,
we consider the action of the center 
\[
\mathcal{Z}(\overline{\mathcal{R}}(\mathcal{D}_{X}^{(0)}))\tilde{=}\mathcal{O}_{T^{*}X^{(1)}}[v]
\]

The action map $\overline{\mathcal{R}}(\mathcal{D}_{X}^{(0)})\otimes_{\mathcal{O}_{X^{(1)}}}\mathcal{T}_{X^{(1)}}(-1)\to\overline{\mathcal{R}}(\mathcal{D}_{X}^{(0)})$
yields (by dualizing) a map 
\[
\Theta:\overline{\mathcal{R}}(\mathcal{D}_{X}^{(0)})\to\overline{\mathcal{R}}(\mathcal{D}_{X}^{(0)})\otimes_{\mathcal{O}_{X^{(1)}}}\Omega_{X^{(1)}}^{1}(1)
\]
which makes $\overline{\mathcal{R}}(\mathcal{D}_{X}^{(0)})$ into
a Higgs sheaf over $X^{(1)}$. In particular we have $\Theta\circ\Theta=0$
and so we can form the complex $\overline{\mathcal{R}}(\mathcal{D}_{X}^{(0)})\otimes_{\mathcal{O}_{X^{(1)}}}\Omega_{X^{(1)}}^{i}(i)$
with the differential induced from $\Theta$. In addition, we can
form the analogue of the Spencer complex, whose terms are $\overline{\mathcal{R}}(\mathcal{D}_{X}^{(0)})\otimes_{\mathcal{O}_{X^{(1)}}}\mathcal{T}_{X^{(1)}}^{i}(-i)$. 

Now let $\varphi:X\to Y$ be a smooth morphism of relative dimension
$d$. Let $X_{Y}^{(1)}\to Y$ be the base change of this morphism
over the absolute Frobenius on $Y$. Then we can perform the above
constructions for $\Omega_{X_{Y}^{(1)}/Y}^{i}$ instead of $\Omega_{X^{(1)}}^{i}$.
We have 
\begin{lem}
\label{lem:Koszul-Res-For-R-bar} The complex $\overline{\mathcal{R}}(\mathcal{D}_{X}^{(0)})\otimes_{\mathcal{O}_{X_{Y}^{(1)}}}\mathcal{T}_{X_{Y}^{(1)}/Y}^{i}(-i)$
is exact except at the right-most term. The image of the map 
\[
\overline{\mathcal{R}}(\mathcal{D}_{X}^{(0)})\otimes_{\mathcal{O}_{X_{Y}^{(1)}}}\mathcal{T}_{X_{Y}^{(1)}/Y}(-1)\to\overline{\mathcal{R}}(\mathcal{D}_{X}^{(0)})
\]
is the central ideal $\mathcal{J}$ generated by $\mathcal{T}_{X_{Y}^{(1)}/Y}\subset\mathcal{Z}(\overline{\mathcal{R}}(\mathcal{D}_{X}^{(0)}))$.
The cokernel of the map, $\overline{\mathcal{R}}(\mathcal{D}_{X}^{(0)})/\mathcal{J}$,
carries the structure of a right $\mathcal{D}_{X/Y}^{(0)}$-module,
of $p$-curvature zero; this action commutes with the natural left
$\overline{\mathcal{R}}(\mathcal{D}_{X}^{(0)})$-module structure
on $\overline{\mathcal{R}}(\mathcal{D}_{X}^{(0)})/\mathcal{J}$. The
cokernel of the associated map $(\overline{\mathcal{R}}(\mathcal{D}_{X}^{(0)})/\mathcal{J})\otimes_{\mathcal{O}_{X/Y}}\mathcal{T}_{X/Y}\to\overline{\mathcal{R}}(\mathcal{D}_{X}^{(0)})/\mathcal{J}$
is isomorphic to $\mathcal{\overline{R}}_{X\to Y}^{(0)}$. 

The complex $\overline{\mathcal{R}}(\mathcal{D}_{X}^{(0)})\otimes_{\mathcal{O}_{X_{Y}^{(1)}}}\Omega_{X_{Y}^{(1)}/Y}^{i}(i)$
is exact except at the right-most term. The cokernel of the map 
\[
\overline{\mathcal{R}}(\mathcal{D}_{X}^{(0)})\otimes_{\mathcal{O}_{X_{Y}^{(1)}}}\Omega_{X_{Y}^{(1)}/Y}^{d-1}(d-1)\to\overline{\mathcal{R}}(\mathcal{D}_{X}^{(0)})\otimes_{\mathcal{O}_{X_{Y}^{(1)}}}\Omega_{X_{Y}^{(1)}/Y}^{d}(d)
\]
denoted $\mathcal{K}_{X/Y}$, carries the structure of a left $\mathcal{D}_{X/Y}^{(0)}$-module,
of $p$-curvature zero; this action commutes with the natural right
$\overline{\mathcal{R}}(\mathcal{D}_{X}^{(0)})$-module structure
on $\mathcal{K}_{X/Y}$. The kernel of the associated connection on
$\mathcal{K}_{X/Y}$ is isomorphic to $\mathcal{\overline{R}}_{Y\leftarrow X}^{(0)}(d)$. 
\end{lem}

\begin{proof}
Choose local coordinates on $X$ for which $\text{Der}(\mathcal{O}_{X})$
is the free module on $\{\partial_{1},\dots\partial_{n}\}$ and $\text{Der}_{\mathcal{O}_{Y}}(\mathcal{O}_{X})=\{\partial_{n-d+1},\dots\partial_{n}\}$.
Then the complex under consideration is simply the Koszul complex
for the elements $\{\partial_{n-d+1}^{p},\dots,\partial_{n}^{p}\}$,
which proves the exactness statements. Furthermore, as the elements
$\{\partial_{n-d+1}^{p},\dots,\partial_{n}^{p}\}$ are central in
$\overline{\mathcal{R}}(\mathcal{D}_{X}^{(0)})$, we see that $\overline{\mathcal{R}}(\mathcal{D}_{X}^{(0)})/\mathcal{J}$
has the structure of a left and right $\mathcal{D}_{X}^{(0)}$-module
(we are here using the fact that $\mathcal{D}_{X}^{(0)}$ is the degree
$0$ part of $\overline{\mathcal{R}}(\mathcal{D}_{X}^{(0)})$). Now,
we have 
\[
\mathcal{\overline{R}}_{X\to Y}^{(0)}=\text{coker}(\overline{\mathcal{R}}(\mathcal{D}_{X}^{(0)})\otimes_{\mathcal{O}_{X_{Y}^{(1)}}}\mathcal{T}_{X/Y}(-1)\to\overline{\mathcal{R}}(\mathcal{D}_{X}^{(0)}))
\]
\[
=\text{coker}((\overline{\mathcal{R}}(\mathcal{D}_{X}^{(0)})/\mathcal{J})\otimes_{\mathcal{O}_{X/Y}}\mathcal{T}_{X/Y}\to\overline{\mathcal{R}}(\mathcal{D}_{X}^{(0)})/\mathcal{J})
\]
since $X\to Y$ is smooth. The second statement follows similarly.
\end{proof}
Now we can give the analogue of \corref{sm-adunction-for-filtered-D}.
It reads: 
\begin{cor}
\label{cor:sm-adjunction-for-R-bar}Let $\varphi:X\to Y$ be smooth
of relative dimension $d$. 

1) There is an isomorphism $R\underline{\mathcal{H}om}{}_{\mathcal{\overline{R}}(\mathcal{D}_{X}^{(0})}(\mathcal{\overline{R}}(\mathcal{D}_{X\to Y}^{(0)}),\mathcal{\overline{R}}(\mathcal{D}_{X}^{(0)}))\tilde{=}\mathcal{\overline{R}}_{Y\leftarrow X}^{(0)}(d)[-d]$
as $(\varphi^{-1}(\mathcal{\overline{R}}(\mathcal{D}_{Y}^{(0)})),\overline{\mathcal{R}}(\mathcal{D}_{X}^{(0)}))$
bimodules.

2) There is an isomorphism of functors 
\[
R\varphi_{*}R\underline{\mathcal{H}om}_{\overline{\mathcal{R}}(\mathcal{D}_{X}^{(0)})}(\varphi^{\dagger,(0)}\mathcal{N}^{\cdot},\mathcal{M}^{\cdot})\tilde{\to}R\underline{\mathcal{H}om}{}_{\mathcal{\overline{R}}(\mathcal{D}_{X}^{(0)})}(\mathcal{N}^{\cdot},\int_{\varphi,0}\mathcal{M}^{\cdot}(d))
\]
for any $\mathcal{N}^{\cdot}\in D(\mathcal{G}(\mathcal{R}(\mathcal{D}_{Y}^{(0)})))$
and any $\mathcal{M}^{\cdot}\in D(\mathcal{G}(\mathcal{R}(\mathcal{D}_{X}^{(0)})))$. 
\end{cor}

\begin{proof}
As in the proof of \corref{smooth-adjunction} below, $2)$ follows
formally from $1)$. To see $1)$, we note that for any $\mathcal{N}\in\mathcal{G}(\mathcal{\overline{R}}(\mathcal{D}_{X}^{(0)}))$
the complex $R\underline{\mathcal{H}om}{}_{\mathcal{\overline{R}}(\mathcal{D}_{X}^{(0})}(\overline{\mathcal{R}}(\mathcal{D}_{X}^{(0)})/\mathcal{J},\mathcal{N})$
can be considered a complex of left $\mathcal{D}_{X/Y}^{(0)}$-modules
with $p$-curvature $0$ (as $\overline{\mathcal{R}}(\mathcal{D}_{X}^{(0)})/\mathcal{J}$
is a right $\mathcal{D}_{X/Y}^{(0)}$-module of $p$-curvature zero,
and this action commutes with the left $\mathcal{\overline{R}}(\mathcal{D}_{X}^{(0})$-action).
As Cartier descent for $\mathcal{D}_{X/Y}^{(0)}$-modules of $p$-curvature
$0$ is an exact functor, applying the previous lemma we obtain
\[
R\underline{\mathcal{H}om}{}_{\mathcal{\overline{R}}(\mathcal{D}_{X}^{(0})}(\mathcal{\overline{R}}(\mathcal{D}_{X\to Y}^{(0)}),\mathcal{\overline{R}}(\mathcal{D}_{X}^{(0)}))\tilde{=}R\underline{\mathcal{H}om}{}_{\mathcal{\overline{R}}(\mathcal{D}_{X}^{(0})}(\overline{\mathcal{R}}(\mathcal{D}_{X}^{(0)})/\mathcal{J},\mathcal{\overline{R}}(\mathcal{D}_{X}^{(0)}))^{\nabla}
\]
\[
(\mathcal{K}_{X/Y}[-d])^{\nabla}=\mathcal{\overline{R}}_{Y\leftarrow X}^{(0)}(d)[-d]
\]
as desired. 
\end{proof}
Now we'll give the relation of our pushforward to the constructions
of \cite{key-11}, section 3.4. We recall that to any morphism $\varphi:X\to Y$
we may attach the diagram 
\[
T^{*}X\xleftarrow{d\varphi}X\times_{Y}T^{*}Y\xrightarrow{\pi}T^{*}Y
\]
and we use the same letters to denote the products of these morphisms
with $\mathbb{A}^{1}$; We have the following analogue of \cite{key-52},
proposition 3.7 (c.f. also \cite{key-11}, theorem 3.11)
\begin{lem}
\label{lem:Bez-Brav}There is an equivalence of graded Azumaya algebras
$(d\varphi^{(1)})^{*}\overline{\mathcal{R}}(\mathcal{D}_{X}^{(0)})\sim(\pi^{(1)})^{*}\overline{\mathcal{R}}(\mathcal{D}_{Y}^{(0)})$. 
\end{lem}

\begin{proof}
Consider the (graded) Azumaya algebra $\mathcal{A}:=(d\varphi^{(1)})^{*}\overline{\mathcal{R}}(\mathcal{D}_{X}^{(0)})\otimes_{\mathcal{O}_{(X\times_{Y}T^{*}Y)^{(1)}}[v]}(\pi^{(1)})^{*}\overline{\mathcal{R}}(\mathcal{D}_{Y}^{(0)})^{\text{opp}}$
on $(X\times_{Y}T^{*}Y)^{(1)}\times\mathbb{A}^{1}$. It is enough
to find a (graded) splitting module for $\mathcal{A}$; i.e., a graded
$\mathcal{A}$-module which is locally free of rank $p^{\text{dim}(X)+\text{dim}(Y)}$
over $\mathcal{O}_{(X\times_{Y}T^{*}Y)^{(1)}}[v]$. 

The graded $(\overline{\mathcal{R}}(\mathcal{D}_{X}^{(0)}),\varphi^{-1}(\overline{\mathcal{R}}(\mathcal{D}_{Y}^{(0)})))$
bimodule $\overline{\mathcal{R}}_{X\to Y}:=\varphi^{*}\overline{\mathcal{R}}(\mathcal{D}_{Y}^{(0)})$
inherits the structure of an $\mathcal{A}$-module; we claim it is
locally free over $\mathcal{O}_{(X\times_{Y}T^{*}Y)^{(1)}}[v]$ of
the correct rank. This can be checked after inverting $v$ and setting
$v=0$; upon inverting $v$ it becomes (via the isomorphisms $\overline{\mathcal{R}}(\mathcal{D}_{X}^{(0)})[v^{-1}]\tilde{=}\mathcal{D}_{X}^{(0)}[v,v^{-1}]$
and $\overline{\mathcal{R}}(\mathcal{D}_{Y}^{(0)})[v^{-1}]\tilde{=}\mathcal{D}_{Y}^{(0)}[v,v^{-1}]$)
a direct consequence of \cite{key-52}, proposition 3.7. After setting
$v=0$ we have $\overline{\mathcal{R}}(\mathcal{D}_{Y}^{(0)})/v=\text{gr}(\mathcal{D}_{Y}^{(0)})=\overline{\mathcal{D}}_{Y}^{(0)}\otimes_{\mathcal{O}_{Y^{(1)}}}\mathcal{O}_{T^{*}Y^{(1)}}$;
and similarly for $X$. Therefore 
\[
\varphi^{*}\overline{\mathcal{R}}(\mathcal{D}_{Y}^{(0)})/v\tilde{=}\mathcal{O}_{X}\otimes_{\varphi^{-1}(\mathcal{O}_{Y})}\varphi^{-1}(\overline{\mathcal{D}}_{Y}^{(0)}\otimes_{\mathcal{O}_{Y^{(1)}}}\mathcal{O}_{T^{*}Y^{(1)}})
\]
\[
\tilde{=}\mathcal{O}_{X}\otimes_{\varphi^{-1}(\mathcal{O}_{Y})}\varphi^{-1}(\overline{\mathcal{D}}_{Y}^{(0)})\otimes_{\varphi^{-1}(\mathcal{O}_{Y^{(1)}})}\varphi^{-1}(\mathcal{O}_{T^{*}Y^{(1)}})
\]
But $\varphi^{-1}(\overline{\mathcal{D}}_{Y}^{(0)})$ is locally free
of rank $p^{\text{dim}(Y)}$ over $\varphi^{-1}(\mathcal{O}_{Y})$,
and $\mathcal{O}_{X}$ is locally free of rank $p^{\text{dim}(X)}$
over $\mathcal{O}_{X^{(1)}}$; so $\varphi^{*}\overline{\mathcal{R}}(\mathcal{D}_{Y}^{(0)})/v$
is locally free of rank $p^{\text{dim}(X)+\text{dim}(Y)}$ over $\mathcal{O}_{(X\times_{Y}T^{*}Y)^{(1)}}$
as claimed. 
\end{proof}
Next, we have the following straightforward:
\begin{lem}
Let $\varphi:X\to Y$ be smooth. Then $d\varphi$ is a closed immersion,
and we may regard the algebra $(d\varphi^{(1)})^{*}\overline{\mathcal{R}}(\mathcal{D}_{X}^{(0)})$
as a (graded) central quotient of $\overline{\mathcal{R}}(\mathcal{D}_{X}^{(0)})$.
The obvious functor $(d\varphi)_{*}:D(\mathcal{G}((d\varphi^{(1)})^{*}\overline{\mathcal{R}}(\mathcal{D}_{X}^{(0)})))\to D(\mathcal{G}(\overline{\mathcal{R}}(\mathcal{D}_{X}^{(0)})))$
admits a right adjoint $(d\varphi)^{!}$ defined by $\mathcal{M}^{\cdot}\to R\underline{\mathcal{H}om}_{\overline{\mathcal{R}}(\mathcal{D}_{X}^{(0)})}((d\varphi^{(1)})^{*}\overline{\mathcal{R}}(\mathcal{D}_{X}^{(0)}),\mathcal{M}^{\cdot})$. 
\end{lem}

Therefore we obtain 
\begin{cor}
\label{cor:Filtered-Bez-Brav}Let $C:D(\mathcal{G}((d\varphi^{(1)})^{*}\overline{\mathcal{R}}(\mathcal{D}_{X}^{(0)})))\to D(\mathcal{G}((\pi^{(1)})^{*}\overline{\mathcal{R}}(\mathcal{D}_{Y}^{(0)})))$
denote the equivalence of categories resulting from \lemref{Bez-Brav}.
Then, when $\varphi:X\to Y$ is smooth of relative dimension $d$,
there is an isomorphism of functors 
\[
\int_{\varphi,0}\tilde{\to}R\pi_{*}^{(1)}\circ C\circ(d\varphi^{(1)})^{!}[-d]:D(\mathcal{G}(\overline{\mathcal{R}}(\mathcal{D}_{X}^{(0)})))\to D(\mathcal{G}(\overline{\mathcal{R}}(\mathcal{D}_{Y}^{(0)})))
\]
Therefore, the functor ${\displaystyle \int_{\varphi,0}}[d]$ agrees,
under the application of the Rees functor, with the pushforward of
conjugate-filtered derived categories constructed in \cite{key-11},
section 3.4. 
\end{cor}

\begin{proof}
(in the spirit of \cite{key-11}, proposition 3.12) We have, for any
$\mathcal{M}^{\cdot}\in D(\mathcal{G}(\overline{\mathcal{R}}(\mathcal{D}_{X}^{(0)})))$,
\[
C\circ(d\varphi)^{!}(\mathcal{M}^{\cdot})=C\circ R\underline{\mathcal{H}om}_{\overline{\mathcal{R}}(\mathcal{D}_{X}^{(0)})}((d\varphi^{(1)})^{*}\overline{\mathcal{R}}(\mathcal{D}_{X}^{(0)}),\mathcal{M}^{\cdot})
\]
\[
\tilde{=}\underline{\mathcal{H}om}_{(d\varphi^{(1)})^{*}\overline{\mathcal{R}}(\mathcal{D}_{X}^{(0)})}(\varphi^{*}\overline{\mathcal{R}}(\mathcal{D}_{Y}^{(0)}),R\underline{\mathcal{H}om}_{\overline{\mathcal{R}}(\mathcal{D}_{X}^{(0)})}((d\varphi^{(1)})^{*}\overline{\mathcal{R}}(\mathcal{D}_{X}^{(0)}),\mathcal{M}^{\cdot}))
\]
Since $\varphi^{*}\overline{\mathcal{R}}(\mathcal{D}_{Y}^{(0)})$
is locally projective over $(d\varphi^{(1)})^{*}\overline{\mathcal{R}}(\mathcal{D}_{X}^{(0)})$,
this is canonically isomorphic to 
\[
R\underline{\mathcal{H}om}_{\overline{\mathcal{R}}(\mathcal{D}_{X}^{(0)})}(\varphi^{*}\overline{\mathcal{R}}(\mathcal{D}_{Y}^{(0)})\otimes_{(d\varphi^{(1)})^{*}\overline{\mathcal{R}}(\mathcal{D}_{X}^{(0)})}(d\varphi^{(1)})^{*}\overline{\mathcal{R}}(\mathcal{D}_{X}^{(0)}),\mathcal{M}^{\cdot}))
\]
\[
=R\underline{\mathcal{H}om}_{\overline{\mathcal{R}}(\mathcal{D}_{X}^{(0)})}(\varphi^{*}\overline{\mathcal{R}}(\mathcal{D}_{Y}^{(0)}),\mathcal{M}^{\cdot}))
\]
so that 
\[
R\pi_{*}^{(1)}\circ C\circ(d\varphi^{(1)})^{!}(\mathcal{M}^{\cdot})\tilde{\to}R\pi_{*}^{(1)}R\underline{\mathcal{H}om}_{\overline{\mathcal{R}}(\mathcal{D}_{X}^{(0)})}(\varphi^{*}\overline{\mathcal{R}}(\mathcal{D}_{Y}^{(0)}),\mathcal{M}^{\cdot}))
\]
\[
\tilde{=}R\varphi_{*}R\underline{\mathcal{H}om}_{\overline{\mathcal{R}}(\mathcal{D}_{X}^{(0)})}(\varphi^{*}\overline{\mathcal{R}}(\mathcal{D}_{Y}^{(0)}),\mathcal{M}^{\cdot}))
\]
But the right-hand functor is canonically isomorphic to ${\displaystyle \int_{\varphi,0}}$
by smooth adjunction (\corref{sm-adjunction-for-R-bar}). 
\end{proof}
From this description it follows directly (c.f. \cite{key-11}, lemma
3.18) that (up to a renumbering) the spectral sequence associated
to the filtration on ${\displaystyle \int_{\varphi}\mathcal{O}_{X}}$
agrees with the usual conjugate spectral sequence; i.e., the ``second
spectral sequence'' for $R\varphi_{dR,*}(\mathcal{O}_{X})$ as discussed
in \cite{key-12}. 

\subsection{Adjunction for a smooth morphism, base change, and the projection
formula}

In this section, we prove adjunction for a smooth morphism $\varphi:\mathfrak{X}\to\mathfrak{Y}$
and the projection formula for an arbitrary morphism; as consequences
we obtain the smooth base change and the and the Kunneth formula,
in fairly general contexts. To start off, let us recall: 
\begin{prop}
For a smooth morphism $\varphi:\mathfrak{X}\to\mathfrak{Y}$ there
is an isomorphism of sheaves $\mathcal{R}\mathcal{H}om_{\widehat{\mathcal{D}}_{\mathfrak{X}}^{(0)}}(\widehat{\mathcal{D}}_{\mathfrak{X}\to\mathfrak{Y}}^{(0)},\widehat{\mathcal{D}}_{\mathfrak{X}}^{(0)})\tilde{=}\widehat{\mathcal{D}}_{\mathfrak{Y}\leftarrow\mathfrak{X}}^{(0)}[-d_{X/Y}]$. 
\end{prop}

This is proved identically to the analogous fact for $\mathcal{D}_{X}^{(0)}$
and $\mathcal{R}(\mathcal{D}_{X}^{(0)})$-modules, as discussed above
in \corref{sm-adunction-for-filtered-D}. 
\begin{prop}
For a smooth morphism $\varphi:\mathfrak{X}\to\mathfrak{Y}$ of relative
dimension $d$, there is an isomorphism $\mathcal{R}\underline{\mathcal{H}om}{}_{\widehat{\mathcal{D}}_{\mathfrak{X}}^{(0,1)}}(\widehat{\mathcal{D}}_{\mathfrak{X}\to\mathfrak{Y}}^{(0,1)},\widehat{\mathcal{D}}_{\mathfrak{X}}^{(0,1)})\tilde{=}\widehat{\mathcal{D}}_{\mathfrak{Y}\leftarrow\mathfrak{X}}^{(0,1)}(d)[-d]$
\end{prop}

\begin{proof}
We have 
\[
\mathcal{R}\underline{\mathcal{H}om}{}_{\widehat{\mathcal{D}}_{\mathfrak{X}}^{(0,1)}}(\widehat{\mathcal{D}}_{\mathfrak{X}\to\mathfrak{Y}}^{(0,1)},\widehat{\mathcal{D}}_{\mathfrak{X}}^{(0,1)})\otimes_{\widehat{\mathcal{D}}_{\mathfrak{X}}^{(0,1)}}^{L}\widehat{\mathcal{D}}_{\mathfrak{X}}^{(0)}
\]
\[
\tilde{=}\mathcal{R}\underline{\mathcal{H}om}{}_{\widehat{\mathcal{D}}_{\mathfrak{X}}^{(0,1)}}(\widehat{\mathcal{D}}_{\mathfrak{X}\to\mathfrak{Y}}^{(0,1)},\widehat{\mathcal{D}}_{\mathfrak{X}}^{(0,1)})\otimes_{D(W(k))}^{L}(W(k)[f,v]/(v-1)
\]
\[
\tilde{=}\mathcal{R}\underline{\mathcal{H}om}{}_{\widehat{\mathcal{D}}_{\mathfrak{X}}^{(0,1)}}(\widehat{\mathcal{D}}_{\mathfrak{X}\to\mathfrak{Y}}^{(0,1)}\otimes_{D(W(k))}^{L}(W(k)[f,v]/(v-1),\widehat{\mathcal{D}}_{\mathfrak{X}}^{(0,1)}\otimes_{D(W(k))}^{L}(W(k)[f,v]/(v-1))
\]
\[
\tilde{=}R\mathcal{H}om_{\widehat{\mathcal{D}}_{\mathfrak{X}}^{(0)}}(\widehat{\mathcal{D}}_{\mathfrak{X}\to\mathfrak{Y}}^{(0)},\widehat{\mathcal{D}}_{\mathfrak{X}}^{(0)})
\]
and 
\[
\mathcal{R}\underline{\mathcal{H}om}{}_{\widehat{\mathcal{D}}_{\mathfrak{X}}^{(0,1)}}(\widehat{\mathcal{D}}_{\mathfrak{X}\to\mathfrak{Y}}^{(0,1)},\widehat{\mathcal{D}}_{\mathfrak{X}}^{(0,1)})\otimes_{\mathcal{\widehat{D}}_{\mathfrak{X}}^{(0,1)}}^{L}\mathcal{D}_{X}^{(0,1)}\tilde{=}\mathcal{R}\underline{\mathcal{H}om}{}_{\widehat{\mathcal{D}}_{\mathfrak{X}}^{(0,1)}}(\widehat{\mathcal{D}}_{\mathfrak{X}\to\mathfrak{Y}}^{(0,1)},\widehat{\mathcal{D}}_{\mathfrak{X}}^{(0,1)})\otimes_{W(k)}^{L}k
\]
\[
\tilde{=}R\mathcal{H}om_{\mathcal{D}_{X}^{(0,1)}}(\mathcal{D}_{X\to Y}^{(0,1)},\mathcal{D}_{X}^{(0,1)})
\]
By the same token, we have

\[
R\underline{\mathcal{H}om}{}_{\mathcal{D}_{X}^{(0,1)}}(\mathcal{D}_{X\to Y}^{(0,1)},\mathcal{D}_{X}^{(0,1)})\otimes_{\mathcal{D}_{X}^{(0,1)}}^{L}\mathcal{R}(\mathcal{D}_{X}^{(1)})\tilde{=}R\underline{\mathcal{H}om}{}_{\mathcal{R}(\mathcal{D}_{X}^{(1)})}(\mathcal{R}_{X\to Y}^{(1)},\mathcal{R}(\mathcal{D}_{X}^{(1)}))
\]
and the analogous statement for $\overline{\mathcal{R}}(\mathcal{D}_{X}^{(0)})$.
Applying the smooth adjunction for $\mathcal{R}(\mathcal{D}_{X}^{(1)})$-modules
(\corref{sm-adunction-for-filtered-D}, part $2)$) to the case where
$\mathcal{N}^{\cdot}=\mathcal{R}(\mathcal{D}_{Y}^{(1)})$ and $\mathcal{M}^{\cdot}=\mathcal{R}(\mathcal{D}_{X}^{(1)})$,
we have an isomorphism 
\[
R\underline{\mathcal{H}om}{}_{\mathcal{R}_{X}}(\mathcal{R}_{X\to Y}^{(1)},\mathcal{R}(\mathcal{D}_{X}^{(1)}))\tilde{=}\mathcal{R}_{Y\leftarrow X}^{(1)}(d)[-d]
\]
and by \corref{sm-adjunction-for-R-bar} we have 
\[
R\underline{\mathcal{H}om}{}_{\mathcal{\overline{R}}(\mathcal{D}_{X}^{(0})}(\mathcal{\overline{R}}(\mathcal{D}_{X\to Y}^{(0)}),\mathcal{\overline{R}}(\mathcal{D}_{X}^{(0)}))\tilde{=}\mathcal{\overline{R}}_{Y\leftarrow X}^{(0)}(d)[-d]
\]
Furthermore, using the relative de Rham resolution for $\widehat{\mathcal{D}}_{\mathfrak{X}}^{(0)}$-modules
(or, equivalently, the previous proposition) we have $\mathcal{R}\mathcal{H}om_{\widehat{\mathcal{D}}_{\mathfrak{X}}^{(0)}}(\widehat{\mathcal{D}}_{\mathfrak{X}\to\mathfrak{Y}}^{(0)},\widehat{\mathcal{D}}_{\mathfrak{X}}^{(0)})\tilde{=}\widehat{\mathcal{D}}_{\mathfrak{Y}\leftarrow\mathfrak{X}}^{(0)}[-d_{X/Y}]$. 

On the other hand, we have the short exact sequence 
\[
\mathcal{\overline{R}}(\mathcal{D}_{X}^{(0)})\to\mathcal{D}_{X}^{(0,1)}\to\mathcal{R}(\mathcal{D}_{X}^{(1)})(-1)
\]
which by \propref{Sandwich!} yields the distinguished triangle 
\[
R\underline{\mathcal{H}om}{}_{\bar{\mathcal{R}}(\mathcal{D}_{X}^{(0)})}(\mathcal{\overline{R}}_{X\to Y},\mathcal{\overline{R}}(\mathcal{D}_{X}^{(0)}))\to R\underline{\mathcal{H}om}{}_{\mathcal{D}_{X}^{(0,1)}}(\mathcal{D}_{X\to Y}^{(0,1)},\mathcal{D}_{X}^{(0,1)})
\]
\[
\to R\underline{\mathcal{H}om}{}_{\mathcal{R}(\mathcal{D}_{X}^{(1)})}(\mathcal{R}_{X\to Y},\mathcal{R}(\mathcal{D}_{X}^{(1)}))(-1)
\]
which implies that $R\underline{\mathcal{H}om}{}_{\mathcal{D}_{X}^{(0,1)}}(\mathcal{D}_{X\to Y}^{(0,1)},\mathcal{D}_{X}^{(0,1)})$
is concentrated in a single homological degree (namely $d$). So,
since $R\underline{\mathcal{H}om}{}_{\widehat{\mathcal{D}}_{\mathfrak{X}}^{(0,1)}}(\widehat{\mathcal{D}}_{\mathfrak{X}\to\mathfrak{Y}}^{(0,1)},\widehat{\mathcal{D}}_{\mathfrak{X}}^{(0,1)})$
is cohomologically complete, we see that $\mathcal{H}^{d}(R\underline{\mathcal{H}om}{}_{\widehat{\mathcal{D}}_{\mathfrak{X}}^{(0,1)}}(\widehat{\mathcal{D}}_{\mathfrak{X}\to\mathfrak{Y}}^{(0,1)},\widehat{\mathcal{D}}_{\mathfrak{X}}^{(0,1)}))$
is $p$-torsion-free and concentrated in degree $0$. We also see,
by \propref{coh-to-coh}, that this module is coherent over $\widehat{\mathcal{D}}_{\mathfrak{X}}^{(0,1)}$
(each of $\mathcal{R}_{Y\leftarrow X}^{(1)}$ and $\mathcal{\overline{R}}_{Y\leftarrow X}^{(0)}$
are coherent, since $X\to Y$ is smooth). Further, since $R\underline{\mathcal{H}om}{}_{\mathcal{D}_{X}^{(0,1)}}(\mathcal{D}_{X\to Y}^{(0,1)},\mathcal{D}_{X}^{(0,1)})\otimes_{\mathcal{D}_{X}^{(0,1)}}^{L}\mathcal{R}(\mathcal{D}_{X}^{(1)})$
and $R\underline{\mathcal{H}om}{}_{\mathcal{D}_{X}^{(0,1)}}(\mathcal{D}_{X\to Y}^{(0,1)},\mathcal{D}_{X}^{(0,1)})\otimes_{\mathcal{D}_{X}^{(0,1)}}^{L}\mathcal{\overline{R}}(\mathcal{D}_{X}^{(0)})$
are concentrated in degree $d$ as well, we see that $\text{im}(f)=\text{ker}(v)$
and $\text{im}(v)=\text{ker}(f)$ on $\mathcal{H}^{d}(R\underline{\mathcal{H}om}{}_{\mathcal{D}_{X}^{(0,1)}}(\mathcal{D}_{X\to Y}^{(0,1)},\mathcal{D}_{X}^{(0,1)}))$
(by \lemref{Basic-Facts-on-Rigid}). Furthermore, the distinguished
triangle above now yields the short exact sequence
\[
\mathcal{\overline{R}}_{Y\leftarrow X}^{(0)}(d)\to\mathcal{H}^{d}(R\underline{\mathcal{H}om}{}_{\mathcal{D}_{X}^{(0,1)}}(\mathcal{D}_{X\to Y}^{(0,1)},\mathcal{D}_{X}^{(0,1)}))\to\mathcal{R}_{Y\leftarrow X}^{(1)}(d-1)
\]
and since $\mathcal{R}_{Y\leftarrow X}^{(1)}$ is $f$-torsion-free,
we see that $\text{im}(v)=\text{ker}(f)=\mathcal{\overline{R}}_{Y\leftarrow X}^{(0)}(d)$
and so $\mathcal{H}^{d}(R\underline{\mathcal{H}om}{}_{\widehat{\mathcal{D}}_{\mathfrak{X}}^{(0,1)}}(\widehat{\mathcal{D}}_{\mathfrak{X}\to\mathfrak{Y}}^{(0,1)},\widehat{\mathcal{D}}_{\mathfrak{X}}^{(0,1)}))$
satisfies the conditions of \propref{Baby-Mazur}. So we may conclude
that the module $\mathcal{H}^{d}(R\underline{\mathcal{H}om}{}_{\widehat{\mathcal{D}}_{\mathfrak{X}}^{(0,1)}}(\widehat{\mathcal{D}}_{\mathfrak{X}\to\mathfrak{Y}}^{(0,1)},\widehat{\mathcal{D}}_{\mathfrak{X}}^{(0,1)}))$
is standard. Furthermore, we see that the grading on $\mathcal{\overline{R}}_{Y\leftarrow X}^{(0)}(d)$
is zero in degrees $<-d$ and is nontrivial in degree $-d$ and above.
Therefore, the index (as defined directly below \defref{Standard!})
is $d$. Since we identified $\mathcal{H}^{d}(R\underline{\mathcal{H}om}{}_{\widehat{\mathcal{D}}_{\mathfrak{X}}^{(0,1)}}(\widehat{\mathcal{D}}_{\mathfrak{X}\to\mathfrak{Y}}^{(0,1)},\widehat{\mathcal{D}}_{\mathfrak{X}}^{(0,1)})^{-\infty})$
with $\widehat{\mathcal{D}}_{\mathfrak{Y}\leftarrow\mathfrak{X}}^{(0)}$
we see that 
\[
\mathcal{H}^{d_{X/Y}}(R\underline{\mathcal{H}om}{}_{\widehat{\mathcal{D}}_{\mathfrak{X}}^{(0,1)}}(\widehat{\mathcal{D}}_{\mathfrak{X}\to\mathfrak{Y}}^{(0,1)},\widehat{\mathcal{D}}_{\mathfrak{X}}^{(0,1)})^{i-d})=\{m\in\widehat{\mathcal{D}}_{\mathfrak{Y}\leftarrow\mathfrak{X}}^{(0)}[p^{-1}]|p^{i}m\in\widehat{\mathcal{D}}_{\mathfrak{Y}\leftarrow\mathfrak{X}}^{(0)}\}
\]
which is exactly the definition of $\widehat{\mathcal{D}}_{\mathfrak{Y}\leftarrow\mathfrak{X}}^{(0,1)}(d)$. 
\end{proof}
From this one deduces 
\begin{cor}
\label{cor:smooth-adjunction}Let $\varphi:\mathfrak{X}\to\mathfrak{Y}$
be smooth of relative dimension $d$; let $\mathcal{M}^{\cdot}\in D_{cc}(\mathcal{G}(\widehat{\mathcal{D}}_{\mathfrak{X}}^{(0,1)}))$
and $\mathcal{N}^{\cdot}\in D_{cc}(\mathcal{G}(\widehat{\mathcal{D}}_{\mathfrak{Y}}^{(0,1)}))$.
Then there is an isomorphism of functors 
\[
R\varphi_{*}R\underline{\mathcal{H}om}{}_{\widehat{\mathcal{D}}_{\mathfrak{X}}^{(0,1)}}(\varphi^{\dagger}\mathcal{N}^{\cdot},\mathcal{M}^{\cdot})\tilde{\to}R\underline{\mathcal{H}om}{}_{\widehat{\mathcal{D}}_{\mathfrak{Y}}^{(0,1)}}(\mathcal{N}^{\cdot},\int_{\varphi}\mathcal{M}^{\cdot}(d))
\]
In particular, if $\varphi$ is also proper, then since both $\varphi^{\dagger}$
and ${\displaystyle \int_{\varphi}}$ preserve $D_{coh}^{b}$, we
obtain that these functors form an adjoint pair on $D_{coh}^{b}$. 

Further, the analogous isomorphism for $\varphi:X\to Y$ holds, and
in this situation the functors are adjoint on $D_{qcoh}^{b}$ in this
setting (even if $\varphi$ is not proper). 
\end{cor}

\begin{proof}
(following \cite{key-4}, Theorem 4.40). We have 
\[
R\underline{\mathcal{H}om}{}_{\widehat{\mathcal{D}}_{\mathfrak{X}}^{(0,1)}}(\varphi^{\dagger}\mathcal{N}^{\cdot},\mathcal{M}^{\cdot})\tilde{\to}R\underline{\mathcal{H}om}{}_{\widehat{\mathcal{D}}_{\mathfrak{X}}^{(0,1)}}(\widehat{\mathcal{D}}_{\mathfrak{X}\to\mathfrak{Y}}^{(0,1)}\widehat{\otimes}_{\varphi^{-1}\widehat{\mathcal{D}}_{\mathfrak{Y}}^{(0,1)}}^{L}\varphi^{-1}\mathcal{N}^{\cdot},\mathcal{M}^{\cdot})[d]
\]
\[
\tilde{\to}R\underline{\mathcal{H}om}{}_{\varphi^{-1}(\widehat{\mathcal{D}}_{\mathfrak{Y}}^{(0,1)})}(\varphi^{-1}\mathcal{N}^{\cdot},R\underline{\mathcal{H}om}{}_{\widehat{\mathcal{D}}_{\mathfrak{X}}^{(0,1)}}(\widehat{\mathcal{D}}_{\mathfrak{X}\to\mathfrak{Y}}^{(0,1)},\mathcal{M}^{\cdot}))[d]
\]
To prove the last isomorphism, one may reduce mod $p$, and then apply
\lemref{basic-hom-tensor} (part $1$), noting that $\mathcal{D}_{X\to Y}^{(0,1)}$
is faithfully flat over $\varphi^{-1}(\mathcal{D}_{Y}^{(0,1)})$. 

Further, we have 
\[
R\underline{\mathcal{H}om}{}_{\widehat{\mathcal{D}}_{\mathfrak{X}}^{(0,1)}}(\widehat{\mathcal{D}}_{\mathfrak{X}\to\mathfrak{Y}}^{(0,1)},\mathcal{M}^{\cdot})\tilde{\leftarrow}R\underline{\mathcal{H}om}{}_{\widehat{\mathcal{D}}_{\mathfrak{X}}^{(0,1)}}(\widehat{\mathcal{D}}_{\mathfrak{X}\to\mathfrak{Y}}^{(0,1)},\widehat{\mathcal{D}}_{\mathfrak{X}}^{(0,1)})\widehat{\otimes}_{\widehat{\mathcal{D}}_{\mathfrak{X}}^{(0,1)}}^{L}\mathcal{M}^{\cdot}\tilde{=}\widehat{\mathcal{D}}_{\mathfrak{Y}\leftarrow\mathfrak{X}}^{(0,1)}\widehat{\otimes}_{\widehat{\mathcal{D}}_{\mathfrak{X}}^{(0,1)}}^{L}\mathcal{M}^{\cdot}(d)[-d]
\]
where the first isomorphism again follows by reduction mod $p$ and
then applying the fact that $\mathcal{D}_{X\to Y}^{(0,1)}$ is (locally)
isomorphic to a bounded complex of projective $\mathcal{D}_{X}^{(0,1)}$-modules
(by \propref{Quasi-rigid=00003Dfinite-homological}) and the second
isomorphism is the previous proposition. Applying this to the previous
isomorphism we obtain 
\[
R\underline{\mathcal{H}om}{}_{\varphi^{-1}(\widehat{\mathcal{D}}_{\mathfrak{Y}}^{(0,1)})}(\varphi^{-1}\mathcal{N}^{\cdot},R\underline{\mathcal{H}om}{}_{\widehat{\mathcal{D}}_{\mathfrak{X}}^{(0,1)}}(\widehat{\mathcal{D}}_{\mathfrak{X}\to\mathfrak{Y}}^{(0,1)},\mathcal{M}^{\cdot}))[d]\tilde{=}R\underline{\mathcal{H}om}{}_{\varphi^{-1}(\widehat{\mathcal{D}}_{\mathfrak{Y}}^{(0,1)})}(\varphi^{-1}\mathcal{N}^{\cdot},\widehat{\mathcal{D}}_{\mathfrak{Y}\leftarrow\mathfrak{X}}^{(0,1)}\widehat{\otimes}_{\widehat{\mathcal{D}}_{\mathfrak{X}}^{(0,1)}}^{L}\mathcal{M}^{\cdot}(d))
\]
Then applying $R\varphi_{*}$ we obtain 
\[
R\varphi_{*}R\underline{\mathcal{H}om}{}_{\widehat{\mathcal{D}}_{\mathfrak{X}}^{(0,1)}}(\varphi^{\dagger}\mathcal{N}^{\cdot},\mathcal{M}^{\cdot})
\]
\[
\tilde{=}R\varphi_{*}R\underline{\mathcal{H}om}{}_{\varphi^{-1}(\widehat{\mathcal{D}}_{\mathfrak{Y}}^{(0,1)})}(\varphi^{-1}\mathcal{N}^{\cdot},\widehat{\mathcal{D}}_{\mathfrak{Y}\leftarrow\mathfrak{X}}^{(0,1)}\widehat{\otimes}_{\widehat{\mathcal{D}}_{\mathfrak{X}}^{(0,1)}}^{L}\mathcal{M}^{\cdot}(d))
\]
\[
\tilde{\to}R\underline{\mathcal{H}om}{}_{\widehat{\mathcal{D}}_{\mathfrak{Y}}^{(0,1)}}(\mathcal{N}^{\cdot},R\varphi_{*}(\widehat{\mathcal{D}}_{\mathfrak{Y}\leftarrow\mathfrak{X}}^{(0,1)}\widehat{\otimes}_{\widehat{\mathcal{D}}_{\mathfrak{X}}^{(0,1)}}^{L}\mathcal{M}^{\cdot}(d)))\tilde{\to}R\underline{\mathcal{H}om}{}_{\widehat{\mathcal{D}}_{\mathfrak{Y}}^{(0,1)}}(\mathcal{N}^{\cdot},\int_{\varphi}\mathcal{M}^{\cdot}(d))
\]
where the final isomorphism, is the adjunction between $\varphi^{-1}$
and $R\varphi_{*}$. One applies analogous reasoning for $\varphi:X\to Y$. 
\end{proof}
Now we prove the projection formula, and then give the the smooth
base change and Kunneth formulas in this context. We start with 
\begin{thm}
\label{thm:Projection-Formula}(Projection Formula) Let $\varphi:\mathfrak{X}\to\mathfrak{Y}$
be a morphism. Let $\mathcal{M}^{\cdot}\in D_{cc}^{b}(\mathcal{G}(\widehat{\mathcal{D}}_{\mathfrak{X}}^{(0,1)}))$
and $\mathcal{N}^{\cdot}\in D_{cc}^{b}(\mathcal{G}(\widehat{\mathcal{D}}_{\mathfrak{Y}}^{(0,1)}))$,
be such that $\mathcal{M}^{\cdot}\otimes_{W(k)}^{L}k\in D_{qcoh}(\mathcal{G}(\mathcal{D}_{X}^{(0,1)}))$
and $\mathcal{N}^{\cdot}\otimes_{W(k)}^{L}k\in D_{qcoh}(\mathcal{G}(\mathcal{D}_{Y}^{(0,1)}))$.
Then we have 
\[
\int_{\varphi}(L\varphi^{*}(\mathcal{N}^{\cdot})\widehat{\otimes}_{D(\mathcal{O}_{\mathfrak{X}})}^{L}\mathcal{M}^{\cdot})\tilde{\to}\mathcal{N}^{\cdot}\otimes_{D(\mathcal{O}_{\mathfrak{Y}})}^{L}\int_{\varphi}\mathcal{M}^{\cdot}
\]
\end{thm}

The proof works essentially the same way as the complex analytic one
(c.f. \cite{key-50}, theorem 2.3.19). In particular, we use \lemref{proj-over-D},
as well as the tensor product juggling lemma \lemref{Juggle}
\begin{proof}
By the left-right interchange it suffices to prove 
\[
\int_{\varphi}(\mathcal{M}_{r}^{\cdot}\widehat{\otimes}_{D(\mathcal{O}_{\mathfrak{X}})}^{L}L\varphi^{*}\mathcal{N}^{\cdot})\tilde{=}\int_{\varphi}(\mathcal{M}_{r}^{\cdot})\widehat{\otimes}_{D(\mathcal{O}_{\mathfrak{Y}})}^{L}\mathcal{N}^{\cdot}
\]
where $\mathcal{M}_{r}^{\cdot}=\omega_{\mathfrak{X}}\otimes_{\mathcal{O}_{\mathfrak{X}}}\mathcal{M}^{\cdot}$.
We have
\[
\int_{\varphi}(\mathcal{M}_{r}^{\cdot})\widehat{\otimes}_{D(\mathcal{O}_{\mathfrak{Y}})}^{L}\mathcal{N}^{\cdot}\tilde{=}\int_{\varphi}(\mathcal{M}_{r}^{\cdot})\widehat{\otimes}_{\mathcal{\widehat{D}}_{\mathfrak{Y}}^{(0,1)}}^{L}(\mathcal{N}^{\cdot}\widehat{\otimes}_{D(\mathcal{O}_{\mathfrak{Y}})}^{L}\mathcal{\widehat{D}}_{\mathfrak{Y}}^{(0,1)})
\]
\[
\tilde{=}R\varphi_{*}(\mathcal{M}_{r}^{\cdot}\widehat{\otimes}_{\mathcal{\widehat{D}}_{\mathfrak{X}}^{(0,1)}}^{L}L\varphi^{*}(\mathcal{N}^{\cdot}\widehat{\otimes}_{D(\mathcal{O}_{\mathfrak{Y}})}^{L}\mathcal{\widehat{D}}_{\mathfrak{Y}}^{(0,1)}))\tilde{=}R\varphi_{*}(\mathcal{M}_{r}^{\cdot}\widehat{\otimes}_{\mathcal{\widehat{D}}_{\mathfrak{X}}^{(0,1)}}^{L}L\varphi^{*}(\mathcal{N}^{\cdot})\widehat{\otimes}_{D(\mathcal{O}_{\mathfrak{X}})}^{L}\varphi^{*}(\mathcal{\widehat{D}}_{\mathfrak{Y}}^{(0,1)}))
\]
\[
\tilde{=}R\varphi_{*}((\mathcal{M}_{r}^{\cdot}\widehat{\otimes}_{\mathcal{\widehat{D}}_{\mathfrak{X}}^{(0,1)}}^{L}L\varphi^{*}(\mathcal{N}^{\cdot}))\widehat{\otimes}_{D(\mathcal{O}_{\mathfrak{X}})}^{L}\mathcal{\widehat{D}}_{\mathfrak{X}\to\mathfrak{Y}}^{(0,1)})\tilde{=}R\varphi_{*}((\mathcal{M}_{r}^{\cdot}\widehat{\otimes}_{D(\mathcal{O}_{\mathfrak{X}})}^{L}L\varphi^{*}(\mathcal{N}^{\cdot}))\widehat{\otimes}_{\mathcal{\widehat{D}}_{\mathfrak{X}}^{(0,1)}}^{L}\mathcal{\widehat{D}}_{\mathfrak{X}\to\mathfrak{Y}}^{(0,1)})
\]
\[
=\int_{\varphi}(\mathcal{M}_{r}^{\cdot}\widehat{\otimes}_{D(\mathcal{O}_{\mathfrak{X}})}^{L}L\varphi^{*}\mathcal{N}^{\cdot})
\]
as claimed; note that the second isomorphism is \lemref{proj-over-D}
which uses the assumption on $\mathcal{M}^{\cdot}$ and $\mathcal{N}^{\cdot}$. 
\end{proof}
Now we turn to the smooth base change. Consider the fibre square of
smooth formal schemes

$$ \begin{CD}  \mathfrak{X}_{\mathfrak{Z}} @>\tilde{\psi} >> \mathfrak{X} \\ @VV\tilde{\varphi}V @VV{\varphi}V \\ \mathfrak{Z}  @>\psi >> \mathfrak{Y} \end{CD} $$where
the bottom row $\psi:\mathfrak{Z}\to\mathfrak{Y}$ is smooth of relative
dimension $d$. 

We have also the analogous square for smooth varieties over $k$. 
\begin{thm}
\label{thm:Smooth-base-change}Suppose that $\mathcal{M}^{\cdot}\in D_{cc}(\mathcal{G}(\widehat{\mathcal{D}}_{\mathfrak{X}}^{(0,1)}))$
and $\mathcal{M}^{\cdot}\otimes_{W(k)}^{L}k\in D_{qcoh}^{b}(\mathcal{G}(\mathcal{D}_{X}^{(0,1)}))$.
There is an isomorphism 
\[
\psi^{\dagger}\int_{\varphi}\mathcal{M}^{\cdot}\tilde{\to}\int_{\tilde{\varphi}}\tilde{\psi}{}^{\dagger}\mathcal{M}^{\cdot}
\]
inside $D_{cc}(\mathcal{G}(\widehat{\mathcal{D}}_{\mathfrak{Z}}^{(0,1)}))$.
The analogous statement holds for smooth varieties over $k$. 
\end{thm}

\begin{proof}
By the adjunction for ${\displaystyle (\tilde{\psi}^{\dagger},\int_{\tilde{\psi}}(d))}$
there is a morphism of functors 
\[
\int_{\varphi}\to\int_{\varphi}\circ\int_{\tilde{\psi}}(\tilde{\psi})^{\dagger}(d)\tilde{=}\int_{\psi}\circ\int_{\tilde{\varphi}}(\tilde{\psi})^{\dagger}(d)
\]
where the last isomorphism follows from the composition of push-forwards
(\lemref{Composition-of-pushforwards}). Now, applying the adjunction
for ${\displaystyle (\psi^{\dagger},\int_{\psi}(d))}$, we obtain
a morphism 
\[
\psi^{\dagger}\int_{\varphi}\to\int_{\tilde{\varphi}}(\tilde{\psi})^{\dagger}
\]
After applying $\otimes_{W(k)}^{L}k$ we obtain the analogous map
over $k$. So it suffices to show that the map is an isomorphism for
varieties over $k$. Furthermore, working locally on $Z$, we reduce
to the case where the map $\psi:Z\to Y$ factors as an etale morphism
$Z\to Z'$ followed by a projection $Z'\tilde{=}Y\times\mathbb{A}^{d}\to Y$.
In the case of an etale morphism, the functor ${\displaystyle \int_{\varphi}}$
agrees with $R\varphi_{*}$, so the result follows from the usual
flat base change for quasicoherent sheaves. In the case of the projection,
we have 
\[
\int_{\tilde{\varphi}}(\tilde{\psi})^{\dagger}\mathcal{M}^{\cdot}\tilde{=}\int_{\text{id}\times\varphi}D(\mathcal{O}_{\mathbb{A}_{k}^{d}})\boxtimes\mathcal{M}^{\cdot}[d]\tilde{=}D(\mathcal{O}_{\mathbb{A}_{k}^{d}})\boxtimes\int_{\varphi}\mathcal{M}^{\cdot}[d]\tilde{=}\psi^{\dagger}\mathcal{M}^{\cdot}
\]
where the second isomorphism follows directly from the definition
of the pushforward; this implies the result in this case.
\end{proof}
From this we deduce the Kunneth formula:
\begin{cor}
Let $\mathcal{M}^{\cdot}\in D_{cc}(\mathcal{G}(\widehat{\mathcal{D}}_{\mathfrak{X}}^{(0,1)}))$
and $\mathcal{N}^{\cdot}\in D_{cc}(\mathcal{G}(\widehat{\mathcal{D}}_{\mathfrak{Y}}^{(0,1)}))$,
so that $\mathcal{M}^{\cdot}\otimes_{W(k)}^{L}k\in D_{qcoh}(\mathcal{G}(\mathcal{D}_{X}^{(0,1)}))$
and $\mathcal{N}^{\cdot}\otimes_{W(k)}^{L}k\in D_{qcoh}(\mathcal{G}(\mathcal{D}_{Y}^{(0,1)}))$.
Then there is an isomorphism 
\[
\mathbb{H}_{\mathcal{G}}^{\cdot}(\mathcal{M}^{\cdot}\boxtimes\mathcal{N}^{\cdot})\tilde{=}\mathbb{H}_{\mathcal{G}}^{\cdot}(\mathcal{M}^{\cdot})\widehat{\otimes}_{W(k)[f,v]}^{L}\mathbb{H}_{\mathcal{G}}^{\cdot}(\mathcal{N}^{\cdot})
\]
(where $\mathbb{H}_{\mathcal{G}}^{\cdot}$ is defined in \defref{Push!})The
analogous statement holds for complexes in $D_{qcoh}(\mathcal{G}(\mathcal{D}_{X}^{(0,1)}))$
and $D_{qcoh}(\mathcal{G}(\mathcal{D}_{Y}^{(0,1)}))$. 
\end{cor}

This is a formal consequence of the projection formula and the smooth
base change (compare, e.g. \cite{key-53}, corollary 2.3.30). 

\section{Operations on Gauges: Duality}

In this section we study the duality functor on $D_{coh}^{b}(\mathcal{G}(\widehat{\mathcal{D}}_{\mathfrak{X}}^{(0,1)}))$
(and on $D_{coh}^{b}(\mathcal{G}(\mathcal{D}_{X}^{(0,1)}))$. Although
neither $\widehat{\mathcal{D}}_{\mathfrak{X}}^{(0,1)}$ nor $\mathcal{D}_{X}^{(0,1)}$
have finite homological dimension, we shall show (using \propref{Sandwich!})
that there is a well-behaved duality functor $\mathbb{D}$ which takes
bounded complexes of coherent modules to bounded complexes of coherent
modules. Further, under suitable conditions this functor commutes
with push-forward, in the following sense: 
\begin{thm}
Let $\varphi:\mathfrak{X}\to\mathfrak{Y}$ be either a smooth proper
morphism or a projective morphism. Then there is an isomorphism of
functors 
\[
\int_{\varphi}\mathbb{D}_{\mathfrak{X}}\tilde{\to}\mathbb{D}_{\mathfrak{Y}}\int_{\varphi}
\]
The analogous statement holds for either a smooth proper or a projective
morphism $\varphi:X\to Y$. In particular; when $\varphi$ is smooth
proper the functors $(\int_{\varphi},\varphi^{\dagger})$ form an
adjoint pair on $D_{coh}^{b}$. 
\end{thm}

The proof, which will essentially occupy this section of the paper,
is somewhat unsatisfactory. The key point is to construct a trace
morphism 
\[
\text{tr}:\int_{\varphi}D(\mathcal{O}_{\mathfrak{X}})[d_{X}]\to D(\mathcal{O}_{\mathfrak{Y}})[d_{Y}]
\]
When $\varphi$ is smooth proper this is done by first constructing
the map in $\mathcal{D}^{(0)}$ and $\mathcal{D}^{(1)}$ modules (using
the Hodge to de Rham spectral sequence), and then deducing its existence
for $\mathcal{D}^{(0,1)}$-modules. When $\varphi$ is a closed immersion
the construction of the trace follows from a direct consideration
of the structure of ${\displaystyle \int_{\varphi}}$ (the transfer
bimodule is easy to describe in this case). For a projective $\varphi$
one defines the trace by breaking up the map into an immersion followed
by a projection. Presumably there is a way to construct the trace
for all proper morphisms at once, but I have been unable to find it. 

To kick things off, we need to define the duality functor and show
that it has finite homological dimension. 
\begin{defn}
Let $\mathcal{M}^{\cdot}\in D_{cc}(\mathcal{G}(\widehat{\mathcal{D}}_{\mathfrak{X}}^{(0,1)}))$.
We define $\mathbb{D}_{\mathfrak{X}}(\mathcal{M}^{\cdot}):=\omega_{\mathfrak{X}}^{-1}\otimes_{\mathcal{O}_{\mathfrak{X}}}R\underline{\mathcal{H}om}(\mathcal{M}^{\cdot},\widehat{\mathcal{D}}_{\mathfrak{X}}^{(0,1)})[d_{X}]\in D_{cc}(\mathcal{G}(\widehat{\mathcal{D}}_{\mathfrak{X}}^{(0,1)}))$
(where we have used the natural right $\widehat{\mathcal{D}}_{\mathfrak{X}}^{(0,1)}$-module
structure on $R\underline{\mathcal{H}om}(\mathcal{M}^{\cdot},\widehat{\mathcal{D}}_{\mathfrak{X}}^{(0,1)})$). 

The same formula defines $\mathbb{D}_{X}$ for a smooth variety $X$
over $k$; and in the analogous way we define the duality functors
for $\mathcal{R}(\mathcal{D}_{X}^{(1)})$ and $\overline{\mathcal{R}}(\mathcal{D}_{X}^{(0)})$. 
\end{defn}

This is really a duality on the category of coherent modules: 
\begin{prop}
Suppose $\mathcal{M}^{\cdot}\in D_{coh}^{b}(\mathcal{G}(\widehat{\mathcal{D}}_{\mathfrak{X}}^{(0,1)}))$
then $\mathbb{D}_{\mathfrak{X}}(\mathcal{M}^{\cdot})\in D_{coh}^{b}(\mathcal{G}(\widehat{\mathcal{D}}_{\mathfrak{X}}^{(0,1)}))$.
Further, the natural transformation $\mathcal{M}^{\cdot}\to\mathbb{D}_{\mathfrak{X}}\mathbb{D}_{\mathfrak{X}}\mathcal{M}^{\cdot}$
is an isomorphism. 

The same result holds for a smooth variety $X$ over $k$. 
\end{prop}

\begin{proof}
By reduction mod $p$ it suffices to prove the result for $X$. Using
\propref{Sandwich!}, and the fact that $\text{ker}(f:\mathcal{D}_{X}^{(0,1)}\to\mathcal{D}_{X}^{(0,1)})\tilde{=}\overline{\mathcal{R}}(\mathcal{D}_{X}^{(0)})(1)$
and $\text{ker}(v:\mathcal{D}_{X}^{(0,1)}\to\mathcal{D}_{X}^{(0,1)})\tilde{=}\mathcal{R}(\mathcal{D}_{X}^{(1)})(-1)$
one reduces to proving the analogous result for $\mathcal{R}(\mathcal{D}_{X}^{(1)})$
and $\overline{\mathcal{R}}(\mathcal{D}_{X}^{(0)})$. But these algebras
have finite homological dimension, as noted above, and the results
follow at once.
\end{proof}

\subsection{Duality for a smooth proper morphism}

Now we turn to defining the trace morphism, and proving the duality,
for a smooth proper map $\mathfrak{X}\to\mathfrak{Y}$ of relative
dimension $d$. In this case the usual Grothendieck duality theory
gives us a canonical morphism 
\[
\text{tr}:R^{d}\varphi_{*}(\omega_{\mathfrak{X}/\mathfrak{Y}})\to\mathcal{O}_{\mathfrak{Y}}
\]
Now consider $\mathcal{O}_{\mathfrak{X}}$ as a module over $\widehat{\mathcal{D}}_{\mathfrak{X}}^{(0)}$.
As the pushforward ${\displaystyle \int_{\varphi,0}\mathcal{O}_{\mathfrak{X}}}$
can be computed by the relative de Rham complex, looking at the Hodge-to-de
Rham spectral sequence in degree $2d$ yields an isomorphism of $\mathcal{O}_{\mathfrak{Y}}$-modules
\[
\mathcal{H}^{d}(\int_{\varphi,0}\mathcal{O}_{\mathfrak{X}})\tilde{=}R^{d}\varphi_{*}(\omega_{\mathfrak{X}/\mathfrak{Y}})
\]
so; composing with the trace morphism above, we obtain a map 
\[
\text{tr}:\mathcal{H}^{d}(\int_{\varphi,0}\mathcal{O}_{\mathfrak{X}})\to\mathcal{O}_{\mathfrak{Y}}
\]
of $\mathcal{\widehat{D}}_{\mathfrak{Y}}^{(0)}$-modules. 

Now consider $\varphi:\mathfrak{X}_{n}\to\mathfrak{Y}_{n}$, the reduction
mod $p^{n}$ of $\varphi$ for each $n\geq0$. Repeating the argument,
we can construct 
\[
\text{tr}:\mathcal{H}^{d}(\int_{\varphi,0}\mathcal{O}_{\mathfrak{X}_{n}})\to\mathcal{O}_{\mathfrak{Y}_{n}}
\]
and, in fact, the inverse limit of these maps is the trace constructed
above. In this setting, the de Rham complex $\Omega_{\mathfrak{X}_{n}/\mathfrak{Y}_{n}}^{\cdot}$
has the structure of a complex of coherent sheaves over the scheme
$W_{n}(\mathcal{O}_{X^{(n)}})$ (here we are identifying the underlying
topological spaces of $\mathfrak{X}_{n}$ and $W_{n}(\mathcal{O}_{X^{(n)}})$).
Thus we may also consider the second spectral sequence for the pushforward
of this complex, and we obtain an isomorphism 
\[
R^{d}\varphi_{*}(\text{coker}(d:\Omega_{\mathfrak{X}_{n}/\mathfrak{Y}_{n}}^{d-1}\to\omega_{\mathfrak{X}_{n}/\mathfrak{Y}_{n}}))\tilde{\to}R^{d}\varphi_{*}(\omega_{\mathfrak{X}/\mathfrak{Y}})
\]
or, equivalently, 
\[
R^{d}\varphi_{*}(\mathcal{D}_{\mathfrak{Y}_{n}\leftarrow\mathfrak{X}_{n}}^{(0)}\otimes_{\mathcal{D}_{\mathfrak{X}_{n}}^{(0)}}\mathcal{O}_{\mathfrak{X}_{n}})\tilde{\to}R^{d}\varphi_{*}(\omega_{\mathfrak{X}_{n}/\mathfrak{Y}_{n}})
\]

Now we consider the the pushforward of $\mathcal{O}_{\mathfrak{X}_{n}}$,
in the category of $\mathcal{D}_{\mathfrak{X}_{n}}^{(1)}$-modules.
By the commutativity of Frobenius descent with push-forward (\cite{key-2},
theoreme 3.4.4), we have 
\[
\int_{\varphi,1}\mathcal{O}_{\mathfrak{X}_{n}}\tilde{=}\int_{\varphi,1}F^{*}\mathcal{O}_{\mathfrak{X}_{n}}\tilde{\to}F^{*}\int_{\varphi,0}\mathcal{O}_{\mathfrak{X}_{n}}
\]
Therefore we obtain a trace map 
\[
\text{tr}:\mathcal{H}^{d}(\int_{\varphi,1}\mathcal{O}_{\mathfrak{X}_{n}})\to\mathcal{O}_{\mathfrak{Y}_{n}}
\]
in the category of $\mathcal{D}_{\mathfrak{Y}_{n}}^{(1)}$-modules;
and, using the second spectral sequence for the pushforward as above,
we have
\[
R^{d}\varphi_{*}(\mathcal{D}_{\mathfrak{Y}_{n}\leftarrow\mathfrak{X}_{n}}^{(1)}\otimes_{\mathcal{D}_{\mathfrak{X}_{n}}^{(1)}}\mathcal{O}_{\mathfrak{X}_{n}})\tilde{\to}\mathcal{H}^{d}(\int_{\varphi,1}\mathcal{O}_{\mathfrak{X}_{n}})
\]

Using these maps, we construct a trace for $\mathcal{D}_{\mathfrak{X}_{n}}^{(0,1)}$-modules:
\begin{lem}
There is a canonical morphism
\[
\text{tr}:R^{d}\varphi_{*}(\mathcal{D}_{\mathfrak{Y}_{n}\leftarrow\mathfrak{X}_{n}}^{(0,1)}\otimes{}_{\mathcal{D}_{\mathfrak{X}_{n}}^{(0,1)}}D(\mathcal{O}_{\mathfrak{X}_{n}}))\to D(\mathcal{O}_{\mathfrak{Y}_{n}})
\]
which has the property that the map $\text{tr}^{\infty}:R^{d}\varphi_{*}(\mathcal{D}_{\mathfrak{Y}_{n}\leftarrow\mathfrak{X}_{n}}^{(0,1)}\otimes{}_{\mathcal{D}_{\mathfrak{X}_{n}}^{(0,1)}}D(\mathcal{O}_{\mathfrak{X}_{n}}))^{\infty}\to D(\mathcal{O}_{\mathfrak{Y}_{n}}){}^{\infty}$
agrees with the trace map for $\mathcal{D}_{\mathfrak{X}_{n}}^{(1)}$-modules
constructed above; and the map $\text{tr}^{-\infty}:R^{d}\varphi_{*}(\mathcal{D}_{\mathfrak{Y}_{n}\leftarrow\mathfrak{X}_{n}}^{(0,1)}\otimes{}_{\mathcal{D}_{\mathfrak{X}_{n}}^{(0,1)}}D(\mathcal{O}_{\mathfrak{X}_{n}}))^{-\infty}\to D(\mathcal{O}_{\mathfrak{Y}_{n}})^{-\infty}$
agrees with the trace map for $\mathcal{D}_{\mathfrak{X}_{n}}^{(0)}$-modules
constructed above. We have the analogous statement for a proper morphism
$\varphi:X\to Y$, as well as in the categories of $\overline{\mathcal{R}}(\mathcal{D}_{X}^{(0)})$-modules
and $\mathcal{R}(\mathcal{D}_{X}^{(1)})$-modules.

This map yields a trace map in the derived category:
\[
\text{tr}:\int_{\varphi}D(\mathcal{O}_{\mathfrak{X}_{n}})[d]\to D(\mathcal{O}_{\mathfrak{Y}_{n}})
\]

Upon taking the inverse limit over $n$, we obtain a map 
\[
\text{tr}:\int_{\varphi}D(\mathcal{O}_{\mathfrak{X}})[d]\to D(\mathcal{O}_{\mathfrak{Y}})
\]
\end{lem}

\begin{proof}
We begin with the case $n=1$; i.e., $\mathfrak{X}_{n}=X$ and $\mathfrak{Y}_{n}=Y$.
We claim that the $\varphi^{-1}(\mathcal{D}_{Y}^{(0,1)})$-gauge $\mathcal{D}_{Y\leftarrow X}^{(0,1)}\otimes{}_{\mathcal{D}_{X}^{(0,1)}}D(\mathcal{O}_{X})$
satisfies the property that $v$ is an isomorphism in degrees $0$
and below and $f$ is an isomorphism in degrees $1$ and above. This
can be checked in local coordinates, where we have the isomorphism
\[
\mathcal{D}_{Y\leftarrow X}^{(0,1)}=\mathcal{J}\backslash\mathcal{D}_{X}^{(0,1)}
\]
where $\mathcal{J}$ is the right ideal generated by $\{\partial_{n-d+1},\dots,\partial_{n},\partial_{n-d+1}^{[p]},\dots,\partial_{n}^{[p]}\}$.
In degrees below $0$, the elements $\{\partial_{n-d+1}^{[p]},\dots,\partial_{n}^{[p]}\}$
act trivially; so that 
\[
(\mathcal{D}_{Y\leftarrow X}^{(0,1)}\otimes{}_{\mathcal{D}_{X}^{(0,1)}}D(\mathcal{O}_{X}))^{i}=\mathcal{O}_{X}/(\partial_{n-d+1},\dots,\partial_{n})
\]
for all $i\leq0$. On the other hand we have 
\[
(\mathcal{D}_{Y\leftarrow X}^{(0,1)}\otimes{}_{\mathcal{D}_{X}^{(0,1)}}D(\mathcal{O}_{X}))^{i}=\mathcal{O}_{X}/(\partial_{n-d+1},\dots,\partial_{n},\partial_{n-d+1}^{[p]},\dots,\partial_{n}^{[p]})
\]
for $i>0$; and the claim about $f$ and $v$ follows immediately.
As the functor $R^{d}\varphi_{*}$ commutes with direct sums, we see
that the gauge
\[
R^{d}(\mathcal{D}_{Y\leftarrow X}^{(0,1)}\otimes{}_{\mathcal{D}_{X}^{(0,1)}}D(\mathcal{O}_{X}))
\]
has the same property: $v$ is an isomorphism in degrees $0$ and
below and $f$ is an isomorphism in degrees $1$ and above. Thus we
may define 
\[
\text{tr}:R^{d}(\mathcal{D}_{Y\leftarrow X}^{(0,1)}\otimes{}_{\mathcal{D}_{X}^{(0,1)}}D(\mathcal{O}_{X}))^{i}\to\mathcal{O}_{Y}
\]
for any $i$ as follows: if $i\leq0$ we have $v_{-\infty}:R^{d}(\mathcal{D}_{Y\leftarrow X}^{(0,1)}\otimes{}_{\mathcal{D}_{X}^{(0,1)}}D(\mathcal{O}_{X}))^{i}\tilde{=}R^{d}\varphi_{*}(\mathcal{D}_{Y\leftarrow X}^{(0)}\otimes_{\mathcal{D}_{X}^{(0)}}\mathcal{O}_{X})$
and so we define the trace as the composition $\text{tr}\circ v_{-\infty}$,
where here $\text{tr}$ denotes the trace for $\mathcal{D}_{X}^{(0)}$-modules
constructed above. If $i>0$ we have $f_{\infty}:R^{d}(\mathcal{D}_{Y\leftarrow X}^{(0,1)}\otimes{}_{\mathcal{D}_{X}^{(0,1)}}D(\mathcal{O}_{X}))^{i}\tilde{=}R^{d}\varphi_{*}(\mathcal{D}_{Y\leftarrow X}^{(1)}\otimes_{\mathcal{D}_{X}^{(1)}}\mathcal{O}_{X})$
and so we define the trace as the composition $\text{tr}\circ f_{\infty}$,
where here $\text{tr}$ denotes the trace for $\mathcal{D}_{X}^{(1)}$-modules
constructed above. In a similar way, we construct the trace map in
the categories of $\overline{\mathcal{R}}(\mathcal{D}_{X}^{(0)})$-modules
and $\mathcal{R}(\mathcal{D}_{X}^{(1)})$-modules.

Now we consider $\mathfrak{X}_{n}$ for $n>1$. Since the functor
$\varphi_{*}$ has homological dimension $d$ (on the category of
quasicoherent sheaves), we have that $(R^{d}\varphi_{*}\mathcal{F})\otimes_{W(k)}k\tilde{=}(R^{d}\varphi_{*}\mathcal{F}\otimes_{W(k)}^{L}k)$
for any $\mathcal{F}\in\text{Qcoh}(\mathfrak{X}_{n})$. So, by Nakayama's
lemma and the result of the previous paragraph, we see that $f$ is
onto in degrees $1$ and above while $v$ is onto in degrees $0$
and below; by the coherence of the sheaves involved we see that these
maps are isomorphisms for $i<<0$ and $i>>0$. Since the target of
the trace map, $D(\mathcal{O}_{\mathfrak{Y}_{n}})$, has the property
that $v$ is an isomorphism in degrees $0$ and below and $f$ is
an isomorphism in degrees $1$ and above, we may define the trace
map in the exact same way as above. 
\end{proof}
\begin{rem}
\label{rem:trace-and-compose}If $\varphi:\mathfrak{X}\to\mathfrak{Y}$
and $\psi:\mathfrak{Y}\to\mathfrak{Z}$, then the trace map for the
composition satisfies ${\displaystyle \text{tr}_{\psi\circ\varphi}=\text{tr}_{\psi}\circ\int_{\psi}\text{tr}_{\varphi}}$.
This follows from the analogous result for the trace map in coherent
sheaf theory. 
\end{rem}

Now, following the usual method of algebraic $\mathcal{D}$-module
theory (c.f. \cite{key-49}, theorem 2.7.2), we have 
\begin{prop}
There is a canonical morphism 
\[
\int_{\varphi}\mathbb{D}_{\mathfrak{X}}\mathcal{M}^{\cdot}\to\mathbb{D}_{\mathfrak{Y}}\int_{\varphi}\mathcal{M}^{\cdot}
\]
 for any $\mathcal{M}^{\cdot}\in D_{cc}(\mathcal{G}(\widehat{\mathcal{D}}_{\mathfrak{X}}^{(0,1)}))$.
The same holds for $\mathcal{M}^{\cdot}\in D(\mathcal{G}(\mathcal{D}_{X}^{(0,1)}))$
when we have a proper map $\varphi:X\to Y$. Further, these maps are
compatible under application of $\otimes_{W(k)}^{L}k$. 
\end{prop}

\begin{proof}
We have
\[
\int_{\varphi}\mathbb{D}_{\mathfrak{X}}\mathcal{M}^{\cdot}=R\varphi_{*}(R\underline{\mathcal{H}om}_{\widehat{\mathcal{D}}_{\mathfrak{X}}^{(0,1)}}(\mathcal{M}^{\cdot},\widehat{\mathcal{D}}_{\mathfrak{X}}^{(0,1)})\widehat{\otimes}_{\widehat{\mathcal{D}}_{\mathfrak{X}}^{(0,1)}}^{L}\widehat{\mathcal{D}}_{\mathfrak{X}\to\mathfrak{Y}}^{(0,1)})\widehat{\otimes}_{\mathcal{O}_{\mathfrak{Y}}}^{L}\omega_{\mathfrak{Y}}^{-1}[d_{X}]
\]
\[
=R\varphi_{*}(R\underline{\mathcal{H}om}_{\widehat{\mathcal{D}}_{\mathfrak{X}}^{(0,1)}}(\mathcal{M}^{\cdot},\widehat{\mathcal{D}}_{\mathfrak{X}\to\mathfrak{Y}}^{(0,1)})\widehat{\otimes}_{\mathcal{O}_{\mathfrak{Y}}}^{L}\omega_{\mathfrak{Y}}^{-1}[d_{X}]
\]
while 
\[
\mathbb{D}_{\mathfrak{Y}}\int_{\varphi}\mathcal{M}^{\cdot}=R\underline{\mathcal{H}om}_{\widehat{\mathcal{D}}_{\mathfrak{Y}}^{(0,1)}}(\int_{\varphi}\mathcal{M}^{\cdot},\widehat{\mathcal{D}}_{\mathfrak{Y}}^{(0,1)})\widehat{\otimes}_{\mathcal{O}_{\mathfrak{Y}}}^{L}\omega_{\mathfrak{Y}}^{-1}[d_{Y}]
\]
To construct a canonical map between these complexes, we begin by
considering ${\displaystyle \int_{\varphi}\widehat{\mathcal{D}}_{\mathfrak{X}\to\mathfrak{Y}}^{(0,1)}}$.
By $\widehat{\mathcal{D}}_{\mathfrak{X}\to\mathfrak{Y}}^{(0,1)}=L\varphi^{*}\widehat{\mathcal{D}}_{\mathfrak{Y}}^{(0,1)}$,
we may apply \thmref{Projection-Formula} to obtain 
\[
\int_{\varphi}\widehat{\mathcal{D}}_{\mathfrak{X}\to\mathfrak{Y}}^{(0,1)}=\int_{\varphi}L\varphi^{*}\widehat{\mathcal{D}}_{\mathfrak{Y}}^{(0,1)}\tilde{\to}(\int_{\varphi}D(\mathcal{O}_{\mathfrak{X}}))\widehat{\otimes}_{\mathcal{O}_{\mathfrak{Y}}[f,v]}^{L}\widehat{\mathcal{D}}_{\mathfrak{Y}}^{(0,1)}
\]
so applying the trace map yields a canonical morphism 
\[
\int_{\varphi}\widehat{\mathcal{D}}_{\mathfrak{X}\to\mathfrak{Y}}^{(0,1)}[d]\to\widehat{\mathcal{D}}_{\mathfrak{Y}}^{(0,1)}
\]
and since $d=d_{X}-d_{Y}$ we have 
\[
\int_{\varphi}\widehat{\mathcal{D}}_{\mathfrak{X}\to\mathfrak{Y}}^{(0,1)}[d_{X}]\to\widehat{\mathcal{D}}_{\mathfrak{Y}}^{(0,1)}[d_{Y}]
\]
Then we have 
\[
R\varphi_{*}(R\underline{\mathcal{H}om}_{\widehat{\mathcal{D}}_{\mathfrak{X}}^{(0,1)}}(\mathcal{M}^{\cdot},\widehat{\mathcal{D}}_{\mathfrak{X}\to\mathfrak{Y}}^{(0,1)})[d_{X}]
\]
\[
\to R\varphi_{*}(R\underline{\mathcal{H}om}_{\varphi^{-1}(\widehat{\mathcal{D}}_{\mathfrak{Y}}^{(0,1)})}(\widehat{\mathcal{D}}_{\mathfrak{Y}\leftarrow\mathfrak{X}}^{(0,1)}\otimes_{\widehat{\mathcal{D}}_{\mathfrak{X}}^{(0,1)}}^{L}\mathcal{M}^{\cdot},\widehat{\mathcal{D}}_{\mathfrak{Y}\leftarrow\mathfrak{X}}^{(0,1)}\otimes_{\widehat{\mathcal{D}}_{\mathfrak{X}}^{(0,1)}}^{L}\widehat{\mathcal{D}}_{\mathfrak{X}\to\mathfrak{Y}}^{(0,1)})[d_{X}]
\]
\[
\to R\underline{\mathcal{H}om}_{\widehat{\mathcal{D}}_{\mathfrak{Y}}^{(0,1)}}(R\varphi_{*}(\widehat{\mathcal{D}}_{\mathfrak{Y}\leftarrow\mathfrak{X}}^{(0,1)}\otimes_{\widehat{\mathcal{D}}_{\mathfrak{X}}^{(0,1)}}^{L}\mathcal{M}^{\cdot}),R\varphi_{*}(\widehat{\mathcal{D}}_{\mathfrak{Y}\leftarrow\mathfrak{X}}^{(0,1)}\otimes_{\widehat{\mathcal{D}}_{\mathfrak{X}}^{(0,1)}}^{L}\widehat{\mathcal{D}}_{\mathfrak{X}\to\mathfrak{Y}}^{(0,1)}))[d_{X}]
\]
\[
=R\underline{\mathcal{H}om}_{\widehat{\mathcal{D}}_{\mathfrak{Y}}^{(0,1)}}(\int_{\varphi}\mathcal{M}^{\cdot},\int_{\varphi}\widehat{\mathcal{D}}_{\mathfrak{X}\to\mathfrak{Y}}^{(0,1)}[d_{X}])\to R\underline{\mathcal{H}om}_{\widehat{\mathcal{D}}_{\mathfrak{Y}}^{(0,1)}}(\int_{\varphi}\mathcal{M}^{\cdot},\widehat{\mathcal{D}}_{\mathfrak{Y}}^{(0,1)}[d_{Y}])
\]
where the last map is the trace. Combining with the above description
yields the canonical map 
\[
\int_{\varphi}\mathbb{D}_{\mathfrak{X}}\mathcal{M}^{\cdot}\to\mathbb{D}_{\mathfrak{Y}}\int_{\varphi}\mathcal{M}^{\cdot}
\]
as desired; the case of a proper map $\varphi:X\to Y$ is identical. 
\end{proof}
Now we turn to 
\begin{thm}
\label{thm:Duality-for-smooth-proper}The canonical map $\int_{\varphi}\mathbb{D}_{\mathfrak{X}}\mathcal{M}^{\cdot}\to\mathbb{D}_{\mathfrak{Y}}\int_{\varphi}\mathcal{M}^{\cdot}$
is an isomorphism for $\mathcal{M}^{\cdot}\in D_{coh}^{b}(\mathcal{G}(\widehat{\mathcal{D}}_{\mathfrak{X}}^{(0,1)}))$.
The same is true for a proper map $\varphi:X\to Y$. 
\end{thm}

The proof of this result will make use of several auxiliary results.
First, we recall a basic computation for pushforwards of $\mathcal{R}(\mathcal{D}_{X}^{(0)})$-modules;
as in the previous section we have the diagram 
\[
T^{*}X\xleftarrow{d\varphi}X\times_{Y}T^{*}Y\xrightarrow{\pi}T^{*}Y
\]
and the result reads
\begin{lem}
Let $\mathcal{M}^{\cdot}\in D_{coh}^{b}(\mathcal{G}(\mathcal{R}(\mathcal{D}_{X}^{(0)})))$.
There is an isomorphism 
\[
(\int_{\varphi}\mathcal{M}^{\cdot})\otimes_{k[f]}^{L}k\tilde{=}R\pi_{*}((d\varphi)^{!}(\mathcal{M}^{\cdot}\otimes_{k[f]}^{L}k))
\]
inside $D_{coh}^{b}(T^{*}Y)$; in this formula $d\varphi^{!}$ is
the extraordinary inverse image in coherent sheaf theory.
\end{lem}

This is a result of Laumon, (c.f. \cite{key-19}, construction 5.6.1).
For a proof in the Rees algebra language, see \cite{key-70}, corollary
3.9.

Next, we need the following Grothendieck duality statement for Azumaya
algebras: 
\begin{lem}
\label{lem:GD-for-Az}Let $X$ and $Y$ be smooth varieties over $k$
and let $\pi:X\to Y$ be a smooth proper morphism, of relative dimension
$d$. Suppose that $\mathcal{A}_{Y}$ be an Azumaya algebra on $Y$.
Set $\mathcal{A}_{X}=\pi^{*}\mathcal{A}_{Y}$, an Azumaya algebra
on $X$. Then there is a trace map $R^{d}\pi_{*}(\mathcal{A}_{X}\otimes_{\mathcal{O}_{X}}\omega_{X})\to\mathcal{A}_{Y}\otimes_{\mathcal{O}_{Y}}\omega_{Y}$
which induces, for any $\mathcal{M}^{\cdot}\in D_{coh}^{b}(\mathcal{A}_{X}-\text{mod})$
a functorial isomorphism 
\[
R\pi_{*}R\mathcal{H}om_{\mathcal{A}_{X}}(\mathcal{M}^{\cdot},\mathcal{A}_{X}\otimes_{\mathcal{O}_{X}}\omega_{X})[d]\tilde{\to}R\mathcal{H}om_{\mathcal{A}_{Y}}(R\pi_{*}\mathcal{M}^{\cdot},\mathcal{A}_{Y}\otimes_{\mathcal{O}_{Y}}\omega_{Y})
\]
inside $D_{coh}^{b}(\mathcal{O}_{Y}-\text{mod})$. 
\end{lem}

\begin{proof}
Via the projection formula we have 
\[
R\pi_{*}(\mathcal{A}_{X}\otimes_{\mathcal{O}_{X}}\omega_{X})=R\pi_{*}(\pi^{*}\mathcal{A}_{Y}\otimes_{\mathcal{O}_{X}}\omega_{X})\tilde{\to}\mathcal{A}_{Y}\otimes_{\mathcal{O}_{Y}}^{L}R\pi_{*}(\omega_{X})
\]
so the usual trace $\text{tr}:R^{d}\pi_{*}(\omega_{X})\to\omega_{Y}$
induces a trace $\text{tr}:R^{d}\pi_{*}(\mathcal{A}_{X}\otimes_{\mathcal{O}_{X}}\omega_{X})\to\mathcal{A}_{Y}\otimes_{\mathcal{O}_{Y}}\omega_{Y}$.
Since $\pi$ has homological dimension $d$, we have $R\pi_{*}(\mathcal{A}_{X}\otimes_{\mathcal{O}_{X}}\omega_{X})[d]\to R^{d}\pi_{*}(\mathcal{A}_{X}\otimes_{\mathcal{O}_{X}}\omega_{X})$
so that there is a map 
\[
R\pi_{*}(\mathcal{A}_{X}\otimes_{\mathcal{O}_{X}}\omega_{X})[d]\to\mathcal{A}_{Y}\otimes_{\mathcal{O}_{Y}}\omega_{Y}
\]
Thus we obtain 
\[
R\pi_{*}R\mathcal{H}om_{\mathcal{A}_{X}}(\mathcal{M}^{\cdot},\mathcal{A}_{X}\otimes_{\mathcal{O}_{X}}\omega_{X})[d]\to R\mathcal{H}om_{\mathcal{A}_{Y}}(R\pi_{*}\mathcal{M}^{\cdot},R\pi_{*}(\mathcal{A}_{X}\otimes_{\mathcal{O}_{X}}\omega_{X})[d])
\]
\[
\to R\mathcal{H}om_{\mathcal{A}_{Y}}(R\pi_{*}\mathcal{M}^{\cdot},\mathcal{A}_{Y}\otimes_{\mathcal{O}_{Y}}\omega_{Y})
\]
for any $\mathcal{M}^{\cdot}\in D_{coh}^{b}(\mathcal{A}_{X}-\text{mod})$.
To prove that this map is an isomorphism, we can can work in the etale
(or flat) topology on $Y$ and so assume that $\mathcal{A}_{Y}$ is
split; i.e., $\mathcal{A}_{Y}=\mathcal{E}nd(\mathcal{E}_{Y})$ for
some vector bundle $\mathcal{E}_{Y}$. This implies $\mathcal{A}_{X}=\mathcal{E}nd(\mathcal{E}_{X})$
where $\mathcal{E}_{X}=\pi^{*}\mathcal{E}_{Y}$. Then for any $\mathcal{M}^{\cdot}\in D_{coh}^{b}(\mathcal{A}_{X}-\text{mod})$
we have $\mathcal{M}^{\cdot}=\mathcal{E}_{X}\otimes_{\mathcal{O}_{X}}\mathcal{N}^{\cdot}$
for a complex $\mathcal{N}^{\cdot}\in D_{coh}^{b}(\mathcal{O}_{X}-\text{mod})$.
Therefore 
\[
R\pi_{*}R\mathcal{H}om_{\mathcal{A}_{X}}(\mathcal{M}^{\cdot},\mathcal{A}_{X}\otimes_{\mathcal{O}_{X}}\omega_{X})[d]\tilde{=}R\pi_{*}R\mathcal{H}om_{\mathcal{A}_{X}}(\mathcal{E}_{X}\otimes_{\mathcal{O}_{X}}\mathcal{N}^{\cdot},\mathcal{E}_{X}\otimes_{\mathcal{O}_{X}}(\mathcal{E}_{X}^{*}\otimes_{\mathcal{O}_{X}}\omega_{X}))[d]
\]
\[
\tilde{=}R\pi_{*}R\mathcal{H}om_{\mathcal{O}_{X}}(\mathcal{N}^{\cdot},\mathcal{E}_{X}^{*}\otimes_{\mathcal{O}_{X}}\omega_{X})[d]\tilde{=}R\pi_{*}R\mathcal{H}om_{\mathcal{O}_{X}}(\mathcal{M}^{\cdot},\omega_{X}[d])
\]
\[
\tilde{\to}R\mathcal{H}om_{\mathcal{O}_{Y}}(R\pi_{*}\mathcal{M}^{\cdot},\omega_{Y})\tilde{=}R\mathcal{H}om_{\mathcal{A}_{Y}}(\mathcal{E}_{Y}\otimes_{\mathcal{O}_{Y}}R\pi_{*}\mathcal{M}^{\cdot},\mathcal{E}_{Y}\otimes_{\mathcal{O}_{Y}}\omega_{Y})
\]
\[
\tilde{\to}R\mathcal{H}om_{\mathcal{A}_{Y}}(R\pi_{*}\mathcal{M}^{\cdot},\mathcal{A}_{Y}\otimes_{\mathcal{O}_{Y}}\omega_{Y})
\]
where the isomorphism $R\pi_{*}R\mathcal{H}om_{\mathcal{O}_{X}}(\mathcal{M}^{\cdot},\omega_{X}[d])\tilde{\to}R\mathcal{H}om_{\mathcal{O}_{Y}}(R\pi_{*}\mathcal{M}^{\cdot},\omega_{Y})$
is Grothendieck duality for coherent sheaves. 
\end{proof}
Now we can proceed to the 
\begin{proof}
(of \thmref{Duality-for-smooth-proper}) By applying $\otimes_{W(k)}^{L}k$
we reduce to the characteristic $p$ situation of a smooth proper
morphism $\varphi:X\to Y$. By induction on the cohomological length,
we may suppose that $\mathcal{M}^{\cdot}$ is concentrated in a single
degree; i.e., $\mathcal{M}^{\cdot}=\mathcal{M}\in\mathcal{G}(\mathcal{D}_{X}^{(0,1)})$.
Then $\mathcal{M}$ admits a short exact sequence
\[
\mathcal{M}_{0}\to\mathcal{M}\to\mathcal{M}_{1}
\]
where $\mathcal{M}_{0}\in\mathcal{G}(\overline{\mathcal{R}}(\mathcal{D}_{X}^{(0)}))$
and $\mathcal{M}_{1}\in\mathcal{G}(\mathcal{R}(\mathcal{D}_{X}^{(1)}))$.
By \propref{Sandwich!} and \propref{Sandwich-push}, we see that
is suffices to prove the analogous statements in $D_{coh}^{b}(\mathcal{G}(\overline{\mathcal{R}}(\mathcal{D}_{X}^{(0)})))$
and $D_{coh}^{b}(\mathcal{G}(\mathcal{R}(\mathcal{D}_{X}^{(1)})))$.
By Frobenius descent (\thmref{Hodge-Filtered-Push}), one sees that
it suffices to prove the result for $D_{coh}^{b}(\mathcal{G}(\overline{\mathcal{R}}(\mathcal{D}_{X}^{(0)})))$
and $D_{coh}^{b}(\mathcal{G}(\mathcal{R}(\mathcal{D}_{X}^{(0)})))$.
These two cases require similar, but slightly different techniques;
we begin with the case of $\mathcal{R}(\mathcal{D}_{X}^{(0)})$. In
this case, since the grading on $\mathcal{R}(\mathcal{D}_{X}^{(0)})$
is concentrated in degrees $\geq0$, the graded Nakayama lemma applies
and so it suffices to prove that the map is an isomorphism after applying
$\otimes_{k[f]}^{L}k$; i.e., we have to prove 
\[
R\pi_{*}(d\varphi)^{!}R\mathcal{H}om_{\mathcal{O}_{T^{*}X}}((\mathcal{M}\otimes_{k[f]}^{L}k),\omega_{T^{*}X})[d]
\]
\[
\tilde{\to}R\mathcal{H}om_{\mathcal{O}_{T^{*}Y}}(R\pi_{*}((d\varphi)^{!}(\mathcal{M}\otimes_{k[f]}^{L}k),\omega_{T^{*}Y})
\]
Since $d\varphi$ is a closed immersion of smooth schemes, we have
\[
(d\varphi)^{!}R\mathcal{H}om_{\mathcal{O}_{T^{*}X}}((\mathcal{M}\otimes_{k[f]}^{L}k),\omega_{T^{*}X})[d]
\]
\[
\tilde{=}R\mathcal{H}om_{\mathcal{O}_{X\times_{Y}T^{*}Y}}((d\varphi)^{!}(\mathcal{M}\otimes_{k[f]}^{L}k),(d\varphi)^{!}\omega_{T^{*}X})[d]
\]
Furthermore, $(d\varphi)^{!}\omega_{T^{*}X}=\omega_{X\times_{Y}T^{*}Y}\tilde{=}\pi^{!}(\omega_{T^{*}Y})$.
Therefore 
\[
R\pi_{*}R\mathcal{H}om_{\mathcal{O}_{T^{*}X}}((d\varphi)^{!}(\mathcal{M}\otimes_{k[f]}^{L}k),(d\varphi)^{!}\omega_{T^{*}X})[d]
\]
\[
\tilde{=}R\pi_{*}R\mathcal{H}om_{\mathcal{O}_{X\times_{Y}T^{*}Y}}((d\varphi)^{!}(\mathcal{M}\otimes_{k[f]}^{L}k),\pi^{!}\omega_{T^{*}Y}[d])
\]
\[
\tilde{\to}R\mathcal{H}om_{\mathcal{O}_{T^{*}Y}}(R\pi_{*}((d\varphi)^{!}(\mathcal{M}\otimes_{k[f]}^{L}k),\omega_{T^{*}Y})
\]
where the last isomorphism is induced by the trace for $X\times_{Y}T^{*}Y\xrightarrow{\pi}T^{*}Y$,
i.e., it is given by Grothendieck duality for $\pi$; this proves
the result for $D_{coh}^{b}(\mathcal{G}(\mathcal{R}(\mathcal{D}_{X}^{(0)})))$. 

In order to handle $D_{coh}^{b}(\mathcal{G}(\overline{\mathcal{R}}(\mathcal{D}_{X}^{(0)})))$,
we apply a similar technique, but working directly\footnote{As the grading on $\overline{\mathcal{R}}(\mathcal{D}_{X}^{(0)})$
in unbounded, the graded Nakayama lemma does not apply} with the Azumaya algebra $\overline{\mathcal{R}}(\mathcal{D}_{X}^{(0)})$.
Since the morphism $\varphi$ is smooth, we can make use of \corref{Filtered-Bez-Brav}
and work with the functor $R\pi_{*}^{(1)}\circ C\circ(d\varphi^{(1)})^{!}$.
We therefore have to prove 
\[
R\pi_{*}^{(1)}\circ C\circ(d\varphi^{(1)})^{!}R\mathcal{H}om_{\overline{\mathcal{R}}(\mathcal{D}_{X}^{(0)})}(\mathcal{M}^{\cdot},\overline{\mathcal{R}}(\mathcal{D}_{X}^{(0)}))\otimes_{\mathcal{O}_{X}}\omega_{X}^{-1}[d]
\]
\[
\tilde{\to}R\mathcal{H}om_{\overline{\mathcal{R}}(\mathcal{D}_{Y}^{(0)})}(R\pi_{*}^{(1)}\circ C\circ(d\varphi^{(1)})^{!}\mathcal{M}^{\cdot},\overline{\mathcal{R}}(\mathcal{D}_{X}^{(0)}))\otimes_{\mathcal{O}_{Y}}\omega_{Y}^{-1}
\]
We proceed as above. We have an isomorphism
\[
C\circ(d\varphi^{(1)})^{!}R\mathcal{H}om_{\overline{\mathcal{R}}(\mathcal{D}_{X}^{(0)})}(\mathcal{M}^{\cdot},\overline{\mathcal{R}}(\mathcal{D}_{X}^{(0)}))\otimes_{\mathcal{O}_{X}}\omega_{X}^{-1}
\]
\[
\tilde{=}R\mathcal{H}om_{(\pi^{(1)})^{*}(\overline{\mathcal{R}}(\mathcal{D}_{Y}^{(0)}))}(C\circ(d\varphi^{(1)})^{!}\mathcal{M}^{\cdot},C\circ(d\varphi^{(1)})^{!}(\overline{\mathcal{R}}(\mathcal{D}_{X}^{(0)})\otimes_{\mathcal{O}_{X}}\omega_{X}^{-1}))
\]
Applying the definition of $(d\varphi^{(1)})^{!}$ and $C$, one deduces
\[
C\circ(d\varphi^{(1)})^{!}(\overline{\mathcal{R}}(\mathcal{D}_{X}^{(0)})\otimes_{\mathcal{O}_{X}}\omega_{X}^{-1})\tilde{=}(\pi^{(1)})^{*}(\overline{\mathcal{R}}(\mathcal{D}_{Y}^{(0)})\otimes_{\mathcal{O}_{Y}}\omega_{Y}^{-1})\otimes_{\mathcal{O}_{(X\times_{Y}T^{*}Y)^{(1)}}}\omega_{(X\times_{Y}T^{*}Y)^{(1)}}
\]
 Therefore 
\[
R\mathcal{H}om_{\pi^{*}(\overline{\mathcal{R}}(\mathcal{D}_{Y}^{(0)}))}(C\circ(d\varphi^{(1)})^{!}\mathcal{M}^{\cdot},C\circ(d\varphi^{(1)})^{!}(\overline{\mathcal{R}}(\mathcal{D}_{X}^{(0)})\otimes_{\mathcal{O}_{X}}\omega_{X}^{-1}))[d]
\]
\[
\tilde{\to}R\mathcal{H}om_{\pi^{*}(\overline{\mathcal{R}}(\mathcal{D}_{Y}^{(0)}))}(C\circ(d\varphi^{(1)})^{!}\mathcal{M}^{\cdot},(\pi^{(1)})^{*}(\overline{\mathcal{R}}(\mathcal{D}_{Y}^{(0)})\otimes_{\mathcal{O}_{Y}}\omega_{Y}^{-1})\otimes_{\mathcal{O}_{(X\times_{Y}T^{*}Y)^{(1)}}}\omega_{(X\times_{Y}T^{*}Y)^{(1)}})
\]
\[
\tilde{\to}R\mathcal{H}om_{\overline{\mathcal{R}}(\mathcal{D}_{Y}^{(0)})}(R\pi_{*}^{(1)}\circ C\circ(d\varphi^{(1)})^{!}\mathcal{M}^{\cdot},\overline{\mathcal{R}}(\mathcal{D}_{X}^{(0)})\otimes_{\mathcal{O}_{Y}}\omega_{Y}^{-1})
\]
where the last isomorphism follows from \lemref{GD-for-Az}; this
proves the result for $\overline{\mathcal{R}}(\mathcal{D}_{X}^{(0)})$. 
\end{proof}
This implies, by an identical argument to theorem 2.7.3 of \cite{key-49}:
\begin{cor}
\label{cor:Smooth-proper-adunction}There is a functorial isomorphism
\[
R\underline{\mathcal{H}om}_{\mathcal{D}_{\mathfrak{Y}}^{(0,1)}}(\int_{\varphi}\mathcal{M}^{\cdot},\mathcal{N}^{\cdot})\tilde{\to}\varphi_{*}R\underline{\mathcal{H}om}_{\mathcal{D}_{\mathfrak{X}}^{(0,1)}}(\mathcal{M}^{\cdot},\varphi^{\dagger}\mathcal{N}^{\cdot})
\]
for all $\mathcal{M}^{\cdot}\in D_{coh}^{b}(\mathcal{G}(\mathcal{D}_{\mathfrak{X}}^{(0,1)}))$
and $\mathcal{N}^{\cdot}\in D_{coh}^{b}(\mathcal{G}(\mathcal{D}_{\mathfrak{Y}}^{(0,1)}))$. 
\end{cor}

\subsection{Duality for a Projective morphism}

Now we turn to constructing the trace map in the case where $\varphi:\mathfrak{X}\to\mathfrak{Y}$
is a closed embedding, of relative dimension $d$. In this case the
pushforward is fairly easy to describe: 
\begin{lem}
\label{lem:transfer-is-locally-free}Let $\varphi:\mathfrak{X}_{n}\to\mathfrak{Y}_{n}$
be the reduction to $W_{n}(k)$ of the closed embedding $\varphi$.
Then the transfer bimodule $\mathcal{D}_{\mathfrak{X}_{n}\to\mathfrak{Y}_{n}}^{(0,1)}$
is locally free over $\mathcal{D}_{\mathfrak{X}_{n}}^{(0,1)}$ and
is coherent over $\mathcal{D}_{\mathfrak{Y}_{n}}^{(0,1),\text{opp}}$.
Thus the functor $\int_{\varphi}^{0}:\mathcal{G}_{coh}(\mathcal{D}_{\mathfrak{X}_{n}}^{(0,1)})\to\mathcal{G}_{coh}(\mathcal{D}_{\mathfrak{Y}_{n}}^{(0,1)})$
is exact. 
\end{lem}

\begin{proof}
Working locally, we can assume that $\mathfrak{X}_{n}=\text{Spec}(B_{n})$,
$\mathfrak{Y}_{n}=\text{Spec}(A_{n})$, and $A_{n}$ admits local
coordinates $\{x_{1},\dots,x_{n}\}$ for which $B_{n}=A_{n}/(x_{1},\dots,x_{m})$.
Then 
\[
\Gamma(\mathcal{D}_{\mathfrak{X}_{n}\to\mathfrak{Y}_{n}}^{(0,1)})=\mathcal{D}_{A_{n}}^{(0,1)}/I\cdot\mathcal{D}_{A_{n}}^{(0,1)}
\]
is coherent over $\mathcal{D}_{A_{n}}^{(0,1),\text{opp}}$. Now, \corref{Local-coords-over-A=00005Bf,v=00005D}
implies that $\mathcal{D}_{A_{n}}^{(0,1)}$ is free over $D(A_{n})$
(c.f. also the proof of \corref{Each-D^(i)-is-free}), with basis
given by the set $\{\partial^{I}(\partial^{[p]})^{J}\}$, where $I=(i_{1},\dots,i_{n})$
is a multi-index with $0\leq i_{j}\leq p-1$ for all $j$ and $J$
is any multi-index with entries $\geq0$. So $\mathcal{D}_{A_{n}}^{(0,1)}/I\cdot\mathcal{D}_{A_{n}}^{(0,1)}$
is free over $\mathcal{D}_{B_{n}}^{(0,1)}$with basis given by $\{\partial^{I}(\partial^{[p]})^{J}\}$,
where $I=(i_{1},\dots,i_{m})$ is a multi-index with $0\leq i_{j}\leq p-1$
for all $j$ and $J=(j_{1},\dots,j_{m})$ is any multi-index with
entries $\geq0$.
\end{proof}
Now we can proceed to analyze this functor, and the pullback $\varphi^{\dagger}$,
in exactly the same way as is done in the usual algebraic $\mathcal{D}$-module
theory. In this case, the existence of the trace map is essentially
deduced from the duality. To start off, we have 
\begin{prop}
Let $\varphi:\mathfrak{X}_{n}\to\mathfrak{Y}_{n}$ be as above. Define
$\varphi^{\sharp}(\mathcal{M}^{\cdot}):=R\underline{\mathcal{H}om}_{\varphi^{-1}(\mathcal{D}_{\mathfrak{Y}_{n}}^{(0,1)})}(\mathcal{D}_{\mathfrak{Y}_{n}\leftarrow\mathfrak{X}_{n}}^{(0,1)},\varphi^{-1}(\mathcal{M}^{\cdot}))$.
Then there is an isomorphism of functors $\varphi^{\dagger}\tilde{=}\varphi^{\sharp}$
on $D(\mathcal{G}(\mathcal{D}_{\mathfrak{Y}_{n}}^{(0,1)}))$. 
\end{prop}

\begin{proof}
This is very similar to \cite{key-49}, propositions 1.5.14 and 1.5.16.
One first shows 
\[
R\underline{\mathcal{H}om}_{\varphi^{-1}(\mathcal{D}_{\mathfrak{Y}_{n}}^{(0,1),\text{opp}})}(\mathcal{D}_{\mathfrak{X}_{n}\to\mathfrak{Y}_{n}}^{(0,1)},\varphi^{-1}(\mathcal{D}_{\mathfrak{Y}_{n}}^{(0,1)}))\tilde{=}\mathcal{D}_{\mathfrak{Y}_{n}\leftarrow\mathfrak{X}_{n}}^{(0,1)}[-d]
\]
by using the Koszul complex to write a locally free resolution for
$\mathcal{D}_{\mathfrak{X}_{n}\to\mathfrak{Y}_{n}}^{(0,1)}$ over
$\varphi^{-1}(\mathcal{D}_{\mathfrak{Y}_{n}}^{(0,1),\text{opp}})$;
note that by the left-right interchange this implies 
\[
R\underline{\mathcal{H}om}_{\varphi^{-1}(\mathcal{D}_{\mathfrak{Y}_{n}}^{(0,1)})}(\mathcal{D}_{\mathfrak{Y}_{n}\leftarrow\mathfrak{X}_{n}}^{(0,1)},\varphi^{-1}(\mathcal{D}_{\mathfrak{Y}_{n}}^{(0,1)}))\tilde{=}\mathcal{D}_{\mathfrak{X}_{n}\to\mathfrak{Y}_{n}}^{(0,1)}[-d]
\]
 Then, we have 
\[
\varphi^{\dagger}(\mathcal{M}^{\cdot})=\mathcal{D}_{\mathfrak{X}_{n}\to\mathfrak{Y}_{n}}^{(0,1)}\otimes_{\varphi^{-1}(\mathcal{D}_{\mathfrak{Y}_{n}}^{(0,1)})}^{L}\varphi^{-1}(\mathcal{M}^{\cdot})[-d]
\]
\[
\tilde{=}R\underline{\mathcal{H}om}_{\varphi^{-1}(\mathcal{D}_{\mathfrak{Y}_{n}}^{(0,1)})}(\mathcal{D}_{\mathfrak{Y}_{n}\leftarrow\mathfrak{X}_{n}}^{(0,1)},\varphi^{-1}(\mathcal{D}_{\mathfrak{Y}_{n}}^{(0,1)}))\otimes_{\varphi^{-1}(\mathcal{D}_{\mathfrak{Y}_{n}}^{(0,1)})}^{L}\varphi^{-1}(\mathcal{M}^{\cdot})
\]
\[
\tilde{\to}R\underline{\mathcal{H}om}_{\varphi^{-1}(\mathcal{D}_{\mathfrak{Y}_{n}}^{(0,1)})}(\mathcal{D}_{\mathfrak{Y}_{n}\leftarrow\mathfrak{X}_{n}}^{(0,1)},\varphi^{-1}(\mathcal{M}^{\cdot}))
\]
where the last isomorphism uses the fact that $\mathcal{D}_{\mathfrak{Y}_{n}\leftarrow\mathfrak{X}_{n}}^{(0,1)}$
admits, locally, a finite free resolution over $\varphi^{-1}(\mathcal{D}_{\mathfrak{Y}_{n}}^{(0,1)})$. 
\end{proof}
In turn, this implies 
\begin{cor}
We have a functorial isomorphism 
\[
R\underline{\mathcal{H}om}_{\mathcal{D}_{\mathfrak{Y}_{n}}^{(0,1)}}(\int_{\varphi}\mathcal{M}^{\cdot},\mathcal{N}^{\cdot})\tilde{\to}\varphi_{*}R\underline{\mathcal{H}om}_{\mathcal{D}_{\mathfrak{X}_{n}}^{(0,1)}}(\mathcal{M}^{\cdot},\varphi^{\dagger}\mathcal{N}^{\cdot})
\]
for all $\mathcal{M}^{\cdot}\in D_{qcoh}^{b}(\mathcal{G}(\mathcal{D}_{\mathfrak{X}_{n}}^{(0,1)}))$
and $\mathcal{N}^{\cdot}\in D_{qcoh}^{b}(\mathcal{G}(\mathcal{D}_{\mathfrak{Y}_{n}}^{(0,1)}))$. 
\end{cor}

\begin{proof}
(Just as in \cite{key-49}, proposition 1.5.25). By the previous proposition,
it suffices to prove the result for $\varphi^{\sharp}$ instead of
$\varphi^{\dagger}$. To proceed, note that we have the local cohomology
functor $\mathcal{N}^{\cdot}\to R\Gamma_{\mathfrak{X}_{n}}(\mathcal{N}^{\cdot})$
which takes $\mathcal{N}^{\cdot}\in D_{qcoh}^{b}(\mathcal{G}(\mathcal{D}_{\mathfrak{Y}_{n}}^{(0,1)}))$
to $D_{qcoh}^{b}(\mathcal{G}(\mathcal{D}_{\mathfrak{Y}_{n}}^{(0,1)}))$.
We have 
\[
R\underline{\mathcal{H}om}_{\mathcal{D}_{\mathfrak{Y}_{n}}^{(0,1)}}(\int_{\varphi}\mathcal{M}^{\cdot},\mathcal{N}^{\cdot})=R\underline{\mathcal{H}om}_{\mathcal{D}_{\mathfrak{Y}_{n}}^{(0,1)}}(\varphi_{*}(\mathcal{D}_{\mathfrak{Y}_{n}\leftarrow\mathfrak{X}_{n}}^{(0,1)}\otimes_{\mathcal{D}_{\mathfrak{X}_{n}}^{(0,1)}}^{L}\mathcal{M}^{\cdot}),\mathcal{N}^{\cdot})
\]
\[
\tilde{=}R\underline{\mathcal{H}om}_{\mathcal{D}_{\mathfrak{Y}_{n}}^{(0,1)}}(\varphi_{*}(\mathcal{D}_{\mathfrak{Y}_{n}\leftarrow\mathfrak{X}_{n}}^{(0,1)}\otimes_{\mathcal{D}_{\mathfrak{X}_{n}}^{(0,1)}}^{L}\mathcal{M}^{\cdot}),R\Gamma_{\mathfrak{X}_{n}}(\mathcal{N}^{\cdot}))
\]
\[
\tilde{=}\varphi_{*}(R\underline{\mathcal{H}om}_{\varphi^{-1}(\mathcal{D}_{\mathfrak{Y}_{n}}^{(0,1)})}(\varphi^{-1}(\varphi_{*}(\mathcal{D}_{\mathfrak{Y}_{n}\leftarrow\mathfrak{X}_{n}}^{(0,1)}\otimes_{\mathcal{D}_{\mathfrak{X}_{n}}^{(0,1)}}^{L}\mathcal{M}^{\cdot})),\varphi^{-1}(R\Gamma_{\mathfrak{X}_{n}}(\mathcal{N}^{\cdot})))
\]
\[
\tilde{=}\varphi_{*}(R\underline{\mathcal{H}om}_{\varphi^{-1}(\mathcal{D}_{\mathfrak{Y}_{n}}^{(0,1)})}(\mathcal{D}_{\mathfrak{Y}_{n}\leftarrow\mathfrak{X}_{n}}^{(0,1)}\otimes_{\mathcal{D}_{\mathfrak{X}_{n}}^{(0,1)}}^{L}\mathcal{M}^{\cdot}),\varphi^{-1}(R\Gamma_{\mathfrak{X}_{n}}(\mathcal{N}^{\cdot})))
\]
\[
\tilde{=}\varphi_{*}R\underline{\mathcal{H}om}_{\mathcal{D}_{\mathfrak{X}_{n}}^{(0,1)}}(\mathcal{M}^{\cdot},R\underline{\mathcal{H}om}_{\varphi^{-1}(\mathcal{D}_{\mathfrak{Y}_{n}}^{(0,1)})}(\mathcal{D}_{\mathfrak{Y}_{n}\leftarrow\mathfrak{X}_{n}}^{(0,1)},\varphi^{-1}(R\Gamma_{\mathfrak{X}_{n}}(\mathcal{N}^{\cdot}))))
\]
\[
\tilde{=}\varphi_{*}R\underline{\mathcal{H}om}_{\mathcal{D}_{\mathfrak{X}_{n}}^{(0,1)}}(\mathcal{M}^{\cdot},\varphi^{\sharp}\mathcal{N}^{\cdot})
\]
where, in both the second isomorphism and the last, we have used the
existence of an exact triangle 
\[
R\Gamma_{\mathfrak{X}_{n}}(\mathcal{N}^{\cdot})\to\mathcal{N}^{\cdot}\to\mathcal{K}^{\cdot}
\]
where $\mathcal{K}^{\cdot}\in D_{qcoh}^{b}(\mathcal{G}(\mathcal{D}_{\mathfrak{Y}_{n}}^{(0,1)}))$
is isomorphic to $Rj_{*}(\mathcal{N}^{\cdot}|_{\mathfrak{Y}_{n}\backslash\mathfrak{X}_{n}})$;
here $j:\mathfrak{Y}_{n}\backslash\mathfrak{X}_{n}\to\mathfrak{Y}_{n}$
is the inclusion. In particular, we have that $R\underline{\mathcal{H}om}(\mathcal{C}^{\cdot},\mathcal{K}^{\cdot})=0$
for any $\mathcal{C}^{\cdot}$ supported along $\mathfrak{X}_{n}$. 
\end{proof}
\begin{cor}
There is a canonical map 
\[
\text{tr}:\int_{\varphi}\mathcal{O}_{\mathfrak{X}_{n}}[d_{X}]\to\mathcal{O}_{\mathfrak{Y}_{n}}[d_{Y}]
\]
After taking inverse limit, we obtain a trace map ${\displaystyle \text{tr}:\int_{\varphi}\mathcal{O}_{\mathfrak{X}}[d_{X}]\to\mathcal{O}_{\mathfrak{Y}}[d_{Y}]}$.
If $\psi:\mathfrak{Y}\to\mathfrak{Z}$ is a smooth morphism, 

We also have a trace map 
\[
\text{tr}:\int_{\varphi}\mathcal{O}_{\mathfrak{X}}[d_{X}]\to\mathcal{O}_{\mathfrak{Y}}[d_{Y}]
\]
given by taking the inverse limit of the above maps; the same compatibility
holds for this trace as well. 
\end{cor}

\begin{proof}
The previous corollary gives an adjunction ${\displaystyle \int_{\varphi}\varphi^{\dagger}\to\text{Id}}$.
Since $\varphi^{\dagger}(\mathcal{O}_{\mathfrak{Y}_{n}})=\mathcal{O}_{\mathfrak{X}_{n}}[d_{X}-d_{Y}]$
we obtain the trace map via this adjunction. 
\end{proof}
Now, by factoring an arbitrary projective morphism as a closed immersion
followed by a smooth projective map, we obtain by composing the trace
maps a trace map for an arbitrary projective morphism. Arguing as
in the classical case (c.f. \cite{key-54}, section 2.10), we see
that this map is independent of the choice of the factorization. Therefore
we obtain 
\begin{thm}
Let $\varphi:\mathfrak{X}\to\mathfrak{Y}$ be a projective morphism.
Then we have a functorial morphism 
\[
\int_{\varphi}\mathbb{D}_{\mathfrak{X}}\mathcal{M}^{\cdot}\to\mathbb{D}_{\mathfrak{Y}}\int_{\varphi}\mathcal{M}^{\cdot}
\]
which is an isomorphism for $\mathcal{M}^{\cdot}\in D_{coh}^{b}(\mathcal{G}(\mathcal{D}_{\mathfrak{X}}^{(0,1)}))$.
Further, we have a functorial isomorphism
\[
R\underline{\mathcal{H}om}_{\mathcal{D}_{\mathfrak{Y}}^{(0,1)}}(\int_{\varphi}\mathcal{M}^{\cdot},\mathcal{N}^{\cdot})\tilde{\to}\varphi_{*}R\underline{\mathcal{H}om}_{\mathcal{D}_{\mathfrak{X}}^{(0,1)}}(\mathcal{M}^{\cdot},\varphi^{\dagger}\mathcal{N}^{\cdot})
\]
for all $\mathcal{M}^{\cdot}\in D_{coh}^{b}(\mathcal{G}(\mathcal{D}_{\mathfrak{X}}^{(0,1)}))$
and $\mathcal{N}^{\cdot}\in D_{coh}^{b}(\mathcal{G}(\mathcal{D}_{\mathfrak{Y}}^{(0,1)}))$. 
\end{thm}

\section{Applications}

In this section we put things together and give the statement and
proof of our generalization of Mazur's theorem for a mixed Hodge module.
We begin with a brief review of the pushforward operation in the world
of mixed Hodge modules. 

Let $X_{\mathbb{C}}$ be a smooth complex variety, and suppose that
$(\mathcal{M}_{\mathbb{C}},F^{\cdot},\mathcal{K}_{\mathbb{Q}},W_{\cdot})$
is a mixed Hodge module on $X_{\mathbb{C}}$. We won't attempt to
recall a complete definition here, instead referring the reader to\cite{key-15},\cite{key-16},
and the excellent survey \cite{key-56}. We will only recall that
$\mathcal{M}_{\mathbb{C}}$ is a coherent $\mathcal{D}$-module which
comes equipped with a good filtration $F^{\cdot}$, a weight filtration
$W_{\cdot}$, and $\mathcal{K}_{\mathbb{Q}}$ is a perverse sheaf
defined over $\mathbb{Q}$ which corresponds to $\mathcal{M}_{\mathbb{C}}$
under the Riemann-Hilbert correspondence. In this paper, our attention
is on the filtration $F^{\cdot}$ and we will mostly suppress the
other aspects of the theory. For the sake of notational convenience,
we will denote simply by $\mathcal{O}_{X_{\mathbb{C}}}$ the mixed
Hodge module whose underlying filtered $\mathcal{D}$-module is $\mathcal{O}_{X_{\mathbb{C}}}$
with its trivial filtration: $F^{i}(\mathcal{O}_{X_{\mathbb{C}}})=\mathcal{O}_{X_{\mathbb{C}}}$
for all $i\geq0$, while $F^{i}(\mathcal{O}_{X_{\mathbb{C}}})=0$
for $i<0$. 

Now let $\varphi:X_{\mathbb{C}}\to Y_{\mathbb{C}}$ be a morphism
of smooth complex varieties. By Nagata's compatification theorem,
combined with Hironaka's resolution of singularities, we can find
an open immersion $j:X_{\mathbb{C}}\to\overline{X}_{\mathbb{C}}$
into a smooth variety, whose compliment is a normal crossings divisor,
and a proper morphism $\overline{\varphi}:\overline{X}_{\mathbb{C}}\to Y_{\mathbb{C}}$,
with $\varphi=\overline{\varphi}\circ j$. 

Then, the following is one of the main results of \cite{key-16} (c.f.
theorem 4.3 and theorem 2.14)
\begin{thm}
Let $\varphi,\overline{\varphi},j$ be morphisms as above. 

1) There is a mixed Hodge module $(j_{\star}(\mathcal{M}_{\mathbb{C}}),F^{\cdot}j_{*}\mathcal{K}_{\mathbb{Q}},W_{\cdot})$,
whose underlying $\mathcal{D}$-module agrees with the usual pushforward
of $\mathcal{D}$-modules under $j$. This defines an exact functor
$j_{\star}:\text{MHM}(X_{\mathbb{C}})\to\text{MHM}(\overline{X}_{\mathbb{C}})$. 

2) There is an object of $D^{b}(\text{MHM}(Y_{\mathbb{C}}))$, $R\overline{\varphi}_{\star}(j_{\star}(\mathcal{M}_{\mathbb{C}}),F^{\cdot}j_{*}\mathcal{K}_{\mathbb{Q}},W_{\cdot})$,
whose underlying complex of filtered $\mathcal{D}$-modules agrees
with ${\displaystyle \int_{\overline{\varphi}}(j_{\star}\mathcal{M}_{\mathbb{C}})}$.
This object of $D^{b}(\text{MHM}(Y_{\mathbb{C}}))$ is, up to isomorphism,
independent of the choice of factorization $\varphi=\overline{\varphi}\circ j$.
Furthermore, the filtration on this complex is strict.
\end{thm}

The reason for stating the theorem this way is that, if $\varphi$
is not proper, the filtered pushforward ${\displaystyle \int_{\varphi}}$
of filtered $\mathcal{D}$-modules does not agree with the pushforward
of mixed Hodge modules. The issue appears already if $Y_{\mathbb{C}}$
is a point and $\mathcal{M}_{\mathbb{C}}=\mathcal{O}_{X_{\mathbb{C}}}$.
In that case, the pushforward $R\varphi_{\star}$ returns\footnote{up to a homological shift, and a re-indexing of the Hodge filtration}
Deligne's Hodge cohomology of $X_{\mathbb{C}}$, while ${\displaystyle {\displaystyle \int_{\varphi}}}$
returns the de Rham cohomology of $X_{\mathbb{C}}$ equipped with
the naive Hodge-to-de Rham filtration; these disagree, e.g., if $X_{\mathbb{C}}$
is affine. 

The construction of the extension $j_{\star}(\mathcal{M}_{\mathbb{C}})$
is, in general, quite deep, and relies on the detailed study of the
degenerations of Hodge structures given in \cite{key-60} and \cite{key-61}.
However, when $\mathcal{M}_{\mathbb{C}}=\mathcal{O}_{X_{\mathbb{C}}}$
is the trivial mixed Hodge module, one can be quite explicit: 
\begin{lem}
\label{lem:Hodge-filt-on-j_push}Let $j:X_{\mathbb{C}}\to\overline{X}_{\mathbb{C}}$
be an open immersion of smooth varieties, whose compliment is a normal
crossings divisor $D_{\mathbb{C}}$. Let $x\in X_{\mathbb{C}}$ be
a point, about which $D_{\mathbb{C}}$ is given by the equation $\{x_{1}\cdots x_{j}=0\}$.
Then as filtered $\mathcal{D}$-modules we have $j_{\star}\mathcal{O}_{X_{\mathbb{C}}}=(j_{*}(\mathcal{O}_{X_{\mathbb{C}}}),F^{\cdot})$
where $F^{l}(j_{*}(\mathcal{O}_{X_{\mathbb{C}}})):=F^{l}(\mathcal{D}_{X_{\mathbb{C}}})\cdot(x_{1}\cdots x_{j})^{-1}$. 

In particular, $F^{l}(j_{*}(\mathcal{O}_{X_{\mathbb{C}}}))$ is spanned
over $\mathcal{O}_{X_{\mathbb{C}}}$ by terms of the form $x_{1}^{-(i_{1}+1)}\cdots x_{j}^{-(i_{j}+1)}$
where ${\displaystyle \sum_{t=1}^{j}i_{t}\leq l}$.
\end{lem}

For a proof, see \cite{key-6}, section 8. This implies that the Hodge
cohomology of $X_{\mathbb{C}}$, as an object in the filtered derived
category of vector spaces, can be computed as ${\displaystyle \int_{\overline{\varphi}}j_{\star}\mathcal{O}_{X_{\mathbb{C}}}(d)[d]}$
where $\overline{\varphi}:\overline{X}_{\mathbb{C}}\to\{*\}$. Of
course, this can be checked directly by comparing the log de Rham
complex with the de Rham complex of a the filtered $\mathcal{D}$-module
$j_{\star}\mathcal{O}_{X_{\mathbb{C}}}$. 

Combining \thmref{Mazur!}with \corref{proper-push-over-W(k)} gives:
\begin{prop}
1) Let $\varphi:\mathfrak{X}\to\mathfrak{Y}$ be a projective morphism,
and let $\mathfrak{D}\subset\mathfrak{X}$ be a (possibly empty) normal
crossings divisor. Let ${\displaystyle j_{\star}D(\mathcal{O}_{\mathfrak{X}})}$
be the gauge of \exaref{Integral-j} on $\mathfrak{X}$. Suppose that
each $\mathcal{H}^{i}({\displaystyle \int_{\varphi}(j_{\star}\mathcal{O}_{\mathfrak{X}})^{-\infty}})$
is a $p$-torsion-free $\widehat{\mathcal{D}}_{\mathfrak{Y}}^{(0)}$-module,
and that each $\mathcal{H}^{i}({\displaystyle (\int_{\varphi}{\displaystyle j_{\star}D(\mathcal{O}_{\mathfrak{X}}})}\otimes_{W(k)}^{L}k)\otimes_{D(k)}^{L}k[f])$
is $f$-torsion-free. Then each $\mathcal{H}^{i}{\displaystyle (\int_{\varphi}{\displaystyle j_{\star}D(\mathcal{O}_{\mathfrak{X}}}}))$
is a standard gauge on $\mathfrak{Y}$. 

2) Let ${\displaystyle j_{!}D(\mathcal{O}_{\mathfrak{X}}):=\mathbb{D}_{\mathfrak{X}}j_{\star}D(\mathcal{O}_{\mathfrak{X}})}$.
The same conclusion holds for $j_{!}D(\mathcal{O}_{\mathfrak{X}})$. 
\end{prop}

When $\mathfrak{Y}$ is a point this recovers the log-version of Mazur's
theorem, as discussed in Ogus' paper \cite{key-18}. 

Now let $R$ be a finite type algebra over $\mathbb{Z}$ so that there
exists smooth (over $R$) models $X_{R},Y_{R}$ for $X_{\mathbb{C}}$
and $Y_{\mathbb{C}}$, respectively, and a projective morphism $\varphi:X_{R}\to Y_{R}$
whose base change to $\mathbb{C}$ is the original morphism. We may
suppose the divisor $D_{\mathbb{C}}$ is defined over $R$ as well. 

Let $\mathcal{D}_{X_{R}}^{(0)}$ be the level zero differential operators
over $X_{R}$, equipped with the symbol filtration; let the associated
Rees algebra be $\mathcal{R}(\mathcal{D}_{X_{R}}^{(0)})$ (as usual
we will use $f$ for the Rees parameter). Since $\text{Rees}(j_{*}\mathcal{O}_{U_{\mathbb{C}}})$
is a coherent $\mathcal{R}(\mathcal{D}_{X_{\mathbb{C}}})$-module,
we can by generic flatness choose a flat model for $\text{Rees}(j_{*}\mathcal{O}_{U_{\mathbb{C}}})$;
in fact, we can describe it explicitly as follows: if $D_{R}$ is
given, in local coordinates, by $\{x_{1}\cdots x_{j}=0\}$, then we
may consider 
\[
\mathcal{D}_{X_{R}}^{(0)}\cdot x_{1}^{-1}\cdots x_{j}^{-1}\subset j_{*}\mathcal{O}_{U_{R}}
\]
with the filtration inherited from the symbol filtration on $\mathcal{D}_{X_{R}}^{(0)}$.
The Rees module of this filtered $\mathcal{D}_{X_{R}}^{(0)}$-module
is a flat $R$-model for $\text{Rees}(j_{*}\mathcal{O}_{U_{\mathbb{C}}})$.
Let us call this sheaf ${\displaystyle j_{\star}\mathcal{O}_{U_{R}}[f]}$;
we will denote the associated filtered $\mathcal{D}_{X_{R}}^{(0)}$-module
by ${\displaystyle j_{\star}\mathcal{O}_{U_{R}}}$. Then, localizing
$R$ if necessary, we have that 
\[
\int_{\varphi}j_{\star}\mathcal{O}_{U_{R}}[f]
\]
is an $f$-torsion-free complex inside $D_{coh}^{b}(\mathcal{D}_{Y_{R}}^{(0)}-\text{mod})$
(since it becomes $f$-torsion-free after base change to $\mathbb{C}$,
as remarked above). By generic flatness, we may also suppose (again,
localizing $R$ if necessary), that each cohomology sheaf ${\displaystyle \mathcal{H}^{i}(\int_{\varphi}j_{\star}\mathcal{O}_{U_{R}})}$
is flat over $R$. Let $k$ be a perfect field of characteristic $p>0$,
for which there is a morphism $R\to W(k)$ (so that $R/p\to k$)\footnote{If we extend $R$ so that it is smooth over $\mathbb{Z}$, then any
map $R/p\to k$ lifts to $R\to W(k)$}. Then, combining this discussion with the previous proposition, we
obtain 
\begin{cor}
\label{cor:Mazur-for-Hodge-1}Let $\mathfrak{X}$ be the formal completion
of $X_{R}\times_{R}W(k)$, and similarly for $\mathfrak{Y}$. Then
each gauge $\mathcal{H}^{i}({\displaystyle (\int_{\varphi}{\displaystyle j_{\star}D(\mathcal{O}_{\mathfrak{X}})}}))$
is a standard, coherent, $F^{-1}$-gauge on $\mathfrak{Y}$. There
is an isomorphism 
\[
\mathcal{H}^{i}(({\displaystyle (\int_{\varphi}{\displaystyle j_{\star}\mathcal{O}_{\mathfrak{X}}[f,v]}})\otimes_{W(k)}^{L}k)\otimes_{D(k)}^{L}k[f])\tilde{\to}F^{*}\mathcal{H}^{i}(\int_{\varphi}j_{\star}\mathcal{O}_{U_{R}}[f]\otimes_{R}^{L}k)
\]
in $\mathcal{G}(\mathcal{R}(\mathcal{D}_{X}^{(1)}))$. In particular,
the Hodge filtration on ${\displaystyle \mathcal{H}^{i}({\displaystyle (\int_{\varphi}{\displaystyle j_{\star}D(\mathcal{O}_{\mathfrak{X}}})}))^{\infty}/p}$
is the Frobenius pullback of the Hodge filtration on ${\displaystyle \mathcal{H}^{i}(\int_{\varphi}j_{\star}\mathcal{O}_{U_{R}}\otimes_{R}^{L}k)}$.
The same holds if we replace ${\displaystyle j_{\star}D(\mathcal{O}_{\mathfrak{X}}})$
by ${\displaystyle j_{!}D(\mathcal{O}_{\mathfrak{X}}})$. The same
statement holds for the pushforward of ${\displaystyle \mathcal{H}^{i}(\int_{\varphi}j_{\star}\mathcal{O}_{U_{R}})}$
under another proper morphism $\psi:Y\to Z$. 
\end{cor}

\begin{proof}
The displayed isomorphism follows immediately from \thmref{Hodge-Filtered-Push}.
Since ${\displaystyle \int_{\varphi}j_{\star}\mathcal{O}_{U_{R}}[f]}$
has $f$-torsion free cohomology sheaves, which are also flat over
$R$, we deduce that ${\displaystyle {\displaystyle ((\int_{\varphi}{\displaystyle j_{\star}\mathcal{O}_{\mathfrak{X}}[f,v]}})\otimes_{W(k)}^{L}k})\otimes_{D(k)}^{L}k[f]$
has $f$-torsion free cohomology sheaves. Comparing the description
of the Hodge filtration on ${\displaystyle ({\displaystyle j_{\star}\mathcal{O}_{\mathfrak{X}}[f,v]}})^{\infty}/p$
with the result of \lemref{Hodge-filt-on-j_push}, the result now
follows from \thmref{F-Mazur}.
\end{proof}
Let us give some first applications of these results. 

Suppose that $X_{\mathbb{C}}$ is an arbitrary (possibly singular)
quasi-projective variety. Let $V_{\mathbb{C}}$ be a smooth quasi-projective
variety such that there is a closed embedding $X_{\mathbb{C}}\to V_{\mathbb{C}}$,
and let $\overline{V}_{\mathbb{C}}$ be a projective compatification
of $V_{\mathbb{C}}$ (i.e., $\overline{V}_{\mathbb{C}}\backslash V_{\mathbb{C}}$
is a normal crossings divisor). Let $U_{\mathbb{C}}\subset X_{\mathbb{C}}$
be an affine open \emph{smooth} subset. Let $\varphi:\tilde{X}_{\mathbb{C}}\to X_{\mathbb{C}}$
denote a resolution of singularities so that $\varphi$ is an isomorphism
over $U_{\mathbb{C}}$ and $\varphi^{-1}(X_{\mathbb{C}}\backslash U_{\mathbb{C}})$
is a normal crossings divisor $\tilde{D}_{\mathbb{C}}\subset\tilde{X}_{\mathbb{C}}$.
The decomposition theorem for Hodge modules implies that the complex
${\displaystyle \int_{\varphi}\mathcal{O}_{\tilde{X}_{\mathbb{C}}}}\in D^{b}(\text{MHM}_{X})$
is quasi-isomorphic to the direct sum of its cohomology sheaves, and
that each such sheaf is a direct sum of simple, pure Hodge modules. 

Therefore, if $j:U_{\mathbb{C}}\to X_{\mathbb{C}}$ (resp. $j':U_{\mathbb{C}}\to\tilde{X}_{\mathbb{C}}$)
denotes the inclusion, then the image of the natural map 
\[
\mathcal{H}^{0}({\displaystyle \int_{\varphi}\mathcal{O}_{\tilde{X}_{\mathbb{C}}}})\to\mathcal{H}^{0}(\int_{\varphi}j'_{\star}\mathcal{O}_{U_{\mathbb{C}}})\tilde{\to}\mathcal{H}^{0}(j_{\star}\mathcal{O}_{U_{\mathbb{C}}})
\]
is the Hodge module $\text{IC}_{X}$; indeed, ${\displaystyle \mathcal{H}^{0}({\displaystyle \int_{\varphi}\mathcal{O}_{\tilde{X}_{\mathbb{C}}}})}=\text{IC}_{X}\oplus\mathcal{M}$
where $\mathcal{M}$ is a pure Hodge module supported on $X_{\mathbb{C}}\backslash U_{\mathbb{C}}$;
its image in $\mathcal{H}^{0}({\displaystyle j_{\star}\mathcal{O}_{U_{\mathbb{C}}})}$
is therefore isomorphic to $\text{IC}_{X}$ (as a Hodge module, and
so in particular as a filtered $\mathcal{D}$-module). 

Now let $\overline{X}_{\mathbb{C}}$ denote the closure of $X_{\mathbb{C}}$
in $\overline{V}_{\mathbb{C}}$, and let $\varphi:\tilde{\overline{X}}_{\mathbb{C}}\to\overline{X}_{\mathbb{C}}$
be a resolution of singularities, whose restriction to $X_{\mathbb{C}}\subset\overline{X}_{\mathbb{C}}$
is isomorphic to $\varphi:\tilde{X}_{\mathbb{C}}\to X_{\mathbb{C}}$,
and so that the inverse image of $\overline{X}_{\mathbb{C}}\backslash X_{\mathbb{C}}$
is a normal crossings divisor (we can modify $\varphi$ if necessary
to ensure that this happens). Let $i:X_{\mathbb{C}}\to\overline{X}_{\mathbb{C}}$
and $i':\tilde{X}_{\mathbb{C}}\to\tilde{\overline{X}}_{\mathbb{C}}$
denote the inclusions. Since Hodge modules on $X_{\mathbb{C}}$ are,
by definition, Hodge modules on $V_{\mathbb{C}}$ which are supported
on $X_{\mathbb{C}}$, the fact that $\overline{V}_{\mathbb{C}}\backslash V_{\mathbb{C}}$
is a divisor implies that $i_{*}$ is an exact functor on the category
of mixed Hodge modules. Therefore the image of the natural map 
\[
\mathcal{H}^{0}(\int_{\varphi}i'_{\star}\mathcal{O}_{\tilde{X}_{\mathbb{C}}})\tilde{=}i_{\star}\mathcal{H}^{0}({\displaystyle \int_{\varphi}\mathcal{O}_{\tilde{X}_{\mathbb{C}}}})\to i_{\star}\mathcal{H}^{0}(\int_{\varphi}j'_{\star}\mathcal{O}_{U_{\mathbb{C}}})\tilde{=}\mathcal{H}^{0}(i\circ j)_{\star}\mathcal{O}_{U_{\mathbb{C}}}
\]
is isomorphic to $i_{\star}(\text{IC}_{X})$ (again, as a Hodge module,
and so in particular as a filtered $\mathcal{D}$-module). 

As above, we now select a finite type $\mathbb{Z}$-algebra $R$ so
that everything in sight is defined and flat over $R$, and let $R\to W(k)$
for some perfect $k$ of characteristic $p>0$. Let $\tilde{\mathfrak{\overline{X}}}\to\mathfrak{\overline{X}}\subset\overline{\mathcal{V}}$
be the formal completion of $\tilde{\overline{X}}_{R}\times_{R}W(k)\to\overline{X}_{R}\times_{R}W(k)\subset\overline{V}_{R}\times_{R}W(k)$.
Abusing notation slightly we'll also denote by $\varphi$ the composed
map $\tilde{\mathfrak{\overline{X}}}\to\widehat{\overline{\mathcal{V}}}$. 
\begin{cor}
\label{cor:Mazur-for-IC}1) The image of the map 
\[
\mathcal{H}^{0}(\int_{\varphi}i'_{\star}D(\mathcal{O}_{\tilde{\mathfrak{X}}}))\to\mathcal{H}^{0}(\int_{\varphi}(i'\circ j')_{*}D(\mathcal{O}_{\mathfrak{U}}))
\]
defines a coherent, standard $F^{-1}$-gauge on $\widehat{\mathbb{P}^{n}}$,
denoted $\text{IC}_{\mathfrak{X}}$. The $\widehat{\mathcal{D}}_{\overline{\mathcal{V}}}^{(0)}$-module
$\text{IC}_{\mathfrak{X}}^{-\infty}$ is isomorphic to the $p$-adic
completion of $\text{IC}_{X_{R}}\otimes_{R}W(k)$, where $\text{IC}_{X_{R}}$
is an $R$-model for $\text{IC}_{X_{\mathbb{C}}}$. The Hodge filtration
on the $\mathcal{D}_{\overline{V}_{k}}^{(1)}$-module $\widehat{\text{IC}_{\mathfrak{X}}^{\infty}}/p\tilde{=}F^{*}\text{IC}_{\mathfrak{X}}^{-\infty}/p$
is equal to the Frobenius pullback of the Hodge filtration on $\text{IC}_{\mathfrak{X}}^{-\infty}/p\tilde{=}\text{IC}_{X_{R}}\otimes_{R}k$
coming from the Hodge filtration on $\text{IC}_{X_{R}}$. 

2) The intersection cohomology groups $\text{IH}^{i}(X_{R})\otimes_{R}W(k):=\mathbb{H}_{dR}^{i}(\text{IC}_{X_{R}})\otimes_{R}W(k)$
satisfy the conclusions of Mazur's theorem; as in \thmref{Mazur-for-IC-Intro}
\end{cor}

\begin{proof}
Since the displayed map is a map of coherent gauges, the image, $i_{\star}\text{IC}_{\mathfrak{X}}$,
is a coherent gauge. Since both ${\displaystyle i'_{\star}D(\mathcal{O}_{\tilde{\mathfrak{X}}})}$
and $(i'\circ j')_{*}D(\mathcal{O}_{\mathfrak{U}}))$ are $F^{-1}$gauges,
and the natural map $i'_{\star}D({\displaystyle \mathcal{O}_{\tilde{\mathfrak{X}}})\to(i'\circ j')_{*}D(\mathcal{O}_{\mathfrak{U}})}$
is $F^{-1}$-equivariant, the same is true of the displayed map, and
so $i_{\star}\text{IC}_{\mathfrak{X}}$ is an $F^{-1}$-gauge. By
\propref{push-and-complete-for-D} (and the exactness of the functor
$\mathcal{M}\to\mathcal{M}^{-\infty}$) we have that the image of
\[
\mathcal{H}^{0}(\int_{\varphi}i'_{\star}D(\mathcal{O}_{\tilde{\mathfrak{X}}}))^{-\infty}\to\mathcal{H}^{0}(\int_{\varphi}(i'\circ j')_{*}D(\mathcal{O}_{\mathfrak{U}}))^{-\infty}
\]
is equal to the image of 
\[
\mathcal{H}^{0}(\int_{\varphi}(i_{\star}\mathcal{O}_{\tilde{\mathfrak{X}}})^{-\infty})\to\mathcal{H}^{0}\int_{\varphi}((i'\circ j')_{*}D(\mathcal{O}_{\mathfrak{U}}))^{-\infty}
\]
in the category of $\widehat{\mathcal{D}}_{\mathfrak{X}}^{(0)}$-modules.
On the other hand, we have $R$-flat filtered $\mathcal{D}_{X_{R}}^{(0)}$-modules
${\displaystyle \mathcal{H}^{0}(\int_{\varphi}i_{\star}\mathcal{O}_{\tilde{X}_{R}})}$
and ${\displaystyle \mathcal{H}^{0}(\int_{\varphi}(i'\circ j')_{*}\mathcal{O}_{U_{R}})}$
such that the $p$-adic completion of $\mathcal{H}^{0}(\int_{\varphi}i_{\star}\mathcal{O}_{\tilde{X}_{R}}){\displaystyle \otimes_{R}W(k)}$
is ${\displaystyle \mathcal{H}^{0}(\int_{\varphi}(i_{\star}\mathcal{O}_{\tilde{\mathfrak{X}}})^{-\infty})}$,
and the $p$-adic completion of ${\displaystyle \mathcal{H}^{0}(\int_{\varphi}(i'\circ j')_{*}\mathcal{O}_{U_{R}})}$
is ${\displaystyle \mathcal{H}^{0}\int_{\varphi}((i'\circ j')_{*}D(\mathcal{O}_{\mathfrak{U}}))^{-\infty}}$,
and, after localizing $R$ if necessary, we may further suppose that
the kernel of the map 
\begin{equation}
{\displaystyle \mathcal{H}^{0}(\int_{\varphi}i_{\star}\mathcal{O}_{\tilde{X}_{R}})}\to\mathcal{H}^{0}(\int_{\varphi}(i'\circ j')_{*}\mathcal{O}_{U_{R}})\label{eq:natural-map-over-R}
\end{equation}
is a summand (in the category of filtered $\mathcal{D}_{X_{R}}^{(0)}$-modules)
of ${\displaystyle {\displaystyle \mathcal{H}^{0}(\int_{\varphi}\mathcal{O}_{\tilde{X}_{R}})}}$
(as this is true over $\mathbb{C}$). Thus the image is flat over
$R$, and so its $p$-adic completion is $p$-torsion-free; therefore
$\text{IC}_{\mathfrak{X}}^{-\infty}$ is $p$-torsion-free, as is
$\text{IC}_{\mathfrak{X}}^{\infty}$ (since $\text{IC}_{\mathfrak{X}}^{-\infty}$
is an $F^{-1}$-gauge; c.f. the proof of \thmref{F-Mazur}) Further,
the map \eqref{natural-map-over-R} is strict with respect to the
Hodge filtration, and so the same is true after taking reduction mod
$p$ and applying $F^{*}$. It follows that $\text{IC}_{\mathfrak{X}}^{\infty}/p/v$
is $f$-torsion-free. 

Thus by \propref{Baby-Mazur}, we see that $\text{IC}_{\mathfrak{X}}$
is a standard gauge; the statement about the Hodge filtration follows
from \corref{Mazur-for-Hodge-1}. This proves part $1)$, and part
$2)$ follows from taking the pushforward to a point. 
\end{proof}
\begin{rem}
The construction above involved a few auxiliary choices- namely, the
ring $R$ and the resolution $\tilde{X}_{R}$. However, any two resolutions
of singularities can be dominated by a third. Therefore, after possibly
localizing $R$, any two definitions of $\text{IC}_{X_{R}}$ agree.
Further, we if we have an inclusion of rings $R\to R'$ then $\text{IC}_{X_{R}}\otimes_{R}R'=\text{IC}_{X_{R'}}$.
Therefore we have $\mathbb{H}_{dR}^{i}(\text{IC}_{X_{R}})\otimes_{R}R'\tilde{\to}\mathbb{H}_{dR}^{i}(\text{IC}_{X_{R'}})$
when both are flat. Since any two finite-type $\mathbb{Z}$-algebras
can be embedded into a third, we also obtain a comparison for any
two such algebras.
\end{rem}

Now suppose $X_{\mathbb{C}}$ is a smooth (quasiprojective) scheme,
and let $i:Y_{\mathbb{C}}\to X_{\mathbb{C}}$ be a closed immersion;
here, $Y_{\mathbb{C}}$ can be singular; let $j:X_{\mathbb{C}}\backslash Y_{\mathbb{C}}\to X_{\mathbb{C}}$
be the open immersion. Now, let $\tilde{j}:X_{\mathbb{C}}\to\overline{X}_{\mathbb{C}}$
be a smooth proper compactification of $X_{\mathbb{C}}$, so that
$\overline{X}_{\mathbb{C}}\backslash X_{\mathbb{C}}$ is a normal
crossings divisor. Choose flat $R$ models for everything in sight.
Then we have 
\begin{cor}
\label{cor:Mazur-for-Ordinary}For each $i$, the Hodge cohomology
group $H^{i}(Y_{\mathbb{C}})$ admits a flat model $H^{i}(Y_{R})$
(as a filtered vector space). Let $k$ is a perfect field such that
$R\to W(k)$. Then there is a standard $F^{-1}$-gauge $H^{i}(Y_{R})_{W(k)}^{\cdot}$
such that $H^{i}(Y_{R})_{W(k)}^{-\infty}\tilde{=}H^{i}(Y_{R})\otimes_{R}W(k)$,
and such that the Hodge filtration on $H^{i}(Y_{R})_{W(k)}^{\infty}/p$
agrees with the Frobenius pullback of the Hodge filtration on $H^{i}(Y_{R})\otimes_{R}k$.
In particular, there is a Frobenius-linear isomorphism of $H^{i}(Y_{R})_{W(k)}[p^{-1}]$
for which the Hodge filtration on $H^{i}(Y_{R})_{W(k)}$ satisfies
the conclusions of Mazur's theorem. The same holds for the compactly
supported Hodge cohomology $H_{c}^{i}(Y_{\mathbb{C}})$. 
\end{cor}

\begin{proof}
As the usual Hodge cohomology and the compactly supported Hodge cohomology
are interchanged under applying the filtered duality functor, it suffices
to deal with the case of the compactly supported cohomology. Let us
recall how to define this in the language of mixed Hodge modules.
We have the morphism
\[
Rj_{!}(\mathcal{O}_{X_{\mathbb{C}}\backslash Y_{\mathbb{C}}})\to\mathcal{O}_{X_{\mathbb{C}}}
\]
in the category of mixed Hodge modules (where $\mathcal{O}$ has its
usual structure as the trivial mixed Hodge module). The cone of this
map is, by definition, the complex of mixed Hodge modules representing
the unit object on $Y_{\mathbb{C}}$; we denote it by $\mathbb{I}_{Y_{\mathbb{C}}}$.
Then we have 
\[
H_{c}^{i}(Y_{\mathbb{C}})=\int_{\varphi}^{d+i}R\tilde{j}_{!}\mathbb{I}_{Y_{\mathbb{C}}}=\int_{\varphi}^{d+i}R\tilde{j}_{!}(\text{cone}(Rj_{!}(\mathcal{O}_{X_{\mathbb{C}}\backslash Y_{\mathbb{C}}})\to\mathcal{O}_{X_{\mathbb{C}}}))
\]
\[
\tilde{=}\int_{\varphi}^{d+i}\text{cone}(R(\tilde{j}\circ j)_{!}(\mathcal{O}_{X_{\mathbb{C}}\backslash Y_{\mathbb{C}}})\to R\tilde{j}_{!}\mathcal{O}_{X_{\mathbb{C}}})
\]
Now, after spreading out both $R(\tilde{j}\circ j)_{!}(\mathcal{O}_{X_{\mathbb{C}}\backslash Y_{\mathbb{C}}})$
and $R\tilde{j}_{!}\mathcal{O}_{X_{\mathbb{C}}})$ over $R$, we can
apply \corref{Mazur-for-Hodge-1}. 
\end{proof}
\begin{rem}
The previous two corollaries also hold for quasiprojective varieties
defined over $\overline{\mathbb{Q}}$. Although the theory of mixed
Hodge modules only exists over $\mathbb{C}$, its algebraic consequences,
such as the strictness of the pushforward of modules of the form $j_{\star}(\mathcal{O}_{X})$,
hold over any field of characteristic $0$. So the above results go
through in this case as well.
\end{rem}

Finally, we wish to give some relations of the theory of this paper
to the Hodge structure of the local cohomology sheaves $\mathcal{H}_{Y_{\mathbb{C}}}^{i}(\mathcal{O}_{X_{\mathbb{C}}})$,
as developed in {[}MP1{]}, {[}MP2{]}. Here, $X_{\mathbb{C}}$ is a
smooth affine variety and $Y_{\mathbb{C}}\subset X_{\mathbb{C}}$
is a subscheme defined by $(Q_{1},\dots,Q_{r})$. In this case, the
nontrivial sheaf is
\[
\mathcal{H}_{Y_{\mathbb{C}}}^{r}(\mathcal{O}_{X_{\mathbb{C}}})\tilde{=}\mathcal{O}_{X_{\mathbb{C}}}[Q_{1}^{-1}\cdots Q_{r}^{-1}]/\sum_{i=1}^{r}\mathcal{O}_{X_{\mathbb{C}}}\cdot Q_{1}^{-1}\cdots\widehat{(Q_{i}^{-1})}\cdots Q_{r}^{-1}
\]
where $\widehat{?}$ stands for ``omitted.'' As above, these sheaves
admits a Hodge structure via 
\[
\mathcal{H}_{Y_{\mathbb{C}}}^{i}(\mathcal{O}_{X_{\mathbb{C}}})\tilde{=}\mathcal{H}^{i}(\int_{\varphi}\int_{j'}\mathcal{O}_{U_{\mathbb{C}}})
\]
where $\varphi:\tilde{X}_{\mathbb{C}}\to X_{\mathbb{C}}$ is a resolution
of singularities such that $\varphi^{-1}(Y_{\mathbb{C}})$ is a normal
crossings divisor; and $j':U_{\mathbb{C}}\to\tilde{X}_{\mathbb{C}}$
is the inclusion. The resulting Hodge filtration is independent of
the choice of the resolution. Taking $R$-models for everything in
sight at above, we obtain a filtered $\mathcal{D}_{X_{R}}^{(0)}$-module
${\displaystyle \mathcal{H}^{i}(\int_{\varphi}j'_{\star}\mathcal{O}_{U_{R}})}$
which (localizing $R$ if necessary) is a flat $R$-model for $\mathcal{H}_{Y_{\mathbb{C}}}^{i}(\mathcal{O}_{X_{\mathbb{C}}})$. 

Now let $R\to W(k)$, and let $\mathfrak{X}$, $\tilde{\mathfrak{X}}$,
etc. be the formal completions of the base-change to $W(k)$ as usual.
Then we have a gauge 
\[
\mathcal{M}_{Y}:=\mathcal{H}^{i}(\int_{\varphi}j'_{\star}D(\mathcal{O}_{\mathfrak{U}}))
\]
which satisfies $\mathcal{M}_{Y}^{-\infty}={\displaystyle \mathcal{H}^{i}(\int_{\varphi}(j'_{\star}\mathcal{O}_{\mathfrak{U}_{W(k)}})^{-\infty})}$. 
\begin{lem}
\label{lem:injectivity-for-local-coh}Let $\widehat{\mathcal{H}_{\mathfrak{Y}}^{i}(\mathcal{O}_{\mathfrak{X}})}:=\mathcal{H}^{i}(Rj_{*}\mathcal{O}_{\mathfrak{U}})$.
(This is simply the $p$-adic completion of the $i$th algebraic  local
cohomology of $\mathfrak{X}$ along $\mathfrak{Y}$). Then the natural
map 
\[
\mathcal{M}_{Y}^{-\infty}\to\widehat{\mathcal{H}_{\mathfrak{Y}}^{i}(\mathcal{O}_{\mathfrak{X}})}
\]
is injective. If $F$ is a lift of Frobenius, the natural map $F^{*}\mathcal{M}_{Y}^{-\infty}\to\widehat{\mathcal{H}_{\mathfrak{Y}}^{i}(\mathcal{O}_{\mathfrak{X}})}$
is also injective.
\end{lem}

\begin{proof}
We have the Hodge filtration on ${\displaystyle \mathcal{H}^{i}(\int_{\varphi}j'_{\star}\mathcal{O}_{U_{R}})}$,
which is is a filtration by coherent $\mathcal{O}_{X_{R}}$-modules;
base changing to $W(k)$ yields a Hodge filtration on ${\displaystyle \mathcal{H}^{i}(\int_{\varphi}j'_{\star}\mathcal{O}_{U_{W(k)}})}$.
The map in question is the $p$-adic completion of the natural map
\[
\mathcal{H}^{i}(\int_{\varphi}j'_{\star}\mathcal{O}_{U_{W(k)}})\to\mathcal{H}_{Y_{W(k)}}^{i}(\mathcal{O}_{X_{W(k)}})
\]
and the right hand module also has a Hodge filtration, which is simply
the restriction of the Hodge filtration on $\mathcal{H}_{Y_{B}}^{r}(\mathcal{O}_{X_{B}})$
where $B=\text{Frac}(W(k))$. So the proof proceeds in an essentially
identical manner to \lemref{Injectivity-of-completion}. 
\end{proof}
Now, fix an integer $m\geq0$. Let us explain how to use this gauge
to obtain an arithmetic description of the Hodge filtration, up to
level $m$. Since $m$ is fixed, we may, after localizing $R$ as
needed, suppose that the image of the map ${\displaystyle \mathcal{H}^{i}(\int_{\varphi}j'_{\star}\mathcal{O}_{U_{R}})})\to\mathcal{H}_{Y_{R}}^{i}(\mathcal{O}_{X_{R}})$
is equal to $F^{m}(\mathcal{H}_{Y_{F}}^{i}(\mathcal{O}_{X_{F}}))\cap\mathcal{H}_{Y_{R}}^{i}(\mathcal{O}_{X_{R}})$.
In particular the map 
\[
F^{m}({\displaystyle \mathcal{H}^{i}(\int_{\varphi}j'_{\star}\mathcal{O}_{U_{R}})})\otimes_{R}k\to\mathcal{H}_{Y_{k}}^{i}(\mathcal{O}_{X_{k}})
\]
is injective; under the isomorphism $F^{*}\mathcal{H}_{Y_{k}}^{i}(\mathcal{O}_{X_{k}})\tilde{\to}\mathcal{H}_{Y_{k}}^{i}(\mathcal{O}_{X_{k}})$,
we also obtain an injection $F^{*}(F^{m}({\displaystyle \mathcal{H}^{i}(\int_{\varphi}j'_{\star}\mathcal{O}_{U_{R}})}))\otimes_{R}k\to\mathcal{H}_{Y_{k}}^{i}(\mathcal{O}_{X_{k}})$.
Then
\begin{prop}
\label{prop:Hodge-for-local-coh!}Let the be notation as above. We
have that the image of $\{g\in F^{*}\mathcal{M}_{Y}^{-\infty}|p^{j}g\in\mathcal{M}_{Y}^{-\infty}\}$
in $\mathcal{H}_{Y_{k}}^{i}(\mathcal{O}_{X_{k}})$ is exactly $F^{*}(F^{j}({\displaystyle \mathcal{H}^{i}(\int_{\varphi}j'_{\star}\mathcal{O}_{U_{R}})}))\otimes_{R}k)$.
For each $0\leq j\leq m$, this is also the image of $\{g\in\widehat{\mathcal{H}_{\mathfrak{Y}}^{i}(\mathcal{O}_{\mathfrak{X}})}|p^{j}g\in\mathcal{M}_{Y}^{-\infty}\}$. 
\end{prop}

\begin{proof}
By construction $\mathcal{M}_{Y}$ is a standard, coherent, $F^{-1}$-gauge
of index $0$ (this can be easily seen as the Hodge filtration is
concentrated in degrees $\geq0$). Therefore, we have $\widehat{\mathcal{M}_{Y}^{\infty}}\tilde{=}F^{*}\mathcal{M}_{Y}$,
and by the previous lemma $F^{*}\mathcal{M}_{Y}\to F^{*}\widehat{\mathcal{H}_{\mathfrak{Y}}^{i}(\mathcal{O}_{\mathfrak{X}})}\tilde{\to}\widehat{\mathcal{H}_{\mathfrak{Y}}^{i}(\mathcal{O}_{\mathfrak{X}})}$
is injective. Since $\mathcal{M}_{Y}$ is standard of index $0$,
we have 
\[
\mathcal{M}_{Y}^{j}=\{m\in\mathcal{M}_{Y}^{\infty}|p^{j}m\in f_{\infty}(\mathcal{M}_{Y}^{0})\}
\]
Note that if $j\leq0$, this means $\mathcal{M}_{Y}^{j}\tilde{=}\mathcal{M}_{Y}^{-\infty}$,
and the map 
\[
\eta_{i}:\mathcal{M}_{Y}^{j}\xrightarrow{f_{\infty}}\mathcal{M}_{Y}^{\infty}\xrightarrow{\widehat{?}}F^{*}\mathcal{M}_{Y}\to\widehat{\mathcal{H}_{\mathfrak{Y}}^{i}(\mathcal{O}_{\mathfrak{X}})}
\]
is simply $p^{-j}$ times the injection $\mathcal{M}_{Y}^{-\infty}\to\widehat{\mathcal{H}_{\mathfrak{Y}}^{i}(\mathcal{O}_{\mathfrak{X}})}$,
and is therefore injective by \lemref{injectivity-for-local-coh}
. If $j>0$, then $p^{j}\cdot\eta_{j}=\eta_{0}\circ v^{j}$ is injective
for the same reason, and so $\eta_{j}$ is injective since everything
in sight is $p$-torsion -free. Thus the entire gauge embeds into
$\widehat{\mathcal{H}_{\mathfrak{Y}}^{i}(\mathcal{O}_{\mathfrak{X}})}$
and the first result follows. 

For the second result, consider the standard gauge $\mathcal{N}_{Y}$
defined by $\mathcal{N}_{Y}^{j}:=\{m\in\widehat{\mathcal{H}_{\mathfrak{Y}}^{i}(\mathcal{O}_{\mathfrak{X}})}|p^{j}m\in\mathcal{M}_{Y}^{0}\}$
(the actions of $f$ and $v$ are inclusion and multiplication by
$p$ as usual). We have the natural injection for each $j$ $\mathcal{M}_{Y}^{j}\to\mathcal{N}_{Y}^{j}$,
which yields a morphism of gauges $\mathcal{M}_{Y}\to\mathcal{N}_{Y}$.
Let us show that for $j\leq m$ the map $\psi:\mathcal{M}_{Y}^{j}\to\mathcal{N}_{Y}^{j}$
is an isomorphism. For $j\leq0$ this is clear by definition, so suppose
it is true for some $j-1\leq m-1$. For any $j$ let $\mathcal{M}_{Y,0}^{j}:=\mathcal{M}_{Y}^{j}/p$
and similarly define $\mathcal{N}_{Y,0}^{j}$. 

We first claim that the isomorphism $\psi:\mathcal{M}_{Y,0}^{j-1}\to\mathcal{N}_{Y,0}^{j-1}$
induces an isomorphism $\text{ker}(f:\mathcal{M}_{Y,0}^{j-1}\to\mathcal{M}_{Y,0}^{j})\tilde{\to}\text{ker}(f:\mathcal{N}_{Y,0}^{j-1}\to\mathcal{N}_{Y,0}^{j})$.
Indeed, we have that $\mathcal{M}_{Y,0}^{j-1}/\text{ker}(f)\tilde{=}F^{j-1}(\mathcal{M}_{Y,0}^{\infty})$
and $\mathcal{N}_{Y,0}^{j-1}/\text{ker}(f)\tilde{=}F^{j-1}(\mathcal{N}_{Y,0}^{\infty})$
(as $\mathcal{M}_{Y}$ and $\mathcal{N}_{Y}$ are standard gauges).
Further, the composed morphism 
\[
F^{j-1}(\mathcal{M}_{Y,0}^{\infty})\to F^{j-1}(\mathcal{N}_{Y,0}^{\infty})\to\mathcal{H}_{Y_{k}}^{i}(\mathcal{O}_{X_{k}})
\]
is injective (since $j-1\leq m$); therefore $F^{j-1}(\mathcal{M}_{Y,0}^{\infty})\to F^{j-1}(\mathcal{N}_{Y,0}^{\infty})$
is injective, and it is clearly surjective since $\mathcal{M}_{Y,0}^{j-1}\to\mathcal{N}_{Y,0}^{j-1}$
is surjective. Therefore it is an isomorphism; and hence so is $\text{ker}(f:\mathcal{M}_{Y,0}^{j-1}\to\mathcal{M}_{Y,0}^{j})\to\text{ker}(f:\mathcal{N}_{Y,0}^{j-1}\to\mathcal{N}_{Y,0}^{j})$
as claimed. 

Now suppose $m\in\text{ker}(\psi:\mathcal{M}_{Y,0}^{j}\to\mathcal{N}_{Y,0}^{j})$.
Then $vm\in\text{ker}(\psi:\mathcal{M}_{Y,0}^{j-1}\to\mathcal{N}_{Y,0}^{j-1})=0$,
so that $m\in\text{ker}(v)=\text{im}(f)$. If $m=fm'$, then we see
$\psi m'\in\text{ker}(f)$; but by the above paragraph this implies
$m'\in\text{ker}(f)$; therefore $m=0$ and $\psi:\mathcal{M}_{Y,0}^{j}\to\mathcal{N}_{Y,0}^{j}$
is injective. Thus the cokernel of $\psi:\mathcal{M}_{Y}^{j}\to\mathcal{N}_{Y}^{j}$
is $p$-torsion-free. On the other hand, we clearly have $p^{j}\cdot\mathcal{N}_{Y}^{j}\subset\mathcal{M}_{Y}^{j}$;
so that the cokernel of $\psi$ is annihilated by $p^{j}$; therefore
the cokernel is $0$ and we see that $\psi:\mathcal{M}_{Y}^{j}\to\mathcal{N}_{Y}^{j}$
is an isomorphism as claimed. 
\end{proof}
Note that this gives a description of the reduction mod $p$ of the
Hodge filtration (up to $F^{m}$) which makes no reference to a resolution
of singularities. It does depend on an $R$-model for the $\mathcal{D}$-module
$\mathcal{H}_{Y_{\mathbb{C}}}^{r}(\mathcal{O}_{X_{\mathbb{C}}})$,
though any two such models agree after localizing $R$ at an element. 

Now let us further suppose that $Y_{\mathbb{C}}\subset X_{\mathbb{C}}$
is a complete intersection of codimension $r$. By {[}MP2{]}, proposition
7.14, (c.f. also section $9$ of loc. cit.) we have 
\[
F^{m}(\mathcal{H}_{Y_{\mathbb{C}}}^{r}(\mathcal{O}_{X_{\mathbb{C}}}))\subset O^{m}(\mathcal{H}_{Y_{\mathbb{C}}}^{r}(\mathcal{O}_{X_{\mathbb{C}}}))=\text{span}_{\mathcal{O}_{X_{\mathbb{C}}}}\{Q_{1}^{-a_{1}}\cdots Q_{r}^{-a_{r}}|\sum a_{i}\leq m+r\}
\]

In {[}MP2{]} the condition $F^{m}(\mathcal{H}_{Y_{\mathbb{C}}}^{r}(\mathcal{O}_{X_{\mathbb{C}}}))=O^{m}(\mathcal{H}_{Y_{\mathbb{C}}}^{r}(\mathcal{O}_{X_{\mathbb{C}}}))$
is discussed at length; and the point of view developed there shows
that the largest $m$ for which there is equality is a subtle measure
of the singularities of $Y_{\mathbb{C}}$. In fact, equality for any
$m$ already implies serious restrictions on the singularities; indeed,
$F^{0}(\mathcal{H}_{Y_{\mathbb{C}}}^{r}(\mathcal{O}_{X_{\mathbb{C}}}))=O^{0}(\mathcal{H}_{Y_{\mathbb{C}}}^{r}(\mathcal{O}_{X_{\mathbb{C}}}))$
is equivalent to $Y_{\mathbb{C}}$ having du Bois singularities (this
is the first case of theorem F of loc. cit.). 

Now, using the methods of this paper, let us show
\begin{cor}
\label{cor:Canonical-Singularities} Suppose $F^{0}(\mathcal{H}_{Y_{\mathbb{C}}}^{r}(\mathcal{O}_{X_{\mathbb{C}}}))=O^{0}(\mathcal{H}_{Y_{\mathbb{C}}}^{r}(\mathcal{O}_{X_{\mathbb{C}}}))$,
i.e., $Q_{1}^{-1}\cdots Q_{r}^{-1}\in F^{0}(\mathcal{H}_{Y_{\mathbb{C}}}^{r}(\mathcal{O}_{X_{\mathbb{C}}}))$.
Then the log-canonical threshold of $Y_{\mathbb{C}}$ is $r$. 
\end{cor}

Combined with the above, this gives a new proof of the famous fact
that du Bois singularities are canonical, in the l.c.i. case at least
(c.f. \cite{key-57}, \cite{key-58}). It is also a (very) special
case of {[}MP2{]}, conjecture 9.11; of course, it also follows from
theorem C of {[}MP2{]}, using the results of \cite{key-57}. 

To prove this result, we will recall a few facts from positive characteristic
algebraic geometry, following {[}BMS{]}. We return to a perfect field
$k$ of positive characteristic and $X$ smooth over $k$. Let $\mathcal{I}\subset\mathcal{O}_{X}$
be an ideal sheaf. For each $m>0$ we let $\mathcal{I}^{[1/p^{m}]}$
be the minimal ideal sheaf such that $\mathcal{I}\subset(F^{m})^{*}(\mathcal{I}^{[1/p^{m}]})$
(here we are using the isomorphism $(F^{m})^{*}\mathcal{O}_{X}\tilde{\to}\mathcal{O}_{X}$;
for any ideal sheaf $\mathcal{J}$ we have $(F^{m})^{*}\mathcal{J}=\mathcal{J}^{[p^{m}]}$,
the ideal locally generated by $p^{m}$th powers of elements of $\mathcal{J}$).
Then, for each $i>0$ one has inclusions 
\[
(\mathcal{I}^{i})^{[1/p^{m}]}\subset(\mathcal{I}^{i'})^{[1/p^{m'}]}
\]
whenever $i/p^{m}\leq i'/p^{m'}$ and $m\leq m'$ (this is {[}BMS{]},
lemma 2.8). These constructions are connected to $\mathcal{D}$-module
theory as follows: for any ideal sheaf $\mathcal{I}$, we have $\mathcal{D}_{X}^{(m)}\cdot\mathcal{I}=(F^{m+1})^{*}(\mathcal{I}^{[1/p^{m+1}]})$
(c.f. \cite{key-64}, remark 2.6, and \cite{key-62}, lemma 3.1). 

Now, fix a number $c\in\mathbb{R}^{+}$. If $x\to\lceil x\rceil$
denotes the ceiling function, then the previous discussion implies
inclusions 
\[
(\mathcal{I}^{\lceil cp^{m}\rceil})^{[1/p^{m}]}\subset(\mathcal{I}^{\lceil cp^{m+1}\rceil})^{[1/p^{m+1}]}
\]
for all $m$. Thus we have a chain of ideals, which necessarily stabilizes,
and so we can define 
\[
\tau(\mathcal{I}^{c})=(\mathcal{I}^{\lceil cp^{m}\rceil})^{[1/p^{m}]}
\]
for all $m>>0$. These ideals are called generalized test ideals.
There is a deep connection to the theory of multiplier ideals in complex
algebraic geometry, which is due to Hara and Yoshida ({[}HY{]}, theorems
3.4 and 6.8). Suppose we have a complex variety $X_{\mathbb{C}}$,
and flat $R$-model $X_{R}$, and an ideal sheaf $\mathcal{I}_{R}$
which is also flat over $R$. Fix a rational number $c$; we may then
choose a flat model $\mathcal{J}(\mathcal{I}_{R}^{c})$ for the multiplier
ideal $\mathcal{J}(\mathcal{I}_{\mathbb{C}}^{c})$. Then for all perfect
fields $k$ of sufficiently large positive characteristic, we have
\[
\mathcal{J}(\mathcal{I}_{R}^{c})\otimes_{R}k=\tau(\mathcal{I}_{k}^{c})
\]

Finally, we note that since $\mathcal{H}_{Y_{\mathbb{C}}}^{r}(\mathcal{O}_{X_{\mathbb{C}}})$
is a coherent $\mathcal{D}_{X_{\mathbb{C}}}$-module, there exists
some $l>0$ such that $\mathcal{H}_{Y_{\mathbb{C}}}^{r}(\mathcal{O}_{X_{\mathbb{C}}})=\mathcal{D}_{X_{\mathbb{C}}}\cdot(Q_{1}\cdot Q_{r})^{-l}$
. Therefore we may obtain an $R$-model by taking the sheaf 
\[
\mathcal{D}_{X_{R}}^{(0)}\cdot(Q_{1}\cdots Q_{r})^{-l}\subset\mathcal{H}_{Y_{R}}^{r}(\mathcal{O}_{X_{R}})
\]
After base change to $F=\text{Frac}(R)$ this agrees with ${\displaystyle \mathcal{H}^{r}(\int_{\varphi}j'_{\star}\mathcal{O}_{U_{R}})}$;
therefore the two models agree after possibly localizing $R$. In
particular, $\mathcal{M}_{Y}^{-\infty}$ is the $p$-adic completion
of $\mathcal{D}_{X_{W(k)}}^{(0)}\cdot(Q_{1}\cdots Q_{r})^{-l}$. 

Now let us turn to the 
\begin{proof}
(of \corref{Canonical-Singularities}) Let $\mathcal{I}_{\mathbb{C}}=(Q_{1},\dots,Q_{r})$
and let us fix a rational number $0<c<r$. Suppose that the ideal
$\mathcal{J}(\mathcal{I}_{\mathbb{C}}^{c})\subsetneq\mathcal{O}_{X_{\mathbb{C}}}$.
We spread everything out over $R$, and reduce to $k$ of large positive
characteristic. Then the above implies $\tau(\mathcal{I}_{k}^{c})\subsetneq\mathcal{O}_{X_{k}}$. 

Now, recall that we have fixed an $R$-model $\mathcal{D}_{X_{R}}^{(0)}\cdot(Q_{1}\cdots Q_{r})^{-l}$
of $\mathcal{H}_{Y_{\mathbb{C}}}^{r}(\mathcal{O}_{X_{\mathbb{C}}})$.
Then the description of the Hodge filtration in \propref{Hodge-for-local-coh!}
implies that $F^{*}(F_{0}(\mathcal{D}_{X_{R}}^{(0)}\cdot(Q_{1}\cdots Q_{r})^{-l}))\otimes_{R}k)$
is the image of $\mathcal{D}_{X_{R}}^{(0)}\cdot(Q_{1}\cdots Q_{r})^{-l}\otimes_{R}k$
in $\mathcal{H}_{Y_{k}}^{r}(\mathcal{O}_{X_{k}})$; in other words,
the $\mathcal{D}_{X_{k}}^{(0)}$-submodule generated by $(Q_{1}\cdots Q_{r})^{-l}$.
Thus the assumption $F_{0}(\mathcal{H}_{Y_{\mathbb{C}}}^{r}(\mathcal{O}_{X_{\mathbb{C}}}))=O_{0}(\mathcal{H}_{Y_{\mathbb{C}}}^{r}(\mathcal{O}_{X_{\mathbb{C}}}))$
is equivalent to the statement 
\[
(Q_{1}\cdots Q_{r})^{-p}\in\mathcal{D}_{X_{k}}^{(0)}\cdot(Q_{1}\cdots Q_{r})^{-l}
\]
inside $\mathcal{H}_{Y_{k}}^{r}(\mathcal{O}_{X_{k}})$. Since $\mathcal{H}_{Y_{k}}^{r}(\mathcal{O}_{X_{k}})$
is the quotient of $\mathcal{O}_{X_{k}}[(Q_{1}\cdots Q_{r})^{-1}]$
by the submodule generated by $\{Q_{1}^{-1}\cdots\widehat{Q_{i}^{-1}}\cdots Q_{r}^{-1}\}_{i=1}^{r}$,
which is contained in the $\mathcal{D}_{X_{k}}^{(0)}$-submodule generated
by $(Q_{1}\cdots Q_{r})^{-l}$, we see that the assumption actually
implies 
\[
(Q_{1}\cdots Q_{r})^{-p}\in\mathcal{D}_{X_{k}}^{(0)}\cdot(Q_{1}\cdots Q_{r})^{-l}
\]
inside $\mathcal{O}_{X_{k}}[(Q_{1}\cdots Q_{r})^{-1}]$. 

To use this, note that the map $(Q_{1}\cdots Q_{r})^{p}\cdot$ is
a $\mathcal{D}_{X_{k}}^{(0)}$-linear isomorphism on $\mathcal{O}_{X_{k}}[(Q_{1}\cdots Q_{r})^{-1}]$.
Thus we see 
\[
\mathcal{O}_{X_{k}}=\mathcal{D}_{X_{k}}^{(0)}\cdot(Q_{1}\cdots Q_{r})^{p-l}\subset F^{*}(\mathcal{I}^{r(p-l)})^{[1/p]}
\]
so that $(\mathcal{I}^{r(p-l)})^{[1/p]}=\mathcal{O}_{X_{k}}$ which
implies $\tau(\mathcal{I}^{r(1-l/p)})=\mathcal{O}_{X_{k}}$. Taking
$p$ large enough so that $r(1-l/p)>c$, we deduce $\tau(\mathcal{I}_{k}^{c})=\mathcal{O}_{X_{k}}$
(the test ideals form a decreasing filtration, by {[}BMS{]}, proposition
2.11); contradiction. Therefore in fact $\mathcal{J}(\mathcal{I}_{\mathbb{C}}^{c})=\mathcal{O}_{X_{\mathbb{C}}}$
for all $c\in(0,r)$ which is the statement.
\end{proof}
As a corollary of this argument, we have: 
\begin{cor}
Suppose $r=1$ in the previous corollary (so that $\mathcal{I}=(Q)$).
Then, under the assumption that $F^{0}(\mathcal{H}_{Y_{\mathbb{C}}}^{r}(\mathcal{O}_{X_{\mathbb{C}}}))=O^{0}(\mathcal{H}_{Y_{\mathbb{C}}}^{r}(\mathcal{O}_{X_{\mathbb{C}}}))$,
we have that, for all $p>>0$, $\tau(Q^{(1-l/p)})=\mathcal{O}_{X_{k}}$. 
\end{cor}

This says that, after reducing mod $p$, for $p>>0$ the $F$-pure
threshold of $Q$ is $\geq1-l/p$. Recall that $l$ is any integer
for which $Q^{-l}$ generates the $\mathcal{D}_{X_{\mathbb{C}}}$-module
$j_{*}(\mathcal{O}_{U_{\mathbb{C}}})$; thus we may take $l$ to be
the least natural number such that $b_{Q}(-l-t)\neq0$ for all $t\in\mathbb{N}$
(here $b_{Q}$ is the $b$-function for $Q$). In this language, this
result was recently reproved (and generalized) in \cite{key-65},
by completely different techniques.

To finish off this section, we'll spell out how the description of
the Hodge filtration in \propref{Hodge-for-local-coh!} relates to
the condition $F_{i}(\mathcal{H}_{Y_{\mathbb{C}}}^{r}(\mathcal{O}_{X_{\mathbb{C}}}))=O_{i}(\mathcal{H}_{Y_{\mathbb{C}}}^{r}(\mathcal{O}_{X_{\mathbb{C}}}))$
when $Y_{\mathbb{C}}$ is a hypersurface inside $X_{\mathbb{C}}$;
$Y_{\mathbb{C}}=Z(Q)$. In this case, we get an intriguing description
in terms of the behavior of $\mathcal{H}_{Y}^{1}(\mathcal{O}_{X})$
in mixed characteristic: 
\begin{cor}
We have $F_{i}(\mathcal{H}_{Y}(\mathcal{O}_{X}))=O_{i}(\mathcal{H}_{Y}^{1}(\mathcal{O}_{X}))$
iff $p^{i}Q^{-(i+1)p}\in\mathcal{D}_{X_{W_{i+1}(k)}}^{(0)}\cdot Q^{-l}$
inside $\mathcal{H}_{Y_{W_{i+1}(k)}}^{1}(\mathcal{O}_{X_{W_{i+1}(k)}})$
for $p>>0$. 
\end{cor}

\begin{proof}
Applying the condition of \propref{Hodge-for-local-coh!}, we see
that $F_{i}(\mathcal{H}_{Y}(\mathcal{O}_{X}))=O_{i}(\mathcal{H}_{Y}^{1}(\mathcal{O}_{X}))$
iff, for all $p>>0$, there exists some $g\in\widehat{\mathcal{H}_{\mathfrak{Y}}^{1}(\mathcal{O}_{\mathfrak{X}})}$,
with $p^{i}g\in\widehat{\mathcal{D}}_{\mathfrak{X}}^{(0)}\cdot Q^{-l}$
whose image, in $\mathcal{H}_{Y_{k}}^{1}(\mathcal{O}_{X_{k}})$ is
$Q^{-(i+1)p}$. This holds iff there is some $g_{1}\in\widehat{\mathcal{H}_{\mathfrak{Y}}^{1}(\mathcal{O}_{\mathfrak{X}})}$
so that 
\[
Q^{-(i+1)p}=g+pg_{1}
\]
inside $\widehat{\mathcal{H}_{\mathfrak{Y}}^{1}(\mathcal{O}_{\mathfrak{X}})}$.
But this is equivalent to 
\[
p^{i}Q^{-(i+1)p}=p^{i}g+p^{i+1}g_{1}=\Phi\cdot Q^{-l}+p^{i+1}g_{1}
\]
for some $\Phi\in\widehat{\mathcal{D}}_{\mathfrak{X}}^{(0)}$. This,
in turn, is a restatement of the corollary.
\end{proof}

\section{\label{sec:Appendix:-an-Inectivity}Appendix: an Injectivity Result}

In this appendix we give a proof of the following technical result
used in \lemref{Hodge-filt-on-log}:
\begin{lem}
The natural map $\text{(\ensuremath{{\displaystyle j_{\star}\mathcal{O}_{\mathfrak{U}}}}})^{-\infty}|_{\mathfrak{V}}\to\widehat{(\mathcal{O}_{\mathfrak{V}}[x_{1}^{-1}\cdots x_{j}^{-1}])}$
(where $\widehat{}$ denotes $p$-adic completion) is injective. 
\end{lem}

Recall that $j_{\star}(\mathcal{O}_{\mathfrak{U}})$ was defined as
the $\widehat{\mathcal{D}}_{\mathfrak{X}}^{(0)}$-module locally generated
by $x_{1}^{-1}\cdots x_{j}^{-1}$, where $x_{1}\cdots x_{j}$ is a
local equation for the divisor $\mathfrak{D}\subset\mathfrak{X}$. 
\begin{proof}
Let $\mathfrak{V}$ be an open affine. On $\mathfrak{V}$, the map
in question is the $p$-adic completion of the inclusion ${\displaystyle {\displaystyle (j_{\star}\mathcal{O}_{\mathfrak{V}})}}^{\text{fin}}\to\mathcal{O}_{\mathfrak{V}}[x_{1}^{-1}\cdots x_{j}^{-1}]$,
where ${\displaystyle {\displaystyle (j_{\star}\mathcal{O}_{\mathfrak{V}})}}^{\text{fin}}$
is the $D_{\mathfrak{V}}^{(0)}$-submodule of $j_{*}(\mathcal{O}_{\mathfrak{V}})$
generated by $x_{1}^{-1}\cdots x_{j}^{-1}$. This is a map of $p$-torsion-free
sheaves, let $\mathcal{C}$ denote its cokernel. Then the kernel of
the completion is given by 
\[
\lim_{\leftarrow}\mathcal{C}[p^{n}]
\]
where $\mathcal{C}[p^{n}]=\{m\in\mathcal{C}|p^{n}m=0\}$, and the
maps in the inverse system are multiplication by $p$. 

Now, both ${\displaystyle {\displaystyle (j_{\star}\mathcal{O}_{\mathfrak{V}})}}^{\text{fin}}$
and $\mathcal{O}_{\mathfrak{V}}[x_{1}^{-1}\cdots x_{j}^{-1}]$ are
filtered by the Hodge filtration; on ${\displaystyle {\displaystyle (j_{\star}\mathcal{O}_{\mathfrak{V}})}}^{\text{fin}}$
it is given by $F^{l}(\mathcal{D}_{\mathfrak{V}}^{(0)})\cdot(x_{1}^{-1}\cdots x_{r}^{-1})$,
which is precisely the span over $\mathcal{O}_{\mathfrak{V}}$ of
terms of the form $I!\cdot x_{1}^{-i_{1}-1}\cdots x_{j}^{-i_{j}-1}=\partial^{I}x_{1}^{-1}\cdots x_{j}^{-1}$
for $|I|\leq l$, here we have denoted $I!=i_{1}!\cdots i_{j}!$.
The Hodge filtration $F^{l}(\mathcal{O}_{\mathfrak{V}}[x_{1}^{-1}\cdots x_{j}^{-1}])$
is defined to be the span over $\mathcal{O}_{\mathfrak{V}}$ of terms
of the form $x_{1}^{-i_{1}-1}\cdots x_{j}^{-i_{j}-1}$ for $|I|\leq l$.
From this description it follows that, in both cases, all of the terms
$F^{i}$ and $F^{i}/F^{i-1}$ are $p$-torsion-free; and the morphism
is strict with respect the the filtrations; i.e., 
\[
F^{i}(\mathcal{O}_{\mathfrak{V}}[x_{1}^{-1}\cdots x_{j}^{-1}])\cap{\displaystyle {\displaystyle (j_{\star}\mathcal{O}_{\mathfrak{V}})}}^{\text{fin}}=F^{i}({\displaystyle {\displaystyle (j_{\star}\mathcal{O}_{\mathfrak{V}})}}^{\text{fin}})
\]

Now we consider the inclusion $\mathcal{R}({\displaystyle {\displaystyle (j_{\star}\mathcal{O}_{\mathfrak{V}})}}^{\text{fin}})\to\mathcal{R}(\mathcal{O}_{\mathfrak{V}}[x_{1}^{-1}\cdots x_{j}^{-1}])$
(where $\mathcal{R}$ stands for the Rees functor with respect to
the Hodge filtrations on both sides). The strictness of the map implies
\[
\text{coker}(\mathcal{R}({\displaystyle {\displaystyle (j_{\star}\mathcal{O}_{\mathfrak{V}})}}^{\text{fin}})\to\mathcal{R}(\mathcal{O}_{\mathfrak{V}}[x_{1}^{-1}\cdots x_{j}^{-1}]))=\mathcal{R}(\mathcal{C})
\]
 We now shall show that the $p$-adic completion\footnote{In this appendix only, we use the completion of the \emph{entire}
Rees module, NOT the graded completion of the rest of the paper; similarly,
the product is the product in the category of all modules, not the
category of graded modules} of this map is injective. The natural map 
\[
\mathcal{R}({\displaystyle {\displaystyle (j_{\star}\mathcal{O}_{\mathfrak{V}})}}^{\text{fin}})=\bigoplus_{i=0}^{\infty}F^{i}({\displaystyle (j_{\star}\mathcal{O}_{\mathfrak{V}})^{\text{fin}}}\to\prod_{i=0}^{\infty}F^{i}({\displaystyle (j_{\star}\mathcal{O}_{\mathfrak{V}})^{\text{fin}}}
\]
is injective, and the cokernel is easily seen to be $p$-torsion-free;
therefore we obtain an injection 
\[
\widehat{\mathcal{R}({\displaystyle {\displaystyle (j_{\star}\mathcal{O}_{\mathfrak{V}})}}^{\text{fin}})}\to\widehat{\prod_{i=0}^{\infty}F^{i}({\displaystyle (j_{\star}\mathcal{O}_{\mathfrak{V}})^{\text{fin}}}}
\]
and the analogous statement holds for $\mathcal{R}(\mathcal{O}_{\mathfrak{V}}[x_{1}^{-1}\cdots x_{j}^{-1}])$.
Further, one has an isomorphism 
\[
\widehat{\prod_{i=0}^{\infty}F^{i}({\displaystyle (j_{\star}\mathcal{O}_{\mathfrak{V}})^{\text{fin}}}}\tilde{=}\prod_{i=0}^{\infty}(\widehat{F^{i}({\displaystyle (j_{\star}\mathcal{O}_{\mathfrak{V}})^{\text{fin}}})}=\prod_{i=0}^{\infty}F^{i}({\displaystyle (j_{\star}\mathcal{O}_{\mathfrak{V}})^{\text{fin}}}
\]
where the last equality is because $F^{i}({\displaystyle (j_{\star}\mathcal{O}_{\mathfrak{V}})^{\text{fin}}}$
is a coherent $\mathcal{O}_{\mathfrak{V}}$-module and therefore $p$-adically
complete; similarly 
\[
\widehat{\prod_{i=0}^{\infty}F^{i}(\mathcal{O}_{\mathfrak{V}}[x_{1}^{-1}\cdots x_{j}^{-1}]}\tilde{=}\prod_{i=0}^{\infty}(\widehat{F^{i}(\mathcal{O}_{\mathfrak{V}}[x_{1}^{-1}\cdots x_{j}^{-1}])}=\prod_{i=0}^{\infty}F^{i}(\mathcal{O}_{\mathfrak{V}}[x_{1}^{-1}\cdots x_{j}^{-1}]
\]
So, since each $F^{i}({\displaystyle (j_{\star}\mathcal{O}_{\mathfrak{V}})^{\text{fin}}}\to F^{i}(\mathcal{O}_{\mathfrak{V}}[x_{1}^{-1}\cdots x_{j}^{-1}]$
is injective, we obtain an injection 
\[
\widehat{\mathcal{R}({\displaystyle {\displaystyle (j_{\star}\mathcal{O}_{\mathfrak{V}})}}^{\text{fin}})}\to\widehat{\mathcal{R}(\mathcal{O}_{\mathfrak{V}}[x_{1}^{-1}\cdots x_{j}^{-1}])}
\]
This means that ${\displaystyle \lim_{\leftarrow}\mathcal{R}(\mathcal{C})[p^{n}]}=0$.
Let $f$ denote the parameter in the Rees ring. Then, for each $n$
we have a short exact sequence
\[
\mathcal{R}(\mathcal{C})[p^{n}]\xrightarrow{f-1}\mathcal{R}(\mathcal{C})[p^{n}]\to C[p^{n}]
\]
Since ${\displaystyle \lim_{\leftarrow}\mathcal{R}(\mathcal{C})[p^{n}]}=0$,
to prove ${\displaystyle \lim_{\leftarrow}\mathcal{C}[p^{n}]=0}$
we must show that $f-1$ acts injectively on ${\displaystyle \text{R}^{1}\lim_{\leftarrow}\mathcal{R}(\mathcal{C})[p^{n}]}$.
Recall that this module is the cokernel of 
\[
\eta:\prod_{n=1}^{\infty}\mathcal{R}(C)[p^{n}]\to\prod_{n=1}^{\infty}\mathcal{R}(C)[p^{n}]
\]
where $\eta(c_{1},c_{2},c_{3},\dots)=(c_{1}-pc_{2},c_{2}-pc_{3},\dots)$.
Now, since each $\mathcal{R}(C)[p^{n}]$ is graded, we may define
a homogenous element of degree $i$ in ${\displaystyle \prod_{n=1}^{\infty}\mathcal{R}(C)[p^{n}]}$
to be an element $(c_{1},c_{2},\dots)$ such that each $c_{j}$ has
degree $i$. Any element of $d\in{\displaystyle \prod_{n=1}^{\infty}\mathcal{R}(C)[p^{n}]}$
has a unique representation of the form ${\displaystyle \sum_{i=0}^{\infty}d_{i}}$
where $d_{i}$ is homogenous of degree $i$ (this follows by looking
at the decomposition by grading of each component). Since the map
$\eta$ preserves the set of homogenous elements of degree $i$, we
have ${\displaystyle \eta(\sum_{i=0}^{\infty}d_{i})=\sum_{i=0}^{\infty}\eta(d_{i})}$. 

Suppose that $(f-1)d=\eta(d')$. Write ${\displaystyle d=\sum_{i=j}^{\infty}d_{i}}$
where $d_{j}\neq0$. Then 
\[
(f-1){\displaystyle \sum_{i=j}^{\infty}d_{j}}=-d_{j}+\sum_{i=j+1}^{\infty}(fd_{i-1}-d_{i})=\sum_{i=0}^{\infty}\eta(d_{i}')
\]
So we obtain $d_{j}=-\eta(d_{j}')$, and $d_{i}=fd_{i-1}+\eta(d_{i}')$
for all $i>j$, which immediately gives $d_{i}\in\text{image}(\eta)$
for all $i$; so $d\in\text{image }(\eta)$ and $f-1$ acts injectively
on $\text{coker}(\eta)$ as required. 
\end{proof}

The University of Illinois at Urbana-Champaign, csdodd2@illinois.edu
\end{document}